\RequirePackage{fix-cm}
\documentclass[twocolumn]{svjour3}          
\smartqed  
\usepackage{graphicx}
\usepackage{amsmath,amssymb}
\usepackage{ifthen}
\usepackage{psfrag}
\usepackage{bm}
\usepackage{mathptmx}      
\usepackage{color}
\newcommand{\dual}[2]{\langle#1\hspace*{.5mm},#2\rangle}
\newcommand{\abs}[1]{\left\vert #1 \right\vert}
\newcommand{\norm}[3][]{#1\|#2#1\|_{#3}}
\newcommand{\snorm}[2]{|#1|_{#2}}
\newcommand{\R}{\ensuremath{\mathbb{R}}}
\newcommand{\N}{\ensuremath{\mathbb{N}}}
\newcommand{\wilde}{\widetilde}
\newcommand{\trace}{\gamma_0}
\newcommand{\supp}{\mathrm{supp}}
\newcommand{\n}{\mathbf{n}}
\newcommand{\dn}{\gamma_1}

\newcommand{\fsol}{\mathrm{G}}
\newcommand{\slo}{V}
\newcommand{\slp}{\wilde\slo}
\newcommand{\dlo}{K}
\newcommand{\dlp}{\wilde\dlo}
\newcommand{\adlo}{\dlo'}
\newcommand{\hyp}{W}

\newcommand{\mesh}{\mathcal{T}}
\newcommand{\el}{T}
\newcommand{\diam}{\mathrm{diam}}
\newcommand{\NN}{\mathcal{N}}
\newcommand{\edges}{\mathcal{E}}
\newcommand{\edge}{e}
\newcommand{\Pp}{\mathcal{P}}
\newcommand{\Sp}{\mathcal{S}}
\newcommand{\MM}{\mathcal{M}}

\newcommand{\loc}{\textrm{loc}}

\newcommand{\nablag}{\nabla_\Gamma}
\newcommand{\dist}{\textrm{dist}}

\newcommand\refine{{\tt{refine}}}
\newcommand\gen{{\rm gen}}

\newcommand\refel{{\rm ref}}

\newtheorem{assumption}[theorem]{Assumption}
\newtheorem{algorithm}[theorem]{Algorithm}

\newtheorem{mylemma}[theorem]{Lemma}
\newtheorem{myproposition}[theorem]{Proposition}
\newtheorem{mycorollary}[theorem]{Corollary}
\newtheorem{mydefinition}[theorem]{Definition}

\newcounter{constantsnumber}
\newcommand{\setc}[1]{
  \ifthenelse{\equal{#1}{cont}}{C_{\rm cnt}}{ 
  \ifthenelse{\equal{#1}{nvb}}{C_{\rm nvb}}{
  \ifthenelse{\equal{#1}{stab}}{C_{\rm stab}}{
  \ifthenelse{\equal{#1}{loc}}{C_{\rm loc}}{
  \ifthenelse{\equal{#1}{inv}}{C_{\rm inv}}{
  \ifthenelse{\equal{#1}{approx}}{C_{\rm apx}}{
  \ifthenelse{\equal{#1}{garding}}{C_{\rm g}}{
  \ifthenelse{\equal{#1}{ell}}{C_{\rm ell}}{
  \ifthenelse{\equal{#1}{cnt}}{C_{\rm cont}}{
  \ifthenelse{\equal{#1}{faermann:assumption:stab}}{C_{\alpha}}{
  \ifthenelse{\equal{#1}{rel}}{C_{\rm rel}}{
  \ifthenelse{\equal{#1}{eff}}{C_{\rm eff}}{
  \ifthenelse{\equal{#1}{sat}}{C_{\rm sat}}{ 
  \ifthenelse{\equal{#1}{sata}}{C_{\rm sata}}{ 
  \ifthenelse{\equal{#1}{opt:estred:aux}}{C_{\rm est}}
  { 
   \refstepcounter{constantsnumber}
   \label{const#1}C_{\theconstantsnumber}}}}}}}}}}}}}}}}}

\def\c#1{
  \ifthenelse{\equal{#1}{cont}}{C_{\rm cnt}}{
  \ifthenelse{\equal{#1}{nvb}}{C_{\rm nvb}}{
  \ifthenelse{\equal{#1}{stab}}{C_{\rm stab}}{
  \ifthenelse{\equal{#1}{loc}}{C_{\rm loc}}{
  \ifthenelse{\equal{#1}{inv}}{C_{\rm inv}}{
  \ifthenelse{\equal{#1}{approx}}{C_{\rm apx}}{
  \ifthenelse{\equal{#1}{ell}}{C_{\rm ell}}{
  \ifthenelse{\equal{#1}{cnt}}{C_{\rm cont}}{
  \ifthenelse{\equal{#1}{garding}}{C_{\rm g}}{
  \ifthenelse{\equal{#1}{faermann:assumption:stab}}{C_{\alpha}}{
  \ifthenelse{\equal{#1}{rel}}{C_{\rm rel}}{
  \ifthenelse{\equal{#1}{eff}}{C_{\rm eff}}{
  \ifthenelse{\equal{#1}{sat}}{C_{\rm sat}}{
  \ifthenelse{\equal{#1}{sata}}{C_{\rm sata}}{ 
  \ifthenelse{\equal{#1}{opt:estred:aux}}{C_{\rm est}}
  { 
    C_{\ref{const#1}}}}}}}}}}}}}}}}}}

\newcommand\ZZ{\mathcal Z}
\newcommand\YY{\mathcal Y}
\newcommand\TT{\mathcal T}
\newcommand\HH{\mathcal H}
\newcommand\OO{\mathcal O}
\newcommand\UU{\mathcal U}
\newcommand\VV{\mathcal V}

\newcommand{\form}[2]{b(#1\,,\,#2)}
\newcommand\XX{\mathcal{X}}
\newcommand\eps{\varepsilon}
\newcommand\A{\mathbb{A}}
\newcommand{\set}[2]{\{#1\,:\,#2\}}
\newcommand\RR[2]{\mathcal{R}(#1,#2)}
\newcommand{\est}[2]{\ifthenelse{\equal{#2}{}}{\eta_{#1}}{\eta_{#1}(#2)}}
\newcommand{\estd}[2]{\ifthenelse{\equal{#2}{}}{\widetilde{\eta}_{#1}}{\widetilde{\eta}_{#1}(#2)}}
\newcommand{\estt}[2]{\ifthenelse{\equal{#2}{}}{\widehat{\eta}_{#1}}{\widetilde{\eta}_{#1}(#2)}}
\newcommand{\muest}[2]{\ifthenelse{\equal{#2}{}}{\mu_{#1}}{\mu_{#1}(#2)}}
\newcommand{\rhoest}[2]{\ifthenelse{\equal{#2}{}}{\rho_{#1}}{\rho_{#1}(#2)}}
\newcommand{\data}[2]{\ifthenelse{\equal{#2}{}}{{\rm data}_{#1}}{{\rm data}_{#1}(#2)}}
\newcommand{\datad}[2]{\ifthenelse{\equal{#2}{}}{\widetilde{{\rm data}}_{#1}}{{\rm data}_{#1}(#2)}}
\newcommand\Ud{\widetilde{U}}

\newcommand\Phid{\widetilde{\Phi}}
\newcommand\phid{\widetilde{\phi}}

\newcommand\hhhalfweak{\mu}
\newcommand\hhhalfweaktilde{\widetilde\mu}
\newcommand\hhhalfhyp{\mu}
\newcommand\hhhalfhyptilde{\widetilde\mu}
\newcommand\estres{\eta}
\newcommand\massmat{\mathbf{M}}
\newcommand\massmathat{\widehat{\mathbf{M}}}
\newcommand\dimPp{M}
\newcommand\idmat{\mathbf{I}}
\newcommand\localmesh{\widehat\mesh_\el}
\newcommand\localgradinPP{\widehat{\mathbf{G}}}
\newcommand\localgradproj{\widehat{\mathbf{D}}}
\newcommand\elref{\el_{\rm ref}}
\newcommand\lagrangemat{\mathbf{L}}
\begin{document}

\title{Adaptive Boundary Element Methods
  \thanks{The research of MF, TF, and DP is supported by the Austrian
    Science Fund (FWF) through the research project \emph{Adaptive boundary
    element method}, funded under grant P21732, see
    \texttt{http://www.asc.tuwien.ac.at/abem/}.
    In addition, the authors MF and DP acknowledge support through
    the FWF doctoral program \emph{Dissipation and dispersion in nonlinear
    PDEs}, funded under grant W1245, see \texttt{http://npde.tuwien.ac.at/}.
    The research of NH is supported by CONICYT projects Anillo ACT1118 (ANANUM)
    and \emph{Non-conforming boundary elements and applications}, funded under grant
    Fondecyt 1110324.
    The research of MK is supported by the CONICYT project 
    \emph{Efficient adaptive strategies for nonconforming boundary element methods},
    funded under grant Fondecyt 3140614.
  }
}
\subtitle{A~posteriori error estimators, adaptivity, convergence, and implementation}

\titlerunning{Adaptive Boundary Element Methods}

\author{Michael Feischl \and
        Thomas F\"uhrer \and
	Norbert Heuer \and\\
        Michael Karkulik \and
	Dirk Praetorius
}

\authorrunning{M.~Feischl, T.~F\"uhrer, N.~Heuer, M.~Karkulik, D.~Praetorius}

\institute{M.~Feischl, T.~F\"uhrer, and D.~Praetorius \at
  Institute for Analysis and Scientific Computing \\
  Vienna University of Technology \\
  Wiedner Hauptstrasse 8-10, 1040 Wien, Austria \\
  \email{ \{michael.feischl,thomas.fuehrer,\\dirk.praetorius\}@tuwien.ac.at} \\
           \and
  N.~Heuer and M.~Karkulik \at
  Facultad de Matem\'aticas \\
  Pontificia Universidad Cat\'olica de Chile \\
  Avenida Vicu\~na Mackenna 4860, Santiago, Chile \\
  \email{ \{nheuer,mkarkulik\}@mat.puc.cl}
}

\date{Received: date / Accepted: date}
\maketitle
\begin{abstract}
  This paper reviews the state of the art and discusses very recent
  mathematical developments in the field of adaptive boundary element methods. 
  This includes an over\-view of available a~posteriori error estimates as well as a 
  state-of-the-art formulation of convergence and quasi-opti\-mality of 
  adaptive mesh-refining algorithms.
\end{abstract}

\hfill
\begin{center}
\end{center}

\setcounter{tocdepth}{1}
\tableofcontents
\section{Introduction}\label{section:introduction}
Many practically relevant PDEs\footnote{partial differential equation (PDE)} 
on bounded or unbounded domains $\Omega\subset\R^d$ can be equivalently formulated 
as integral equations on the $(d-1)$-dimensional boundary 
$\Gamma=\partial\Omega$.
This reformulation is then discretized and solved numerically by 
BEM\footnote{boundary element method (BEM)}. Striking advantages of BEM over 
FEM\footnote{finite element method (FEM)}
rely on the dimension reduction, the natural treatment of unbounded domains,
as well as a potentially high rate of convergence with respect to both, the
natural energy norm as well as the pointwise error. On the other hand, high
convergence rates are only achieved if the (given) data as well as the 
(unknown) exact solution are sufficiently smooth 
or if the possible singularities are appropriately resolved. In practice, one thus observes
a huge gap between the theoretically possible optimal rate and the empirical 
convergence behavior, if the meshes are refined uniformly. The remedy is to
use appropriately graded meshes which resolve the possible singularities of
data and exact solution. To this end, a~posteriori error estimation and 
related adaptive mesh-refine\-ment have themselves proven to be important tools 
for scientific computing, cf.~\cite{ao00,v13}. 
First, they allow to monitor the actual error and to stop the computation if the computed solution is
accurate enough. Second, they may also drive the problem-adapted discretization and thus the
appropriate resolution of the possible singularities. While the convergence and quasi-optimality of
AFEM\footnote{adaptive finite element method (AFEM)} has been mathematically 
analyzed within the last decade~\cite{bdd04,ckns,doerfler,mns00,steve07},
analogous results for ABEM\footnote{adaptive 
boundary element method (ABEM)}~\cite{affkp13-A,ffkmp13,ffkmp13-A,fkmp13,gantumur} have only
been achieved very recently, see also~\cite{dhs07,g08,ths07,ku12} for adaptive wavelet-based BEM.

\subsection{Galerkin BEM and C\'ea lemma}\label{section:cea}
Throughout, our main focus is on Galerkin BEM. Here, the mathematical frame
reads as follows: Let $\XX$ be a real Hilbert space with norm $\norm\cdot\XX$, 
which will be an appropriate Sobolev space in the applications in mind. Let 
$b:\XX\times\XX\to\R$ be a continuous and elliptic bilinear form, i.e.,
there are constants $C_{\rm cont},C_{\rm ell}>0$ such that
\begin{align}\label{intro:continuous}
 b(v,w) \le C_{\rm cont}\,\norm{v}{\XX}\norm{w}{\XX}
 \quad\text{for all }v,w\in\XX
\end{align}
and 
\begin{align}\label{intro:elliptic}
 b(v,v) \ge C_{\rm ell}\,\norm{v}{\XX}^2
 \quad\text{for all }v\in\XX.
\end{align}
Given a linear and continuous functional $F:\XX\to\R$,
the so-called weak formulation (or variational formulation) 
of the BIE\footnote{boundary integral equation (BIE)} reads: Find the 
exact solution $u\in\XX$ of 
\begin{align}\label{intro:weakform}
 b(u,v) = F(v)
 \quad\text{for all }v\in\XX.
\end{align}
Based on a triangulation $\TT$ of the underlying spatial domain, let 
$\XX_\TT\subset\XX$ be a finite-dimen\-sional subspace. The Galerkin BEM discretization
reads: Find $U\in\XX_\mesh$ such that
\begin{align}\label{intro:galerkin}
 b(U,V) = F(V)
 \quad\text{for all }V\in\XX_\mesh.
\end{align}
For both, the continuous as well as the discrete formulations~\eqref{intro:weakform}
and~\eqref{intro:galerkin}, the Lax-Milgram lemma applies and proves
the existence and uniqueness of $u\in\XX$ resp. $U\in\XX_\mesh$. Moreover,
a direct computation with the Galerkin orthogonality
\begin{align}\label{intro:orthogonality}
 b(u-U,V)=0
 \quad\text{for all }V\in\XX_\mesh,
\end{align}
provides the C\'ea lemma
\begin{align}\label{intro:cea}
 \frac{C_{\rm ell}}{C_{\rm cont}}\,\norm{u-U}{\XX}
 \le \min_{V\in\XX_\mesh}
 \norm{u-V}{\XX}
 \le \norm{u-U}{\XX}
\end{align}
i.e., the \emph{computable} Galerkin solution $U\in\XX_\mesh$ is
a quasi-best approximation of $u$ among all functions $V$ in the discrete space $\XX_\mesh$.

\subsection{Adaptive algorithm}\label{section:algorithm}
As a consequence of the C\'ea lemma~\eqref{intro:cea}, a natural question is how to choose
the discrete space $\XX_\mesh$ (resp. the mesh $\mesh$).
Ideally, one should choose the discrete space $\XX_\mesh$
such that the best approximation error in~\eqref{intro:cea}
is minimal with respect to the number of degrees of freedom.
Usually the necessary a priori knowledge is not available (even if the generic
singularities appear to be known), such that it is infeasible to address this question.
Another possibility is to choose a sequence of meshes such that
the best approximation error shows an ``optimal decay'' with increasing dimension
of $\XX_\mesh$.
Usually this question is empirically addressed by adaptive algorithms which start from
an initial mesh $\TT_0$ and generate a sequence of (locally) refined meshes $\TT_\ell$ for 
$\ell\in\N_0$ by iterating the loop
\begin{align}\label{intro:algorithm}
 \boxed{\texttt{solve}}
 \,\,\,\,\to\,\,\,\,
 \boxed{\texttt{estimate}}
 \,\,\,\,\to\,\,\,\,
 \boxed{\texttt{mark}}
 \,\,\,\,\to\,\,\,\,
 \boxed{\texttt{refine}}.
\end{align}
It provides a sequence of Galerkin solutions\- $U_\ell\in\XX_\ell := \XX_{\mesh_\ell}$ with
nested discrete spaces $\XX_\ell\subset\XX_{\ell+1}\subset\XX$ for all
$\ell\ge0$.
Adaptive algorithms thus work with a sequence of meshes and need to solve in every step.
Yet, it can be observed in model problems that they outperform algorithms
which uniformly refine a coarse mesh up to a given number of degrees of freedom
and finally solve only once. This superiority appears in terms
of memory versus error as well as time consumption versus error, cf.~\cite{hilbert,afgkmp12}.

The module $\boxed{\texttt{solve}}$ consists of the direct or iterative
solution of the linear system corresponding to~\eqref{intro:galerkin} to
compute the (approximate) Galerkin solution $U_\ell\in\XX_\ell$. Mathematical
questions arise from the fact that, first, BEM matrices are 
densely populated (i.e., the number of non-zero entries is roughly equivalent to the overall number of entries) and hence 
have to be 
treated by matrix compression techniques like 
FMM\footnote{fast multipole method (FMM)}~\cite{fmm,osw06},
$\HH$-matrices\footnote{hierarchical matrices ($\HH$-matrices)}~\cite{hackbusch99,hackbusch09},
panel clustering~\cite{hn89},
or ACA\footnote{adaptive cross approximation (ACA)}~\cite{bebendorf00,br03,bg06},
see also the monograph~\cite{rs07} on this subject. In particular, this 
prevents the use of direct solvers for problems of practical interest. 
Second, the condition number of BEM matrices grows if the mesh is 
refined, i.e., one needs cheap and effective preconditioners 
which build on the hierarchical structure of the nested discrete spaces.
Finally, the right-hand side $F$ in~\eqref{intro:weakform} often involves
evaluations of integral operators applied to the given data. Then, the
computation of the right-hand side in~\eqref{intro:galerkin} can hardly
be done analytically. Instead, appropriate and reliable data approximation 
and/or quadrature has to be employed, and this additional consistency error 
has to be controlled.

The module $\boxed{\texttt{estimate}}$ comprises the computation of a 
numerically computable a~posteriori error estimator 
\begin{align}\label{intro:eta}
 \eta_\ell = \Big(\sum_{T\in\TT_\ell}\eta_\ell(T)^2\Big)^{1/2}
\end{align}
whose local contributions $\eta_\ell(T)$ measure ---at least heuris\-ti\-cally---
the Galerkin error $u-U_\ell$ on an element $T$ of the current 
triangulation $\TT_\ell$.
For this purpose, different types of error estimators have been proposed 
in the literature which range from simple two-grid error estimators over
residual-based strategies to estimators which build on the BEM inherent 
Calder\'on system.

The module $\boxed{\texttt{mark}}$ uses the local refinement indicators 
$\eta_\ell(T)$ and selects certain elements for refinement, where refinement
can either be a geometric bisection of the element (so-called $h$-refinement)
or the increase of the local approximation order (so-called $p$-refinement). 

Finally, the module $\boxed{\texttt{refine}}$ uses the prior information to 
generate a new mesh $\TT_{\ell+1}$ as well as a related enriched space 
$\XX_{\ell+1}\supset\XX_\ell$.
For now, we denote this by\linebreak $\mesh_{\ell+1}\in\refine(\mesh_\ell)$.
In the later sections, this will be specified further.
Usually, the numerical analysis requires certain
care for the (otherwise simple) operation $\refine(\cdot)$ to ensure that, e.g., hanging nodes 
are avoided, the quotient of the diameters of neighboring elements does not deteriorate and neither do the elements' angles. In particular,
this leads to additional refinement of non-marked elements. As over-refinement
might affect observed convergence rates with respect to the degrees of
freedom, this requires mathematical care if it comes to the proof of optimal
convergence rates.

The design of an adaptive algorithm usually consists in making appropriate choices for the
different parts of the adaptive loop~\eqref{intro:algorithm}. For the $h$-version, which is the
focus of this work (although we will also briefly discuss
$hp$-versions,\linebreak where a mixture of $h$- and $p$-refinement takes place),
it is common to write the loop~\eqref{intro:algorithm} in pseudo-code in the following form:
\begin{algorithm}[Adaptive mesh refinement]\label{opt:algorithm}\ \linebreak
    \textsc{Input}: initial mesh $\mesh_0$ and adaptivity parameter $0<\theta\leq 1$.\\
  \textsc{Output}: sequence of solutions $(U_\ell)_{\ell\in\N_0}$, sequence of estimators $(\est{\ell}{})_{\ell\in\N_0}$, and sequence of meshes $(\mesh_\ell)_{\ell\in\N_0}$.\\
  \textsc{Iteration}: For all $\ell=0,1,2,3,\ldots$ do {\rm (i)}--{\rm (iv)}.
  \begin{itemize}
    \item[\rm(i)] Compute solution $U_\ell$ of~\eqref{intro:galerkin}.
    \item[\rm(ii)] Compute error indicators $\est{\ell}{\el}$ for all elements $\el\in\mesh_\ell$.
    \item[\rm(iii)] Find a set of (minimal) cardinality $\MM_\ell\subseteq\mesh_\ell$ such that
    \begin{align}\label{opt:bulkchasing}
    \theta\est{\ell}{}^2\leq \sum_{\el\in\MM_\ell}\est{\ell}{\el}^2.
    \end{align}
    \item[\rm(iv)] Refine at least the marked elements $T\in\MM_\ell$ to obtain the new mesh 
    $\mesh_{\ell+1}\in\refine(\mesh_\ell)$.
    \end{itemize}
\end{algorithm}

\begin{figure}[t]
\psfrag{Symms integral equation on slit}{}
\psfrag{error}[c][c]{\footnotesize error in energy norm}
\psfrag{number of elements}[c][c]{\footnotesize number of elements}
\psfrag{O12}{\scriptsize$\OO(N^{-1/2})$}
\psfrag{O32}{\hspace*{-5mm}\scriptsize$\OO(N^{-3/2})$}
\psfrag{O52}{\hspace*{-3mm}\scriptsize$\OO(N^{-5/2})$}
\psfrag{errorP0}{\scriptsize$p=0$, adaptive}
\psfrag{errorP1}{\scriptsize$p=1$, adaptive}
\psfrag{errorP0-unif}{\scriptsize$p=0$, uniform}
\psfrag{errorP1-unif}{\scriptsize$p=1$, uniform}
\centering
\includegraphics[width=0.46\textwidth]{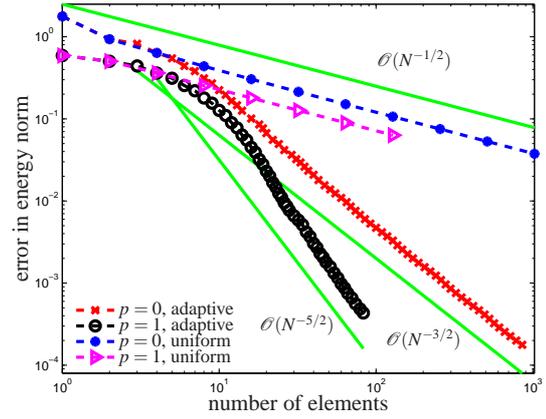}
\caption{BEM error for piecewise constants ($p=0$) and piecewise linears ($p=1$) on uniform and adaptive meshes for example~\eqref{intro:example}.}
\label{intro:2d:error}
\end{figure}

\begin{figure}[t]
\psfrag{Symms integral equation on slit p=1}{}
\psfrag{error, estimators}[c][c]{\footnotesize error resp.\ error estimator}
\psfrag{number of elements}[c][c]{\footnotesize number of elements}
\psfrag{O12}{\scriptsize$\OO(N^{-1/2})$}
\psfrag{O32}{\hspace*{-3mm}\scriptsize$\OO(N^{-3/2})$}
\psfrag{error}{\scriptsize error, $p=0$, adaptive}
\psfrag{est}{\scriptsize error, $p=0$, uniform}
\psfrag{error-unif}{\scriptsize estimator, $p=0$, adaptive}
\psfrag{est-unif}{\scriptsize estimator, $p=0$, uniform}
\centering
\includegraphics[width=0.46\textwidth]{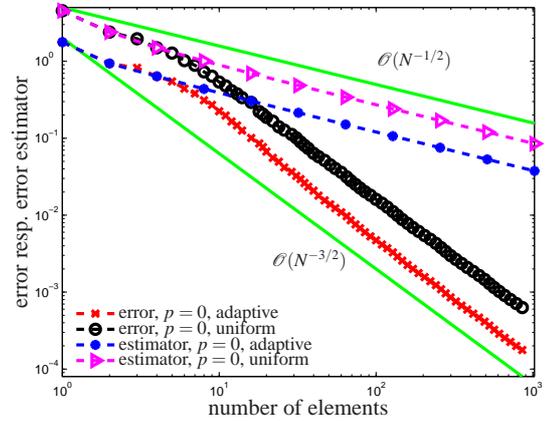}
\caption{BEM error and error estimator for piecewise constants ($p=0$)
on uniform and adaptive meshes for example~\eqref{intro:example}.}
\label{intro:2d:errest}
\end{figure}

\begin{figure*}[t]
 \centering
 \def\fig#1#2{\begin{minipage}[t]{33mm}%
   \centering\hspace*{-5.5mm}\includegraphics[height=30mm]{anisotrop_#1.eps}\\{\hspace*{-2mm}\scriptsize#2}%
 \end{minipage}}
 \fig{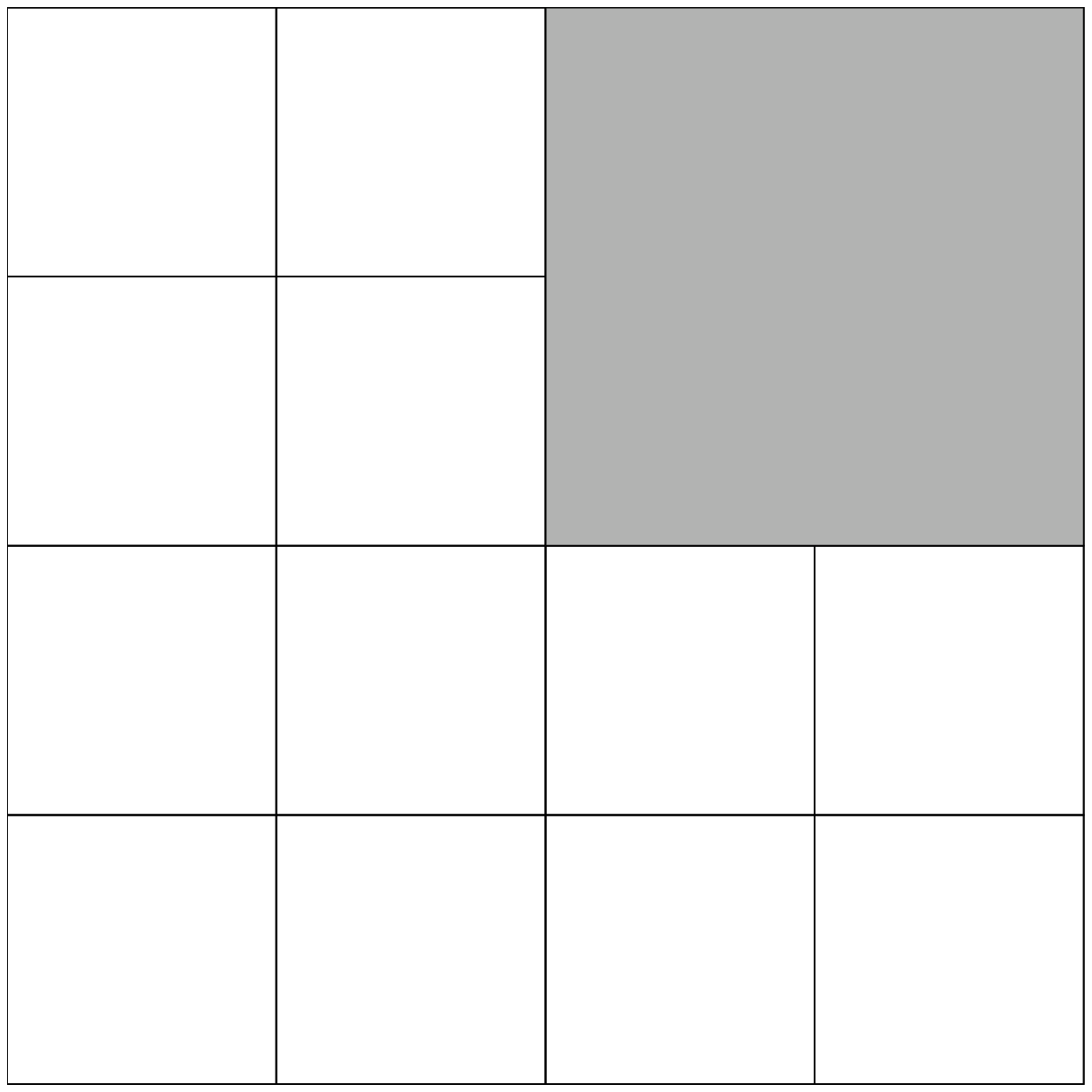}{initial mesh $\TT_0$}
 \fig{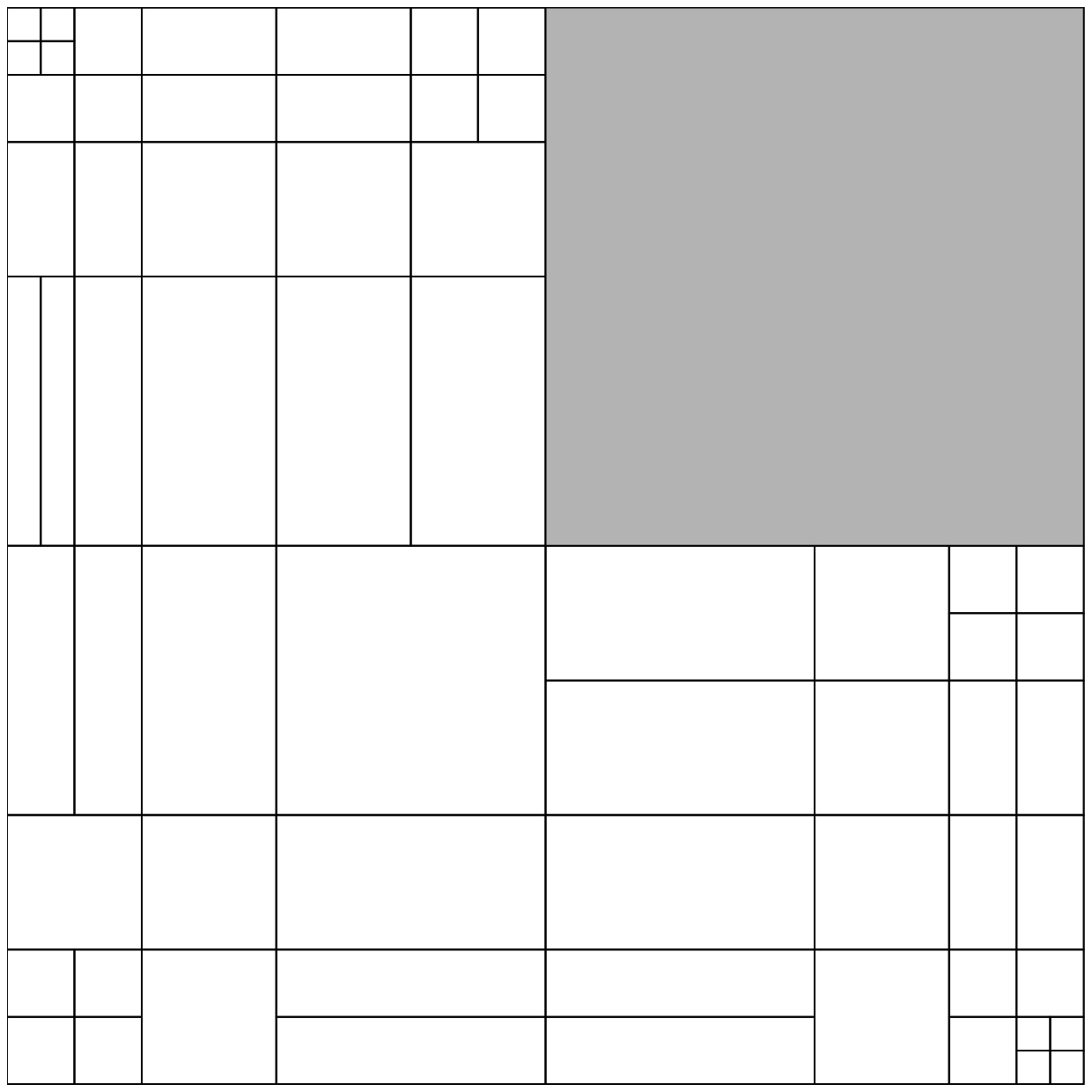}{$\TT_3$}
 \fig{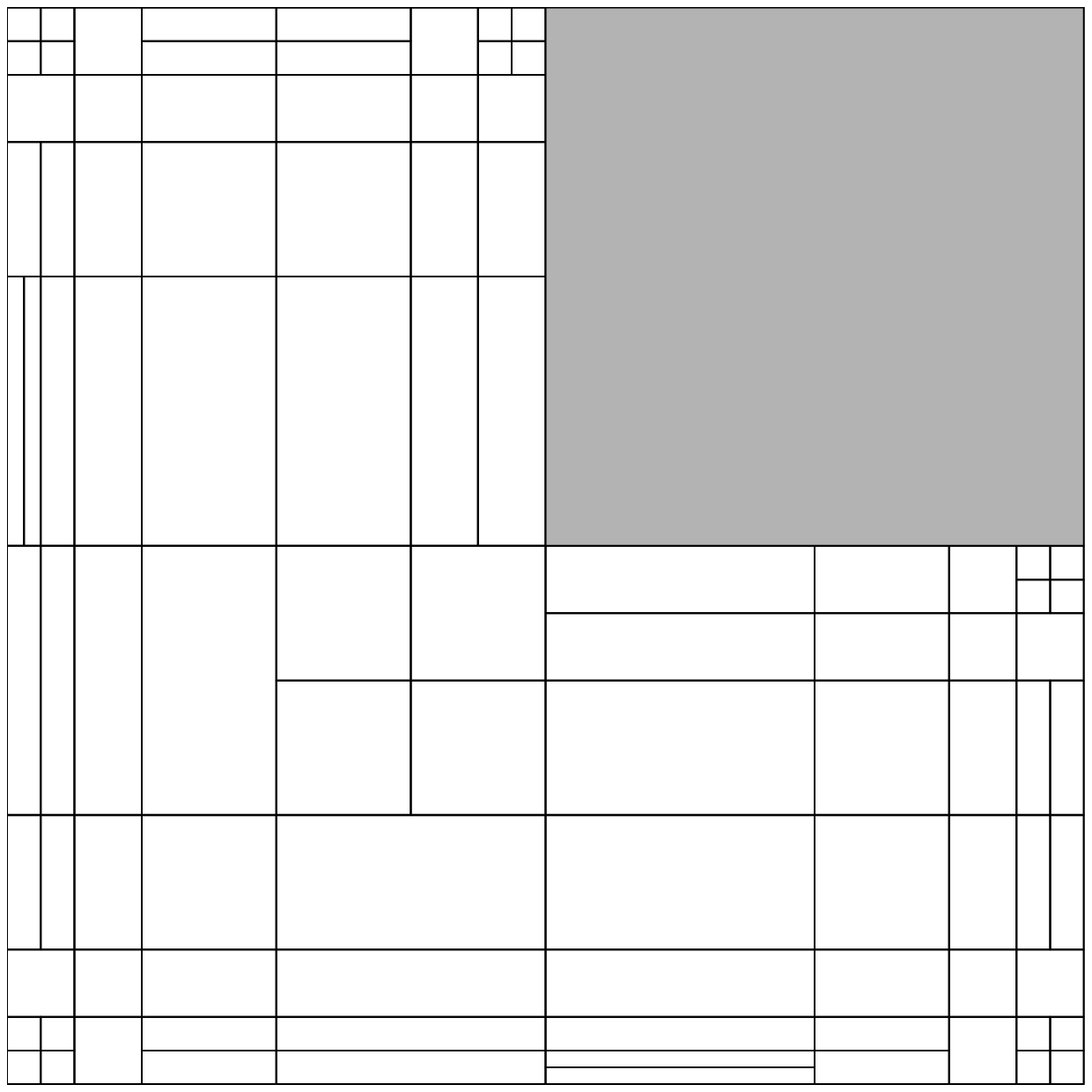}{$\TT_6$}
 \fig{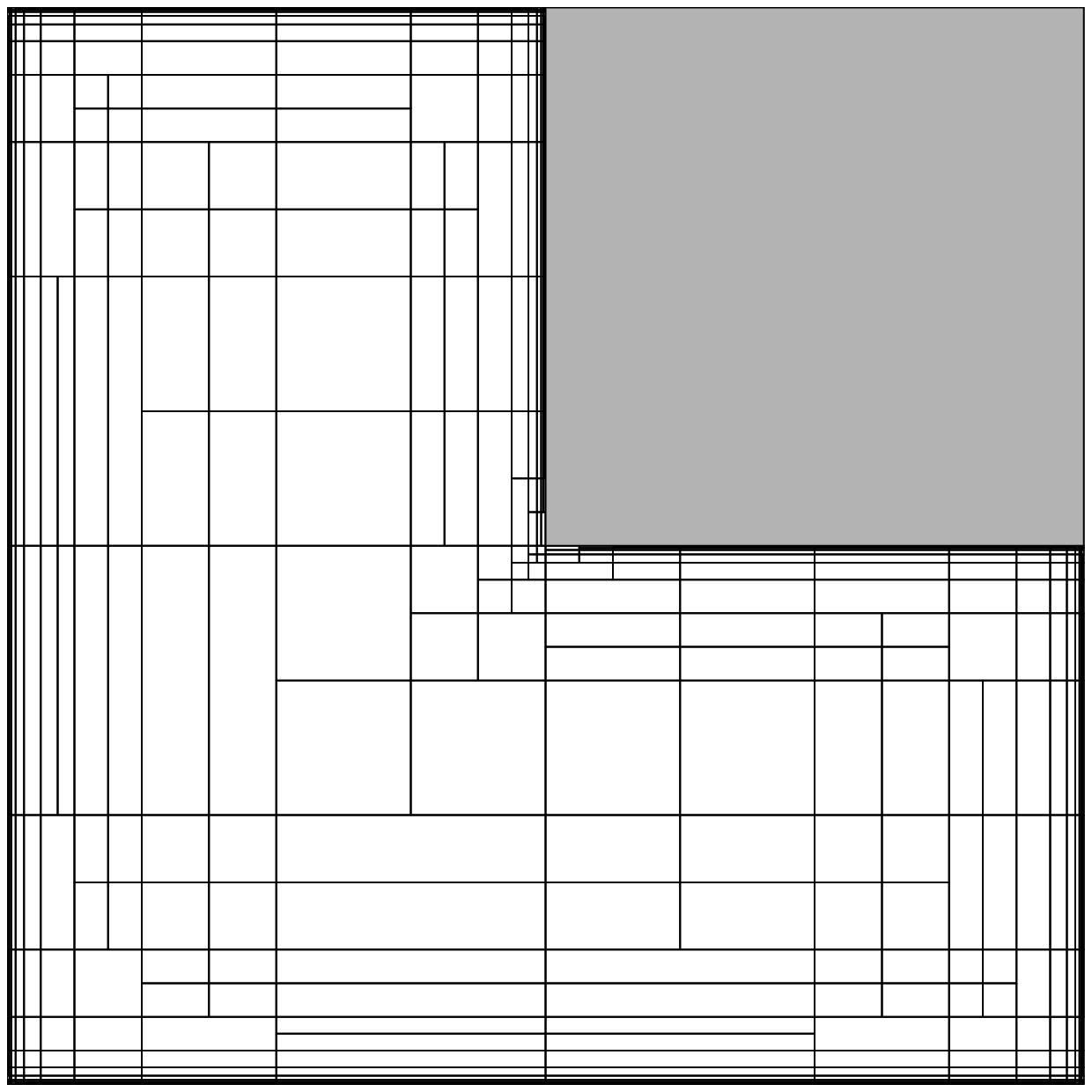}{$\TT_9$}
 \fig{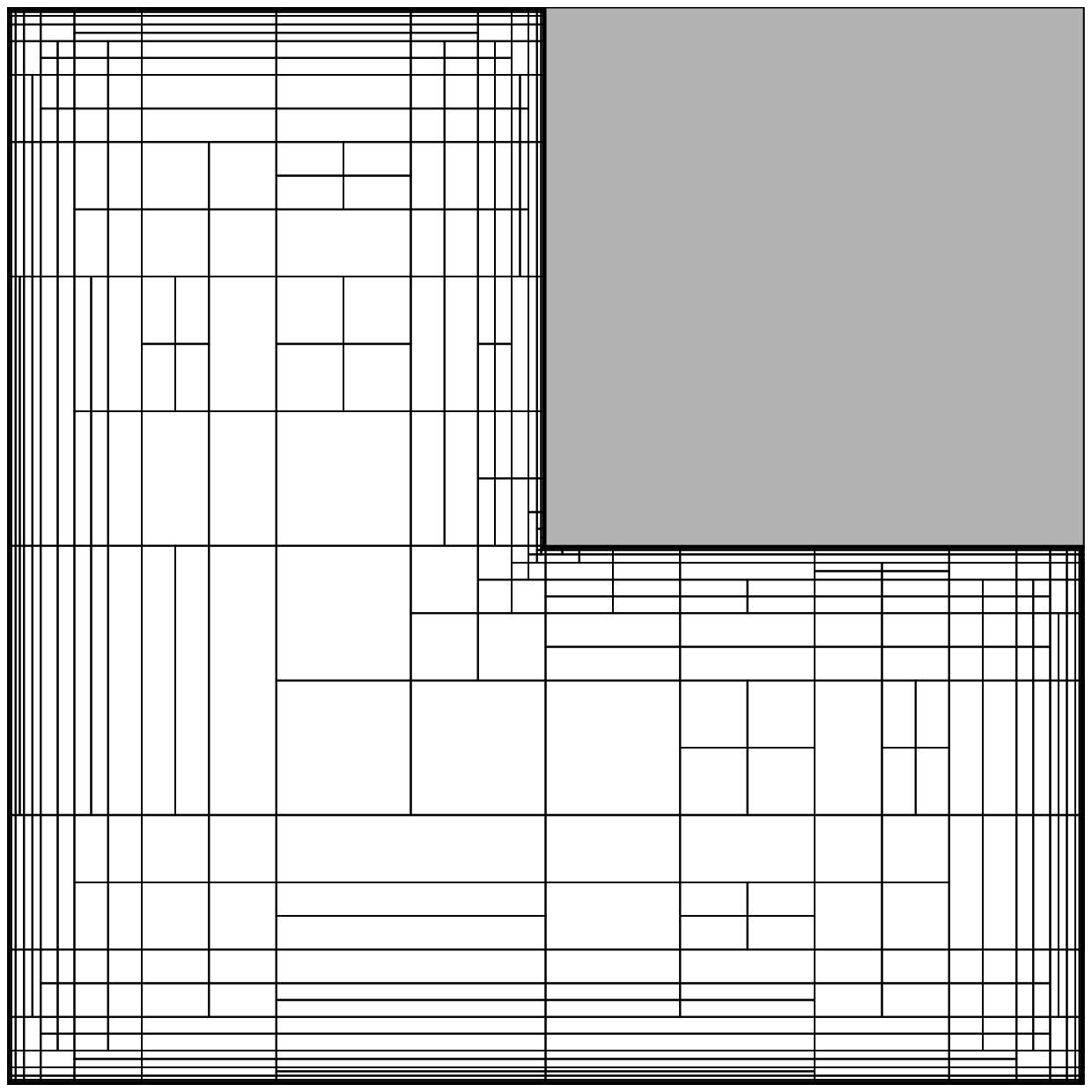}{$\TT_{12}$}
  \caption{Sequence of adaptively generated meshes for 3D BEM with anisotropic
mesh refinement for example~\eqref{intro:example3d}.}
  \label{intro:3d:meshes}
\end{figure*}

\begin{figure*}[t]
 \centering
 \def\T#1{\psfrag{T#1}{\hspace*{-2mm}{$\boldsymbol{T_{#1}}$}}}
 \def\fig#1#2{\begin{minipage}[t]{40mm}%
   \centering\hspace*{-5.5mm}\includegraphics[height=37mm]{intro_#1.eps}\\[-4mm]{\scriptsize#2}%
 \end{minipage}}
 \T{} \T1 \T2 \T3 \T4
 \fig{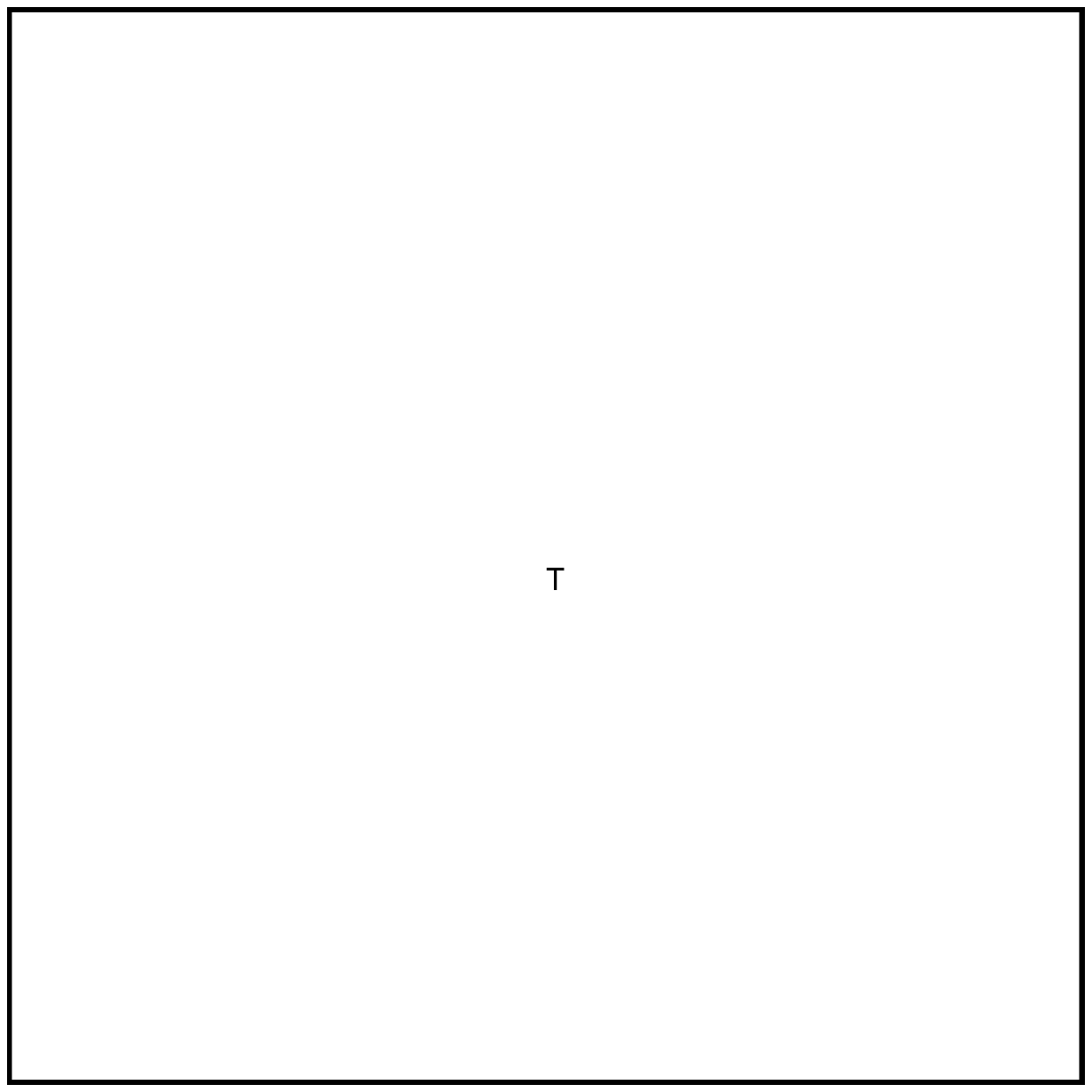}{marked element}
 \fig{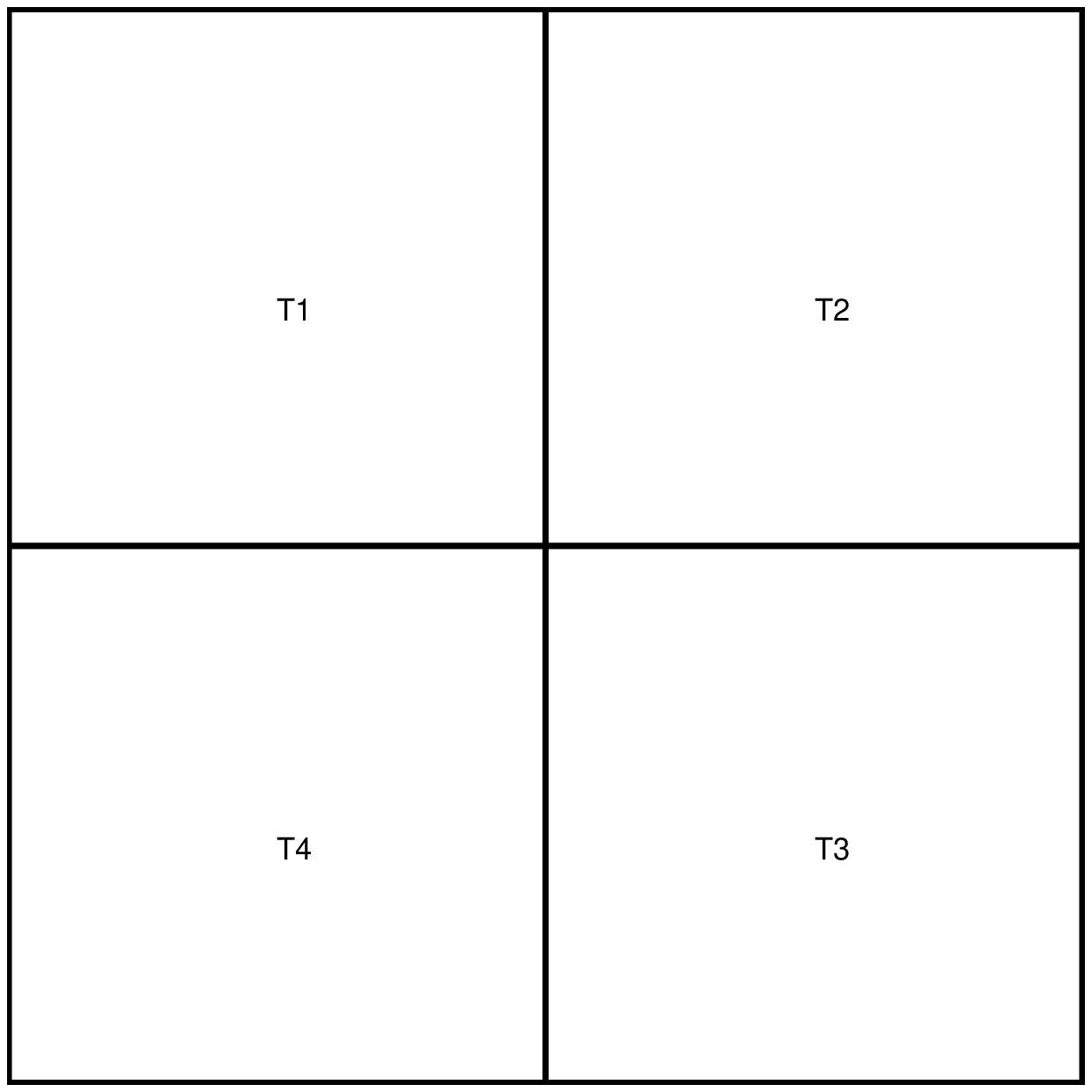}{isotropic refinement}
 \fig{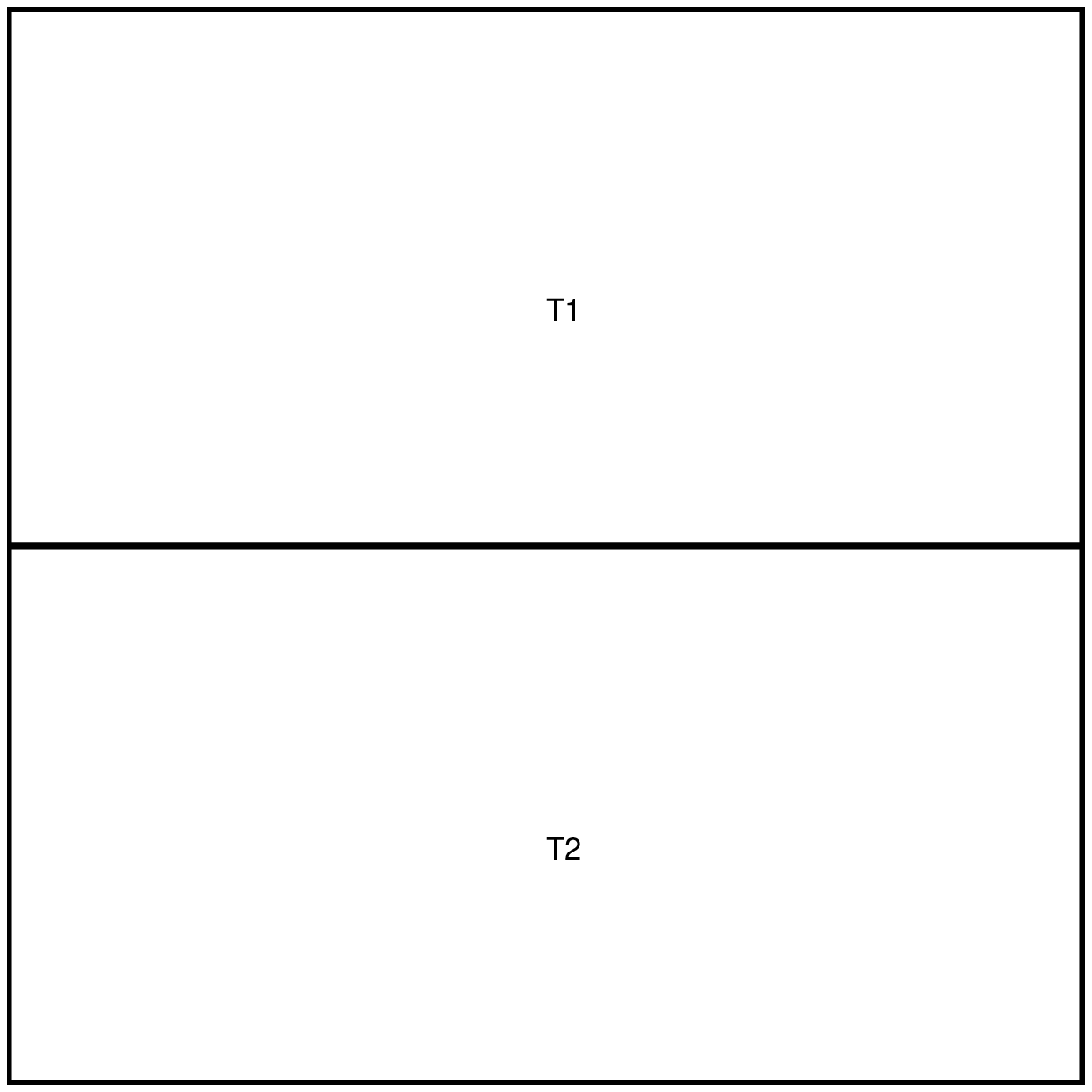}{vertical refinement}
 \fig{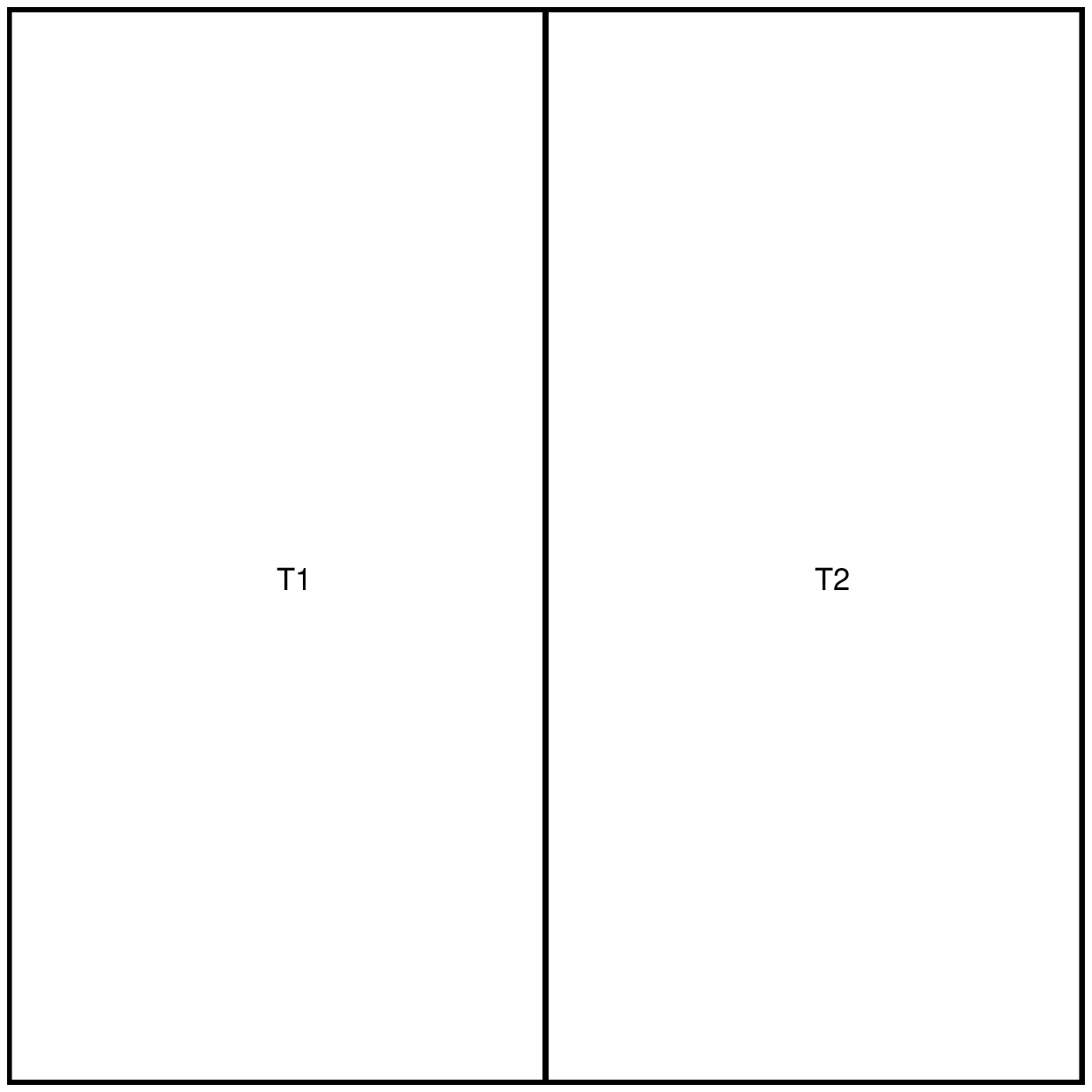}{horizontal refinement}
  \caption{For 3D BEM, each marked rectangle $T\in\TT_\ell$ (left) is either 
refined isotropically into four elements 
or anisotropically into two elements.}
  \label{intro:3d:refinement}
\end{figure*}

\subsection{Mathematical questions}
To illustrate some of the mathematical questions which have to be addressed, 
we consider simple toy problems for the 2D and 3D Laplacian: First, we consider 
the weakly singular integral equation
\begin{align}\label{intro:example}
 \slo\phi = f
 \quad\text{on the slit }\Gamma = (-1,1)\times\{0\},
\end{align} 
where $\slo$ is the simple-layer integral operator of the 2D Laplacian
(see Section~\ref{section:bio} below).
For given $f(s,0) = s$, the unique solution of~\eqref{intro:example} is
known to be
\begin{align}\label{intro:example:solution}
 u(s,0) = 2s/\sqrt{1-s^2}.
\end{align}
We consider BEM with piecewise constant ansatz and test functions
($p=0$)
as well as with discontinuous piecewise linear ansatz and test functions
($p=1$).
The adaptive mesh refinement is driven by some $(h-h/2)$-type
error estimator (see Section~\ref{section:est:hh2} below). The initial mesh 
consists of one line segment of length $2$.

Figs.~\ref{intro:2d:error} and \ref{intro:2d:errest} show the outcome of the 
numerical computations, where we compare uniform vs.\ adaptive mesh-refine\-ment.
We plot the error (measured in the natural\linebreak $\widetilde H^{-1/2}$-norm) and the 
computed a~posteriori error estimator versus the number $N$ of elements. 
If $u$ was smooth, the generically optimal order of convergence would be\linebreak
$\OO(N^{-p-3/2})$, see~\cite{ss11}. However, in the present example, the exact 
solution has strong (generic) singularities at the tips of the slit and thus 
lacks the required regularity. For uniform mesh refinement, where all line 
segments are bisected to obtain $\TT_{\ell+1}$ from $\TT_\ell$, we observe 
a poor convergence rate of $\OO(N^{-1/2})$ for both, piecewise constants
and piecewise linears, see Fig.~\ref{intro:2d:error}.
Conse\-quent\-ly, the use of higher-order 
polynomials does not pay on uniform meshes. However, if we use the adaptive
algorithm which automatically enforces an appropriate grading of the mesh 
towards the singularities of $u$, we observe the optimal convergence 
behavior $\OO(N^{-3/2})$ for piecewise constants $p=0$ and
$\OO(N^{-5/2})$ for piecewise linears $p=1$. 

Mathematically, this observation gives rise to the following questions:
\begin{itemize}
  \item[\Large$\bullet$] Does the adaptive algorithm (Algorithm~\ref{opt:algorithm})
    guarantee\\convergence?
  More precisely, is it true that
  \begin{align}\label{eq:plain:conv}
    \norm{u - U_\ell}{\XX} \rightarrow 0 \quad\text{ as } \ell\rightarrow\infty?
  \end{align}
\end{itemize}
For uniform mesh refinement, the C\'ea lemma~\eqref{intro:cea} and appropriate 
approximation results for smooth functions guarantee that Galerkin BEM 
always lead to convergence $U_\ell \to u$, independently of the overall regularity 
or possible singularities of the unknown solution $u$. As observed, this convergence can be
slow. On the other hand, the adaptive algorithm does not guarantee that the
mesh-size will tend to zero as $\ell\to\infty$. Consequently, the convergence
analysis for uniform meshes does not carry over to adaptive meshes.
However, a~priori arguments guarantee that nestedness $\XX_\ell\subseteq\XX_{\ell+1}$ 
for all $\ell\ge0$ implies convergence of $U_\ell$ towards some limit
$U_\infty$ (see Lemma~\ref{estred:lem:convortho}), but raises the important
question whether we can identify 
$u=U_\infty$. In particular, the numerical check for convergence will always
be affirmative even if the adaptive algorithm does wrong, i.e., $u\neq U_\infty$.
\begin{itemize}
  \item[\Large$\bullet$] Empirically, the adaptive algorithm~\ref{opt:algorithm}
    does not only lead to convergence,but even ensures \emph{linear} convergence, i.e.,
    $\norm{u-U_{\ell+1}}\XX\le q\,\norm{u-U_\ell}\XX$ for some uniform 
    constant $0<q<1$.
\end{itemize}
For some symmetric and elliptic bilinear form $b(\cdot,\cdot)$ and the
induced norm $\norm{v}\XX = \sqrt{b(v,v)}$, the C\'ea lemma~\eqref{intro:cea}
holds with $C_{\rm ell} = 1 = C_{\rm cont}$. This implies at least
$\norm{u-U_{\ell+1}}\XX\le\norm{u-U_\ell}\XX$, and the numerical experiment 
also shows that pre-asymptotically even $q\approx1$ can be observed, see
Fig.~\ref{intro:2d:error}.
\begin{itemize}
  \item[\Large$\bullet$] Does the adaptive algorithm~\ref{opt:algorithm}
    recover the optimal rate of convergence?
\end{itemize}
Clearly, the last two questions are strongly related to the a~posteriori error 
estimator which drives the adaptive mesh refinement. As the adaptive
algorithm does not see the actual error, but only the error estimator,
an a~posteriori error estimator is called \emph{reliable}, if it provides
an upper bound for the unknown error
\begin{align}\label{intro:reliable}
 \norm{u-U_\ell}\XX \le C_{\rm rel}\,\eta_\ell
\end{align}
up to some generic constant $C_{\rm rel}>0$. If the adaptive algorithm thus
drives the error estimator to zero, this implies convergence of the overall
scheme. Conversely, $\eta_\ell$ is called \emph{efficient}, if it provides
a lower bound  for the unknown error
\begin{align}\label{intro:efficient}
 \eta_\ell \le C_{\rm eff}\, \norm{u-U_\ell}\XX
\end{align}
up to some generic constant $C_{\rm eff}>0$. If $\eta_\ell$ is both,
efficient and reliable, the adaptive algorithm monitors the convergence
behavior. 

\begin{figure}[t]
\psfrag{error}[c][c]{\footnotesize error in energy norm}
\psfrag{degrees of freedom}[c][c]{\footnotesize number of elements}
\psfrag{N14}{\scriptsize$\OO(N^{-1/4})$}
\psfrag{N34}{\hspace*{-6mm}\scriptsize$\OO(N^{-3/4})$}
\psfrag{uniform}{\scriptsize uniform}
\psfrag{isotropic}{\scriptsize isotropic}
\psfrag{anisotropic}{\scriptsize anisotropic}
\centering
\includegraphics[width=0.46\textwidth]{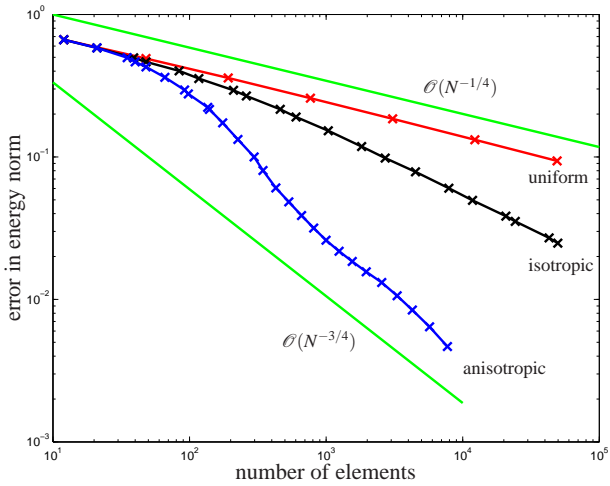}
\caption{BEM error for piecewise constants on uniform and adaptive meshes
for example~\eqref{intro:example3d} with isotropic and anisotropic
refinement.}
\label{intro:3d:error}
\end{figure}

\begin{figure}[t]
\psfrag{error}[c][c]{\footnotesize error resp.\ error estimator}
\psfrag{degrees of freedom}[c][c]{\footnotesize number of elements}
\psfrag{N14}{\scriptsize$\OO(N^{-1/4})$}
\psfrag{O34}{\scriptsize$\OO(N^{-3/4})$}
\psfrag{uniform}{\scriptsize uniform}
\psfrag{isotropic}{\scriptsize isotropic}
\psfrag{anisotropic}{\scriptsize anisotropic}
\centering
\includegraphics[width=0.46\textwidth]{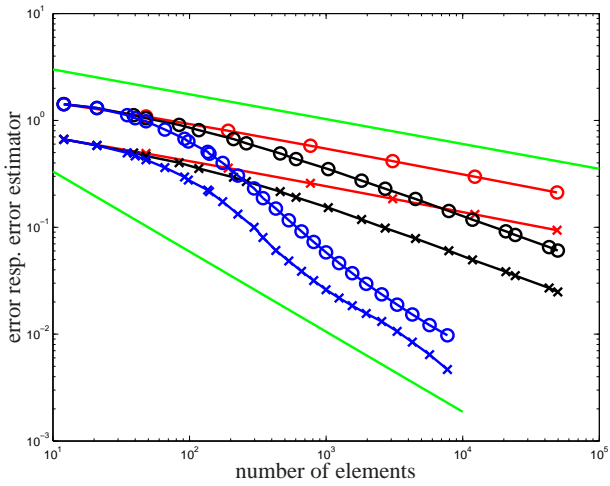}
\caption{BEM error and error estimator for piecewise constants
on uniform and adaptive meshes for example~\eqref{intro:example3d} with isotropic and anisotropic
refinement.}
\label{intro:3d:errest}
\end{figure}

Fig.~\ref{intro:2d:errest} displays error as well as $(h-h/2)$-type
error estimator for uniform and adaptive mesh refinement and lowest-order
elements $p=0$. As the curves of error and error estimator are parallel, 
independently of the mesh refinement, it is observed that 
$\eta_\ell$ is reliable and efficient. The same observation is 
obtained for higher-order polynomials (not displayed).
\begin{itemize}
\item[\Large$\bullet$] In which situations can it be mathematically guaranteed
that the BEM error is estimated reliably and efficiently?
\end{itemize}
Mathematical details (and restrictions) for the $(h-h/2)$-error estimator 
are discussed in Section~\ref{section:est:hh2} below.

Moreover, optimal convergence behavior is also constrained by the 
mesh refinement used. To illustrate this, we consider the weakly singular 
integral equation
\begin{align}\label{intro:example3d}
 \slo\phi = f
 \text{ on the L-screen }\Gamma = \bigl((-1,1)^2\backslash(0,1)^2\bigr)\times\{0\},
\end{align} 
where $\slo$ now is the simple-layer integral operator of the 3D Laplacian.
We consider lowest-order BEM with piecewise constant ansatz and test functions.
The adaptive mesh refinement is driven by some $(h-h/2)$-type
error estimator (see Section~\ref{section:est:hh2} below). The initial mesh 
consists of $12$ uniform squares with edge length $1/2$, see
Fig.~\ref{intro:3d:meshes} (left).

It is well-known that the solutions of~\eqref{intro:example3d} suffer from
edge singularities. For $f(x)=1$, we therefore compare uniform 
mesh refinement, where each
square is divided into four similar squares of half edge length, with
adaptive iso\-tro\-pic resp.\ anisotropic mesh refinement. In the isotropic
\linebreak case,
marked squares are refined into four similar squares of half edge length.
In the anisotropic case, we also allow that the (rectangular) elements are
only refined along one edge into two rectangles, see 
Fig.~\ref{intro:3d:refinement}. For the anisotropic refinement, some 
adaptively refined meshes are shown in Fig.~\ref{intro:3d:meshes}.

The overall outcome of these computations is visualized in 
Figs.~\ref{intro:3d:error} and~\ref{intro:3d:errest}, where we compare 
uniform and adaptive isotropic and anisotropic mesh refinement. We plot
the error (measured in the natural $\widetilde H^{-1/2}$-norm) and the computed
a~posteriori error estimator versus the number $N$ of elements. If $u$
was smooth, the generically optimal order of convergence would be $\OO(h^{3/2})$
for the uniform mesh-size $h$. For 3D BEM, this corresponds to an optimal
decay $\OO(N^{-3/4})$ with respect to the number of elements. However,
the exact solution exhibits generic singularities along the edges of $\Gamma$.
For uniform mesh refinement, where all elements are refined isotropically, 
we observe a poor rate of convergence $\OO(N^{-1/4})$ for the error.
For the adaptive strategy with aniso\-tropic elements, we observe the optimal
rate of convergence $\OO(N^{-3/4})$, while adaptive isotropic refinement
leads to approximately $\OO(N^{-1/2})$. We note that heuristic arguments
show that $\OO(N^{-1/2})$ is the optimal rate of convergence in the presence of
generic edge singularities, if one restricts to isotropic elements~\cite{cmps}.
This is also illustrated by the adaptive meshes shown in 
Fig.~\ref{intro:anisotropic3d} as well as Fig.~\ref{intro:isotropic3d}
which show adaptively generated anisotropic resp.\ isotropic meshes with
(almost) the same number of elements.

\begin{figure}
\centering
\includegraphics[width=0.44\textwidth]{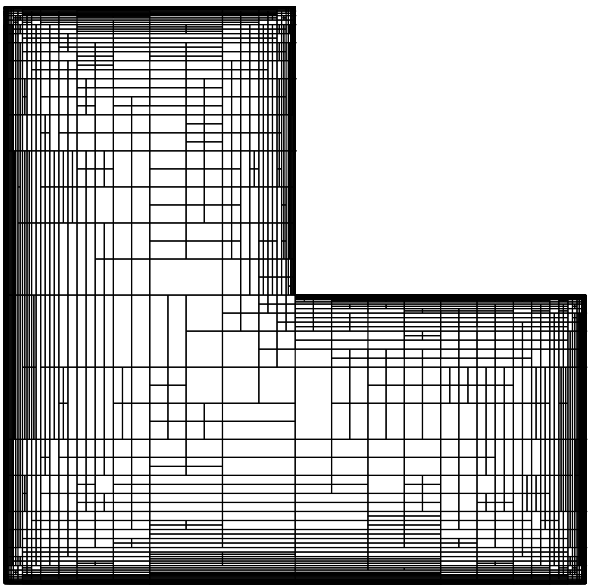}
\caption{Adaptively generated mesh with anisotropic elements.}
\label{intro:anisotropic3d}
\end{figure}

\begin{figure}
\centering
\includegraphics[width=0.44\textwidth]{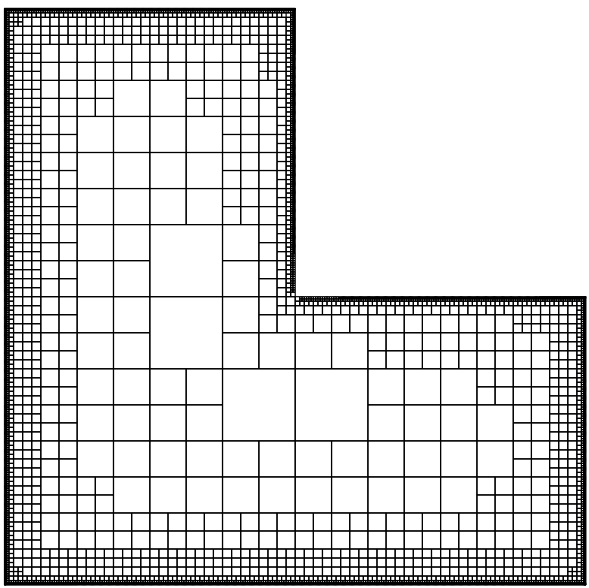}
\caption{Adaptively generated mesh with isotropic elements.}
\label{intro:isotropic3d}
\end{figure}

Fig.~\ref{intro:3d:errest} shows the BEM error as well as the $(h-h/2)$-type
error estimators. As for the 2D example~\eqref{intro:example}, we
observe that the a~posteriori error estimators used are 
reliable and efficient.

\subsection{Outline}
Essential ingredients of the mathematical theory of BEM will be collected
in Section~\ref{section:bem}. The fundamental function spaces in BEM are
Sobolev spaces, which will be introduced briefly in Section~\ref{section:sobolev}.
There will be no further explanations on the connection of BIEs and PDEs, but
in Section~\ref{section:bio} we will define the boundary integral operators that constitute
the equations that are to be solved (Sections~\ref{section:weaksing} and~\ref{section:hypsing}).
To emphasize the significance of local mesh refinement, Section~\ref{sec_reg} briefly summarizes
the regularity theory in the context of this work.
The terms associated with discrete spaces, such as meshes, piecewise polynomials, and so on, will
be defined in Section~\ref{section:bem:discrete}, and the resulting discrete equations are
given in Section~\ref{section:galerkin}.

As a basis for a~posteriori error estimation, Section~\ref{section:localization} focuses exclusively
on the localization techniques for fractional order Sobolev norms. The results of this section
will be used frequently in this work, and the two omnipresent approaches, localization by 
local fractional norms (Section~\ref{section:localization:slob}) and localization by
approximation (Section~\ref{section:localization:apx}), will be\linebreak
treated in particular.

Section~\ref{section:aposteriori} gives an overview of the different a~posteriori error estimators
for BEM that have been proposed in the mathematical literature. The estimators are
classified into five different groups:
\begin{itemize}
  \item Residual error estimators (Section~\ref{section:residual}),
  \item estimators based on space enrichment (Section~\ref{sec:enrich}),
  \item averaging on large patches (Section~\ref{section:averaging}),
  \item ZZ-type estimators (Section~\ref{section:zzest}),
  \item and estimators based on the Calder\'on sytem (Section~\ref{section:twoeq}).
\end{itemize}
In addition, Section~\ref{section:dataapproximation} deals with the question of how to estimate
data approximation errors. The estimators presented up to this point might serve also in
higher-order BEM, but are analyzed only with respect to mesh refinement.
In contrast, a~posteriori estimators and associated adaptive algorithms
for $p$ and $hp$-versions of the BEM are shown in Section~\ref{section:aposteriori:hp}.

The question of convergence of $h$-adaptive algorithms of the type~\eqref{eq:plain:conv} will be
addressed in Section~\ref{section:estred}, which deals with the so-called
\textit{estimator reduction principle}.
This is a rather general concept dealing with convergence of adaptive algorithms. This
will be explained in detail in Section~\ref{section:estred:principle}.
The results of Section~\ref{section:estred} are
tailored to certain concrete model problems and estimators which are
given in Sections~\ref{section:estred:hh2}--\ref{section:estred:res} for a~posteriori error estimation
with $(h-h/2)$, $ZZ$, and weighted residual estimators, as well as in
Sections~\ref{section:estred:data}\linebreak
and~\ref{section:estred:data:res} including data approximation.
We also comment on convergence in the presence of anisotropic mesh refinement in
Section~\ref{section:estred:anisotropic}.

The properties of module $\texttt{refine}(\cdot)$, responsible for mesh refinement, play an
important role in the analysis of optimal rates of adaptive mesh-refining algorithms.
Section~\ref{section:meshrefinement} explains the requirements for $\texttt{refine}(\cdot)$
and summarizes\linebreak available results from the
literature to account for local refinement in 2D BEM as well as 3D BEM.

Regarding the question of optimal convergence of Algorithm~\ref{opt:algorithm},
the following Section~\ref{section:convergence} introduces an abstract framework
that was recently laid out even in a more general setting in~\cite{axioms}.
At this point, we will have fixed all the parts of Algorithm~\ref{opt:algorithm}
except $\boxed{\texttt{estimate}}$. Based on certain assumptions
(called (A1)--(A4)) on the estimator $\eta_\ell$ that is employed in $\boxed{\texttt{estimate}}$,
convergence of Algorithm~\ref{opt:algorithm} will be shown in Section~\ref{opt:section:convergence},
generalising the results of Section~\ref{section:estred}. Optimal convergence
of Algorithm~\ref{opt:algorithm} within the abstract framework
will be shown in Section~\ref{opt:section:optimal}.
In\linebreak
Sections~\ref{opt:sec:example1}--\ref{opt:sec:example4}, it is shown
how to apply the abstract setting to concrete model problems, i.e., the assumptions (A1)--(A4) will be
checked for different error estimators. More precisely, we obtain linear convergence for
$(h-h/2)$-based estimators and optimal convergence for weighted residual estimators.
From Section~\ref{opt:section:data} on, we deal with convergence and optimality of ABEM including data
approximation. To that end, an extended algorithm (Algorithm~\ref{opt:algorithmd})
will be formulated that
differs from Algorithm~\ref{opt:algorithm} only in that Galerkin solutions are computed with respect
to an approximate right-hand side and that error control for data approximation is included in
the error estimation. For this extended error estimators, we again formulate
assumptions (called ($\widetilde{\rm A1}$)--($\widetilde{\rm A6}$)), and show not only
convergence (Section~\ref{opt:section:data:conv})
but also optimality (Section~\ref{opt:section:data:optimal})
of Algorithm~\ref{opt:algorithmd}.
We show how to apply this abstract framework to concrete model problems in
Sections~\ref{opt:sec:example2d} and~\ref{opt:sec:example4d}.

The final Section~\ref{section:implementation} is devoted to details in implementation.
We give detailed explanations how to implement the $L_2$-orthogonal projection
(Section~\ref{section:implementation:l2proj}) as well as the Scott-Zhang projection
(Section~\ref{section:implementation:sz}). Furthermore, we show how to implement
the two-level error estimator, the $(h-h/2)$ based error estimator, as well as the
weighted residual error estimator for $d=2$ and in the lowest order case.
The ideas that we present for implementation transfer immediately to $d=3$ and higher-order
polynomials.

Throughout the paper, $a\lesssim b$ means that $a\le cb$ with a generic constant $c>0$
that is independent of involved mesh parameters or functions. Similarly, the notation
$a\gtrsim b$ and $a\simeq b$ is used.

\section{Mathematical foundation of the BEM}\label{section:bem}
This section briefly introduces the mathematical framework for boundary element methods.
Definitive books in this respect are~\cite{hw08,mclean,ned01}, which deal exclusively with
boundary integral equations and their analytical underpinning, and \cite{ss11,s08},
which focus to a great extent on boundary element discretizations.
Let us also note that the analysis of finite elements for the
discretization of boundary integral equations
of the first kind goes back to N\'ed\'elec and Planchard
\cite{np73}, and Hsiao and Wendland \cite{hw77}.

\subsection{Sobolev spaces}\label{section:sobolev}
For a rigorous treatise of Sobolev spaces, we refer to the standard reference~\cite{a75}.
For $\Omega\subset\R^d$ an open and bounded set and $p\in [1,\infty]$, $L_p(\Omega)$
denotes the space of all measurable functions $u:\Omega\rightarrow\R$ whose $p$-th power
is integrable, i.e.,
$\norm{u}{L_p(\Omega)}<\infty$, where
\begin{align*}
  \norm{u}{L_p(\Omega)} :=
  \begin{cases}
    \left(\int_\Omega \snorm{u}{}^p \right)^{1/p} &\text{ for } p<\infty,\\
    \inf_{\substack{M\subset\Omega\\\abs{M}=0}}
    \sup_{x\in\Omega\setminus M} \abs{u(x)} &\text{ for } p=\infty.
  \end{cases}
\end{align*}
The space $L_2(\Omega)$ is a Hilbert space with inner product and norm
\begin{align*}
  \dual{u}{w}_\Omega := \int_\Omega u w\,dx,\qquad \norm[]{u}{L_2(\Omega)} := \dual{u}{u}_\Omega^{1/2}.
\end{align*}
The space $C^\infty_0(\Omega)$ is the space of smooth $\phi\in C^\infty(\Omega)$ with
$\supp(\varphi)\subset\Omega$.
If, for $u\in L_2(\Omega)$,
a locally integrable function $w:\Omega\rightarrow\R^d$ exists such that, for all
$\varphi\in C^\infty_0(\Omega)$
\begin{align*}
  \int_\Omega u(x) \nabla \varphi(x)\,dx = - \int_\Omega w(x)\varphi(x)\,dx,
\end{align*}
then $w$ is called the \textit{weak gradient} of $u$, abbreviated by $\nabla u:=w$.
It follows from the fundamental lemma of calculus of variations that the weak gradient
is uniquely defined almost everywhere, and integration by parts shows that it therefore coincides
with the classical gradient of $u$ if it exists.
The space of all functions $u\in L_2(\Omega)$ with weak gradient $\nabla u\in L_2(\Omega)$
is the Sobolev space $H^1(\Omega)$.
This is again a Hilbert space with inner product and norm
\begin{align*}
  \dual{u}{w}_{H^1(\Omega)} &:= \dual{u}{w}_\Omega + \dual{\nabla u}{\nabla w}_\Omega,\\
  \norm{u}{H^1(\Omega)} &:= \dual{u}{u}_{H^1(\Omega)}^{1/2}.
\end{align*}
For $s\in(0,1)$, the fractional order Sobolev space $H^s(\Omega)$ consists
of all $u\in L_2(\Omega)$ with $\norm{u}{H^s(\Omega)} < \infty$, where
inner product and norm are
\begin{align*}
  \dual{u}{w}_{H^s(\Omega)} &:= \dual{u}{w}_{\Omega}\\
  &\quad+\int_\Omega\int_\Omega \frac{(u(x)-u(y))(w(x)-w(y))}{\abs{x-y}^{d+2s}}
  \,dx\,dy,\\
  \norm{u}{H^s(\Omega)} &:= \dual{u}{u}_{H^s(\Omega)}^{1/2}.
\end{align*}
Generally, for a non-empty set $\omega\subset\Omega$ and $s\in(0,1)$, the associated seminorm is denoted by
\begin{align*}
  \snorm{u}{H^s(\omega)}^2 := \int_\omega\int_\omega \frac{(u(x)-u(y))^2}{\abs{x-y}^{d+2s}}\,dx\,dy.
\end{align*}
From now on we assume that $\Omega$ is simply connected and has a Lipschitz
boundary $\partial\Omega$, i.e., local orthogonal coordinates may be introduced to
represent $\partial\Omega$ locally as a Lipschitz function over a $(d-1)$-dimensional domain.
Then, we can define Sobolev spaces $H^s(\Gamma)$ for $\Gamma\subseteq\partial\Omega$
and associated inner products and norms for $s\in[0,1)$ exactly as for $(d-1)$-dimensional
domains but using surface integrals instead of integrals over domains.
The definition of the surface integral does
not depend on the parametrization used, so neither does the space $H^s(\Gamma)$ and its inner
product or norm for $s\in[0,1)$.
Independently of the chosen parametrization of $\Gamma$, a weak surface gradient $\nablag$ can be
defined, cf.~\cite[Def.~1.9]{verchota} or~\cite[Appendix~A.3]{cwgls12}, and hence a space $H^1(\Gamma)$.
The surface gradient is tangential to $\Gamma$, and for smooth functions $u$ in $\R^d$ there holds
$\nabla u = \nablag u + (\n \cdot \nabla u) \cdot \n$.
For $d=2$, i.e., $\Gamma$ a one-dimensional curve, the notation $u^\prime$ will be used to
denote the gradient of $u$.
For $s\in[0,1]$, define
\begin{align*}
  \wilde H^s(\Gamma) := \left\{ u\in H^s(\partial\Omega)\mid \supp(u) \subset \overline\Gamma \right\}
\end{align*}
with norm
\begin{align}\label{eq:wilde:H:norm}
  \norm{u}{\wilde H^s(\Gamma)} := \norm{\wilde u}{H^s(\partial\Omega)},
\end{align}
where $\wilde u$ denotes the extension of $u$ by zero on $\partial\Omega$.
Clearly, if $\Gamma = \partial\Omega$ is the boundary
of a bounded Lipschitz domain, it holds that $\wilde H^s(\Gamma) = H^s(\Gamma)$ for all $s\in[0,1]$.
The same space can be defined for a domain $\Omega$ instead of $\Gamma$ by using
$\R^d$ instead of $\partial\Omega$.
A different characterization of the spaces
$\wilde H^s(\Omega)$ can be given; to that end, we introduce the \textit{trace operator} $\trace$, which
is defined for smooth functions $u$ as $\trace u := u|_{\partial\Omega}$. It can be shown that
$\norm{\trace u}{H^{s-1/2}(\partial\Omega)} \leq C_s \norm{u}{H^s(\Omega)}$ for $s\in (1/2,1]$,
hence $\trace$ can be extended to a linear and continuous operator from $H^s(\Omega)$ to
$H^{s-1/2}(\Gamma)$ for $s\in (1/2,1]$. It is known that
\begin{align*}
  \wilde H^s(\Omega) &= \left\{ u \in H^s(\Omega) \mid \trace u = 0 \right\}\quad&\text{ for }
  s\in(1/2,1],\\
  \wilde H^s(\Omega) &= H^s(\Omega)\quad&\text{ for } s\in [0,1/2).
\end{align*}
Furthermore, in these cases, the norms on $\wilde H^s(\Omega)$ from~\eqref{eq:wilde:H:norm}
and the norms on $H^s(\Omega)$ are equivalent, and the equivalence constants depend on $s$ and
$\Omega$, cf.~\cite[Lem.~1.3.2.6 and Thm.~1.4.4.4]{grisvard}.
If $s\in\left\{ 0,1 \right\}$, the norms coincide.
Note that the case $s=1/2$ is excluded.
Likewise, an operator $\dn$ can be defined which extends the (co-)normal derivative.

For a linear operator $B:\XX\rightarrow\YY$ between two normed linear spaces, denote its operator
norm by
\begin{align*}
  \norm{B}{\XX\rightarrow\YY} := \sup_{0 \neq x\in \XX}\frac{\norm{B(x)}{\YY}}{\norm{x}{\XX}}.
\end{align*}
The operator $B$ is called \textit{bounded} if
$\norm{B}{\XX\rightarrow\YY}<\infty$.
The dual space of a normed linear space $\XX$, denoted by $\XX'$, consists of all
linear and bounded operators (so-called \textit{functionals}) $f: \XX\rightarrow \R$. A norm on
$\XX'$ is given by
\begin{align*}
  \norm{f}{\XX'} := \norm{f}{\XX\rightarrow\R} = \sup_{0\neq x \in \XX}\frac{\abs{f(x)}}{\norm{x}{\XX}}.
\end{align*}
For the Sobolev space $H^s(\Gamma)$, the dual space can be characterized
by the concept of the so-called
\textit{Gelfand triple},\linebreak
cf.~\cite[Sec. 2.1.2.4]{ss11}, using the fact that for densely embedded
Hilbert spaces $\VV\subset \UU$, their dual spaces are also densely embedded, i.e., $\UU'\subset \VV'$.
Then, identifying $\UU$ with its dual $\UU'$, the scalar product $\dual{\cdot}{\cdot}_{\UU}$
can be extended to a duality pairing between $\VV$ and its dual $\VV'$. The space $\UU$ is called
\textit{pivot space}.
The described concept is used to define the duality pairing between $H^s(\Gamma)$ for $s>0$ and its dual
$\wilde H^{-s}(\Gamma) := H^s(\Gamma)'$, where $L_2(\Gamma)$ is used as pivot space. The consequence
is that, if $u\in H^s(\Gamma)$ and $v\in \wilde H^{-s}(\Gamma)$ are both in $L_2(\Gamma)$, then
$\dual{u}{v}_{H^s(\Gamma)\times \wilde H^{-s}(\Gamma)} = \dual{u}{v}_{\Gamma}$ coincides
with the $L_2(\Gamma)$ scalar product, and the last
expression will be used from now on to denote the duality pairing.
The dual space of $\wilde H^s(\Gamma)$ for $s>0$ will be denoted by $H^{-s}(\Gamma)$.

Certain equations that will be considered have a kernel, hence a quotient space will be needed
to solve them. For $s\in[-1,1]$, define
\begin{align*}
  H^s_0(\Gamma) := \left\{ u\in H^s(\Gamma) \mid \dual{u}{1}_\Gamma = 0 \right\}.
\end{align*}
\begin{remark}
  In the literature, cf.~\cite{ss11,s08}, Sobolev spaces are defined with a fixed
  parametrization and associated partition of unity on $\Gamma$. This gives an equivalent
  definition to ours, with constants that depend on the chosen parametrization. To see this,
  denote by $a$ a specific parametrization and partition of unity on $\Gamma$,
  and associated norms $\norm{\cdot}{s,a}$.
  It follows immediately that $\norm{\cdot}{s,a} \leq C_a \norm{\cdot}{H^s(\Gamma)}$,
  and the reverse inequality can be proven with the same arguments as
  in the proof of Theorem~\ref{thm:loc:slobodeckij:upper} below.
\end{remark}
\subsection{Boundary integral operators}\label{section:bio}
From now on, $\Omega\subset\R^d$ will always denote a bounded, simply connected,
$d$-dimensional domain with Lipschitz
boundary $\partial\Omega$ and outer normal vector $\n(y)$ for $y\in\partial\Omega$, and $\Gamma$ will
denote a $(d-1)$-dimensional subset $\Gamma\subset\partial\Omega$.
For simplicity, in the case $d=2$, we assume that ${\rm cap}(\partial\Omega)<1$, see~\cite{mclean},
which can always be fulfilled by scaling $\Omega$ such that its diameter is smaller than $1$.
In order to transform a given PDE into an equivalent boundary integral equation,
a fundamental solution of the PDE at hand needs to be available.
For the Laplace operator $-\Delta$, the fundamental solution is given by
\begin{align*}
  \fsol(z) :=
  \begin{cases}
    -\frac{1}{2\pi}\log\abs{z}\quad&\text{ for } d=2,\\
    \frac{1}{4\pi}\frac{1}{\abs{z}}\quad&\text{ for } d=3.
  \end{cases}
\end{align*}
For densities $\phi,v:\Gamma\rightarrow\R$ and $x\in\R^d\setminus\Gamma$,
define the following potentials:
\begin{itemize}
  \item the \textbf{single layer potential} of $\phi$ as
    \begin{align*}
      \slp \phi(x) := \int_\Gamma \fsol(x-y)\phi(y)\,d\Gamma(y),
    \end{align*}
  \item and the \textbf{double layer potential} of $v$ as
    \begin{align*}
      \dlp v(x) := \int_\Gamma \partial_{\n(y)}\fsol(x-y)v(y)\,d\Gamma(y).
    \end{align*}
\end{itemize}
At least for $\phi,v\in L_1(\Gamma)$, these operators are smooth away from $\Gamma$, i.e.,
$\slp\phi,\dlp v\in C^\infty(\R^d\setminus\Gamma)$, and also harmonic, i.e.,
$\Delta\slp\phi = 0 = \Delta\dlp v$ on $\R^d\setminus\Gamma$.
Starting from these definitions, boundary integral operators are defined as
\begin{align*}
  \slo &:= \trace \slp, &\dlo &:= 1/2 + \trace\dlp,\\
  \hyp &:= -\dn\dlp, &\adlo &:= -1/2 + \dn\slp.
\end{align*}
The operator $\slo$ is called the \textit{single layer operator}, $\hyp$ the
\textit{hypersingular operator}, and $\dlo$ and $\adlo$ the \textit{double layer operator}
and its \textit{adjoint}, respectively.
The two following results recall the stability and ellipticity properties of these
boundary integral operators. For proofs and further references, we refer
to~\cite{Costabel_88_BIO,mclean,ss11,verchota}.
\begin{theorem}\label{thm:stability}
  For $\Gamma = \partial\Omega$ a Lipschitz boundary and\linebreak $s\in[-1/2,1/2]$,
  the boundary integral operators are bounded as mappings
  \begin{align*}
    \slo:&\,H^{-1/2+s}(\Gamma)\rightarrow H^{1/2+s}(\Gamma)\\
    \dlo:&\,H^{1/2+s}(\Gamma)\rightarrow H^{1/2+s}(\Gamma)\\
    \adlo:&\,H^{-1/2+s}(\Gamma)\rightarrow H^{-1/2+s}(\Gamma)\\
    \hyp:&\,H^{1/2+s}(\Gamma)\rightarrow H^{-1/2+s}(\Gamma).
  \end{align*}
  If $\Gamma\subset\partial\Omega$, $\partial\Omega$ again a Lipschitz boundary, it holds
  \begin{align*}
    \slo:&\,\wilde H^{-1/2+s}(\Gamma)\rightarrow H^{1/2+s}(\Gamma)\\
    \hyp:&\,\wilde H^{1/2+s}(\Gamma)\rightarrow H^{-1/2+s}(\Gamma).
  \end{align*}
\end{theorem}
\begin{theorem}\label{thm:ellipticity}
  If $\Gamma=\partial\Omega$ is the boundary of a Lipschitz domain $\Omega$,
  then there holds ellipticity
  \begin{align*}
    \dual{\slo\phi}{\phi}_\Gamma \geq \c{ell}\norm{\phi}{H^{-1/2}(\Gamma)}^2
    \quad\text{ for all } \phi\in H^{-1/2}(\Gamma),\\
    \dual{\hyp u}{u}_\Gamma + \dual{u}{1}_\Gamma^2 \geq \c{ell}\norm{u}{H^{1/2}(\Gamma)}^2
    \quad\text{ for all } u\in H^{1/2}(\Gamma).
  \end{align*}
  If $\Gamma\subsetneq\partial\Omega$ is only a subset, then there holds
  \begin{align*}
    \dual{\slo\phi}{\phi}_\Gamma &\geq \c{ell}\norm{\phi}{\wilde H^{-1/2}(\Gamma)}^2
    \quad\text{ for all } \phi\in\wilde H^{-1/2}(\Gamma),\\
    \dual{\hyp u}{u}_\Gamma &\geq \c{ell}\norm{u}{\wilde H^{1/2}(\Gamma)}^2
    \quad\text{ for all } u\in \wilde H^{1/2}(\Gamma).
  \end{align*}
  The constant $\c{ell}$ depends only on $\Gamma$.
\end{theorem}
\subsection{Weakly singular integral equations}\label{section:weaksing}
According to Theorems~\ref{thm:stability} and~\ref{thm:ellipticity},
$\dual{\slo\cdot}{\cdot}_\Gamma$ is a scalar product on $\wilde H^{-1/2}(\Gamma)$,
such that the Riesz representation theorem immediately yields solutions to the following
variational formulations.
\begin{myproposition}[Weakly singular integral equation]\label{prop:weaksing}
  Denote by $\Omega\subset\R^d$ a Lipschitz domain with ${\rm cap}(\partial\Omega)<1$ for $d=2$ and $\Gamma\subset\partial\Omega$.
  Given $f\in H^{1/2}(\Gamma)$, there is a unique solution $\phi\in\wilde H^{-1/2}(\Gamma)$ of the 
  variational problem
  \begin{align*}
    \dual{V\phi}{\psi}_\Gamma = \dual{f}{\psi}_\Gamma
    \quad\text{ for all }\psi\in \wilde H^{-1/2}(\Gamma).
  \end{align*}
\end{myproposition}
\begin{myproposition}[Dirichlet problem]\label{prop:dirichlet}
  Denote by $\Omega\subset\R^d$ a\linebreak
  Lipschitz domain with ${\rm cap}(\partial\Omega)<1$ for $d=2$ and $\Gamma=\partial\Omega$.
  Given $f\in H^{1/2}(\Gamma)$, there is a unique solution $\phi\in H^{-1/2}(\Gamma)$ of the 
  variational problem
  \begin{align*}
    \dual{V\phi}{\psi}_\Gamma = \dual{(1/2+K)f}{\psi}_\Gamma
    \quad\text{ for all }\psi\in H^{-1/2}(\Gamma).
  \end{align*}
\end{myproposition}
\subsection{Hypersingular integral equations}\label{section:hypsing}
Likewise, Theorems~\ref{thm:stability} and~\ref{thm:ellipticity} state that
$\dual{\hyp\cdot}{\cdot}_\Gamma$ is a scalar product on $\wilde H^{1/2}(\Gamma)$ if
$\Gamma\subsetneq\partial\Omega$ is an open surface, while
$\dual{\hyp\cdot}{\cdot}_\Gamma + \dual{\cdot}{1}_\Gamma\dual{\cdot}{1}_\Gamma$
is a scalar product on $H^{1/2}(\Gamma)$ in case of $\Gamma=\partial\Omega$.
\begin{myproposition}[Hypersingular integral equation]\label{prop:hypsing}
  Denote\linebreak
  by $\Omega\subset\R^d$ a Lipschitz domain and $\Gamma\subsetneq\partial\Omega$ a simply
  connected, open surface. Given $\phi\in H^{-1/2}(\Gamma)$, there is a unique solution
  $u\in \wilde H^{1/2}(\Gamma)$ of the variational problem
  \begin{align*}
    \dual{\hyp u}{v}_\Gamma = \dual{\phi}{v}_\Gamma
    \quad\text{ for all }v\in \wilde H^{1/2}(\Gamma).
  \end{align*}
  If $\Gamma=\partial\Omega$, then
  there is a unique solution $u\in H^{1/2}(\Gamma)$ of the variational problem
  \begin{align*}
    \dual{\hyp u}{v}_\Gamma + \dual{u}{1}_\Gamma\dual{v}{1}_\Gamma = \dual{\phi}{v}_\Gamma
    \quad\text{ for all }v\in H^{1/2}(\Gamma).
  \end{align*}
  Provided that $\phi\in H^{-1/2}_0(\Gamma)$, the solution satisfies\linebreak
  $u\in H^{1/2}_0(\Gamma)$.
\end{myproposition}
\begin{myproposition}[Neumann problem]\label{prop:neumann}
  Denote by $\Omega\subset\R^d$ a Lipschitz domain and $\Gamma=\partial\Omega$
  Given $\phi\in H^{-1/2}_0(\Gamma)$, there is a unique solution
  $u\in H_0^{1/2}(\Gamma)$ of the variational problem
  \begin{align*}
    \dual{\hyp u}{v}_\Gamma +\dual{u}{1}_\Gamma\dual{v}{1}_\Gamma= \dual{(1/2-\adlo)\phi}{v}_\Gamma\\
    \quad\text{ for all }v\in H^{1/2}(\Gamma).
  \end{align*}
\end{myproposition}
\subsection{Regularity of solutions}\label{sec_reg}
It is well known that solutions to BVPs\footnote{boundary value problem (BVP)} on non-smooth
domains have in general limited regularity, even for smooth data. For
polygonal/polyhedral domains and standard elliptic operators of second order
there exists a precise regularity theory that proves that this regularity
reduction is due to the presence of so-called corner singularities
(on polygons and polyhedra) and corner-edge singularities (on polyhedra).
In this paper we are studying the solution of integral equations of the first
kind where unknowns are Cauchy data of BVP. Therefore, through trace operations
(extended restriction and normal derivative), singular behavior of solutions
to BVP imply in a natural way singular behavior of solutions
to such integral equations.

For an overview of regularity theory for BVP on non-smooth domains we refer to
the monograph by Dauge \cite{Dauge_88_EBV}.
The singularity expressions by Dauge have been extensively studied by
Stephan and von Petersdorff
\cite{tvp:thesis,vonPetersdorffS_90_DEC,vonPetersdorffS_90_RMB}.
Their main contribution is tensor product expansions of singularities so that
they are accessible to approximation analysis by piecewise polynomial functions.
In this way, precise predictions can be made about convergence orders of
FE and BE approximations.
In two dimensions, the study of corner singularities goes back to the
seminal paper by Kondratiev \cite{Kondratiev67} and, of course, the structure
of the appearing singularities is much simpler.

We note that for an optimal error analysis of $hp$-methods with geometric mesh
refinement, a more specific regularity analysis based on countably normed
weighted spaces is in order. We refer to
\cite{HeuerS_98_BIO,HeuerS_00_PSO} for a corresponding regularity theory
of boundary integral equations on polygons. To our knowledge \cite{Maischak_pers},
Maischak and Stephan have a manuscript analyzing the case of the hypersingular
integral equation (governing the Laplacian) on polyhedral surfaces.

Low-order methods severely suffer from the presence of singularities.
They limit the order of convergence of the boundary element method when
quasi-uniform meshes are used. Adaptive methods refine meshes locally by using
information that stems from a~posteriori error estimation. In this way, adaptivity
aims at recovering the orders of convergence that one would obtain
for smooth solutions and quasi-uniform meshes.

In the following, for some typical cases, we recall what are the principal
singularities that one has to expect in solutions to the hypersingular and
weakly singular boundary integral equations. These results stem from the
previously mentioned publications
\cite{Kondratiev67,Dauge_88_EBV,vonPetersdorffS_90_DEC,vonPetersdorffS_90_RMB}.

\paragraph{Two space dimensions.}
Let $\Gamma$ be the boundary of a simply connected polygon with edges
$\Gamma^j$, vertices $t_j$ and angles $\omega_j$ at the vertices ($j=1,\ldots,J$).
We consider the weakly singular integral equation from Proposition~\ref{prop:weaksing}
(with solution $\phi$ and right-hand side function $f$) and the hypersingular integral equation
from Proposition~\ref{prop:hypsing} (with solution $u$ and right-hand side function $g$).
For piecewise analytic data $f$, $g$, the solutions $\phi$ and $u$
behave singularly at the corners of the polygon and are smooth elsewhere.

To be precise we consider a partition of unity
$(\chi_1,\ldots,\chi_J)$ where
$\chi_j$ is the restriction of a $C_0^\infty(\R^2)$ function to $\Gamma$
such that $\chi_j=1$ in a neighborhood of the vertex $t_j$
and $\supp (\chi_j) \subset \Gamma^{j-1}\cup\{t_j\}\cup\Gamma^j$
($\Gamma^0=\Gamma^J$).
In this way we may write any function $\varphi$ on $\Gamma$ like
\[
   \varphi=\sum_{j=1}^J (\varphi_-, \varphi_+)\chi_j
\]
where a pair $(\varphi_-, \varphi_+)$ corresponds
to $\varphi$ on $\Gamma^{j-1}\cup\{t_j\}\cup\Gamma^j$
with
\[
   \varphi_- = \varphi|_{\Gamma^{j-1}}
   \quad\mbox{and}\quad
   \varphi_+ = \varphi|_{\Gamma^j}.
\]
 From \cite{CostabelS_85_BIE,HeuerS_96_hpB} we cite the following result.

Let $\alpha_{jk} := k {\pi\over \omega_j}$ (integer $k\ge 1$, $j=1,\ldots,J$)
and, for $t\ge 1/2$, let $n$ be an integer with
$n+1 > {\omega_j \over \pi} (t-1/2)\ge n$.

\noindent (i)
If $f$ is a piecewise analytic function, then
there exists a function $\phi_0$ with
$\phi_0|_{\Gamma^j}\in H^{t-1}(\Gamma^j )$ such that, for the
solution $\phi$ of the weakly singular integral equation, there holds
\[
  \phi
  =
  \sum_{j=1}^J \sum_{k=1}^n
  \left((\phi_{jk})_-, (\phi_{jk})_+\right)\chi_j + \phi_0.
\]
Here,
\[
   (\phi_{jk})_\pm(x) = c|x-t_j|^{\alpha_{jk}- 1}
\]
if $\alpha_{jk}$ is not an integer and
\[
   (\phi_{jk})_\pm(x) = c_1|x-t_j|^{\alpha_{jk}- 1}
                      + c_2|x-t_j|^{\alpha_{jk}- 1}\log|x-t_j|
\]
if $\alpha_{jk}$ is an integer.

\noindent (ii)
If $g$ is a piecewise analytic function, then
there exists a function $u_0$ with
$u_0|_{\Gamma^j}\in H^t (\Gamma^j)$ such that, for the
solution $u$ of the hypersingular integral equation, there holds
\[
  u
  =
  \sum_{j=1}^J \sum_{k=1}^n
  \left((u_{jk})_-, (u_{jk})_+\right)\chi_j + u_0.
\]
Here,
\[
   (u_{jk})_\pm(x) = c|x-t_j|^{\alpha_{jk}}
\]
if $\alpha_{jk}$ is not an integer and
\[
   (u_{jk})_\pm(x) = c_1|x-t_j|^{\alpha_{jk}}
                   + c_2|x-t_j|^{\alpha_{jk}}\log|x-t_j|
\]
if $\alpha_{jk}$ is an integer.

The constants $c$, $c_1$ and $c_2$ above (in (i) and (ii)) are generic.

The representation of singularities above is valid also in the case of
open curves, by setting the angles $\omega_j=2\pi$ at the endpoints.
For example, $\Gamma$ being an interval in $\R^2$, the solution $\phi$ of
the weakly singular integral equation has (with $t_0$ being any endpoint
of $\Gamma$) singularities of the form
\[
   \phi(x)\sim |t_0-x|^{-1/2}\quad\mbox{($x$ close to $t_0$)}
\]
and the solution $u$ of the hypersingular integral equation behaves like
\[
   u(x)\sim |t_0-x|^{1/2}\quad\mbox{($x$ close to $t_0$)}.
\]
An illustration of both cases is given in
Figures~\ref{fig:sl:sing:2d}, \ref{fig:hy:sing:2d}.

Concluding, the solutions of the integral equations are smooth away from the corners
and have reduced regularity at the corners. In the case of the weakly singular
equation, the solution can be unbounded at corners and in the case of the
hypersingular equation, gradients (derivatives with respect to the arc-length)
can be unbounded there. In the extreme case of an open polygon,
the singularity $|\cdot -t_0|^{-1/2}$ prevents $\phi$ from being an
$L_2(\Gamma)$-function and, similarly, $u$ with its
$|\cdot -t_0|^{1/2}$-singularity is not an element of $H^1(\Gamma)$.

\begin{figure}[t]
\begin{center}
\includegraphics[width=0.25\textwidth]{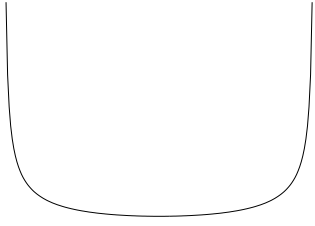}
\end{center}
\caption{Typical singular solution of the weakly singular integral equation on
         an interval}
\label{fig:sl:sing:2d}
\end{figure}

\begin{figure}[t]
\begin{center}
\includegraphics[width=0.25\textwidth]{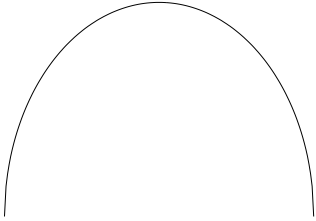}
\end{center}
\caption{Typical singular solution of the hypersingular integral equation on
         an interval}
\label{fig:hy:sing:2d}
\end{figure}

\paragraph{Three space dimensions.}
For simplicity we restrict our presentation of singularities in three dimensions
to the case of $\Gamma$ being a plane open surface with polygonal boundary.
We use the results from \cite{vonPetersdorffS_90_DEC,vonPetersdorffS_90_RMB} and follow
the notation from \cite{bespalov:heuer:05,bespalov:heuer:10}, see also \cite{SchwabS_96_OpA}.

Again, we consider the weakly singular integral equation from Proposition~\ref{prop:weaksing}
(with solution $\phi$ and right-hand side function $f$) and the hypersingular integral equation
from Proposition~\ref{prop:hypsing} (with solution $u$ and right-hand side function $g$).
As before, we assume that $f$ and $g$ are sufficiently smooth. In the following
we present the singularities of $\phi$ and $u$ together.

Let $V$ and $E$ denote the sets of vertices and edges of $\Gamma$, respectively.
For $v\in V$, let $E(v)$ denote the set of edges with $v$ as an endpoint.
Then, $\phi$ and $u$ are of the form
\begin{align*}
   \phi &= \phi_{\rm reg} + \sum_{e\in E}\phi^e + \sum_{v\in V}\phi^v
         + \sum_{v\in V}\sum_{e\in E(v)} \phi^{ev},
   \\
   u    &= u_{\rm reg} + \sum_{e\in E}u^e + \sum_{v\in V}u^v
         + \sum_{v\in V}\sum_{e\in E(v)} u^{ev},
\end{align*}
where, using local polar and Cartesian coordinate systems $(r_v,\theta_v)$ and $(x_{e1},x_{e2})$
with origin $v$, there hold the following representations:

(i) The regular parts satisfy $\phi_{\rm reg}\in H^k(\Gamma)$,
    $u_{\rm reg}\in H^{k+1}(\Gamma)$, with $k>0$.

(ii) The edge singularities $\phi^e$, $u^e$ have the form
\begin{align*}
   \phi^e &= \sum_{j=1}^{m_e}
                  \left( \sum_{s=0}^{s_j^e} b_{js}^e(x_{e1})|\log x_{e2}|^s \right)
                  x_{e2}^{\gamma_j^e-1}\,\chi_1^e(x_{e1})\chi_2^e(x_{e2}),
   \\
   u^e &= \sum_{j=1}^{m_e}
                  \left( \sum_{s=0}^{s_j^e} b_{js}^e(x_{e1})|\log x_{e2}|^s \right)
                  x_{e2}^{\gamma_j^e}\,\chi_1^e(x_{e1})\chi_2^e(x_{e2}),
\end{align*}
where $\gamma_{j+1}^e\ge\gamma_j^e\ge \frac 12$, and $m_e$, $s_j^e$ are
integers. Here, $\chi_1^e$, $\chi_2^e$ are $C^\infty$ cut-off functions
with $\chi_1^e=1$ in a certain distance to the endpoints of $e$,
and $\chi_1^e=0$ in a neighbourhood of these vertices.
Moreover, $\chi_2^e=1$ for $0\le x_{e2}\le \delta_e$
and $\chi_2^e=0$ for $x_{e2}\ge 2\delta_e$ with some $\delta_e \in (0,\frac 12)$.
The functions $b_{js}^e\chi_1^e$ are in $H^m(e)$ for $m$ as large as required.

(iii) The vertex singularities $\phi^v$, $u^v$ have the form
\begin{align*}
   \phi^v &= \chi^v(r_v)
             \sum_{i=1}^{n_v}\sum_{t=0}^{q_i^v} B_{it}^v |\log r_v|^t
                                                r_v^{\lambda_i^v-1} w_{it}^v(\theta_v),
   \\
   u^v &= \chi^v(r_v)
             \sum_{i=1}^{n_v}\sum_{t=0}^{q_i^v} B_{it}^v |\log r_v|^t
                                                r_v^{\lambda_i^v} w_{it}^v(\theta_v),
\end{align*}
where $\lambda_{i+1}^v\ge\lambda_i^v>0$, $n_v$, $q_i^v\ge 0$ are integers,
and $B_{it}^v$ are real numbers.
Here, $\chi^v$ is a $C^\infty$ cut-off function with
$\chi^v=1$ for $0\le r_v\le\tau_v$ and $\chi^v=0$ for $r_v\ge 2\tau_v$
with some $\tau_v \in (0,\frac 12)$. The functions $w_{it}^v$ are in $H^q(0,\omega_v)$
for $q$ as large as required. Here, $\omega_v$ denotes the interior angle
(on $\Gamma$) between the edges meeting at $v$.

(iv) The edge-vertex singularities $\phi^{ev}$, $u^{ev}$ have the form
\[
   \phi^{ev} = \phi_1^{ev} + \phi_2^{ev},\quad
   u^{ev} = u_1^{ev} + u_2^{ev},
\]
where
\begin{multline*}
   \phi_1^{ev} = \sum_{j=1}^{m_e}\sum_{i=1}^{n_v}\left(
              \sum_{s=0}^{s_j^e}\sum_{t=0}^{q_i^v}\sum_{l=0}^s
              B_{ijlts}^{ev}|\log x_{e1}|^{s+t-l}|\log x_{e2}|^l
              \right)
   \\
              x_{e1}^{\lambda_i^v-\gamma_j^e}
              x_{e2}^{\gamma_j^e-1}
              \,\chi^v(r_v) \chi^{ev}(\theta_v),
\end{multline*}
\begin{multline*}
   u_1^{ev} = \sum_{j=1}^{m_e}\sum_{i=1}^{n_v}\left(
              \sum_{s=0}^{s_j^e}\sum_{t=0}^{q_i^v}\sum_{l=0}^s
              B_{ijlts}^{ev}|\log x_{e1}|^{s+t-l}|\log x_{e2}|^l
              \right)
   \\
              x_{e1}^{\lambda_i^v-\gamma_j^e}
              x_{e2}^{\gamma_j^e}
              \,\chi^v(r_v) \chi^{ev}(\theta_v)
\end{multline*}
and
\begin{align*}
   \phi_2^{ev} &= \sum_{j=1}^{m_e}\sum_{s=0}^{s_j^e}
              B_{js}^{ev}(r_v) |\log x_{e2}|^s x_{e2}^{\gamma_j^e-1}
              \,\chi^v(r_v)\chi^{ev}(\theta_v),
   \\
   u_2^{ev}    &= \sum_{j=1}^{m_e}\sum_{s=0}^{s_j^e}
              B_{js}^{ev}(r_v) |\log x_{e2}|^s x_{e2}^{\gamma_j^e}
              \,\chi^v(r_v)\chi^{ev}(\theta_v),
\end{align*}
with
\[
   B_{js}^{ev}(r_v) = \sum_{l=0}^s B_{jsl}^{ev}(r_v)|\log r_v|^l.
\]
Here, $q_i^v$, $s_j^e$, $\lambda_i^v$, $\gamma_j^e$, $\chi^v$ are as above,
$B_{ijlts}^{ev}$ are real numbers,
and $\chi^{ev}$ is a $C^\infty$ cut-off function with
$\chi^{ev}=1$ for $0\le\theta_v\le\beta_v$ and $\chi^{ev}=0$ for
$\frac 32\beta_v\le\theta_v\le\omega_v$ for some
$\beta_v \in (0,\min\{\omega_v/2,\linebreak \pi/8\}]$.
The functions $B_{jsl}^{ev}$ may be chosen such that
\[
    B_{js}^{ev}(r_v)\,\chi^v(r_v)\chi^{ev}(\theta_v)
    =
    \chi_{js}(x_{e1},x_{e2})\,\chi_2^e(x_{e2}),
\]
where the extension of $\chi_{js}$ by zero onto
\[
   \R^{2+}:=\{(x_{e1},x_{e2});\; x_{e2}>0\}
\]
lies in $H^m(\R^{2+})$ for $m$ as large as required.
Here, $\chi_2^e$ is a $C^\infty$ cut-off function as in {\rm (ii)}.

Concluding, on open surfaces, both integral equations have solutions with singularities.
The strongest ones are of the edge-type $\dist(\cdot,\partial\Gamma)^{-1/2}$
for the solution $\phi$ of the weakly singular equation. In this case, $\phi$ is
not an element of $L_2(\Gamma)$, as in the case of two dimensions on open curves.
Correspondingly, the strongest singularities of the solution $u$ of the hypersingular
equation are of the type $\dist(\cdot,\partial\Gamma)^{1/2}$ so that $u\not\in H^1(\Gamma)$,
again analogously to the case in two dimensions.
In Figures~\ref{fig:sl:sing:3d}, \ref{fig:hy:sing:3d} we present typical solutions to
both integral equations on the open surface $\Gamma=(0,1)\times(0,1)\times\{0\}$.
It remains to mention that on polyhedral surfaces, singularities have the same structure
but with larger exponents defining the edge and edge-vertex singularities,
cf.~\cite{vonPetersdorffS_90_DEC,vonPetersdorffS_90_RMB} for more details.

\begin{figure}[t]
\begin{center}
\includegraphics[width=0.5\textwidth]{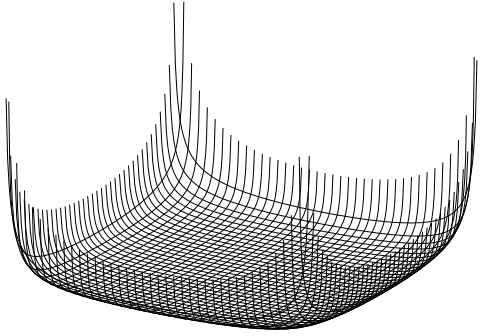}
\end{center}
\caption{Typical singular solution of the weakly singular integral equation on
         the open surface $(0,1)\times(0,1)\times\{0\}$}
\label{fig:sl:sing:3d}
\end{figure}

\begin{figure}[t]
\begin{center}
\includegraphics[width=0.5\textwidth]{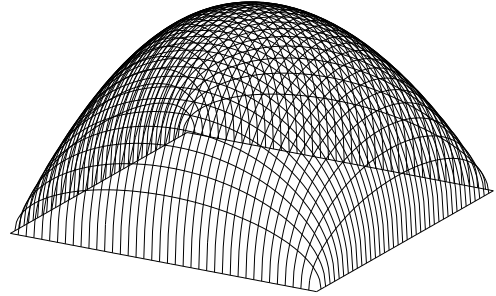}
\end{center}
\caption{Typical singular solution of the hypersingular integral equation on
         the open surface $(0,1)\times(0,1)\times\{0\}$}
\label{fig:hy:sing:3d}
\end{figure}
\subsection{Discrete spaces}\label{section:bem:discrete}
The discrete spaces that will be used to approximate the solutions of the problems given by
Propositions~\ref{prop:weaksing}--\ref{prop:neumann} are spaces of piecewise polynomials over
a \text{mesh} of $\Gamma$. The related terms will be introduced in this section.
\subsubsection{Meshes}
\begin{mydefinition}
  A mesh $\mesh$ on $\Gamma\subset\partial\Omega$ is a finite, mutually disjoint partition
  $\mesh=\left\{ \el_1, \dots, \el_M \right\}$ with the following properties:
  \begin{itemize}
    \item every element $\el\in\mesh$ is a $d$-simplex, i.e., the interior of the convex hull
      of $d$ points $x_1,\dots,x_{d}$,
    \item $\overline\Gamma = \bigcup_{i=1}^M\overline\el_i$,
    \item the intersection $\overline\el\cap\overline\el^\prime$ is either empty,
      a common point, or a common edge of both $\el$ and $\el^\prime$.
  \end{itemize}
\end{mydefinition}
The collection of all points $\NN := \left\{ x_1, \dots, x_N \right\}$ that constitute the
elements is called the set of \textit{nodes} of $\mesh$. Associated to a mesh $\mesh$ is the
local mesh-width function $h_\mesh\in L_\infty(\Gamma)$, given $\mesh$-element-wise as
$h_\mesh|_\el := h_\mesh(\el) := \abs{\el}^{1/(d-1)}$.
On certain occasions the index $\mesh$ will be omitted if no confusion can arise, i.e.,
$h_\el$ will be used instead of $h_\mesh|_\el$.
The quantity
\begin{align*}
  \sigma_\mesh :=
  \begin{cases}
  \displaystyle
  \sup_{\el,\el^\prime\in\mesh, \overline\el\cap\overline\el^\prime\neq\emptyset}
  \frac{\diam(\el)}{\diam(\el')}&\quad\text{ for } d=2,\\
  \displaystyle
  \max_{\el\in\mesh} \frac{\diam(\el)^{d-1}}{\abs{\el}}&\quad\text{ for } d\geq3,
  \end{cases}
\end{align*}
is usually called the \textit{shape-regularity constant} of $\mesh$, and it is a measure
for the degeneracy of the elements $\el$.
\begin{remark}
  To say that \textit{a constant $C$ in a statement depends on shape-regularity}
  means that, given some constant $\sigma>0$, there is a constant $C(\sigma)$,
  depending only on $\sigma$, such that the statement holds true for all
  meshes $\mesh$ as long as $\sigma_\mesh \leq \sigma$.
\end{remark}
For a node $z\in\NN$, the node-patch $\omega_z$ is the collection of all elements $\el\in\mesh$
which share $z$, same idea for the element-patch $\omega_\el$, i.e.,
\begin{align*}
  \omega_z &:= \left\{ \el\in\mesh \mid z\in\overline\el \right\}\\
  \omega_\el &:= \left\{ \el'\in\mesh \mid \overline{\el'}\cap\overline\el\neq\emptyset \right\}.
\end{align*}
An important concept is the so-called \textit{reference element} $\el_\refel$, which
is chosen to be fixed throughout, e.g., as the interior of the convex hull of $(0,0)$, $(1,0)$, and
$(0,1)$. Every element $\el\in\mesh$ with nodes $\left\{ x_0, x_1, x_2 \right\}$
is then the image $\el=F_\el(\el_\refel)$ of $\el_\refel$ under the affine mapping
\begin{align*}
  F_\el:
  \begin{cases}
    \R^2 \rightarrow \R^3 \\
    x \mapsto B_\el x + x_0,
  \end{cases}
\end{align*}
with matrix $B_\el = \left( x_1 - x_0 \mid x_2 - x_0 \right)\in\R^{3\times 2}$.
\subsubsection{Polynomial spaces}
Denoting for $p\geq0$ the polynomial space on the reference element by
\begin{align*}
  \Pp^p(\el_\refel) := \textrm{span}\left\{ (x,y)\in\R^2\mapsto x^iy^k \mid 0 \leq i+k\leq p \right\},
\end{align*}
polynomial spaces on a mesh $\mesh$ are defined by
\begin{align*}
  \Pp^p(\mesh) &:= \left\{ u\in L_\infty(\Gamma) \mid u\circ F_\el\in \Pp^p(\el_\refel)
    \text{ for all } \el\in\mesh \right\}\\
    \Sp^p(\mesh) &:= \Pp^p(\mesh)\cap C^0(\Gamma).
\end{align*}
For $\Gamma\subsetneq\partial\Omega$, define the space with vanishing boundary conditions
\begin{align*}
  \wilde\Sp^p(\mesh) := \wilde H^{1/2}(\Gamma) \cap \Sp^p(\mesh).
\end{align*}
\begin{remark}
  If a mesh carries an index, e.g., $\mesh_{\ell}$, it's associated quantities are also
  equipped with this index, e.g., $\sigma_\ell$ denotes the shape-regularity constant,
  $h_\ell$ the mesh-width, $\NN_\ell$ the set of nodes, and so on.
\end{remark}
\subsection{Galerkin formulation}\label{section:galerkin}
Discrete approximations to the exact solutions $\phi$ and $u$ of the Problems in
Propositions~\ref{prop:weaksing}--\ref{prop:neumann} can be computed by changing the infinite
dimensional spaces $\wilde H^{\pm 1/2}(\Gamma)$ to the discrete spaces
$\Pp^p(\mesh)$ and $\Sp^p(\mesh)$.
The fact that we can also use functions of the discrete spaces in the
variational formulations provides best-approximation estimates (C\'ea's Lemma).
\begin{myproposition}[Galerkin for weakly singular]\label{prop:galerkin:weaksing}
  There is a\linebreak
  unique solution $\Phi\in\Pp^p(\mesh)$ of
    \begin{align*}
    \dual{V\Phi}{\Psi}_\Gamma = \dual{f}{\Psi}_\Gamma
    \quad\text{ for all }\Psi\in \Pp^p(\mesh).
  \end{align*}
\end{myproposition}
\begin{myproposition}[Galerkin for Dirichlet]\label{prop:galerkin:dirichlet}
  There is a unique solution $\Phi\in\Pp^p(\mesh)$ of
    \begin{align*}
    \dual{V\Phi}{\Psi}_\Gamma = \dual{(1/2+\dlo)f}{\Psi}_\Gamma
    \quad\text{ for all }\Psi\in \Pp^p(\mesh).
  \end{align*}
\end{myproposition}
\begin{mylemma}
  If $\phi\in \wilde H^{-1/2}(\Gamma)$ is the solution of Proposition~\ref{prop:weaksing}
  or Proposition~\ref{prop:dirichlet}, and $\Phi\in\Pp^p(\mesh)$ is the solution of
  Proposition~\ref{prop:galerkin:weaksing} or~\ref{prop:galerkin:dirichlet}, then
  \begin{align*}
    \norm{\phi-\Phi}{\wilde H^{-1/2}(\Gamma)} \leq\frac{\c{cnt}}{\c{ell}}
    \min_{\Psi\in\Pp^p(\mesh)}\norm{\phi-\Psi}{\wilde H^{-1/2}(\Gamma)}.
  \end{align*}
  Here, $\c{cnt}=\norm{\slo}{\wilde H^{-1/2}(\Gamma)\rightarrow H^{1/2}(\Gamma)}$
  is the stability constant of $V$, cf. Theorem~\ref{thm:stability}, and $\c{ell}$ is its
  ellipticity constant, cf. Theorem~\ref{thm:ellipticity}.
\end{mylemma}

\begin{myproposition}[Galerkin for hypersingular]\label{prop:galerkin:hypsing}
  In the case\linebreak
  $\Gamma\subsetneq\partial\Omega$, there is a unique solution $U\in\wilde\Sp^p(\mesh)$ of
  \begin{align*}
    \dual{\hyp U}{V}_\Gamma = \dual{\phi}{V}_\Gamma
    \quad\text{ for all }V\in \wilde \Sp^p(\mesh).
  \end{align*}
  In the case $\Gamma=\partial\Omega$, there is a unique solution $U\in\Sp^p(\mesh)$ 
  such that for all $V\in \Sp^p(\mesh)$
  \begin{align*}
    \dual{\hyp U}{V}_\Gamma + \dual{U}{1}_\Gamma\dual{V}{1}_\Gamma
    = \dual{\phi}{V}_\Gamma.
  \end{align*}
  If $\phi\in H^{-1/2}_0(\Gamma)$, it holds that $\dual{U}{1}_\Gamma=0$.
\end{myproposition}
\begin{myproposition}[Galerkin for Neumann]\label{prop:galerkin:neumann}
  There is a unique solution $U\in\Sp^p(\mesh)$ such that for all
  $V\in \Sp^p(\mesh)$
  \begin{align*}
    \dual{\hyp U}{V}_\Gamma + \dual{U}{1}_\Gamma\dual{V}{1}_\Gamma
    = \dual{(1/2-\adlo)\phi}{V}_\Gamma.
  \end{align*}
\end{myproposition}
\begin{mylemma}
  If $u\in \wilde H^{1/2}(\Gamma)$ is the solution of Proposition~\ref{prop:hypsing}
  or Proposition~\ref{prop:neumann}, and $U\in\wilde\Sp^p(\mesh)$ resp.
  $U\in\Sp^p(\mesh)$ is the solution of
  Proposition~\ref{prop:galerkin:hypsing} or~\ref{prop:galerkin:neumann}, then
  \begin{align*}
    \norm{u-U}{\wilde H^{1/2}(\Gamma)} \leq\frac{\c{cnt}}{\c{ell}}
    \min_{V \in\Sp^p(\mesh)}\norm{u-V}{\wilde H^{1/2}(\Gamma)}.
  \end{align*}
  Here, $\c{cnt}=\norm{\hyp}{\wilde H^{1/2}(\Gamma)\rightarrow H^{-1/2}(\Gamma)}$
  is the stability constant of $\hyp$, cf. Theorem~\ref{thm:stability}, and $\c{ell}$ is its
  ellipticity constant, cf. Theorem~\ref{thm:ellipticity}.
\end{mylemma}
Note that the discrete formulations of
Propositions~\ref{prop:galerkin:weaksing}--~\ref{prop:galerkin:neumann} are, indeed,
linear systems of equations. A distinct feature of boundary element methods is that, due
to the non-locality of the boundary integral operators, the
system matrices are dense, and therefore sophisticated data compression techniques are used
to reduce complexity for assembling and solving. In Propositions~\ref{prop:galerkin:dirichlet}
and~\ref{prop:galerkin:neumann}, also the right-hand sides contain boundary integral operators.
There are fast methods to compute the right-hand sides, cf.~\cite{cps,schwab94},
but if one wants to re-use the fast method that is employed for system matrices, the
data $f$ resp. $\phi$ needs to be approximated by discrete functions.
\begin{myproposition}\label{prop:galerkin:dirichlet:data}
  \textnormal{\textbf{(Galerkin for Dirichlet with data approximation)}}
  Denote by $J_\mesh:H^{1/2}(\Gamma)\rightarrow\Sp^{p+1}(\mesh)$ a $H^{1/2}(\Gamma)$
  stable projection. Then, there is a unique solution\linebreak
  $\Phi\in\Pp^p(\mesh)$ of
    \begin{align*}
    \dual{V\Phi}{\Psi}_\Gamma = \dual{(1/2+\dlo)J_\mesh f}{\Psi}_\Gamma
    \quad\text{ for all }\Psi\in \Pp^p(\mesh).
  \end{align*}
\end{myproposition}
\begin{myproposition}\label{prop:galerkin:neumann:data}
  \textnormal{\textbf{(Galerkin for Neumann with data approximation)}}
  Denote by $\pi_\mesh^{p-1}:L_2(\Gamma)\rightarrow\Pp^{p-1}(\mesh)$ the $L_2(\Gamma)$-orthogonal
  projection. Then, there is a unique solution $U\in\Sp^p(\mesh)$ such that for all
  $V\in \Sp^p(\mesh)$
  \begin{align*}
    \dual{\hyp U}{V}_\Gamma + \dual{U}{1}_\Gamma\dual{V}{1}_\Gamma
    = \dual{(1/2-\adlo)\pi_\mesh^{p-1}\phi}{V}_\Gamma.
  \end{align*}
\end{myproposition}
\section{Localization of fractional order Sobolev norms}\label{section:localization}
The numerical analysis of boundary element methods takes place in the Sobolev spaces
$H^s(\Gamma)$ for $s\in[-1,1]$ that are defined in Section~\ref{section:bem}.
Apart from the exceptional cases $H^0(\Gamma) = L_2(\Gamma)$ and $H^1(\Gamma)$,
all other spaces are equipped with norms that are either non-local ($s\in(0,1)$) or additionally
impossible to compute ($s\in[-1,0)$). A norm is understood to be non-local if it is not possible
to split its square into contributions on the elements of a mesh, i.e., if one cannot write
\begin{align}\label{eq:norm:local}
  \norm{u}{H^s(\Gamma)}^2 \simeq \sum_{\el\in\mesh}\norm{u}{H^s(\el)}^2.
\end{align}
The spaces and dual spaces that are used in the variational
formulations of integral equations are usually equipped with non-local norms.
For example, the residual $V\Phi - f$ of a weakly singular integral equation is
computable but has to be measured in the non-local norm of the space $H^{1/2}(\Gamma)$.
However, only the knowledge of the residuals local contributions enables us to define
local error indicators that can be used for local mesh refinement in adaptive algorithms.
Different possibilities to localize a non-local norm are available and will be presented
in this section.
\subsection{Localization by local fractional order norms}\label{section:localization:slob}
In general,~\eqref{eq:norm:local} does not hold equivalently with constants that
are independent of the mesh $\mesh$. Fortunately, this is no longer true if
additional properties of the functions under consideration are assumed.
The following result is shown in~\cite{f00} for $d=2$ and in~\cite{f02} for $d=3$.
\begin{theorem}\label{thm:loc:slobodeckij:upper}
  If $\mesh$ is a mesh on $\Gamma$ and $s\in(0,1)$, then it holds for all $v\in H^s(\Gamma)$ that
  \begin{align}\label{thm:loc:slobodeckij:upper:eq}
    \snorm{v}{H^s(\Gamma)}^2 \leq
    \sum_{z\in\NN}\snorm{v}{H^s(\omega_z)}^2
    &+ \c{loc}\sum_{\el\in\mesh} h_\el^{-2s}\norm{v}{L_2(\el)}^2,
  \end{align}
  where $\c{loc}$ depends only on $s$ and $\Gamma$.
\end{theorem}
\begin{proof}
  In the following, we use the abbreviation
  \begin{align*}
    \int_Y \int_X := \int_Y \int_X\frac{\abs{v(x)-v(y)}^2}{\abs{x-y}^{d-1+2s}}\,dx\,dy.
  \end{align*}
  The idea of the proof is to write
  \begin{align}\label{thm:loc:slob:eq1}
    \snorm{v}{H^s(\Gamma)}^2 = \sum_{\el\in\mesh} \int_\el\int_{\omega_\el}
    + \sum_{\el\in\mesh} \int_\el\int_{\Gamma\setminus\omega_\el}
  \end{align}
  and bound the second term via the triangle inequality
  \begin{align*}
    \int_\el\int_{\Gamma\setminus\omega_\el}\lesssim\,
    &\int_\el \abs{v(y)}^2 \int_{\Gamma\setminus\omega_\el} \abs{x-y}^{-d+1-2s}\,dx\,dy\,+\\
    &\int_{\Gamma\setminus\omega_\el} \abs{v(x)}^2 \int_\el \abs{x-y}^{-d+1-2s}\,dx\,dy.
  \end{align*}
  One shows that the sum over all $\el\in\mesh$ of the second part on the right-hand side
  is the same as the sum over the first part, hence
  \begin{align*}
    \sum_{\el\in\mesh}\int_\el\int_{\Gamma\setminus\omega_\el}\lesssim\,
    \sum_{\el\in\mesh}\int_\el \abs{v(y)}^2 \int_{\Gamma\setminus\omega_\el}
    \abs{x-y}^{-d+1-2s}\,dx\,dy.
  \end{align*}
  Finally, direct calculation for $d=2$ and the use of polar coordinates for $d=3$ shows
  \begin{align*}
    \int_{\Gamma\setminus\omega_\el} \abs{x-y}^{-d+1-2s}\,dx \lesssim h_\el^{-2s}.
  \end{align*}
  The first term on the right-hand side of~\eqref{thm:loc:slob:eq1} can be estimated
  immediately via
  \begin{align*}
    \sum_{\el\in\mesh} \int_\el\int_{\omega_\el} \leq \sum_{z\in\NN}\snorm{v}{H^s(\omega_z)}^2,
  \end{align*}
  which finishes the proof.$\hfill\qed$
\end{proof}
The estimate~\eqref{thm:loc:slobodeckij:upper:eq} already provides a reliable localization of the
non-local $H^s$-norm, independent of the shape-regularity of the mesh.
However, choosing $v$ constant on $\Gamma$ shows that the reverse inequality
to~\eqref{thm:loc:slobodeckij:upper:eq}
cannot hold in general, i.e., the bound~\eqref{thm:loc:slobodeckij:upper:eq} is not efficient.
However, efficiency can be shown to hold when certain orthogonality is available.
More precisely, the following estimate from~\cite[Lemma~3.4]{f02} enables us to bound the $L_2$-terms on the right-hand side
of~\eqref{thm:loc:slobodeckij:upper:eq} by local $H^s$-terms.
The benefit will be twofold: First, it will enable us to show efficiency of the
localization~\eqref{thm:loc:slobodeckij:upper:eq} on shape-regular meshes.
Second, it provides us with another localization which is always efficient
as well as reliable on shape-regular meshes (Theorem~\ref{thm:loc:slob}).
\begin{mylemma}\label{lem:Hs:poinc}
  Let $\omega\subseteq\Gamma$ be a measurable set, $s\in(0,1)$, and $u\in H^s(\omega)$.
  Then,
  \begin{align*}
    \norm{u}{L_2(\omega)}^2 \leq \frac{\diam(\omega)^{d-1+2s}}{2\abs{\omega}}
    \snorm{u}{H^s(\omega)}^2 + \frac{1}{\abs{\omega}}\left(\int_\omega u(x)\,dx\right)^2.
  \end{align*}
  In particular, if $\mesh$ is a mesh on $\Gamma$ and $\dual{u}{\Psi_\el}_\Gamma = 0$\linebreak
  for $\Psi_\el\in\Pp^0(\mesh)$ the characteristic function of an element $\el\in\mesh$,
  \begin{align}\label{lem:Hs:poinc:eq2}
    \norm{u}{L_2(\el)}^2 \leq\frac{\sigma_\mesh h_\el^{2s}}{2}\snorm{u}{H^s(\el)}^2.
  \end{align}
  If $\dual{u}{\Psi_z}_\Gamma = 0$
  for $\Psi_z\in\Sp^1(\mesh)$ the hat-function function of a node $z\in\NN$,
  \begin{align}\label{lem:Hs:poinc:eq3}
    \norm{u}{L_2(\omega_z)}^2 \leq C \frac{\diam(\omega_z)^{d-1+2s}}{\abs{\omega_z}}
    \snorm{u}{H^s(\omega_z)}^2,
  \end{align}
  where the constant $C>0$ depends only on $d$.
\end{mylemma}
\begin{proof}
  To see the first estimate, note that
  \begin{align*}
    2\abs{\omega}\norm{u}{L_2(\omega)}^2 &= \int_\omega\int_\omega \abs{u(x)}^2\,dx\,dy
    + \int_\omega\int_\omega \abs{u(y)}^2\,dx\,dy\\
    &=\int_\omega\int_\omega \abs{u(x)-u(y)}^2\,dx\,dy\\
    &\qquad + 2\left(\int_\omega u(x)\,dx\right)^2.
  \end{align*}
  Since for all $x,y\in\omega$ with $x\neq y$ it holds
  \begin{align*}
    1 \leq \diam(\omega)^{d-1+2s}\frac{1}{\abs{x-y}^{d-1+2s}},
  \end{align*}
  the first estimate follows.
  To show~\eqref{lem:Hs:poinc:eq2}, choose $\omega=\el$ and observe that
  \begin{align*}
    \frac{\diam(\el)^{d-1+2s}}{2\abs{\el}} \leq \frac{\sigma_\mesh h_\el^{2s}}{2}
    \quad\text{ and }\\\quad\dual{u}{\Psi_\el}_\Gamma = \int_\el u(x)\,dx = 0.
  \end{align*}
  To show~\eqref{lem:Hs:poinc:eq3} note first that
  \begin{align*}
    \int_{\omega_z} u\, \Psi_z\,dx = 0 \quad\text{ for all } z\in\NN_\mesh,
  \end{align*}
  where $\Psi_z$ is the $\mesh$-piecewise linear and continuous basis function for $z$.
  The Cauchy-Schwarz inequality shows
  \begin{align*}
    \abs{\int_{\omega_z}u\,dx} = \abs{\int_{\omega_z}(1-\Psi_z)u\,dx}
    \leq \norm{1-\Psi_z}{L_2(\omega_z)} \norm{u}{L_2(\omega_z)}.
  \end{align*}
  A direct calculation shows that 
  \begin{align*}
    \norm{1-\Psi_z}{L_2(\omega_z)}^2 =
    \begin{cases}
      \abs{\omega_z}/2 \quad\text{ for } d=3,\\
      \abs{\omega_z}/3 \quad\text{ for } d=2,\\
    \end{cases}
  \end{align*}
  hence
  \begin{align*}
    \abs{\int_{\omega_z}u\,dx} \leq q\abs{\omega_z}^{1/2}\norm{u}{L_2(\omega_z)},
  \end{align*}
  for some $q<1$ depending only on $d$. This concludes~\eqref{lem:Hs:poinc:eq3}.
  $\hfill\qed$
\end{proof}
The two preceding results show that the localization~\eqref{thm:loc:slobodeckij:upper:eq}
is always reliable and, on shape-regular meshes, also efficient. We can combine the results also
to obtain a localization which is always efficient, and, on shape-regular meshes also reliable.
\begin{theorem}\label{thm:loc:slob}
  Denote by $\mesh$ a mesh on $\Gamma$ and let $s\in(0,1)$.
  Suppose that $\XX_\mesh$ is a discrete space with $\Pp^p(\mesh)\subseteq \XX_\mesh$
  or $\Sp^p(\mesh)\subseteq \XX_\mesh$. Then for all $v\in H^s(\Gamma)$ with
  $\int_\Gamma v \Psi\,dx = 0$ for all $\Psi\in \XX_\mesh$, it holds that
  \begin{align}\label{thm:loc:slob:eq}
    \norm{v}{H^s(\Gamma)}^2 \simeq
    \sum_{z\in\NN}\snorm{v}{H^s(\omega_z)}^2.
  \end{align}
  The constant in the lower bound depends only on $s$ and $\Gamma$, while the constant in the upper
  bound additionally depends on shape-regularity.
\end{theorem}
\begin{proof}
  The lower bound $\gtrsim$ follows immediately from the definition of the norms and
  the fact that every $\el\in\mesh$ is part of at most $d$ node-patches $\omega_z$.
  Now, we show the upper bound if $\Pp^p(\mesh)\subseteq \XX_\mesh$. Since
  $\Pp^0(\mesh)\subseteq\Pp^p(\mesh)$, the estimate~\eqref{lem:Hs:poinc:eq2} can be
  used in~\eqref{thm:loc:slobodeckij:upper:eq} to show
  \begin{align*}
    \snorm{u}{H^s(\Gamma)}^2 \leq \sum_{z\in\NN}\snorm{u}{H^s(\omega_z)}^2 +
    \c{loc}\frac{\sigma_\mesh}{2}\sum_{\el\in\mesh}\snorm{u}{H^s(\el)}^2,
  \end{align*}
  and furthermore, due to $h_\el^{2s}\leq\sigma_\mesh^s\abs{\el}^s\leq\sigma_\mesh^s\abs{\Gamma}^s$,
  \begin{align*}
    \norm{u}{L_2(\Gamma)}^2 \leq \frac{\sigma_\mesh^{1+s}\abs{\Gamma}^s}{2}
    \sum_{\el\in\mesh}\snorm{u}{H^s(\el)}^2.
  \end{align*}
  Using again the fact that every element $\el\in\mesh$ it part of at most $d$ patches $\omega_z$
  finally shows the upper bound in~\eqref{thm:loc:slob:eq}.

  Suppose now that $\Sp^p(\mesh)\subseteq \XX_\mesh$.
  Since\linebreak
  $\Sp^1(\mesh)\subseteq\Sp^p(\mesh)$, the upper bound follows from the same arguments
  as in the case $\Pp^p(\mesh)\subseteq \XX_\mesh$,
  using~\eqref{lem:Hs:poinc:eq3} instead
  of~\eqref{lem:Hs:poinc:eq2}.\hfill\qed
\end{proof}
The sets to which the norm on the left-hand side of~\eqref{thm:loc:slobodeckij:upper:eq}
is localized are overlapping. This overlap can be omitted if further assumptions on $u$
are imposed. The proof of the next Lemma first appeared in~\cite[Lemma~3.2]{tvp:thesis} if
norms are defined by a method called complex interpolation, and in~\cite[Thm.~4.1]{amt99}
if norms are defined by a method called real interpolation.
\begin{mylemma}\label{lem:loc1}
  Suppose $u\in \wilde H^s(\Gamma)$, $s\in[0,1]$ and $u|_\el\in \wilde H^s(\el)$ for all elements
  $\el\in\mesh$. Then,
  \begin{align*}
    \norm{u}{\wilde H^s(\Gamma)}^2 \leq \c{loc} \sum_{\el\in\mesh} \norm{u|_\el}{\wilde H^s(\el)}^2,
  \end{align*}
  where $\c{loc}$ is a constant independent of all the other involved quantities.
\end{mylemma}
\begin{proof}
  The cases $s=0,1$ are obvious, so choose $s\in(0,1)$. We will use the same abbreviation as in
  Theorem~\ref{thm:loc:slobodeckij:upper}. As
  $\norm{\wilde u}{H^s(\partial\Omega)}^2
  = \norm{\wilde u}{L_2(\partial\Omega)}^2 + \snorm{\wilde u}{H^s(\partial\Omega)}^2$,
  it suffices to consider the seminorm on the right-hand side. As $\wilde u = 0$ outside $\Gamma$,
  it follows $\wilde u(x)- \wilde u(y)=0$ outside
  $\Gamma\times\partial\Omega\cup\partial\Omega\times\Gamma$, hence
  \begin{align*}
    \snorm{\wilde u}{H^s(\partial\Omega)}^2
    \leq 2 \int_\Gamma \int_{\partial\Omega}
    \frac{\abs{\wilde u(x)-\wilde u(y)}^2}{\abs{x-y}^{d-1+2s}}\,dx\,dy
    =: 2 \int_\Gamma \int_{\partial\Omega}.
  \end{align*}
  We split the double integral in
  \begin{align*}
    \int_\Gamma \int_{\partial\Omega}
    = \sum_{\el\in\mesh}\int_\el\int_\el
    + \sum_{\el\in\mesh}\int_\el\int_{\Gamma\setminus\el}
    + \sum_{\el\in\mesh}\int_\el\int_{\partial\Omega\setminus\Gamma}.
  \end{align*}
  Now, for the first term,
  \begin{align*}
    \int_\el\int_\el
    &= \int_\el\int_\el\frac{\abs{\wilde u(x)-\wilde u(y)}^2}{\abs{x-y}^{d-1+2s}}\,dx\,dy\\
    &= \int_\el\int_\el\frac{\abs{u|_\el(x)- u|_\el(y)}^2}{\abs{x-y}^{d-1+2s}}\,dx\,dy
    \leq \norm{u|_\el}{\wilde H^s(\el)}^2.
  \end{align*}
  For the second term, estimate as in the proof of Theorem~\ref{thm:loc:slobodeckij:upper}
  \begin{align*}
    \sum_{\el\in\mesh}\int_\el\int_{\Gamma\setminus\el}\lesssim\,
    \sum_{\el\in\mesh}\int_\el
    \abs{\wilde u(y)}^2 \int_{\Gamma\setminus\el} \abs{x-y}^{-d+1-2s}\,dx\,dy
  \end{align*}
  where we note that the integral on the right-hand side exists due to
  $\wilde u|_\el\in \wilde H^{s}(\el)$. We conclude
  \begin{align*}
    &\sum_{\el\in\mesh}
    \int_\el \abs{\wilde u(y)}^2 \int_{\Gamma\setminus\el} \abs{x-y}^{-d+1-2s}\,dx\,dy\lesssim\\
    &\quad\sum_{\el\in\mesh}\int_{\partial\Omega}\int_{\partial\Omega}
    \frac{\abs{\wilde{u|_\el}(x)-\wilde{u|_\el}(y)}^2}{\abs{x-y}^{d-1+2s}}\,dx\,dy
    =\sum_{\el\in\mesh} \norm{u|_\el}{\wilde H^s(\el)}^2.
  \end{align*}
  For the last term, note that
  \begin{align*}
    \int_\el\int_{\partial\Omega\setminus\Gamma}
    &= \int_\el\int_{\partial\Omega\setminus\Gamma}
    \frac{\abs{\wilde u(x)-\wilde u(y)}^2}{\abs{x-y}^{d-1+2s}}\,dx\,dy\\
    &= \int_{\partial\Omega}\int_{\partial\Omega}
    \frac{\abs{\wilde{u|_\el}(x)-\wilde{u|_\el}(y)}^2}{\abs{x-y}^{d-1+2s}}\,dx\,dy
    \lesssim \norm{u|_\el}{\wilde H^s(\el)}^2.
  \end{align*}
  $\hfill\qed$
\end{proof}
\subsection{Localization by approximation}\label{section:localization:apx}
This localization technique is employed to derive a~posteriori error estimators based on the so-called
$(h-h/2)$ methodology, cf. Section~\ref{section:est:hh2}.
The idea of this approach is to compare the solution on the current mesh with a solution
on a finer mesh. The energy norm of the difference of the two solutions is an efficient
and (under a certain assumption) reliable error estimator.
To localize the energy norm of the difference of two solutions, one uses
approximation operators to bound the (non-local) energy norm from above by a stronger integer
order (hence local) norm. This can be done
by employing approximation estimates for the approximation operators used.
The energy spaces of weakly singular and hypersingular equations require a different amount
of smoothness, and therefore discontinuous as well as continuous approximation operators
will be considered in this Section.
To show efficiency of the $(h-h/2)$ type estimators, \textit{inverse estimates} will be needed.
These are the counterparts to approximation estimates and are important tools in finite
and boundary elements.
\begin{mydefinition}\label{def:L2projection}
  For a given mesh $\mesh$, denote by $\pi_\mesh^p:L_2(\Gamma)\rightarrow \Pp^p(\mesh)$
  and $\Pi_\mesh^p:L_2(\Gamma)\rightarrow\Sp^p(\mesh)$
  the $L_2(\Gamma)$-orthogonal projections, which are uniquely characterized by
  \begin{align*}
    \dual{\pi_\mesh^p \phi}{\Phi}_\Gamma &= \dual{\phi}{\Phi}_\Gamma \quad \text{ for all }
    \Phi\in\Pp^p(\mesh),\\
    \dual{\Pi_\mesh^p u}{U}_\Gamma &= \dual{u}{U}_\Gamma \quad \text{ for all }
    U\in\Sp^p(\mesh).
  \end{align*}
\end{mydefinition}
We also write $\pi_\mesh:=\pi^0_\mesh$ and $\Pi_\mesh := \Pi^1_\mesh$.
The approximation properties of $\pi_\mesh^p$ can be stated as follows.
\begin{mylemma}\label{lem:L2:Pp:apx}
  For $\phi\in H^s(\Gamma)$, $s\in[0,1]$, and $r\in(0,1]$ holds
  \begin{align*}
    \norm{\phi-\pi_\mesh^p\phi}{L_2(\el)} &\leq \c{approx}h_\el^s \snorm{\phi}{H^s(\el)},\\
    \norm{\phi-\pi_\mesh^p\phi}{H^{-r}(\Gamma)}^2 &\leq
    \c{approx}\sum_{\el\in\mesh}h_\el^{2(s+r)} \snorm{\phi}{H^s(\el)}^2.
  \end{align*}
  The second estimate is also true if $\wilde H^{-r}(\Gamma)$ is used instead of
  $H^{-r}(\Gamma)$.
  The constant $\c{approx}>0$ depends only on $s$, and in the cases $s\in(0,1)$ or $r\in (0,1)$
  it additionally depends on the shape-regularity $\sigma_\mesh$.
\end{mylemma}
\begin{proof}
  For $p=1$, the first estimate is proven by a scaling argument, cf.~\cite[Thm.~10.2]{s08}.
  This special case extends immediately to general $p\in\N$ by the best approximation property
  of $\pi^p_\mesh$.
  If $s\in\left\{ 0,1 \right\}$, the constant $\c{approx}$ does only depend on $s$ but not
  on shape-regularity, which is seen by a careful inspection of the scaling argument.
  The second estimate is a slight refinement
  of~\cite[Thm.~10.3]{s08}: For $v\in \wilde H^r(\Gamma)$ holds
  \begin{align*}
    \dual{\phi-\pi_\mesh^p\phi}{v}_\Gamma
    &= \dual{h^r (\phi-\pi_\mesh^p\phi)}{h^{-r}(v-\pi_\mesh^p v)}_\Gamma\\
    &\leq \norm{h^r (\phi-\pi_\mesh^p\phi)}{L_2(\Gamma)} \norm{h^{-r}(v-\pi_\mesh^p v)}{L_2(\Gamma)}\\
    &\lesssim \left( \sum_{\el\in\mesh}h_\el^{2(s+r)} \snorm{\phi}{H^s(\el)}^2 \right)^{1/2}
    \norm{v}{\wilde H^r(\Gamma)},
  \end{align*}
  The same estimate holds true if we choose $v\in H^r(\Gamma)$.
  The result follows by the dual definition of the\linebreak
  $H^{-r}(\Gamma)$ and
  $\wilde H^{-r}(\Gamma)$-norm.
  $\hfill\qed$
\end{proof}
Inverse estimates in the context of the last lemma are proven in~\cite[Thm.~3.6]{ghs05}:
\begin{mylemma}\label{lem:Pp:invest}
  For $\mesh$ a mesh on $\Gamma$ and $s\in[0,1]$ holds
  \begin{align*}
    \norm{h_\mesh^{s}\Phi}{L_2(\Gamma)} \leq \c{inv}\norm{\Phi}{\wilde H^{-s}(\Gamma)}
    \quad\text{ for all } \Phi\in\Pp^p(\mesh).
  \end{align*}
  The constant $\c{inv}>0$ depends only on the shape-regularity of $\mesh$, $p\in\N$, and $s$.
\end{mylemma}
The approximation in fractional order spaces by continuous functions is a little bit more involved.
For the proof of the following two lemmata, we refer to~\cite{kop}
and~\cite[Prop.~5 and Lem.~7]{affkp13}.
\begin{mylemma}\label{lem:Sp:apx}
  For $s\in[0,1]$, each $\wilde H^s(\Gamma)$-stable projection
  $J_\mesh:\wilde H^s(\Gamma)\rightarrow\wilde\Sp^p(\mesh)$ satisfies
  \begin{align*}
    \norm{v-J_\mesh v}{\wilde H^s(\Gamma)}
    \leq \c{approx} \min_{V\in\Sp^p(\mesh)}\norm{h_\mesh^{1-s}\nablag(v-V)}{L_2(\Gamma)}
  \end{align*}
  for all $v\in \wilde H^1(\Gamma)$. The constant $\c{approx}$ depends only on
  $\Gamma$, $p\in\N$, $s$, shape-regularity of $\mesh$, and the stability constant of
  $J_\mesh$.
\end{mylemma}
\begin{mylemma}\label{lem:Sp:invest}
  For $\mesh$ a mesh on $\Gamma$ and $s\in[0,1]$ holds
  \begin{align*}
    \norm{h_\mesh^{1-s}\nablag V}{L_2(\Gamma)} \leq \c{inv} \norm{V}{H^s(\Gamma)}
    \quad\text{ for all } V\in\Sp^p(\mesh).
  \end{align*}
  The constant $\c{inv}>0$ depends only on the shape-regularity of $\mesh$, $p\in\N$, and $s$.
\end{mylemma}
Lemma~\ref{lem:Sp:apx} holds for any projection $J_\mesh:\wilde H^s(\Gamma)\rightarrow\Sp^p(\mesh)$
which is stable, i.e., for all $v\in \wilde H^s(\Gamma)$ holds
\begin{align*}
  \norm{J_\mesh v}{\wilde H^s(\Gamma)} \leq \c{stab} \norm{v}{\wilde H^s(\Gamma)},
\end{align*}
and the constant $\c{stab}$ does not depend on $v$.
For an implementation of an associated $(h-h/2)$ error estimator,
the operator $J_\mesh$ needs to be computed. Possible candidates are presented in the following.
\subsubsection{The $L_2(\Gamma)$ projection onto $\Sp^p(\mesh)$}\label{section:L2}
The $L_2(\Gamma)$ orthogonal projection $\Pi_\mesh^p$ onto $\Sp^p(\mesh)$ from
Definition~\ref{def:L2projection} is an easy-to-implement candidate for $J_\mesh$ in
Lemma~\ref{lem:Sp:apx}. The parameter $s$ will subsequently be chosen to be greater than $0$, such that
the $\wilde H^s(\Gamma)$-stability of $\Pi_\mesh^p$
\begin{align}\label{eq:L2projection:stab}
  \norm{\Pi_\mesh^p u}{\wilde H^s(\Gamma)} \leq \c{stab}\norm{u}{\wilde H^s(\Gamma)},
\end{align}
needs to be available to use Lemma~\ref{lem:Sp:apx}.
While there holds~\eqref{eq:L2projection:stab} for $s=0$ and $\c{stab}=1$
without any assumption on $\mesh$, this might not be the case for $s>0$.
It can be shown that~\eqref{eq:L2projection:stab} holds for $s>0$ on a sequence of meshes
where the quotient of the biggest and smallest element stays bounded,
and $\c{stab}$ depends on this bound, cf.~\cite{bx91}. However, as we will deal with
adaptively refined meshes, this quotient will not stay bounded on the (infinite) sequence of meshes
that we investigate. However, the fact that a sequence of adaptively refined meshes exhibits a strong
structure can be used in order to show useful results. Existing works to this topic
include~\cite{by13,bps02,c02,c04,ct87,ej95,kpp13,s01}.
We refer to Section~\ref{section:L2projection:H1stability} for a detailed discussion.
\subsubsection{The Scott-Zhang projection}\label{section:sz}
The Scott-Zhang projection, developed in~\cite{sz90}, is widely used in
numerical analysis. It is a linear and bounded projection onto $\Sp^p(\mesh)$
which is defined on $H^{1/2+\varepsilon}(\Gamma)$ for $\varepsilon>0$.
The energy spaces in BEM, usually variants of $H^{1/2}(\Gamma)$, lack the regularity
necessary for the classical definition. However, if the energy space
is $\wilde H^{1/2}(\Gamma)$, an operator that maps into a space with zero boundary conditions
is needed. A slightly modified derivation is therefore necessary and will be presented here.
For ease of presentation, we present the details only for $\Sp^1(\mesh)$.
Suppose that $\left\{ z_i \right\}_{i=1}^N$ is the collection of degrees of freedom
for $\Sp^1(\mesh)$, which are ordered in a way such that $\left\{ z_i \right\}_{i=\wilde N+1}^N$
are on the boundary $\partial\Gamma$ (if $\Gamma$ is open, of course).
For every $z_i$, we choose an element $T_i$ with $z_i\in\overline\el_i$.
Denote by $\left\{ \phi_{i,j} \right\}_{j=1}^d$ the nodal basis of $\Pp^1(\el_i)$,
and by $\left\{ \psi_{i,k} \right\}_{k=1}^d$ an $L_2(\el_i)$-dual basis defined by
\begin{align}\label{eq:dualbasis}
  \int_{\el_i} \psi_{i,k}\phi_{i,j}\,dx = \delta_{k,j}.
\end{align}
Set $\psi_{i}$ to be the dual basis function of $\phi_{i,j}$ with $\phi_{i,j}(z_i)=1$,
and denote by $\left\{ \eta_i \right\}_{i=1}^N$ the nodal basis of $\Sp^1(\mesh)$
with $\eta_j(z_k)=\delta_{j,k}$.
The Scott-Zhang operators are defined for $v\in L_1^\loc(\Gamma)$ via
\begin{align*}
  J_\mesh v = \sum_{i=1}^N \eta_i \int_{\el_i}\psi_i v\,dx
  \quad\text{ and }\quad
  \wilde J_\mesh v = \sum_{i=1}^{\wilde N} \eta_i \int_{\el_i}\psi_i v\,dx.
\end{align*}
The following results can be shown using arguments\linebreak
from~\cite{sz90}, cf.~\cite{affkp13}.
\begin{mylemma}\label{lem:sz}
  The operator $J_\mesh: L_1^\loc(\Gamma)\rightarrow\Sp^p(\mesh)$
  is stable for all $s\in[0,1]$, i.e.,
  \begin{align*}
    \norm{J_\mesh v}{H^s(\Gamma)} \leq \c{stab}\norm{v}{H^s(\Gamma)}
    \quad\text{ for all } v\in H^s(\Gamma).
  \end{align*}
  The operator $\wilde J_\mesh: L_1^\loc(\Gamma)\rightarrow\wilde\Sp^p(\mesh)$
  is stable for all $s\in[0,1]$, i.e.,
  \begin{align*}
    \norm{J_\mesh v}{\wilde H^s(\Gamma)} \leq \c{stab}\norm{v}{\wilde H^s(\Gamma)}
    \quad\text{ for all } v\in \wilde H^s(\Gamma).
  \end{align*}
\end{mylemma}
\subsubsection{Nodal interpolation}\label{section:nodal}
The nodal interpolator $J_\mesh:C(\overline\Gamma)\rightarrow\Sp^p(\mesh)$
is without doubt the easiest approximation operator when
it comes to implementation, as
\begin{align*}
  J_\mesh v = \sum_{i=1}^N v(z_i)\eta_i,
\end{align*}
where $\left\{ z_i \right\}_{i=1}^N$ and $\left\{ \eta_i \right\}_{i=1}^N$ are again the degrees
of freedom and it's associated nodal basis, i.e., $\eta_j(z_k)=\delta_{j,k}$.
However, $J_\mesh$ needs point evaluation, which is
only a stable operation on curves (i.e., $d=2$, cf. Section~\ref{section:bio}) due
to the Sobolev embedding theorem. The following result is essentially proved in~\cite[Thm.~1]{c97}
and~\cite[Cor.~3.4]{cp07} for $p=1$, but transfers verbatim to $p\geq 1$.
\begin{mylemma}\label{lem:nodalinterpolation:1d}
  For $d=2$, i.e., $\Gamma\subseteq\Omega$ a one-dimensional curve, it holds for all
  $v\in H^1(\Gamma)$ that
  \begin{align*}
    \norm{v - J_\mesh v}{H^s(\Gamma)} \leq \c{approx}\norm{h_\mesh^{1-s}v^\prime}{L_2(\Gamma)},
  \end{align*}
  where $\c{approx}>0$ depends only $\Gamma$, $\sigma_\mesh$, and $p$.
\end{mylemma}
Contrary, for surfaces, i.e., $d=3$, point evaluation is not a stable operation and a result
analogous to the last one cannot hold.
A remedy that can be used at least in $(h-h/2)$ error
estimation is that nodal interpolation can be shown to be a stable operation when the function
to be approximated is discrete on a finer scale, and the scales do not differ too much.
The following result captures this idea in a mathematical sense, for a proof see~\cite{affkp13}.
\begin{mylemma}\label{lem:nodalinterpolation}
  Consider a mesh $\mesh$ together with its uniform refinement $\widehat\mesh$, cf.
  Section~\ref{section:meshrefinement}. For $q\geq p$, the nodal interpolation operator
  $J_\mesh:\wilde\Sp^q(\widehat\mesh)\rightarrow\wilde\Sp^p(\mesh)$ satisfies for $s\in[0,1]$
  \begin{align*}
    \norm{(1-J_\mesh)\widehat V}{\wilde H^s(\Gamma)}
    \leq \c{approx} \min_{V\in\Sp^p(\mesh)}\norm{h_\mesh^{1-s}\nablag(\widehat V - V )}{L_2(\Gamma)}
  \end{align*}
  as well as
  \begin{align*}
    \norm{h_\mesh^{1-s}\nablag(1-J_\mesh)\widehat V}{L_2(\Gamma)}
    \leq \c{stab}\norm{h_\mesh^{1-s}(1-\Pi_\mesh^p)\nablag\widehat V}{L_2(\Gamma)}
  \end{align*}
  for all $\widehat V\in\wilde\Sp^q(\widehat\mesh)$. The constant $\c{approx}>0$
  depends only on $\Gamma$, $p$, $q$, $s$, and the shape-regularity $\sigma_\mesh$,
  whereas the constant $\c{stab}$ depends only on the shape-regularity $\sigma_\mesh$.
\end{mylemma}
The implementation of the nodal interpolation operator is straight forward and will not be discussed
further.
\subsection{Localization by multilevel norms}\label{section:localization:multilevel}
The localization techniques discussed in
Sections~\ref{section:localization:slob}--\ref{section:localization:multilevel}
depend on either orthogonality and/or approximation.
If neither of those properties is available, one can still use so-called
\textit{multilevel norms} to localize fractional-order Sobolev norms of discrete functions.
The following theorem and its proof are found in~\cite{oswald:98}. We will not give
the proof here, as it involves deeper mathematical results such as Besov spaces.
\begin{theorem}\label{thm:multilevel}
  Let $\left( \mesh_\ell \right)_{\ell\in\N_0}$ be a uniform sequence of meshes on $\Gamma$ with
  corresponding mesh-width $h_\ell$. Denote by\linebreak
  $\pi_\ell:L_2(\Gamma)\rightarrow \Pp^0(\mesh_\ell)$ the
  $L_2$ orthogonal projections and $\pi_{-1}=0$.
  Then, there are constants $\setc{ml1}, \setc{ml2}>0$ such that for all $L\geq0$ and all
  $\Phi_L\in\Pp^0(\mesh_L)$ it holds that
  \begin{align*}
    \c{ml1}\norm{\Phi_L}{\wilde H^{-1/2}(\Gamma)}^2
    &\leq \sum_{\ell=0}^L h_\ell \norm{(\pi_\ell-\pi_{\ell-1})\Phi_L}{L_2(\Gamma)}^2\\
    &\leq \c{ml2}(L+1)^2\norm{\Phi_L}{\wilde H^{-1/2}(\Gamma)}^2.
  \end{align*}
\end{theorem}
The last theorem gives a reliable, computable bound for the $\wilde H^{-1/2}$-norm of a discrete
function. The upper bound in contrast depends on the number of levels $L$ that are involved.
This upper bound cannot be improved in general, cf.~\cite{oswald:98}.
\section{A posteriori error estimators for the $h$-version}\label{section:aposteriori}
The continuous solution $u$ of any of our model problems is in general unknown, and so
is the error $u-U$, where $U$ is a discrete approximation to $u$
from a discrete space $\XX_\mesh$ based on a mesh $\mesh$.
The idea of a~posteriori error estimation is to estimate the error $u-U$ in order to
\begin{itemize}
  \item use an estimate for the \textit{global} error as a stopping criterion, or
  \item to use local contributions of the error to decide where to refine the mesh locally.
\end{itemize}
Adaptive algorithms clearly need estimators that provide local contributions, which will
be written as
\begin{align*}
  \eta_\mesh = \left( \sum_{\el\in\mesh}\eta_\el^2 \right)^{1/2}
  \text{ or } \eta_\mesh = \left( \sum_{j\in J} \eta_j^2 \right)^{1/2},
\end{align*}
where $J$ is a certain index set (e.g., the set of nodes).
If the error estimator does not underestimate the error, i.e.,
\begin{align*}
  \norm{u-U}{}\leq \c{rel}\eta_\mesh
\end{align*}
holds true, then $\eta_\mesh$ is said to be \textit{reliable}.
Likewise, if the error estimator does not overestimate the error, i.e.,
\begin{align*}
  \eta_\mesh\leq \c{eff}\norm{u-U}{}
\end{align*}
holds true, it is called \textit{efficient}.
Here, $\norm{\cdot}{}$ is a norm of interest, typically the energy
norm of the problem. Usually, $\c{rel},\c{eff}>0$ are unknown (except for $(h-h/2)$ estimators,
where $\c{eff}=1$), but do not depend on the current mesh $\mesh$.

In this Section, we present different approaches for a~posteriori error estimation in boundary
element methods that are available in the mathematical literature.
The focus is to show reliability and efficiency and give an overview on the available approaches.
Frequently, we will identify a bilinear form $b:\XX\times\XX\rightarrow\R$ with an operator
$B:\XX\rightarrow\XX'$ via
\begin{align*}
  b(v,w) = \dual{Bv}{w}_{\XX'\times\XX},
\end{align*}
where $\dual{\cdot}{\cdot}_{\XX'\times\XX}$ is the chosen duality pairing.
\subsection{Residual type estimators}\label{section:residual}
While the exact solution $u$ of an equation is unknown, the residual
$R := F - BU$ is a computable quantity. Here, $B$ is the involved operator (i.e., the simple
layer $\slo$ or the hypersingular operator $\hyp$), and $F$ is the corresponding right-hand side.
In boundary element methods, the residual is usually measured in a non-local fractional
Sobolev norm, and the different approaches for residual error estimation differ in their approach
for localization of this norm.
\subsubsection{Babu\v{s}ka-Rheinboldt-estimators}\label{section:est:br}
In~\cite{f98}, Faermann extended the estimators that were developed for finite element methods
by Babu\v{s}ka and\linebreak
Rheinboldt~\cite{br78} to fractional order Sobolev norms.
We will sketch the ideas for the case of the hypersingular
integral equation and lowest-order discretization, and comment on the other cases afterwards.
To that end, denote by $u\in H^{1/2}(\Gamma)$ the exact solution to the Neumann
problem, see Proposition~\ref{prop:neumann}, and by $U\in\Sp^1(\mesh)$
the Galerkin solution, see Proposition~\ref{prop:galerkin:neumann}
\begin{mydefinition}\label{def:hyp:br}
  Denote by $\left\{ \Psi_j \right\}_{j=1}^N$ the nodal basis of $\Sp^1(\mesh)$
  and by $R:=(1/2-K')\phi - \hyp U$ the residual.
  The BR-type error estimator is defined by
  \begin{align*}
    \eta_\mesh^2 = \sum_{j=1}^N \eta_j^2,\quad \text{ where }\quad
    \eta_j :=
    \sup_{\substack{
      v\in H^{1/2}(\Gamma)\\
      \Psi_jv \neq 0
    }}
    \frac{\dual{R}{\Psi_jv}_\Gamma}{\norm{\Psi_jv}{H^{1/2}(\Gamma)}}
  \end{align*}
\end{mydefinition}
For the hypersingular integral operator, the BR-type estimators are reliable and efficient.
\begin{theorem}\label{thm:br:hyp}
  There are constants $\c{rel},\c{eff}>0$ such that
  \begin{align}\label{thm:br:hyp:eq}
    \c{eff}^{-2}\eta_\mesh^2 \leq \norm{u-U}{\wilde H^{1/2}(\Gamma)}^2 \leq \c{rel}^2 \eta_\mesh^2.
  \end{align}
  The constant $\c{rel}$ depends on the shape-regularity of $\mesh$,
  whereas the constant $\c{eff}$ depends also on $\Gamma$.
\end{theorem}
\begin{proof}
  We sketch the ideas of~\cite[Sect. 6]{cf}.
  The exact solution satisfies $u\in H^{1/2}_0(\Gamma)$, hence
  $\hyp u = (1/2-K')\phi$ in $H^{1/2}(\Gamma)$.
  As $\hyp^{-1}$ is linear and bounded, it follows that
  $\norm{u-U}{\wilde H^{1/2}(\Gamma)} \lesssim \norm{R}{H^{-1/2}(\Gamma)}$, and
  \begin{align*}
    \norm{R}{H^{-1/2}(\Gamma)} &= \sup_{\norm{v}{H^{1/2}(\Gamma)}=1}\dual{R}{v}_\Gamma\\
    &\leq \eta_\mesh \sup_{\norm{v}{H^{1/2}(\Gamma)}=1}
    \left( \sum_{j=1}^N \norm{(v-z_j)\Psi_j}{H^{1/2}(\Gamma)}^2 \right)^{1/2},
  \end{align*}
  with arbitrary $z_j\in\R$ for $j=1,\dots,N$. Now, with\linebreak
  $\omega_j = \supp(\Psi_j)$, it holds
  \begin{align*}
    \snorm{(v-z_j)\Psi_j}{H^{1/2}(\Gamma)} &\lesssim
    \snorm{v}{H^{1/2}(\omega_j)}\\
    &\quad+ \diam(\omega_j)^{-1/2}\norm{v-z_j}{L_2(\omega_j)}.
  \end{align*}
  Choosing $z_j$ according to the variant of the Poincar\'e inequality
  of~\cite[Thm.~7.1]{ds80} shows the upper bound in~\eqref{thm:br:hyp:eq}.
  To show the lower bound,
  note that the index set $J=\left\{ 1, \dots, N \right\}$ can be
  decomposed into at most $M$ pairwise disjoint subsets $J_k$, $k=1,\dots, M$, and $M$ depends
  only on the shape-regularity of $\mesh$, such that the supports of the basis functions
  $\left\{ \Psi_j \right\}_{j\in J_k}$ are pairwise disjoint.
  Due to the latter property, it is possible to choose for an arbitrary collection
  of functions
  $v_j\in H^{1/2}(\Gamma)$, $j\in J_k$, a function $w_j\in H^{1/2}(\Gamma)$ such that
  on the support of $\Psi_j$ it holds
  \begin{align*}
    w_j = \frac{\dual{R}{\Psi_j v_j}_\Gamma}{\norm{\Psi_j v_j}{H^{1/2}(\Gamma)}^2} v_j.
  \end{align*}
  It follows that
  \begin{align*}
    \sum_{j\in J_k}\norm{\Psi_j w_j}{H^{1/2}(\Gamma)}^2
    = \sum_{j\in J_k} \frac{\dual{R}{\Psi_j v_j}_\Gamma^2}{\norm{\Psi_j v_j}{H^{1/2}(\Gamma)}^2}
    = \dual{R}{\sum_{j\in J_k}\Psi_j w_j}_\Gamma,
  \end{align*}
  and hence
  \begin{align*}
    \sum_{j\in J_k} \frac{\dual{R}{\Psi_j v_j}_\Gamma^2}{\norm{\Psi_j v_j}{H^{1/2}(\Gamma)}^2}
    &= \frac{
      \left(\dual{R}{\sum_{j\in J_k} \Psi_j w_j}_\Gamma\right)^2
    }{\sum_{j\in J_k}\norm{\Psi_j w_j}{H^{1/2}(\Gamma)}^2}\\
    &\leq \c{loc} \norm{R}{H^{-1/2}(\Gamma)}^2,
  \end{align*}
  where $\c{loc}$ is the constant from Lemma~\ref{lem:loc1}.
  Since the $v_j$ can be chosen arbitrarily, we conclude that
  \begin{align*}
    \sum_{j\in J_K}\eta_j^2 \leq \c{loc}\norm{R_N}{H^{-1/2}(\Gamma)}^2
    \leq \c{loc}\norm{\hyp}{}^2 \norm{u-U_N}{\wilde H^{1/2}(\Gamma)}^2,
  \end{align*}
  where $\norm{\hyp}{} = \norm{\hyp}{\wilde H^{1/2}(\Gamma)\rightarrow H^{-1/2}(\Gamma)}$.
  Hence, the lower bound in~\eqref{thm:br:hyp:eq} follows with
  $\c{eff}^2 = M\c{loc}\norm{\hyp}{}^2$.\hfill\qed
\end{proof}
In~\cite{f98}, a result like Theorem~\ref{thm:br:hyp} is proven for bijective, continuous
operators $B:H^\alpha(\Gamma)\rightarrow H^{-\alpha}(\Gamma)$, $\alpha\in\R$,
that satisfy the G\r{a}rding inequality
\begin{align}\label{eq:garding}
  \abs{b(v,v)} \geq \c{ell}\norm{v}{H^\alpha(\Gamma)}^2 -
  \c{garding}\norm{v}{H^{\alpha-\delta}(\Gamma)}
\end{align}
for all $v\in H^\alpha(\Gamma)$ and some $\delta > 0 $, where $\c{ell}>0$ and $\c{garding}\geq0$.
As in the case of the hypersingular operator, estimators of BR-type are associated to a basis
$\left\{ \Psi_j \right\}_{j=1}^N$ of the discrete trial space $\XX_\mesh$,
which is supposed to fulfill the following assumptions.
\begin{assumption}\label{ass:br}
  There are constants $M\in\N$ and $\c{faermann:assumption:stab}>0$ such that
  \begin{itemize}
    \item[(i)] The basis $\left\{ \Psi_j \right\}_{j=1}^N$ can be partitioned into $M$
      sets of basis functions with mutually disjoint
      support, i.e., there are at most $M$ disjoint subsets $I_k$, $k=1, \dots, M$
      with $I_k \subseteq \left\{ 1,\dots, N \right\}$
      such that
      \begin{align*}
	\supp\left( \Psi_m \right) \cap \supp\left( \Psi_n \right) = \emptyset \quad\text{ for } 
	m,n\in I_k, m\neq n.
      \end{align*}
    \item[(ii)] The basis $\left\{ \Psi_j \right\}_{j=1}^N$ is a partition of unity, i.e.,
      \begin{align*}
	\sum_{j=1}^N \Psi_j(x) = 1 \quad\text{ for almost all } x\in\Gamma.
      \end{align*}
    \item[(iii)] For every function $v\in H^{\max\left\{ \alpha,0 \right\}}(\Gamma)$
      there is a function $V\in \XX_\mesh$ such that
      \begin{align*}
	\sum_{j=1}^N \norm{\Psi_j(v-V)}{H^\alpha(\Gamma)}^2 \leq \c{faermann:assumption:stab}
	\norm{v}{H^\alpha(\Gamma)}^2.
      \end{align*}
    \item[(iv)] For all $v\in H^{\alpha+\delta}(\Gamma)$ exists a function $V_N\in \XX_\mesh$ such that
      \begin{align*}
	\norm{v-V_N}{H^{\alpha}(\Gamma)} \leq \c{approx}\max_{\el\in\mesh}h^\delta
	\norm{v}{H^{\alpha+\delta}(\Gamma)},
      \end{align*}
      and $\c{approx}$ depends only on the shape-regularity $\sigma_\mesh$ and $\Gamma$.
  \end{itemize}
\end{assumption}
For $d=2$, the standard bases of $\Pp^p(\mesh)$ or $\Sp^p(\mesh)$ always fulfill \textit{(i)}.
By standard bases, we mean functions having support only on one element for $\Pp^p(\mesh)$,
and the classical finite-element \textit{hat-functions} for $\Sp^p(\mesh)$.
For $d=3$, the standard basis of $\Pp^p(\mesh)$ always fulfills \textit{(i)},
while for the standard basis of $\Sp^p(\mesh)$ the constant $M$ depends on shape-regularity.
Independent of $d$, the assumption \textit{(ii)} can always be fulfilled as long as there
are no boundary conditions imposed, i.e., in the case of $\Gamma$ being not a closed boundary,
the space $\wilde\Sp^p(\mesh)$ cannot fulfill \textit{(ii)}.
Assumption \textit{(iii)} holds for $d=2$ for $\Pp^0(\mesh)$ and $\alpha=-1/2$. In the case
of $\Sp^1(\mesh)$ and $\alpha=1/2$, the constant $\c{faermann:assumption:stab}$
depends on the shape-regularity of $\mesh$.
In $d=3$, it holds for $\Pp^0(\mesh)$ and $\alpha=-1/2$ or $\Sp^1(\mesh)$ and $\alpha=1/2$, and
the constant $\c{faermann:assumption:stab}$ depends on the shape-regularity of $\mesh$ in both cases.
\begin{remark}
  Note that the basis $\Psi_j$ is only needed for the computation of the BR-indicators.
\end{remark}
In the general case, the following result together with a proof can be found
in~\cite[Thm.~5.2]{f98}.
\begin{theorem}
  Denote by $B:H^\alpha(\Gamma)\rightarrow H^{-\alpha}(\Gamma)$, $\alpha\in\R$,
  a linear and bounded operator which satisfies the G\r{a}rding inequality~\eqref{eq:garding}.
  If $\c{garding}>0$, assume that $B:H^{\alpha-\delta}(\Gamma)\rightarrow\linebreak
  H^{-\alpha-\delta}(\Gamma)$
  is bijective and continuous, and denote by $\XX_\mesh\subset H^{\alpha}(\Gamma)$ a discrete
  space which fulfills Assumption~\ref{ass:br}.
  If $u\in H^{\max\left\{ \alpha,0 \right\}}(\Gamma)$ and $U\in \XX_\mesh$
  are the exact and the Galerkin solution of
  \begin{align*}
    \dual{Bu}{v}_\Gamma &= \dual{F}{v}_\Gamma \quad \text{ for all } v\in \wilde H^{-\alpha}(\Gamma),\\
    \dual{BU}{V}_\Gamma &= \dual{F}{V}_\Gamma \quad \text{ for all } V\in \XX_\mesh,
  \end{align*}
  denote the residual by $R:= F - BU$ and define a BR-type estimator by
  \begin{align*}
    \eta_\mesh^2 = \sum_{j=1}^N \eta_j^2,\quad \text{ where }\quad
    \eta_j :=
    \sup_{\substack{
      v\in H^{\alpha}(\Gamma)\\
      \Psi_jv \neq 0
    }}
    \frac{\dual{R}{\Psi_jv}_\Gamma}{\norm{\Psi_jv}{H^{\alpha}(\Gamma)}}.
  \end{align*}
  Then, the following holds:
  \begin{itemize}
    \item If $B$ is elliptic, i.e., $\c{garding}=0$, $\eta_\mesh$ is reliable, i.e.,
      \begin{align}\label{br:rel}
	\norm{u-U}{H^{\alpha}(\Gamma)} \leq \c{rel} \eta_\mesh.
      \end{align}
    \item If $B$ is not elliptic but satisfies a G\r{a}rding inequality,\linebreak
      then~\eqref{br:rel} holds
      if $h_\mesh$ is sufficiently small.
    \item For $\alpha\geq0$ holds efficiency
      \begin{align}\label{br:eff}
	\eta_\mesh \leq \c{eff}\norm{u-U}{H^\alpha(\Gamma)},
      \end{align}
      where $\c{eff}$ depends on $B$, $\Gamma$, and $M$.
    \item For $\alpha\in\R$ holds efficiency~\eqref{br:eff}, where $\c{eff}$ depends on
      $\dim (X)$.
    \item For $d=2$ and $\alpha\in (-1/2,0)$ holds~\eqref{br:eff} if the mesh is sufficiently small,
      where $\c{eff}$ depends on $A$, $\Gamma$, $\alpha$, and the shape-regularity of $\mesh$.
  \end{itemize}
\end{theorem}
As the BR-type estimators are defined as a supremum, they are not computable. By definition,
any $v\in H^\alpha(\Gamma)$ with\linebreak
$\Psi_jv \neq 0$ fulfills
\begin{align*}
  \frac{\dual{R}{\Psi_jv}_\Gamma}{\norm{\Psi_jv}{H^{\alpha}(\Gamma)}} \leq \eta_j,
\end{align*}
providing a lower, computable bound by choosing, e.g., $v = \Psi_j$. Computable upper bounds
are more involved. In~\cite{f98}, it is shown that for $d=2$, it holds
\begin{align*}
  \eta_j \lesssim
  \begin{cases}
    \diam(\omega_j)^{\alpha} \norm{R_N}{L_2(\omega_j)} \quad\text{ for } \alpha\geq0,\\
    \snorm{R_N}{H^{-\alpha}(\omega_j)}^2 + \\
    \quad\sum_{j=0}^{[-\alpha]} \diam(\omega_j)^{2(j+\alpha)}\snorm{R_N}{H^j(\omega_j)}^2
    \quad\text{ for } \alpha<0.
  \end{cases}
\end{align*}
For the case of the hypersingular integral operator ($\alpha=1/2$),
this upper bound corresponds to the
weighted residual error estimator which will be considered in Section~\ref{section:est:wres}.
\subsubsection{RYW-estimators}
These types of estimators were the first ones available for boundary element methods.
Developed and analyzed by Rank \cite{rank:arfec:86} and Wendland-Yu~\cite{wy88}, they
were labeled\\ \textit{RYW}-estimators in \cite{f98}. These estimators are connected
to the Babu\v{s}ka-Rheinboldt estimators from Section~\ref{section:est:br}.
Again we sketch the ideas for the case of the hypersingular
integral equation and lowest-order discretization first.
To that end, denote by $u\in H^{1/2}(\Gamma)$ the exact solution to the Neumann
problem of Proposition~\ref{prop:neumann}, and by $U\in\Sp^1(\mesh)$
the Galerkin solution of the discrete version of Proposition~\ref{prop:galerkin:neumann}.
\begin{mydefinition}
  Denote by $\left\{ \Psi_j \right\}_{j=1}^N$ the nodal basis of $\Sp^1(\mesh)$.
  For every $j=1,\dots, N$, consider the space
  \begin{align*}
    H_j := \wilde H^{1/2}(\supp(\Psi_j)^\circ),
  \end{align*}
  which is a closed subspace of $H^{1/2}(\Gamma)$.
  Define $\zeta_j\in H_j$ as the unique solution of
  \begin{align*}
    \dual{\hyp\zeta_j}{v_j}_\Gamma = \dual{\hyp(u-U_N)}{v_j}_\Gamma\quad \text{ for all } v_j\in H_j,
  \end{align*}
  and set
  \begin{align*}
    \eta_\mesh^2 = \sum_{j=1}^N \eta_j^2 \quad \text{ where }
    \eta_j := \norm{\zeta_j}{H^{1/2}(\Gamma)}.
  \end{align*}
\end{mydefinition}
\begin{theorem}\label{thm:ryw:hyp}
  There are constants $\c{rel},\c{eff}>0$ such that
  \begin{align}\label{thm:ryw:hyp:eq}
    \c{eff}^{-2}\eta_\mesh^2 \leq \norm{u-U}{H^{1/2}(\Gamma)}^2 \leq \c{rel}^2 \eta_\mesh^2.
  \end{align}
  The constant $\c{rel}>0$ depends only on the shape-regularity of $\mesh$.
\end{theorem}
\begin{proof}
  Denote by $\eta_{\textrm{BR}}$ the Babu\v{s}ka-Rheinboldt estimator\linebreak
  from
  Definition~\ref{def:hyp:br} with its local contributions $\eta_{\textrm{BR},j}$
  and by $R:=(1/2-K')\phi - \hyp U_N$ the residual.
  Now, if $v\in H^{1/2}(\Gamma)$ with $\Psi_j v\neq 0$, it follows
  $\Psi_j v\in H_j$ and hence
  \begin{align*}
    \dual{R}{\Psi_j v}_\Gamma = \dual{\hyp\zeta_j}{\Psi_j v}_\Gamma
    \lesssim \norm{\zeta_j}{H^{1/2}(\Gamma)} \norm{\Psi_j v}{H^{1/2}(\Gamma)}.
  \end{align*}
  Taking the supremum over all those $v$ yields
  \begin{align*}
    \eta_{\textrm{BR},j} \leq \eta_j,
  \end{align*}
  such that reliability, i.e., the upper bound in~\eqref{thm:ryw:hyp:eq}, follows
  from Theorem~\ref{thm:br:hyp}. To show efficiency, choose for $\delta>0$ a function
  $v_j^{(\delta)}$ such that
  \begin{align*}
    \norm{\Psi_j v_j^{(\delta)} - \zeta_j}{H^{1/2}(\Gamma)} \leq \delta.
  \end{align*}
  Then,
  \begin{align*}
    \eta_j^2 \simeq \dual{W \zeta_j}{\zeta_j}_\Gamma &=
    \dual{\hyp \zeta_j}{\Psi_j v_j^{(\delta)}}_\Gamma
    + \dual{\hyp \zeta_j}{\zeta_j - \Psi_j v_j^{(\delta)}}_\Gamma,
  \end{align*}
  and
  \begin{align*}
    \dual{\hyp \zeta_j}{\Psi_j v_j^{(\delta)}}_\Gamma &\leq \eta_{\textrm{BR},j}
    \norm{\Psi_j v_j^{(\delta)}}{H^{1/2}(\Gamma)}, \\
    \dual{\hyp \zeta_j}{\zeta_j - \Psi_j v_j^{(\delta)}}_\Gamma &\lesssim
    \delta \norm{\zeta_j}{H^{1/2}(\Gamma)}.
  \end{align*}
  Due to $\norm{\Psi_j v_j^{(\delta)}}{H^{1/2}(\Gamma)} \leq \eta_j + \delta$ it follows that
  \begin{align*}
    \eta_j^2 \lesssim \eta_{\textrm{BR},j} (\eta_j+\delta) + \delta \eta_j,
  \end{align*}
  such that the limit $\delta \rightarrow 0$ finishes the proof of efficiency.
  $\hfill\qed$
\end{proof}
In~\cite{wy88}, estimators of this type are analyzed for bijective, continuous operators
$B:H^\alpha(\Gamma)\rightarrow H^{-\alpha}(\Gamma)$, $\alpha\in\R$, that
satisfy the G\r{a}rding inequality~\eqref{eq:garding} with discretizations that
satisfy $(ii)$ and $(iii)$ of Assumption~\ref{ass:br}. The following Theorem
summarizes the available results on the RYW-estimators.
\begin{theorem}[{\cite{f98,wy88}}]\label{thm:ryw}
  Denote by $B:H^\alpha(\Gamma)\rightarrow H^{-\alpha}(\Gamma)$, $\alpha\in\R$,
  a linear and bounded operator which satisfies the G\r{a}rding inequality~\eqref{eq:garding}.
  If $\c{garding}>0$, assume that\linebreak $B:H^{\alpha-\delta}(\Gamma)\rightarrow H^{-\alpha-\delta}(\Gamma)$
  is bijective and continuous,
  and denote by $\XX_\mesh\subset H^{\alpha}(\Gamma)$ a discrete space which fulfills $(ii)$ and
  $(iii)$ of Assumption~\ref{ass:br}.
  Suppose that there is a basis $\left\{ \Psi_j \right\}_{j=1}^N$ of $\XX_\mesh$
  such that there are at most $M$ disjoint subsets $I_k$, $k=1, \dots, M$,
  $I_k \subseteq \left\{ 1,\dots, N \right\}$, which satisfy the\linebreak strengthened Cauchy-Schwarz
  inequality
  \begin{align*}
    \dual{A v_m}{v_n}_\Gamma \leq (\# I_k)^{-1} \dual{A v_m}{v_m}_\Gamma^{1/2}
    \dual{A v_n}{v_n}_\Gamma^{1/2}
  \end{align*}
  for all $m\neq n \in I_k$ and $v_j \in H_j$, with
  \begin{align*}
    H_j := \wilde H^{\alpha}(\supp(\Psi_j)^\circ).
  \end{align*}
  If $u\in H^{\alpha}(\Gamma)$ and $U\in \XX_\mesh$
  are the exact and the Galerkin solution of
  \begin{align*}
    \dual{Au}{v}_\Gamma &= \dual{F}{v}_\Gamma \quad \text{ for all } v\in \wilde H^{-\alpha}(\Gamma),\\
    \dual{AU}{V}_\Gamma &= \dual{F}{V}_\Gamma \quad \text{ for all } V\in \XX_\mesh,
  \end{align*}
  denote the residual by $R:= F - BU$ and define a RYW-type estimator by
  \begin{align*}
    \eta_\mesh^2 = \sum_{j=1}^N \eta_j^2,\quad \text{ where }\quad
    \eta_j := \norm{\zeta_j}{H^{\alpha}(\Gamma)},
  \end{align*}
  where $\zeta_j\in H_j$ is defined by
  \begin{align*}
    \dual{A\zeta_j}{v_j}_\Gamma = \dual{R}{v_j}_\Gamma \quad\text{ for all } v_j\in H_j.
  \end{align*}
  Then, the following holds:
    \begin{itemize}
    \item If $B$ is elliptic, i.e., $\c{garding}=0$, $\eta_\mesh$ is reliable, i.e.,
      \begin{align}\label{ryw:rel}
	\norm{u-U}{\wilde H^{\alpha}(\Gamma)} \leq \c{rel} \eta_\mesh.
      \end{align}
    \item If $B$ is not elliptic but satisfies a G\r{a}rding inequality,
      i.e., $\c{garding}>0$, then~\eqref{ryw:rel} holds
      if $h_\mesh$ is sufficiently small.
    \item For $\alpha\geq0$ holds efficiency
      \begin{align}\label{ryw:eff}
	\eta_\mesh \leq \c{eff}\norm{u-U}{\wilde H^\alpha(\Gamma)},
      \end{align}
      where $\c{eff}$ depends on $B$, $\Gamma$, and $M$.
    \item For $\alpha\in\R$ holds efficiency~\eqref{br:eff}, where $\c{eff}$ depends on
      $M$.
    \item For $d=2$ and $\alpha\in (-1/2,0)$ efficiency~\eqref{ryw:eff} holds
      if the mesh is sufficiently small. The efficiency constant $\c{eff}>0$ depends
    on $B$, $\Gamma$, $\alpha$ and on the shape-regularity of $\mesh$.
  \end{itemize}
\end{theorem}
\begin{proof}
  The proof was first shown in~\cite{wy88}, with efficiency always dependent on $M$.
  Later, Faermann~\cite{f98} showed equivalence of the RYW and the BR estimators, thereby
  obtaining efficiency without dependence on $M$ for $\alpha\geq 0$ and $d=2$ and
  $\alpha\in(-1/2,0)$.
  $\hfill\qed$
\end{proof}
The constant $M$ in the last theorem can always be chosen as $M=N$ by
decomposing the set of degrees of freedom into it's single elements. In~\cite{yu87}
it is postulated that $M$ can be chosen even much smaller.
\subsubsection{Weighted residual estimators}\label{section:est:wres}
Estimators of this kind usually employ orthogonality properties to 
localize the residuals' fractional norm by a weighted norm of integer order.
The very first paper in this sense is~\cite{cs95}, where the following idea
was carried out in a more general Banach space setting:
Suppose that $\XX_\mesh\subset\XX$ is a discrete space and denote by $U$ the Galerkin
approximation to $u\in\XX$, cf.~\eqref{intro:galerkin}, and by $R := Bu-BU$ the residual.
Due to the open mapping theorem, $B^{-1}$ is bounded and it holds
\begin{align*}
  \norm{u-U}{\XX} \leq \norm{B^{-1}}{\XX'\rightarrow\XX} \norm{R}{\XX'}.
\end{align*}
We assume that there are spaces $\XX_0,\XX_1$ with $\XX_0'\supseteq\XX'\supseteq\XX_1'$,
such that $\XX_\mesh\subset\XX_0$, $R\in\XX_1'$, and
\begin{align}
  \begin{split}\label{wres:eq1}
    \norm{u}{\XX'} \leq C \norm{u}{\XX_0'}^{1-s}\norm{u}{\XX_1'}^{s}
    &\quad \text{ for all } u\in\XX_1',\\
    0 = \dual{V}{R}_{\XX_0\times\XX_0'} &\quad \text{ for all } V\in\XX_\mesh.
  \end{split}
\end{align}
Then, according to the theorem of Hahn-Banach, there is $\rho\in\XX_0$ with
\begin{align*}
  \norm{\rho}{\XX_0}^2 = \norm{R}{\XX_0'}^2 = \dual{\rho}{R}_{\XX_0\times\XX_0'}
  = \dual{\rho-V}{R}_{\XX_0\times\XX_0'}
\end{align*}
for all $V\in\XX_\mesh$, and from~\eqref{wres:eq1} we infer, using the Cauchy-Schwarz inequality,
that
\begin{align*}
  \norm{u-U}{\XX} \lesssim \norm{R}{\XX_1'}^{s}
  \inf_{V \in\XX_\mesh}\norm{\rho - V}{\XX_0}^{1-s}.
\end{align*}
For example, in the context of weakly singular integral equations in $d\geq 2$ one chooses
$B=\slo$, $\XX=\wilde H^{-1/2}(\Gamma)$, $\XX_0'=L_2(\Gamma)$, $\XX_1'=H^1(\Gamma)$, $s=1/2$,
and $\XX_\mesh := \Pp^p(\mesh)$.
Then, $\rho=R$ and due to Lemma~\ref{lem:L2:Pp:apx},
\begin{align*}
  \inf_{V \in\Pp^p(\mesh)}\norm{\rho - V}{L_2(\Gamma)} \lesssim \norm{h_\mesh \nablag R}{L_2(\Gamma)}.
\end{align*}
This shows the following, cf.~\cite[Thm.~2]{cs95}.
\begin{theorem}\label{thm:wres:cs95}
  If $\phi\in\wilde H^{-1/2}(\Gamma)$ is the exact solution of Proposition~\ref{prop:weaksing}
  or~\ref{prop:dirichlet} with $f\in H^1(\Gamma)$ and $\Phi\in\Pp^p(\mesh)$ is the respective Galerkin
  approximation from Proposition~\ref{prop:galerkin:weaksing} or~\ref{prop:galerkin:dirichlet},
  then
  \begin{align*}
    \norm{\phi-\Phi}{\wilde H^{-1/2}(\Gamma)} \lesssim
    \norm{R}{H^1(\Gamma)}^{1/2} \norm{h_\mesh \nablag R}{L_2(\Gamma)}^{1/2}
  \end{align*}
\end{theorem}
This method can be applied to problems involving hypersingular integrals~\cite[Thms.~3,~4]{cs95}
as well as transmission problems~\cite[Sec.~5]{cs95}. However, the a~posteriori error estimates
based on this method are of the form
\begin{align*}
  \eta_\mesh^2 := \left( \sum_{\el\in\mesh} \eta_\el^2 \right)^{1/2} \cdot
  \left( \sum_{\el\in\mesh} h_\el^2\eta_\el^2 \right)^{1/2},
\end{align*}
which reflects the fact that this method does not fully localize a fractional norm, see the discussions
in~\cite[Sec.~6]{cs95} and~\cite[Sec.~1]{cs96}. This issue can be overcome when considering
a uniform sequence of meshes, where the result of Theorem~\ref{thm:wres:cs95} clearly reduces to
\begin{align}
  \label{eq:wres:weaksing:rel}
  \begin{split}
  \norm{\phi-\Phi}{\wilde H^{-1/2}(\Gamma)} \lesssim
  \eta_\mesh := \left( \sum_{\el\in\mesh}\eta_\el^2 \right)^{1/2}\\
  \text{ with } \eta_\el := \norm{h_\mesh^{1/2} \nablag R}{L_2(\el)}.
  \end{split}
\end{align}
Further works on weighted residual error estimation in BEM focus on establishing
the reliability estimate~\eqref{eq:wres:weaksing:rel} also for locally refined meshes.
The first one to mention is~\cite{cs96}. For weakly singular equations, it is shown that
for $d=2$ it holds
\begin{align*}
  \norm{\phi-\Phi}{\wilde H^{-1/2}(\Gamma)} \lesssim \sigma_\mesh^{1/2}
  \sum_{\el\in\mesh}h_\el^{1/2}(h_\el^2+1)^{1/4}\norm{R'}{L_2(\el)},
\end{align*}
where $\sigma_\mesh$ is the shape-regularity constant of $\mesh$
and $(\cdot)'$ abbreviates the arclength derivative $\nablag(\cdot)$ for $d=2$.
An analogous result holds
for equations involving the hypersingular operator. The most advanced results regarding
reliable a~posteriori estimation by weighted residuals are due to~\cite{c97} for $d=2$
and~\cite{cmps,cms} for $d=3$. The first theorem that will be presented is concerned with the
a~posteriori error estimation for weakly singular integral equations,
cf.~\cite[Ex.~1]{c97} for $d=2$ and~\cite[Cor.~4.2]{cms} for $d=3$. The idea of the proof
will be presented briefly.
\begin{theorem}\label{thm:est:wres:weaksing:rel}
  If $\phi\in\wilde H^{-1/2}(\Gamma)$ is the exact solution of Proposition~\ref{prop:weaksing}
  or~\ref{prop:dirichlet} with $f\in H^1(\Gamma)$ and $\Phi\in\Pp^p(\mesh)$
  is the respective Galerkin
  approximation from Proposition~\ref{prop:galerkin:weaksing} or~\ref{prop:galerkin:dirichlet},
  then
  \begin{align*}
    \norm{\phi-\Phi}{\wilde H^{-1/2}(\Gamma)} \leq \c{rel}
    \norm{h_\mesh^{1/2} \nablag R}{L_2(\Gamma)} =: \eta_\mesh,
  \end{align*}
  where $R$ denotes the residual, i.e., $R=f-V\Phi$ in the case of Proposition~\ref{prop:weaksing}
  and $R=(1/2+\dlo)f - V\Phi$ in the case of Proposition~\ref{prop:dirichlet}.
  The constant $\c{rel}>0$ depends only on $\Gamma$, the shape-regularity $\sigma_\mesh$,
  and on the polynomial degree $p$.
\end{theorem}
\begin{proof}
  We will show the result for $\Gamma = \partial\Omega$ a closed boundary. The case
  of an open boundary then follows easily.
  Stability of $V^{-1}$ shows
  \begin{align*}
    \norm{\phi-\Phi}{\wilde H^{-1/2}(\Gamma)} \lesssim \norm{R}{H^{1/2}(\Gamma)}.
  \end{align*}
  The set of nodes $\NN$ of the mesh $\mesh$ can be split into $m>0$ subsets,
  $m$ depending only on $\sigma_\mesh$, into sets $\NN^i$, $i=1, \dots, m$, such that
  \begin{align*}
    \NN &= \bigcup_{i=1}^m \NN^i,
  \end{align*}
  and $\supp(\varphi_{z_1}) \cap \supp(\varphi_{z_2}) = \emptyset \text{ for } z_1,z_2\in \NN^i,$
  where $\phi_z$ denotes the hat function associated to a vertex $z\in\NN$.
  As $\left( \sum_{j=1}^m a_j \right)^2 \leq m \left( \sum_{j=1}^m a_j^2 \right)$, it follows
  from the triangle inequality and Lemma~\ref{lem:loc1} that
  \begin{align*}
    \norm{R}{H^{1/2}(\Gamma)}^2 
    &\leq m \sum_{i=1}^m \norm{\sum_{z\in\NN^i}\varphi_{z}R}{\wilde H^{1/2}(\Gamma)}^2\\
    &\leq \c{loc}m\sum_{z\in\NN}\norm{\varphi_z R}{\wilde H^{1/2}(\omega_z)}^2,
  \end{align*}
  where $\omega_z :=\supp(\varphi_z)$. Friedrich's inequality shows
  \begin{align*}
    \norm{\varphi_z R}{\wilde H^{1/2}(\omega_z)}^2
    \lesssim h_z(1+h_z^2)^{1/2}\norm{\nablag (\varphi_z R)}{L_2(\omega_z)}^2,
  \end{align*}
  where $h_z:=\diam(\omega_z)$.
  Now, as $R$ is orthogonal to piecewise constants, a Poincar\'e inequality shows
  \begin{align*}
    \norm{R}{L_2(\omega_z)} \lesssim \diam(\omega_z) \norm{\nablag R}{L_2(\omega_z)},
  \end{align*}
  and taking into account $\norm{\varphi_z}{L_\infty(\Gamma)} \simeq 1$ and
  $\norm{\nablag \varphi_z}{L_\infty(\Gamma)} \simeq h_z^{-1}$
  shows the result.
  $\hfill\qed$
\end{proof}
An analogous estimate holds for hypersingular integral equations,
cf.~\cite[Ex.~5]{c97} for $d=2$ and~\cite[Thm.~4.2]{cmps} for $d=3$.
\begin{theorem}\label{thm:est:wres:hypsing:rel}
  If $u\in\wilde H^{1/2}(\Gamma)$ is the exact solution of Proposition~\ref{prop:hypsing}
  or~\ref{prop:neumann} with $\phi\in L_2(\Gamma)$ and $U\in\Sp^p(\mesh)$ is the respective Galerkin
  approximation from Proposition~\ref{prop:galerkin:hypsing} or~\ref{prop:galerkin:neumann},
  then
  \begin{align*}
    \norm{u-U}{\wilde H^{1/2}(\Gamma)} \leq \c{rel}
    \norm{h_\mesh^{1/2} R}{L_2(\Gamma)},
  \end{align*}
  where $R$ denotes the residual, i.e., $R= \phi-\hyp U$ in the case of Prop.~\ref{prop:hypsing}
  and $R=(1/2-\adlo)\phi - \hyp U$ in the case of Prop.~\ref{prop:neumann}.
  The constant $\c{rel}>0$ depends only on $\Gamma$, the shape-regularity $\sigma_\mesh$,
  and on the polynomial degree $p$.
\end{theorem}
The preceding two theorems provide reliable and fully localized
error estimators for Galerkin methods for
weakly singular and hypersingular integral equations. Up to now these estimators
are the only ones which can be mathematically shown to drive adaptive BEM algorithms with
optimal rates, cf. Section~\ref{opt:sec:example2}.
The efficiency of this type of estimator is more involved and requires a careful analysis
of the (possible) singular behavior of the solutions of the problem at hand. In
the current optimality theory for adaptive algorithms, efficiency can be used
to characterize approximation classes and therefore provides a means to work only with
the error estimator, cf.~\cite{affkp13-A}.

Efficiency results for the estimators of Theorem~\ref{thm:est:wres:weaksing:rel}
and~\ref{thm:est:wres:hypsing:rel} have first been proved for $d=2$ and globally quasi-uniform
meshes in~\cite{c96}. We state the idea in the context and with notation of
Theorem~\ref{thm:est:wres:weaksing:rel}.
As we consider globally quasi-uniform meshes, we treat the mesh-with $h_\mesh$ of mesh $\mesh$
as a constant rather than a function.
For a given mesh $\mesh$ with mesh size $h_\mesh$, suppose that $\mesh_\star\geq\mesh$ is a finer mesh
with mesh size $h_{\mesh_\star} \leq h_\mesh$. Recall that $\pi_\mesh$ and $\pi_{\mesh_\star}$
denote the $L_2$ projections onto $\Pp^0(\mesh)$ and $\Pp^0(\mesh_\star)$, respectively,
and that $R$ denotes the residual on $\mesh$, using the Galerkin solution $\Phi\in\Pp^p(\mesh)$.
As $(1-\pi_{\mesh_\star}) = (1-\pi_{\mesh_\star})(1-\pi_{\mesh_\star})$, it follows
from the approximation properties of $\pi_{\mesh_\star}$ that
\begin{align*}
  \norm{(1-\pi_{\mesh_\star})\phi}{\wilde H^{-1/2}(\Gamma)}
  \lesssim h_{\mesh_\star}^{1/2}\norm{(1-\pi_{\mesh_\star})\phi}{L_2(\Gamma)},
\end{align*}
and an inverse estimate then shows
\begin{align*}
  &h_\mesh^{1/2}\norm{\pi_{\mesh_\star}\phi-\Phi}{L_2(\Gamma)}
  \lesssim\\
  &\quad h_\mesh^{1/2}
    \norm{\pi_{\mesh_\star}\phi-\phi}{\wilde H^{-1/2}(\Gamma)}
    + \left(\frac{h_\mesh}{h_{\mesh_\star}}\right)^{1/2}\norm{\phi-\Phi}{\wilde H^{-1/2}(\Gamma)}.
\end{align*}
Finally, this gives
\begin{align}\label{eq:wres:eff:eq1}
  \begin{split}
    h_\mesh^{1/2}&\norm{R}{H^1(\Gamma)} \lesssim h_\mesh^{1/2}\norm{\phi-\Phi}{L_2(\Gamma)}\\
    &\lesssim h_\mesh^{1/2}\norm{\phi-\pi_{\mesh_\star}\phi}{L_2(\Gamma)}
    + \left(\frac{h_\mesh}{h_{\mesh_\star}}\right)^{1/2}\norm{\phi-\Phi}{\wilde H^{-1/2}(\Gamma)}.
  \end{split}
\end{align}
Given $\mesh$ and an arbitrary $q<1$, the fine mesh $\mesh_\star$ can always be chosen such that
\begin{align}\label{eq:wres:eff:eq2}
  \norm{\phi-\pi_{\mesh_\star}\phi}{L_2(\Gamma)} \leq q \norm{\phi-\pi_\mesh\phi}{L_2(\Gamma)}
  \leq q \norm{\phi - \Phi}{L_2(\Gamma)}
\end{align}
holds. It follows that~\eqref{eq:wres:eff:eq1} and~\eqref{eq:wres:eff:eq2} yield
\begin{align}\label{est:wres:eff:eq3}
  h_\mesh^{1/2}&\norm{R}{H^1(\Gamma)} \lesssim
  \left(\frac{h_\mesh}{h_{\mesh_\star}}\right)^{1/2}\norm{\phi-\Phi}{\wilde H^{-1/2}(\Gamma)}.
\end{align}
The mesh $\mesh_\star$ depends on $\mesh$ and on $q$ (it is a refinement of $\mesh$
that fulfills~\eqref{eq:wres:eff:eq2}).
However, from~\eqref{est:wres:eff:eq3} we see that efficiency can only hold if it is guaranteed that
$h_\mesh \leq C h_{\mesh_\star}$, where $0<C<1$ only depends on $q$.
This can be done by exploiting explicit knowledge of the qualitative behavior of $\phi$,
cf.~\cite[Prop.~1]{c96} for globally quasi-uniform meshes and weakly singular and hypersingular
equations.
The presented approach is analyzed in a local fashion in~\cite{affkp13-A} to
obtain efficiency results on locally refined meshes for the weakly singular case.
The corresponding result is the following.
\begin{theorem}\label{thm:est:wres:weaksing:eff}
  Suppose $d=2$ and that the data fulfill $f\in H^1(\Gamma)$ and $f\in H^s$
  for some $s>2$ on the different sides of the polygonal boundary $\Gamma$. Then,
  \begin{align*}
    \c{eff}^{-1} \norm{h_\mesh^{1/2}\nablag R}{L_2(\Gamma)}
    \leq &\norm{\phi-\Phi}{\wilde H^{-1/2}} + \\
    &C(s,\varepsilon)\left( \sum_{\el\in\mesh} h_{\mesh}(\el)^{\min\{2s,5\} -1 -\varepsilon } \right)
  \end{align*}
  for all $\varepsilon>0$. The constant $\c{eff}>0$ depends only $\Gamma$ and $\sigma_\mesh$,
  whereas $C(s,\varepsilon)$ additionally depends on $s$ and $\varepsilon$.
\end{theorem}
\subsubsection{Faermann's local double norm estimators}\label{section:est:faermann}
It was suggested in~\cite{fhk96} to split the outer integral of the $H^s$-norm of the residual
in contributions on different faces of the mesh. The authors used this approach to present an
adaptive boundary element algorithm
for the solution of the Helmholtz equation, but nevertheless this procedure does not give
fully localized indicators.
Their approach was refined in~\cite{f00,f02}, where localization techniques for Sobolev-Slobodeckij norms
(cf. Section~\ref{section:localization}) were deduced and put into action to derive fully
localized error indicators. Up to now, this is the only way to obtain localized estimators that
are both reliable and efficient (on shape-regular meshes)
without further conditions or additional analysis. Their operational area is restricted
to continuous and bijective operators
\begin{align*}
  B: \wilde H^{s+2\alpha}(\Gamma) \rightarrow H^{s}(\Gamma),
  \quad s\in(0,1), \alpha\in\R.
\end{align*}
The upper bound on $s$ stems from the fact that we deal with Lipschitz domains,
but this bound can be enlarged on\linebreak
smoother domains.
For arbitrary $d$, denote by $\phi\in\wilde H^{s+2\alpha}(\Gamma)$ the exact solution of the
equation
\begin{align*}
  \dual{B \phi}{\psi}_\Gamma = \dual{F}{\psi}_\Gamma \quad \text{ for all }
  \psi\in \wilde H^{-s}(\Gamma).
\end{align*}
For a discrete space $\XX_\mesh\subseteq \wilde H^{s+2\alpha}(\Gamma)$, denote by $\Phi\in \XX_\mesh$
the Galerkin solution
\begin{align*}
  \dual{B \Phi}{V}_\Gamma = \dual{F}{V} \quad \text{ for all } V\in \XX_\mesh.
\end{align*}
Denote by $R:=F-B \Phi \in H^s(\Gamma)$ the residual.
The localization result of Theorem~\ref{thm:loc:slobodeckij:upper} immediately provides
a localized a~posteriori estimator.
\begin{theorem}
  Suppose that the assumptions and notations from the beginning of this section hold.
  For a mesh $\mesh$, define the a~posteriori error estimator
  \begin{align*}
    \eta_\mesh^2 &:= \sum_{z\in\NN}\eta_z^2 + \sum_{\el\in\mesh}\eta_\el^2\quad \text{ with }\\
    \eta_z^2 &:= \snorm{R}{H^s(\omega_z)}^2, \quad
    \eta_\el^2 := h_\el^{-2s}\norm{R}{L_2(\el)}^2.
  \end{align*}
  Then, $\eta_\mesh$ is always reliable, i.e.,
  \begin{align*}
    \norm{\phi-\Phi}{\wilde H^{s+2\alpha}(\Gamma)} \leq \c{rel} \eta_\mesh,
  \end{align*}
  and $\c{rel}$ depends only on $s$ and $\Gamma$. If $\Pp^p(\mesh)\subseteq \XX_\mesh$ or
  $\Sp^p(\mesh)\subseteq \XX_\mesh$, then we have efficiency
  \begin{align*}
    \eta_\mesh \leq \c{eff} \norm{\phi-\Phi_N}{\wilde H^{s+2\alpha}(\Gamma)},
  \end{align*}
  and $\c{eff}$ depends only on the shape-regularity $\sigma_{\mesh}$.
\end{theorem}
\begin{proof}
  To show reliability, note first that $B^{-1}$ is bounded due to the bounded inverse theorem.
  This gives
  \begin{align*}
    \norm{\phi-\Phi_N}{\wilde H^{s+2\alpha}(\Gamma)}^2
    \lesssim \norm{R_N}{H^s(\Gamma)}^2  = \norm{R_N}{L_2(\Gamma)}^2
    + \snorm{R_N}{H^s(\Gamma)}^2.
  \end{align*}
  The $H^s$-part can be bounded immediately with Theorem~\ref{thm:loc:slobodeckij:upper}.
  There is a constant $C(\Gamma)>0$ such that for all $\el\in\mesh$ it holds that
  $h_\el\leq C(\Gamma)$, and the $L_2$-part can be hence bounded by
  \begin{align*}
    \norm{R}{L_2(\Gamma)}^2
    \leq C(\Gamma)^{2s} \sum_{\el\in\mesh}h_\el^{-2s}\norm{R}{L_2(\el)}^2.
  \end{align*}
  This yields reliability.
  To show efficiency, we note that due to the assumptions $\Pp^p(\mesh)\subseteq \XX_\mesh$
  or $\Sp^p(\mesh)\subseteq \XX_\mesh$ it follows that $\dual{R}{\Psi}_\Gamma=0$
  for all discrete functions $\Psi\in \XX_\mesh$. Lemma~\ref{lem:Hs:poinc} shows that
  \begin{align*}
    \sum_{\el\in\mesh}h_\el^{-2s}\norm{R_N}{L_2(\el)}^2
    &\leq C(\sigma_\mesh) \sum_{z\in\NN}\snorm{R_N}{H^s(\omega_z)}^2.
  \end{align*}
  Hence,
  \begin{align*}
    \eta_\mesh^2 \lesssim \sum_{z\in\NN}\snorm{R_N}{H^s(\omega_z)}^2 \lesssim 
    \norm{R_N}{H^s(\Gamma)}^2,
  \end{align*}
  and continuity of $B$ shows the efficiency.
  \hfill\qed
\end{proof}
The estimator of the last theorem is always reliable, and on shape-regular meshes it is also
efficient. With Theorem~\ref{thm:loc:slob}, the reverse situation can be generated.
\begin{theorem}
  Suppose that the assumptions and notations from the beginning of this section hold.
  For a mesh $\mesh$, define the a~posteriori error estimator
  \begin{align*}
    \eta_\mesh^2 := \sum_{z\in\NN}\eta_z^2\quad \text{ with }\quad
    \eta_z^2 := \snorm{R}{H^s(\omega_z)}^2.
  \end{align*}
  Then, $\eta_\mesh$ is always efficient, i.e.,
  \begin{align*}
    \eta_\mesh \leq \c{eff} \norm{\phi-\Phi}{\wilde H^{s+2\alpha}(\Gamma)},
  \end{align*}
  and $\c{eff}$ depends only on $s$ and $\Gamma$. If $\Pp^p(\mesh)\subseteq \XX_\mesh$ or
  $\Sp^p(\mesh)\subseteq \XX_\mesh$, then it is also reliable,
  \begin{align*}
    \norm{\phi-\Phi}{\wilde H^{s+2\alpha}(\Gamma)} \leq \c{rel} \eta_\mesh,
  \end{align*}
  and $\c{rel}$ depends only on the shape-regularity $\sigma_\mesh$.
\end{theorem}
\subsection{Estimators based on space enrichment}\label{sec:enrich}
The principal idea for the construction of error estimators based on space enrichment
is that, for a given approximation, the Galerkin error can be approximated by replacing
the exact solution with an improved approximation from an enriched discrete space.
In the following, we introduce the basic setting for this methodology. Afterwards,
in Subsections~\ref{section:est:2level} and~\ref{section:est:hh2}, we discuss specific variants within this framework.

To fix notation, let us consider the variational problem specified in
Section~\ref{section:cea}, i.e.,
find $u$ in a Hilbert space $\XX$ such that $u$ is a solution of equation~\eqref{intro:weakform},
where $b$ is a continuous and elliptic bilinear form, i.e.~\eqref{intro:continuous}
and~\eqref{intro:elliptic} hold true with constants $\c{cnt}, \c{ell}>0$.
Recall that for a discrete approximation $U\in\XX_\mesh\subset\XX$, there holds C\'ea's
Theorem~\eqref{intro:cea}. Now, for the error estimation, one considers an enriched
approximation space $\XX_\mesh \subset \widehat\XX_\mesh \subset \XX$ with corresponding Galerkin
approximation $\widehat U\in\widehat\XX_\mesh$. Under appropriate conditions,
\begin{align} \label{estimH}
  \overline\eta_\mesh := \norm{\widehat U-U}{\XX}
\end{align}
is a good approximation of the error $\norm{u-U}{\XX}$.

The estimator $\overline\eta_\mesh$ is usually not practical for two reasons.
First, it requires the calculation of the improved approximation $\widehat U$,
which is expensive. Second, considering\linebreak
boundary element methods for integral
equations of the first kind, the $\XX$-norm is non-local so that $\overline\eta_\mesh$ does
not immediately provide local informations that could be used for adaptivity.
Therefore, further techniques are needed to avoid these problems. The resulting
methods are called {\em error estimators based on space enrichment}. First and common
step for their analysis is to study reliability and efficiency of $\overline\eta_\mesh$ for
the estimation of $\norm{u-U}{\XX}$. Reliability of the estimator
is based on the ``richness'' of $\widehat\XX_\mesh$ which is usually formulated as the
following \emph{saturation assumption}.

\begin{assumption}[saturation] \label{ass:satH}
Let $(\XX_\ell)_{\ell=1}^\infty$ be a sequence\linebreak
of approximation spaces $\XX_\ell\subset \XX$ with
corresponding sequence of enriched spaces $(\widehat\XX_\ell)_\ell^\infty$ and Galerkin projections
$U_\ell\in \XX_\ell$, $\widehat U_\ell\in\widehat\XX_\ell$. There exists a constant 
$\c{sat}\in[0,1)$ such that
\begin{align*}
  \norm{u-\widehat U_\ell}{\XX} \leq \c{sat} \norm{u-U_\ell}{\XX}
  \quad \text{ for all } \ell\in\N.
\end{align*}
In the following, when referring to this assumption and to simplify notation,
we will simply write
\begin{align*}
  \norm{u-\widehat U}{\XX} \leq \c{sat}\norm{u-U}{\XX}
\end{align*}
in the sense that $U\in \XX_\mesh$ is an element of a family of Galerkin approximations
and that $\widehat U$ is an improved approximation from an enriched space $\widehat\XX_\mesh$.
\end{assumption}

Reliability and efficiency of $\overline\eta_\mesh$ (or variants) in this or similar situations
are based on the saturation assumption and have been studied many times in the literature,
to our knowledge first in~\cite{bank:smith:93}. An immediate consequence of the
saturation assumption, combined with the triangle inequality, is the following
two-sided estimate, showing reliability and efficiency of the global estimator \eqref{estimH}.

\begin{myproposition} \label{prop:estimH}
  The estimator $\overline\eta_\mesh$ is efficient, i.e.,
  \begin{align*}
    \overline\eta_\mesh \leq \frac{\c{cnt}}{\c{ell}} \norm{u-U}{\XX}.
  \end{align*}
  In the situation of Assumption~\ref{ass:satH}, $\overline\eta_\mesh$ is also reliable, i.e.,
  \begin{align*}
    \norm{u-U}{\XX} \leq (1-\c{sat})^{-1}\overline\eta_\mesh
  \end{align*}
\end{myproposition}
Hence, $\overline\eta_\mesh$ is an efficient measure for the error, whereas it's reliability hinges
on the saturation assumption.
In the case that the bilinear form $b(\cdot,\cdot)$ is symmetric, both estimates can
be improved when they are formulated in terms of the so-called \emph{energy norm}
$\norm{\cdot}{b}:=\sqrt{b(\cdot,\cdot)}$.
In many publications, the saturation assumption is formulated with respect
to this norm anyhow, whereas Assumption~\ref{ass:satH} uses the $\XX$-norm.
Of course, both norms are equivalent, i.e.,
\begin{align}\label{eq:normeq}
  \c{ell}\norm{v}{\XX}^2 \leq \norm{v}{b}^2 \leq \c{cnt}\norm{v}{\XX}^2\text{ for all } v\in \XX.
\end{align}
Using this norm equivalence, Assumption~\ref{ass:satH} yields
\begin{align*}
  \norm{u-\widehat U}{b} \leq \sqrt{\c{cnt}/\c{ell}}\;\c{sat} \norm{u-U}{b}.
\end{align*}
However, from~\eqref{intro:continuous} and~\eqref{intro:elliptic} it follows that
\begin{align*}
  1 \leq \c{cnt}/\c{ell},
\end{align*}
which does not provide a saturation assumption in the energy norm.
Proposition~\ref{prop:estim:sat} below shows how to overcome this problem
For convenience, we separately formulate the saturation assumption for the energy norm first.

\begin{assumption}[saturation in energy norm] \label{ass:sata}
Let us consider the situation of Assumption~\ref{ass:satH} and let the bilinear form
$b(\cdot,\cdot)$ be symmetric. We assume that there exists a constant 
$\setc{sata}\in[0,1)$ such that
\begin{align*}
  \norm{u-\widehat U_\ell}{b} \leq \c{sata} \norm{u-U_\ell}{b}
  \quad \text{ for all } \ell\in\N.
\end{align*}
As previously in Assumption~\ref{ass:satH}, we will use this estimate for
a single discrete space $\XX_\mesh$ understanding that it is an element of a family of spaces
(with corresponding Galerkin approximations and enrichments).
\end{assumption}

We are ready to present the results corresponding to\linebreak Proposition~\ref{prop:estimH}
in the case of a symmetric bilinear form and in terms of the energy norm.

\begin{myproposition} \label{prop:estim:sat}
Let the bilinear form $b(\cdot,\cdot)$ be symmetric and $\widehat\XX_\mesh$ be an enriched
space of $\XX_\mesh\subset \XX$. Define an estimator by
\begin{align*}
  \eta_\mesh := \norm{\widehat U - U}{b}.
\end{align*}
Then, the estimator $\eta_\mesh$ is efficient, i.e.,
\begin{align*}
  \eta_\mesh \leq \norm{u-U}{b}.
\end{align*}
If additionally Assumption~\ref{ass:sata} holds, then $\eta_\mesh$ is also reliable, i.e.,
\begin{align*}
  \norm{u-U}{b} \leq (1-\c{sata}^2)^{-1/2} \eta_\mesh.
\end{align*}
Furthermore, Assumption~\ref{ass:satH} implies Assumption~\ref{ass:sata}.
\end{myproposition}
\begin{proof}
Symmetry of $b(\cdot,\cdot)$, Galerkin orthogonality and the saturation assumption~\ref{ass:sata}
immediately yield
\begin{align*}
  \norm{\widehat U - U}{b}^2 \leq \norm{u-U}{b}^2 \leq
  \c{sata}^2 \norm{u-U}{b}^2 + \norm{\widehat U - U}{b}^2.
\end{align*}
This proves both reliability and efficiency. To show that saturation in the norm $\norm{\cdot}{\XX}$
implies saturation in the norm $\norm{\cdot}{b}$, use the reliability of
Proposition~\ref{prop:estimH} and the norm equivalence~\eqref{eq:normeq} to see
\begin{align*}
  \norm{u-U}{b} \leq \frac{\sqrt{\c{cnt}/\c{ell}}}{1-\c{sat}}\norm{\widehat U-U}{b}.
\end{align*}
Galerkin orthogonality then yields
\begin{align*}
  \norm{u-\widehat U}{b}^2 &= \norm{u-U}{b}^2 - \norm{\widehat U-U}{b}^2\\
  &\leq \norm{u-U}{b}^2 - \left(\frac{\sqrt{\c{cnt}/\c{ell}}}{1-\c{sat}}\right)^{-2}
  \norm{u-U}{b}^2
\end{align*}
and saturation in the energy norm $\norm{\cdot}{b}$ follows with
\begin{align*}
  \c{sata} = \left( 1 - \left(\frac{\sqrt{\c{cnt}/\c{ell}}}{1-\c{sat}}\right)^{-2} \right)^{1/2}.
\end{align*}
As $1\leq\c{cnt}/\c{ell}$ and $\c{sat}\in[0,1)$ it follows that $\c{sata}\in[0,1)$.
$\hfill\qed$
\end{proof}

\subsubsection{Two-level estimators} \label{section:est:2level}
The term two-level estimator refers to the fact that, using the notation $\widehat\XX_\mesh$ and
$\XX_\mesh\subset \XX$ from Section~\ref{sec:enrich}, the space $\widehat\XX_\mesh$ is generated like
\begin{align} \label{2level}
   \widehat\XX_\mesh = \XX_\mesh \oplus \ZZ_\mesh.
\end{align}
That means $\widehat\XX_\mesh$ is generated by adding to the
approximation space $\XX_\mesh$ a second level as enrichment. In other words, $\widehat\XX_\mesh$ has
a hierarchical two-level decomposition like \eqref{2level}. In order to produce local
contributions to the final error estimator, the second level is usually further decomposed
so that
\begin{align} \label{2level:dec}
  \widehat\XX_\mesh = \ZZ_{\mesh,0} \oplus \ZZ_{\mesh,1} \oplus \ZZ_{\mesh,2} \oplus \cdots \oplus \ZZ_{\mesh,L}
\end{align}
with $\ZZ_{\mesh,0}:=\XX_\mesh$ if we want to be consistent with \eqref{2level}.
Here, the number $L$ of subspaces $\ZZ_{\mesh,j}\subset \ZZ_\mesh$ can be fixed or can vary
with the dimension of $\XX_\mesh$.

In the following, let us consider the simplest case of a symmetric (and elliptic, continuous)
bilinear form $b(\cdot,\cdot)$. Based on the decomposition \eqref{2level:dec} one defines
error indicators
\begin{align} \label{estim:2level:loc}
  \eta_j := \|P_j(\widehat U-U)\|_b,\quad j=0,\ldots,L,
\end{align}
with
\[
  P_j:\;\widehat\XX_\mesh\to \ZZ_{\mesh,j}:\quad b(P_jv, w) = b(v,w)\quad\text{ for all } w\in \ZZ_{\mesh,j}.
\]
The projectors $P_j$ are called \emph{additive Schwarz projectors} and
$P:=\sum_{j=0}^LP_j$ is the \emph{additive Schwarz operator} corresponding to the
decomposition \eqref{2level:dec}, cf.~\cite{qv:99,sbg:96,tw:05}.
The operator $P$ corresponds to a preconditioned stiffness matrix and is related
to techniques from domain decomposition when \eqref{2level:dec} is constructed via
such a decomposition. However, in principle, \eqref{2level:dec} can be generated
by any means, in particular to allow for indicators aimed at anisotropic mesh
refinement, cf., e.g., \cite{eh06}. Finally, having at hand the indicators $\eta_j$,
an error estimator is defined by
\begin{align} \label{estim:2level}
  \eta_\mesh := \Bigl(\sum_{j=0}^L \eta_j^2)^{1/2}.
\end{align}
The following simple result shows that the calculation of $\eta_\mesh$ is not expensive
if the dimensions of $\ZZ_{\mesh,j}$ ($j>0$) are small and $\ZZ_{\mesh,0}=\XX_\mesh$. In particular, there is
no need to calculate the improved Galerkin approximation $\widehat U\in\widehat\XX_\mesh$.

\begin{mylemma} \label{lem:ASMprop}
The additive Schwarz projections $P_j(\widehat U-U)$ can be calculated
by solving problems in the subspaces $\ZZ_{\mesh,j}$ without knowing $\widehat U$,
\[
  b(P_j(\widehat U-U),V) = L(V)-b(U,V)\quad\text{ for all } V\in \ZZ_{\mesh,j}.
\]
Moreover, if $\ZZ_{\mesh,0}\subset \XX_\mesh$ then $\eta_0=0$.
\end{mylemma}

\begin{proof}
These properties follow immediately by the definition of the projectors
and the Galerkin orthogonality.
$\hfill\qed$
\end{proof}

To show reliability and efficiency of $\eta$ one usually shows stability of
the decomposition \eqref{2level:dec}. This can be formulated in different equivalent
ways as follows (see, e.g.,~\cite{qv:99,sbg:96,tw:05}).

\begin{myproposition} \label{prop:ASM}
Let $P$ be the additive Schwarz operator related to the symmetric bilinear form $b(\cdot,\cdot)$
and discrete space $\widehat\XX_\mesh$ with decomposition \eqref{2level:dec}.
Then, for two positive numbers $\lambda_0$, $\lambda_1$ the following statements are
equivalent.

\noindent (i)
There hold the bounds
$\lambda_\mathrm{min}(P)\ge\lambda_0$ and $\lambda_\mathrm{max}(P)\le\lambda_1$
for the minimum and maximum eigenvalues of $P$, respectively.
\begin{align*}
   (ii)\quad
   \lambda_0 \sum_{j=0}^L b(v_j,v_j) \le b(v,v) \le \lambda_1 \sum_{j=0}^L b(v_j,v_j)
\end{align*}
for all $v=\sum_{j=0}^L v_j\in \widehat\XX_\mesh$ with $v_j\in \ZZ_{\mesh,j}$ ($j=0,\ldots,L$).
\[
   (iii)\quad
   \lambda_0\|v\|_b^2 \le b(Pv,v) \le \lambda_1\|v\|_b^2\quad\text{ for all } v\in\widehat\XX_\mesh.
\]
\end{myproposition}

For a decomposition of $\widehat\XX_\mesh$ that, unlike \eqref{2level:dec},
is not direct, the spectral properties of $P$ are characterized slightly differently.
In the following we will consider only direct decompositions \eqref{2level:dec} of
$\widehat\XX_\mesh$.

In most applications, the stability of \eqref{2level:dec} is ensured by two independent
steps, first the enrichment of $\XX_\mesh$ by a second level $\ZZ_\mesh$ so that the decomposition
\eqref{2level} is stable and, second, a stable decomposition of the second level,
\begin{align} \label{Z:dec}
  \ZZ_\mesh=\ZZ_{\mesh,1}\oplus\cdots\oplus \ZZ_{\mesh,L}.
\end{align}
Of course, the stability of \eqref{2level} is optimal
when $\XX_\mesh$ and $\ZZ_\mesh$ are orthogonal,
\[
   b(v,v) = b(x,x) + b(z,z)
\]
for all $v=x+z\in\widehat\XX_\mesh$ with $x\in \XX_\mesh$ and $z\in \ZZ_\mesh$,
cf. Proposition~\ref{prop:ASM}, (iii).
A generalization of this case is the so-called \emph{strengthened Cauchy-Schwarz inequality}.

\begin{mydefinition} \label{def:CSI}
The decomposition \eqref{2level} satisfies a \emph{strengthened Cauchy-Schwarz inequality}
if there exists a number $\gamma\in[0,1)$ such that
\[
   b(x,z) \le \gamma \|x\|_b \|z\|_b\quad\forall x\in \XX_\mesh,\; z\in \ZZ_\mesh.
\]
\end{mydefinition}

Immediate implication of the strengthened Cauchy-\linebreak Schwarz inequality is the stability
of the two-level decomposition.

\begin{mylemma} \label{la:CSI}
Let the decomposition \eqref{2level} satisfy a strengthened Cauchy-Schwarz inequality
(with constant $\gamma$) and let \eqref{Z:dec} be a stable decomposition
with constants $\lambda_0^\ZZ$ and $\lambda_1^\ZZ$,
\begin{align} \label{Z:dec:stab}
   \lambda_0^\ZZ \sum_{j=1}^L b(v_j,v_j) \le b(v,v) \le \lambda_1^\ZZ \sum_{j=1}^L b(v_j,v_j)
\end{align}
for all $v=\sum_{j=1}^L v_j\in \ZZ_\mesh$ with $v_j\in \ZZ_{\mesh,j}$ ($j=1,\ldots,L$).
Then, \eqref{2level:dec} is stable in the sense of Proposition~\ref{prop:ASM}
with
\[
   \lambda_0\ge (1-\gamma)\, \mathrm{min}\{1,\lambda_0^\ZZ\}   \quad \text{and}\quad
   \lambda_1\le (1+\gamma)\, \mathrm{max}\{1,\lambda_1^\ZZ\}.
\]
\end{mylemma}

\begin{proof}
The strengthened Cauchy-Schwarz inequality implies that
\begin{align*}
   (1-\gamma) \bigl(\|v_0\|_b^2 + \|v_\ZZ\|_b^2\bigr)
   &\le
   \|v_0+v_\ZZ\|_b^2\\
   &\le
   (1+\gamma) \bigl(\|v_0\|_b^2 + \|v_\ZZ\|_b^2\bigr)
\end{align*}
for all $v_0\in \XX_\mesh$ and $v_\ZZ\in \ZZ_\mesh$. The assertion then follows immediately
by application of \eqref{Z:dec:stab}.
$\hfill\qed$
\end{proof}

A combination of Propositions~\ref{prop:estim:sat} and \ref{prop:ASM}
leads to the following general result on the efficiency and reliability of a two-level
error estimator.

\begin{theorem} \label{thm:2level:estim}
Let the bilinear form $b(\cdot,\cdot)$ be symmetric and $\widehat\XX_\mesh$ be an enriched
space of $\XX_\mesh\subset \XX$. Assume that the decomposition \eqref{2level:dec} is stable
in the sense that there exist positive numbers $\lambda_0$, $\lambda_1$ that satisfy
the relations of Proposition~\ref{prop:ASM}.
Then the estimator $\eta_\mesh$ from \eqref{estim:2level} defined by the local
projections \eqref{estim:2level:loc} is efficient,
\[
   \lambda_1^{-1/2} \eta_\mesh \le \|u-U\|_b.
\]
If, additionally, Assumption~\ref{ass:sata} holds then $\eta_\mesh$ is also reliable,
\[
   \|u-U\|_b \le (1-\c{sata}^2)^{-1/2} \lambda_0^{-1/2}\; \eta_\mesh.
\]
Here, $U\in \XX_\mesh$ is the Galerkin projection of the exact solution $u\in \XX$ of the
abstract problem \eqref{intro:weakform}.
\end{theorem}

\begin{proof}
By definition of $\eta_\ell$ and the projectors $P_j$, and using the characterization
by Proposition~\ref{prop:ASM}, (iii), there holds
\begin{align*}
  \eta_\mesh^2 = &\sum_{j=0}^L \eta_j^2 = 
  \sum_{j=0}^L b\bigl(P_j(\widehat U-U), P_j(\widehat U-U)\bigr)\\
            = &\sum_{j=0}^L b\bigl(\widehat U-U, P_j(\widehat U-U)\bigr)
            =               b\bigl(\widehat U-U, P(\widehat U-U)\bigr)\\
   &\left\{
   \begin{array}{l}
      \le \lambda_1 \|\widehat U-U\|_b^2
        = \lambda_1 \eta_\mesh^2,\\
      \ge \lambda_0 \|\widehat U-U\|_b^2
        = \lambda_0 \eta_\mesh^2,
   \end{array}
   \right.
\end{align*}
where $\eta_\ell$ is the estimator defined in Proposition~\ref{prop:estim:sat}.
The assertions follow from the properties of $\eta_\ell$.
$\hfill\qed$
\end{proof}

Having set the abstract (additive Schwarz) framework for two-level error estimators
we proceed considering the specific cases of low order approximations to solutions
of weakly singular and hypersingular integral equations.

\paragraph{Weakly singular operator:}
Let us consider the weakly singular integral equation (see Proposition~\ref{prop:weaksing})
on an open or closed polyhedral surface $\Gamma$, with solution $\phi\in\wilde H^{-1/2}(\Gamma)$.
For simplicity we write $\wilde H^{-1/2}(\Gamma)=H^{-1/2}(\Gamma)$ also on a closed surface.
For a mesh $\mesh$ of shape-regular triangles and quadrilaterals, and discrete space
$\Pp^0(\mesh)$ of piecewise constant functions, $\Phi\in\Pp^0(\mesh)$ denotes the
Galerkin approximation of $\phi$, cf. Proposition~\ref{prop:galerkin:weaksing}.
We stress the fact that the mesh needs not be quasi-uniform and the {\bf quadrilaterals
can be anisotropic} but must be convex and satisfy a minimum angle condition
In the notation introduced previously,
\[
   b(u,v) = \dual{\slo u}{v}_\Gamma,\quad
   \XX=\wilde H^{-1/2}(\Gamma),\quad
   \XX_\mesh=\Pp^0(\mesh).
\]
Now, in order to define a two-level estimator for the error
$\|\phi-\Phi\|_{\wilde H^{-1/2}(\Gamma)}$, we define the second level space $\ZZ_\mesh$ as
piecewise constant functions on a refined mesh $\widehat\mesh$ with the restriction that the
functions have integral-mean zero on any element of $\mesh$:
\[
  \ZZ_\mesh = \{V\in \Pp^0(\widehat\mesh)\;|\; \dual{V}{1}_T=0\ \text{ for all } T\in\mesh\}.
\]
The enriched space is
\[
   \widehat\XX_\mesh = \Pp^0(\mesh) \oplus \ZZ_\mesh = \Pp^0(\widehat\mesh).
\]
Here, we generate $\widehat\mesh$ by refining every element of $\mesh$ in such a way that
elements of $\widehat\mesh$ are shape-regular, see Figure~\ref{fig:weaksing:refine}.
In this enrichment step the objective is two-fold. Essential is to make the
saturation assumption hold. Second, if one wants to perform anisotropic mesh refinement
then one needs sufficiently many unknowns on every old element that allow for
direction indicators. Some more details will be given below.

\begin{figure}
\begin{center}
\includegraphics[width=0.25\textwidth]{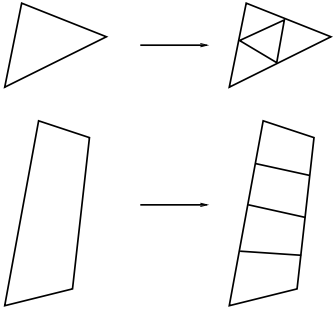}
\end{center}
\caption{Some elements of $\mesh$ (on the left) and their refinements
         to shape-regular elements of $\wilde\mesh$ (on the right)}
\label{fig:weaksing:refine}
\end{figure}

In this relatively general setting one can show reliability (based on saturation) and
efficiency of the element-based error estimator
\begin{align} \label{weaksing:estim:2level}
   \eta_\mesh := \Bigl(\sum_{T\in\mesh} \eta_T^2\Bigr)^{1/2},\quad
   \eta_T:=\|P_T(\widehat\Phi-\Phi)\|_b.
\end{align}
Here, $\widehat\Phi\in\widehat\XX_\mesh$ is the improved Galerkin approximation and,
for any $T\in\mesh$ and $\ZZ_T:=\{v\in \ZZ_\mesh\;|\; \supp(v)\subset\bar T\}$,
\[
   P_T:\;\widehat\XX_\mesh\to \ZZ_T:\quad
   \dual{\slo P_T v}{w}_T=\dual{\slo v}{w}_T\quad \forall w\in \ZZ_T
\]
and
\[
  \|v\|_b^2 = \dual{\slo v}{v}_T\quad\text{ for all } v\in \ZZ_T.
\]
\begin{theorem} \label{thm:weaksing:estim:2level}
The error estimator $\eta_\mesh$ defined by \eqref{weaksing:estim:2level} is efficient:
there exists a constant $\c{eff}>0$ such that, for any mesh $\mesh$ with shape-regular
refinement $\widehat\mesh$, there holds
\[
   \eta_\mesh \le \c{eff} \|\phi-\Phi\|_b.
\]
Furthermore, if Assumption~\ref{ass:sata} holds, then $\eta_\mesh$ is also reliable:
there exists a constant $c>0$ such that, with $\c{rel}=(1-\c{sata}^2)^{-1/2} c$, there
holds for any mesh $\mesh$ with shape-regular refinement $\wilde\mesh$ the estimate
\[
   \|\phi-\Phi\|_b \le \c{rel}\; \eta_\mesh.
\]
\end{theorem}

For a detailed proof of Theorem~\ref{thm:weaksing:estim:2level} we refer to
\cite{eh06}, where the vector case of the weakly singular operator for the Stokes
problem is analyzed. As indicated by Theorem~\ref{thm:2level:estim}, a proof
boils down to a stability analysis of the underlying decomposition
\begin{align} \label{2level:dec:weaksing}
   \widehat\XX_\mesh = \Pp^0(\mesh)\oplus \bigoplus_{T\in\mesh} \ZZ_T.
\end{align}
This analysis uses estimates for norms from fractional order Sobolev spaces.
Therefore, a major ingredient is to find a Sobolev norm that is equivalent
to the energy norm $\|\cdot\|_b$. For a fixed surface $\Gamma$, this is
the $\wilde H^{-1/2}(\Gamma)$-norm according to Theorems~\ref{thm:stability}
and \ref{thm:ellipticity}. However, for an element $v\in \ZZ_T$, there holds the
equivalence
\[
   \|v\|_b^2 = \dual{\slo v}{v}_T \simeq \|v\|_{\wilde H^{-1/2}(T)}^2,
\]
and it is not immediately clear how the corresponding equivalence numbers depend
on $T$. One has to find a Sobolev norm that is uniformly equivalent to the
energy norm for shape-regular elements $T\in\widehat\mesh$. By an affine mapping
of $T$ to a reference element $T_\refel$ one finds that
\[
  \|v\|_b^2 \simeq h_\el^{2d-3} \|\hat v\|_b^2.
\]
Here, $\hat v$ is the affinely transformed function defined on $\el_\refel$.
This equivalence is immediate by the two Jacobians of the double integral in
$\dual{\slo \cdot}{\cdot}_T$ and by the scaling property of the weakly singular kernel,
\[
   \frac 1{\abs{x-y}} = \frac 1{\abs{F_\el(\hat x)-F_\el(\hat y)}}
   \simeq h_\el^{-1} \frac 1{\abs{\hat x-\hat y}},
\]
where $x = F_\el(\hat x), y = F_\el(\hat y)$.
On the other hand,
\[
   \|v\|_{\wilde H^{-1/2}(T)}
   =
   \sup_{\varphi\in H^{1/2}(T)\setminus\{0\}}
   \frac {\dual{v}{\varphi}_T}{\|\varphi\|_{H^{1/2}(T)}}
\]
is certainly not uniformly equivalent to the energy norm since the duality
in the numerator scales under affine transformations but the denominator does not
(the seminorm $|\cdot|_{H^{1/2}(T)}$ behaves differently from the $L_2(T)$-norm
under affine mappings). To fix this mismatch, one uses an $H^{1/2}(T)$-norm with
weighted $L_2(T)$-term,
\[
   \|v\|_{H^{1/2}_h(T)}^2 := h_\el^{-1}\|v\|_{L_2(T)}^2 + |v|_{H^{1/2}(T)}^2,
\]
and defines a scalable $\wilde H^{-1/2}(T)$-norm by duality:
\[
   \|v\|_{\wilde H^{-1/2}_h(T)}
   :=
   \sup_{\varphi\in H^{1/2}(T)\setminus\{0\}}
   \frac {\dual{v}{\varphi}_T}{\|\varphi\|_{H^{1/2}_h(T)}}.
\]
This norm is uniformly equivalent to the energy norm
under affine mappings that maintain shape regularity, as long as the functions
under consideration have integral-mean zero.
This integral-mean zero condition is essential and the reason for the
particular construction of our second level space $\ZZ$.

A proof of stability of the decomposition \eqref{2level:dec:weaksing} then
reduces to the following three steps.
\begin{enumerate}
\item Replace the energy norm by the uniformly equivalent scalable Sobolev norm
      $\|\cdot\|_{\wilde H^{-1/2}_h(T)}$ in the spaces $\ZZ_T$.
\item One shows (see \cite[Lemma~3.2]{eh06}) that
\[
   \|v\|_{\wilde H^{-1/2}(\Gamma)}^2
   \lesssim
   \sum_{T\in\mesh} \|v|_T\|_{\wilde H^{-1/2}_h(T)}^2
\]
for all $v\in\wilde H^{-1/2}(\Gamma)$ with $v|_T\in\wilde H^{-1/2}(T)$ and
$\dual{v}{1}_T=0 \text{ for all } T\in\mesh$.
\item By scalability and equivalence of norms in finite-dimen\-sional spaces one
proves (see \cite[(3.16)]{eh06}) that
\[
   \sum_{T\in\mesh} \|V|_T\|_{\wilde H^{-1/2}_h(T)}^2
   \lesssim
   \|V\|_{\wilde H^{-1/2}(\Gamma)}^2
   \quad\forall V\in \ZZ_\mesh.
\]
\end{enumerate}
Finally, having shown the stability of \eqref{2level:dec:weaksing} and making
use of the saturation assumption, Theorem~\ref{thm:weaksing:estim:2level} is
proved by application of Theorem~\ref{thm:2level:estim}.

\begin{remark}
The indicators $\eta_T$ defined so far give information only with respect to the
location of elements. By simple changes, it is easy to define indicators with
respect to directions, so that anisotropic refinements can be considered. One only
has to use slightly different local spaces $\ZZ_T$, further split so that corresponding
projections give the direction indicators. In Figure~\ref{fig:tri:dec} we have
illustrated this for a single triangle (on the left) which is decomposed into
two triangles in three different ways (on the right). The plus and minus signs
indicate that one has to use piecewise constant functions (positive on the triangle
with the plus sign and negative on the other) so that the function has
integral-mean value zero. In this way, on each triangle $T\in\mesh$, one has
three spaces and together they generate the second level $\ZZ$ on $T$. A refinement
algorithm with direction control would consider, e.g., the triangle refinement that
corresponds to the space among the three whose error indicator is largest.
Similar constructions work on quadrilaterals.
Throughout, in the refinement procedure with direction steering,
one has to consider a minimum angle condition.

The inclusion of this direction control in the stability analysis of the
decomposition of $\ZZ_\mesh$ is straightforward by selecting the previously scalable
Sobolev norm. One only has to use an argument from equivalence of norms in
finite-dimensional spaces. For more details we refer to \cite{eh06}.
\end{remark}

\begin{figure}
\begin{center}
\includegraphics[width=0.25\textwidth]{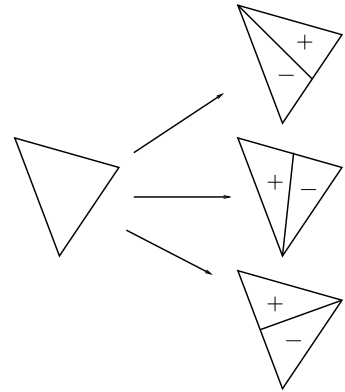}
\end{center}
\caption{A triangle of $\mesh$ (on the left) and its three partitions
         (on the right) for the construction of error indicators with
         direction control.}
\label{fig:tri:dec}
\end{figure}

Let us comment on other publications on two-level error estimators for
weakly singular integral equations. In \cite{msw98}, Mund, Stephan and Wei{\ss}e
analyze the situation we have considered above for the particular case of uniform
meshes of squares. In this case, several of the required norm estimates
can be calculated exactly so that general arguments from fractional order
Sobolev spaces (that we have discussed above) can be avoided. Also, \cite{msw98}
reports on numerical experiments on curved surfaces.

In \cite{ms99}, Mund and Stephan study two-level error estimators for the
coupling of finite elements and boundary elements. The model problem is a
transmission problem in two dimensions with nonlinear behavior in a
bounded domain, coupled with the Laplacian in the exterior. The variational
formulation and its discretization involves the weakly singular operator (on
a curve). The proposed error estimator is of the two-level kind with additive
Schwarz theory. Here, the authors prove stability of the boundary element
contribution up to a perturbation of the type $h^{-\epsilon}$ with $\epsilon>0$
and $h$ being the mesh size.

So far we have only discussed the case of symmetric and elliptic bilinear forms.
This theory can be extended to indefinite problems. In particular,
in~\cite{bank:smith:93} Bank and Smith analyze in an abstract setting the
general case of a variational form with bounded bilinear form that only satisfies
the continuous and discrete inf-sup conditions
(to guarantee existence and uniqueness of a continuous and discrete solution).
Apart from the saturation assumption in the form of Assumption~\ref{ass:satH}
(with respect to a Sobolev norm rather than energy norm), the analysis is based
on a strengthened Cauchy-Schwarz inequality in the corresponding Sobolev norm
(cf. Definition~\ref{def:CSI} in the energy norm).
More specifically, for the boundary element method, Maischak, Mund and Stephan
analyze in~\cite{mms97} two-level error estimators for weakly singular integral equations governing
the Helmholtz problem with small wave number. Their theory follows the setting
from~\cite{bank:smith:93}, by showing that it is enough to have a stable decomposition
corresponding to the elliptic part of the operator, and that the compact perturbation
due to non-zero wave number does not change the behavior of the two-level error
estimator. However, proofs are given for the two-dimensional case. In three dimensions,
numerical results verify the expected behavior of the error estimator.

Finally, we note that in \cite{jl99} the authors have studied two-level error estimators
for boundary element discretizations of weakly singular and hypersingular operators in
two dimensions. However, there are several unresolved theoretical hickups
involving subtle issues with fractional order Sobolev spaces that the authors replaced
with several assumptions. We prefer not to discuss the outcomes in detail.

\paragraph{Hypersingular operator:}

In two dimensions, that means for boundary integral equations on curves, additive
Schwarz theory for weakly singular operators is equivalent to the one for hypersingular
operators. This is due to the fact that Sobolev spaces of orders plus and minus one half
are being mapped among them by differentiation and integration with respect to the
arc-length. Correspondingly, basis functions are being transformed. The only,
purely technical, difficulty is an integral-mean zero condition for functions in
$H^{-1/2}$ along the elements or the curve. For an early observation and application
of this fact, see~\cite{eps:s:89}.

In three dimensions, however, the situation is different. In this case, rather than
simple differentiation and integration, pseudo-differential operators act as appropriate
mappings. Possible operators are the square root of the negative Laplacian
(more precisely, of the negative Laplace-Beltrami operator) and its inverse
operator. These operators do not map piecewise polynomials onto piecewise polynomials.
Therefore, in three dimensions on surfaces, the stability analyses of two-level
decompositions of discrete spaces  in $H^{1/2}$ and $H^{-1/2}$ are substantially different.

We do not know of any mathematical publication on two-level error estimators for
hypersingular integral equations on surfaces that do not also consider the $p$-version
(where approximations are improved by increasing polynomial degrees).
Therefore, we postpone the discussion of this case to Section~\ref{section:aposteriori:hp}
which deals with the $hp$-version.
\subsubsection{$(h-h/2)$ estimators}\label{section:est:hh2}
The starting point for this type of error estimators is Proposition~\ref{prop:estim:sat}, which states
\begin{align*}
  \norm{\widehat U-U}{b} \leq \norm{u-U}{b} \leq (1-\c{sata}^2)^{-1/2} \norm{\widehat U-U}{b}
\end{align*}
for a symmetric bilinear form $b(\cdot,\cdot)$, where the upper bound holds under
the saturation assumption~\ref{ass:sata}. Here, $U\in \XX_\mesh$
denotes the Galerkin solution with respect to a mesh $\mesh$ and
$\widehat U\in \widehat\XX_\mesh := \XX_{\widehat\mesh}$
denotes the Galerkin solution with respect to a uniformly refined mesh $\widehat\mesh$. The term
\begin{align*}
  \eta := \norm{\widehat U -U}{b}
\end{align*}
is computable in the sense that it does not contain any unknowns and that it can be evaluated
easily as a matrix-vector product
\begin{align*}
  \norm{\widehat U -U}{b}^2
  = (\mathbf{\widehat U} - \mathbf{U})\cdot\mathbf{\widehat B}\cdot(\mathbf{\widehat U} - \mathbf{U}),
\end{align*}
where $\mathbf{\widehat B}$ is the Galerkin matrix on the space $\widehat\XX_\mesh$ and
$\mathbf{\widehat U}$ and $\mathbf{U}$ are the coefficient vectors of the Galerkin solutions
with respect to the chosen basis for $\widehat\XX_\mesh$. Now, the idea of $(h-h/2)$ estimators is to
overcome the following two problems:
\begin{itemize}
  \item The computation of both, $U$ and $\widehat U$ is necessary.
  \item The inherent non-locality of the norm prevents us to use
    $\norm{\widehat U -U}{b}$ as refinement indicator in an adaptive algorithm.
\end{itemize}
To motivate a remedy for the first problem, we note that, due to best approximation properties
of Galerkin solutions, $\widehat U$ will always be a better solution than $U$ which therefore
becomes only a temporary result.
Furthermore, as soon as $\widehat U$ is computed, the computational
cost of computing $U$ is quite high in contrast to a simple postprocessing of $\widehat U$.
Hence, to avoid the (expensive) computation of $U$,
we use $\Pi \widehat U$ instead, where $\Pi$ is a (preferably cheap) projection onto the space $\XX_\mesh$,
which is supposed to fulfill the following properties for all $\widehat U\in\widehat\XX_\mesh$:
\begin{align}
    \norm{(1-\Pi)\widehat U}{b} &\leq 
    \c{approx}\min_{V \in \XX_\mesh}\norm{\widehat U - V}{h^s}\label{sec:hh2:ass1}\\
    \norm{\widehat U}{h^s} &\leq \c{inv} \norm{\widehat U}{b}\label{sec:hh2:ass2},
\end{align}
where $\norm{\cdot}{h^s}$ denotes an adequate $h^s$-weighted, integer order seminorm.
In BEM, the norm $\norm{\cdot}{b}$ is equivalent to a fractional order Sobolev norm.
Hence, the first estimate in~\eqref{sec:hh2:ass1} is an approximation property for the operator $\Pi$,
whereas the estimate~\eqref{sec:hh2:ass2} corresponds to an inverse estimate.
Note, however, that $\widehat U\in\widehat\XX_\mesh$ is based on a fine mesh $\widehat\mesh$, whereas
$h$ in~\eqref{sec:hh2:ass2} corresponds to the mesh $\mesh$. In other words,~\eqref{sec:hh2:ass2} requires that
the mesh-size $\widehat h$ of $\widehat\mesh$ must not be too small in comparison with the mesh-size $h$ of $\mesh$.
Now, the best approximation properties of Galerkin methods show immediately that
\begin{align*}
  \eta \leq \wilde\eta := \norm{(1 - \Pi) \widehat U}{b}.
\end{align*}
From the estimate~\eqref{sec:hh2:ass1} follows immediately that
\begin{align*}
  \wilde\eta = \norm{(1 - \Pi) \widehat U}{b} \lesssim \norm{(1-\Pi) \widehat U}{h^s} =: \wilde \mu.
\end{align*}
The estimator $\wilde \mu$ has all the desired properties as it is local and avoids
the computation of $U$.
Next, the estimates~\eqref{sec:hh2:ass2} and~\eqref{sec:hh2:ass1} show
\begin{align*}
  \wilde\mu = \norm{(1-\Pi)\widehat U}{h^s} \lesssim \norm{(1-\Pi)\widehat U}{b} \lesssim \norm{\widehat U-U}{h^s} =: \mu.
\end{align*}
Finally it follows from estimate~\eqref{sec:hh2:ass2} that
\begin{align*}
  \mu = \norm{\widehat U-U}{h^s} \lesssim \norm{\widehat U - U}{b} = \eta
\end{align*}
The strength of the resulting estimators is that they are conceptually simple and require nearly
no overhead in implementation. In the following, we will specify the involved quantities to
obtain estimates in the weakly singular and hypersingular case.
\paragraph{Weakly singular operator:}
For weakly singular integral equations, i.e., integral equations involving the single layer
potential $\slo$, the presented approach was analyzed in detail for
$d=2$ in~\cite{effp09} and for $d=3$ in~\cite{fp08}. The cited works discuss only the
lowest-order case $p=0$, therefore we sketch the proof for general $p\geq 0$.
The energy norm is given in this case by
\begin{align*}
  \norm{u}{b}^2 := \norm{u}{\slo}^2 := \dual{\slo u}{u}_\Gamma.
\end{align*}
The first result regarding reliability and efficiency of the estimator $\eta$ follows
directly from Proposition~\ref{prop:estim:sat}, cf.~\cite[Prop.~1.1]{fp08}
and~\cite[Prop.~3.1]{effp09}.
\begin{theorem}\label{thm:hh2:weaksing}
  Suppose that $\phi\in\wilde H^{-1/2}(\Gamma)$ is the exact solution of
  Proposition~\ref{prop:weaksing} or~\ref{prop:dirichlet}. Given a mesh $\mesh$ and its uniform
  refinement $\widehat\mesh$, denote by $\Phi\in\Pp^p(\mesh)$ and
  $\widehat\Phi\in\Pp^p(\widehat\mesh)$ the respective Galerkin
  approximations from Proposition~\ref{prop:galerkin:weaksing} or~\ref{prop:galerkin:dirichlet}.
  Then,
  \begin{align*}
    \eta_\mesh := \norm{\Phi - \widehat\Phi}{\slo} \leq\norm{\phi - \Phi}{\slo},
  \end{align*}
  i.e., the estimator $\eta_\mesh$ is efficient with $\c{eff}=1$.
  Under the saturation assumption~\ref{ass:sata}, $\eta_\mesh$ is reliable, i.e.,
  \begin{align*}
    \norm{\phi - \Phi}{\slo} \leq (1-\c{sata}^2)^{-1/2} \eta_\mesh.
  \end{align*}
\end{theorem}
The localization of $\eta_\mesh$ and the avoidance of the computation of $\Phi$ is
done by using the $L_2$-orthogonal projection $\pi^p_\mesh$ from Definition~\ref{def:L2projection},
and the seminorm $\norm{\cdot}{h^s}$ will be the $h_\mesh^{1/2}$-weighted $L_2$-norm in this case.
Note that by Theorem~\ref{thm:stability} and~\ref{thm:ellipticity}, $\norm{\cdot}{b}$ is an
equivalent norm on $\wilde H^{-1/2}(\Gamma)$.
Lemma~\ref{lem:L2:Pp:apx} with $r=1/2$ and $s=0$ provides the approximation
properties for the derivation of~\eqref{sec:hh2:ass1}, and Lemma~\ref{lem:Pp:invest} provides the
inverse estimate that is needed in~\eqref{sec:hh2:ass2}.
Note that the projection property of $\pi_\mesh^p$ is used to arrive at the minima.
The resulting estimators and equivalences are stated in the following theorem,
cf.~\cite[Thms.~3.2,~3.4]{fp08}.
\begin{theorem}\label{thm:def:hh2:weaksing}
  Define the following a~posteriori error estimators:
  \begin{align*}
    \begin{split}
    \begin{array}{rclcrcl}
      \eta_\mesh &:=& \norm{\widehat\Phi - \Phi}{\slo},
      &\quad&
      \mu_\mesh &:=& \norm{h_\mesh^{1/2}(\widehat\Phi - \Phi)}{L_2(\Gamma)},\\
      \widetilde\eta_\mesh &:=& \norm{(1-\pi^p_\mesh)\widehat\Phi}{\slo},
      &\quad&
      \widetilde\mu_\mesh &:=& \norm{h_\mesh^{1/2}(1-\pi^p_\mesh)\widehat\Phi}{L_2(\Gamma)}.
    \end{array}
    \end{split}
  \end{align*}
  Then, it holds
  \begin{align*}
    \eta_\mesh &\leq \wilde\eta_\mesh\leq C_\slo^{1/2}\c{approx}\wilde\mu_\mesh,\\
    \wilde\mu_\mesh &\leq \mu_\mesh \leq \c{inv}\c{ell}^{-1/2} \eta_\mesh,
  \end{align*}
  where $C_V = \norm{\slo}{\wilde H^{-1/2}(\Gamma)\rightarrow H^{1/2}(\Gamma)}$ and
  $\c{ell}$ are the stability and ellipticity constants of the single layer operator $\slo$,
  $\c{approx}$ is the constant of Lemma~\ref{lem:L2:Pp:apx}, and $\c{inv}$ is the
  constant of the inverse estimate of Lemma~\ref{lem:Pp:invest}.
\end{theorem}
The last theorem shows that all estimators are equivalent up to constants that depend
only on $\Gamma$, $p$, and the shape-regularity constant $\sigma_\mesh$. In particular,
Theorem~\ref{thm:hh2:weaksing} shows that all estimators are efficient and (under the
saturation assumption~\ref{ass:sata}) reliable.
\begin{remark}
  The work~\cite{fp08} uses the quantity $\rho_\mesh$ instead of $h_\mesh$
  to define the estimators $\mu_\mesh$ and $\wilde\mu_\mesh$, where $\rho_\mesh$ is
  defined $\mesh$-elementwise as the diameter of the largest sphere centered at a point in
  $\el\in\mesh$ whose intersection with $\Gamma$ lies entirely in $\el$.
  The reason for this is that~\cite{fp08} also uses the estimator $\wilde\mu_\mesh$
  to steer an adaptive anisotropic mesh refinement on quadrilaterals, for which
  $\rho_\mesh$ is more appropriate than $h_\mesh$.
  After an element has been selected for refinement, the choice on the refinement directions
  is based on the expansion of $\widehat\Phi$ in a series of functions on $\widehat\mesh$
  indicating the possible refinement directions. The resulting adaptive algorithms
  behave reasonable and the authors observe the optimal convergence rate
  $\OO(N^{-3/2})$, where $N$ is the number of degrees of freedom. We refer
  to~\cite{afp12,fp08} as well as Section~\ref{section:estred:anisotropic} for further details.
\end{remark}
\paragraph{Hypersingular operator:}
For hypersingular integral equations, the analysis of $(h-h/2)$-type estimators is given
in~\cite{cp07,efgp12} for $d=2$ and the lowest-order case $p=1$, and in~\cite{affkp13} for $d=3$
and general $p\geq 1$. The energy norm is given by
\begin{align*}
  \norm{u}{b}^2 := \norm{u}{\hyp}^2 :=
  \begin{cases}
    \dual{\hyp u}{u}_\Gamma &\text{ for } \Gamma\subsetneq\partial\Omega\\
    \dual{\hyp u}{u}_\Gamma + \dual{u}{1}_\Gamma^2 &\text{ for } \Gamma = \partial\Omega,
  \end{cases}
\end{align*}
cf. Section~\ref{section:bem}. Again, Proposition~\ref{prop:estim:sat} shows reliability
and efficiency of the estimator.
\begin{theorem}\label{thm:hh2:hypsing}
  Suppose that $u\in\wilde H^{1/2}(\Gamma)$ is the exact solution of Proposition~\ref{prop:hypsing}
  or~\ref{prop:neumann}. Given a mesh $\mesh$ and its uniform refinement $\widehat\mesh$,
  denote by $U\in\wilde\Sp^p(\mesh)$ and $\widehat U\in\wilde\Sp^p(\widehat\mesh)$ the respective
  Galerkin approximations from Proposition~\ref{prop:galerkin:hypsing} or~\ref{prop:galerkin:neumann}.
  Then,
  \begin{align*}
    \eta_\mesh := \norm{U - \widehat U}{\hyp} \leq \norm{u-U}{\hyp},
  \end{align*}
  i.e., the estimator $\eta_\mesh$ is efficient with $\c{eff}=1$. Under the saturation
  assumption~\ref{ass:sata}, $\eta_\mesh$ is reliable, i.e.,
  \begin{align*}
    \norm{u-U}{\hyp} \leq (1-\c{sata}^2)^{-1/2} \eta_\mesh.
  \end{align*}
\end{theorem}
The estimator $\eta_\mesh$ will be localized by the seminorm
\begin{align*}
  \norm{\cdot}{h^s} = \norm{h_\mesh^{1/2}\nablag(\cdot)}{L_2(\Gamma)}.
\end{align*}
Note first that $\norm{\cdot}{b}$ is an equivalent
norm on $\wilde H^{1/2}(\Gamma)$ by Theorem~\ref{thm:stability} and~\ref{thm:ellipticity}.
Estimate~\eqref{sec:hh2:ass2} is valid due to the inverse estimate of Lemma~\ref{lem:Sp:invest}
with $s=1/2$, as long as $\Pi$ is a projection.
To show~\eqref{sec:hh2:ass1}, Lemma~\ref{lem:Sp:apx} with $s=1/2$ can be employed
as long as $\Pi$ is an $\wilde H^{1/2}$ stable projection.
Sections~\ref{section:L2}--\ref{section:nodal} present different projection operators that can be used
in this context.
\begin{itemize}
  \item The Scott-Zhang operator $\wilde J_\mesh$ (resp. $J_\mesh$), which is an $\wilde H^{1/2}(\Gamma)$
    stable projection due to Lemma~\ref{lem:sz}.
  \item For $d=2$, the nodal interpolation operator $J_\mesh$, which is $\wilde H^{1/2}(\Gamma)$
    stable due to Lemma~\ref{lem:nodalinterpolation:1d}.
  \item On a sequence of meshes that is generated by certain mesh refinement rules, the
    the $L_2$ projection $\Pi_\mesh^p$ onto $\wilde\Sp^p(\mesh)$
    can be shown to be stable in $\wilde H^{1/2}(\Gamma)$.
    Present proofs for this property require certain restrictions on the mesh refinement
    and the polynomial degree, cf. Section~\ref{section:meshrefinement} for details.
  \item For $d=3$, the nodal interpolation operator $J_\mesh$ is not $\wilde H^{1/2}(\Gamma)$
    stable. However, Lemma~\ref{lem:nodalinterpolation} with $s=1/2$ and $q=p$ can be used.
    Indeed, the estimate~\eqref{sec:hh2:ass1} is given explicitly in
    Lemma~\ref{lem:nodalinterpolation}.
\end{itemize}
The resulting estimators and equivalences are summarized in the following Theorem,
cf.~\cite{cp07,efgp12} for $d=2, p=1$, and\linebreak\cite{affkp13} for $d=3, p\geq1$.
\begin{theorem}\label{thm:def:hh2:hypsing}
  Denote by $P_\mesh$ either
  \begin{itemize}
    \item[(a)] the Scott-Zhang operator or
    \item[(b)] the $L_2$ orthogonal projection onto $\Sp^p(\mesh)$
      (given that it is $\wilde H^{1/2}(\Gamma)$ stable, cf. Section~\ref{section:meshrefinement})
    \item[(c)] the nodal interpolation operator.
  \end{itemize}
  Define the following a~posteriori error estimators:
  \begin{align*}
    \begin{split}
    \begin{array}{rclcrcl}
      \eta_\mesh &:=& \norm{\widehat U - U}{\hyp},
      &\,&
      \mu_\mesh &:=& \norm{h_\mesh^{1/2}\nablag(\widehat U - U)}{L_2(\Gamma)},\\
      \widetilde\eta_\mesh &:=& \norm{(1-P_\mesh)\widehat U}{\hyp},
      &\,&
      \widetilde\mu_\mesh &:=& \norm{h_\mesh^{1/2}\nablag(1-P_\mesh)\widehat U}{L_2(\Gamma)}.
    \end{array}
    \end{split}
  \end{align*}
  Then, it holds that
  \begin{align*}
    \eta_\mesh &\leq \wilde\eta_\mesh\leq C_\hyp^{1/2}\c{approx}\wilde\mu_\mesh,\\
    \wilde\mu_\mesh &\leq \c{stab}\mu_\mesh \leq \c{inv}\c{ell}^{-1/2} \eta_\mesh,
  \end{align*}
  where $C_\hyp = \norm{\hyp}{\wilde H^{1/2}(\Gamma)\rightarrow H^{-1/2}(\Gamma)}$ and
  $\c{ell}$ are the stability and ellipticity constants of the hypersingular operator $\hyp$,
  $\c{inv}$ is the constant of the inverse estimate of Lemma~\ref{lem:Sp:invest}, and, depending
  on the choice of $P_\mesh$,
  \begin{itemize}
    \item[(a)] $\c{approx}$ is the constant of Lemma~\ref{lem:Sp:apx} and
      $\c{stab}$ depends solely on the operator norm $\norm{J_\mesh}{H^1(\Gamma)}$, or
    \item[(b)] $\c{approx}$ is the constant of Lemma~\ref{lem:Sp:apx} and
      $\c{stab} = \c{inv}\c{approx}$.
    \item[(c)] For $d=2$, $\c{approx}$ is the constant from Lemma~\ref{lem:nodalinterpolation:1d}
      and $\c{stab}=1$, and for $d=3$, $\c{approx}$ and $\c{stab}$ are the constants from
      Lemma~\ref{lem:nodalinterpolation}.
  \end{itemize}
\end{theorem}
The estimators of the last theorem always apply an operator $P_\mesh$ to the solution $\widehat U$.
However, as Lemma~\ref{lem:nodalinterpolation} shows, also the gradient $\nablag\widehat U$ could
be projected locally on the coarse mesh, which is much cheaper.
\begin{theorem}\label{thm:def:hh2:hypsing:grad}
  Define the a~posteriori error estimator
  \begin{align*}
    \overline\mu_\mesh := \norm{h_\mesh^{1/2}(1-\pi_\mesh^{p-1})\nablag \widehat U}{L_2(\Gamma)}.
  \end{align*}
  Then, it holds that
  \begin{align*}
    \eta_\mesh \leq C_\hyp \c{approx}\c{stab} \overline\mu_\mesh \leq \c{inv}\c{ell}^{-1/2}\eta_\mesh,
  \end{align*}
  where $\c{approx},\c{stab}>0$ are the constants from Lemma~\ref{lem:nodalinterpolation}, $\c{inv}>0$
  is the constant of the inverse estimate of Lemma~\ref{lem:Sp:invest}, and
  $C_\hyp = \norm{\hyp}{\wilde H^{1/2}(\Gamma)\rightarrow H^{-1/2}(\Gamma)}$ and
  $\c{ell}$ are the stability and ellipticity constants of the hypersingular operator $\hyp$,
\end{theorem}
\begin{remark}
  The concept of $(h-h/2)$ type error estimators has recently been extended to nonconforming
  boundary element methods for hypersingular integral equations,
  see~\cite{DominguezH_PEA,HeuerK_ACR}.
\end{remark}
\subsection{Averaging estimators}\label{section:averaging}
The advantage of space-enrichment based error estimators 
(Section~\ref{sec:enrich}) is that their implementation essentially only
requires a simple postprocessing of the Galerkin data and the computed Galerkin
solution. For the $(h-h/2)$-type error estimators from 
Section~\ref{section:est:hh2}, one theoretical drawback is that the Galerkin 
solution has to computed on the fine-mesh $\widehat\TT$, while the error 
estimators only estimates the coarse-mesh error, cf.\ 
Theorem~\ref{thm:hh2:weaksing} for the weakly singular integral equation and 
Theorem~\ref{thm:hh2:hypsing} for the hypersingular integral equation. 
Although the two-level error estimators from Section~\ref{section:est:2level}
avoid the computation of the fine-mesh solution, their computation requires
the assembly of the fine-mesh Galerkin data. Since the latter is the most
time consuming part of BEM computations, neither of these error estimators
seems to be attractive at the first glance. 

This section discusses error estimation by averaging on large patches. 
On an abstract level, the approach can be outlined as follows: 
Let $ u\in\XX$ denote the unknown exact solution of~\eqref{intro:weakform}.
Suppose that $\TT$ is a given mesh with uniform refinement $\widehat\TT$
and that we are given a space $\XX(\widehat\TT)$ with low-order polynomials 
on the fine mesh and a space $\widehat\XX(\TT)$ with higher-order 
polynomials on the coarse mesh. The goal is to derive a computable error estimator
$\eta_\TT$ which estimates the fine-mesh error $\norm{u-\widehat U}{\XX}$ of 
the Galerkin solution $\widehat U\in\XX(\widehat\TT)$ of~\eqref{intro:galerkin}
with $\XX = \XX(\widehat\TT)$. To that end, let $G:\XX\to\widehat\XX(\TT)$ denote the
Galerkin projection, i.e., for all $w\in\XX$, $Gw \in \widehat\XX(\TT)$ is 
the unique solution of the linear system
\begin{align}\label{eq:averaging:galerkin}
 b(Gw,v) = b(w,v)
 \quad\text{for all }v\in\widehat\XX(\TT).
\end{align}
With this notation, we define the computable error estimator
\begin{align}
 \eta_\TT := \norm{(1-G)\widehat U}{\XX}.
\end{align}
The following abstract theorem is found, e.g., 
in\linebreak
\cite[Thm.~2.1]{cp07:lecture:notes}.

\begin{theorem}\label{theorem:averaging}
Define the quantities
\begin{align}\label{eq:averaging:q}
 q &:= \frac{\norm{(1-G)u}{\XX}}{\norm{u-\widehat U}{\XX}},\\
 \label{eq:averaging:lambda}
 \lambda &:= \max_{V\in\widehat\XX(\TT)}\min_{\widehat V\in\XX(\widehat\TT)}
 \frac{\norm{V-\widehat V}{\XX}}{\norm{V}{\XX}}.
\end{align}
Then, the error estimator $\eta_\TT$ is efficient
\begin{align}\label{eq:averaging:efficient}
 \eta_\TT \le (\c{cnt}/\c{ell}+q)\,\norm{u-\widehat U}{\XX}.
\end{align}
Provided that the ellipticity and continuity constant of $b(\cdot,\cdot)$
satisfy $q+\lambda < \c{ell}/\c{cnt}$, there also holds reliability
\begin{align}\label{eq:averaging:reliable}
 \norm{u-\widehat U}{\XX}
 \le \frac{\c{cnt}}{\c{ell}-\c{cnt}(q+\lambda)}\,\eta_\TT.
\end{align}
\end{theorem}

\begin{proof}
Let $\dual\cdot\cdot_\XX$ denote the scalar product on the Hilbert space
$\XX$ which gives rise to the norm $\norm\cdot\XX$.
Recall that the C\'ea lemma~\eqref{intro:cea} also applies for
$\widehat\XX(\TT)$ and hence $G$. The efficiency estimate~\eqref{eq:averaging:efficient}
therefore follows from the triangle inequality
\begin{align*}
 \eta_\TT 
 &\le \norm{(1-G)(u-\widehat U)}{\XX}
 + \norm{(1-G)u}{\XX}\\
 &\le (\c{cnt}/\c{ell}+q)\,\norm{u-\widehat U}\XX.
\end{align*}
For the proof of the reliability estimate~\eqref{eq:averaging:reliable}, we define
$e:=u-\widehat U$. Let $E\in\widehat\XX(\TT)$ denote the best approximation
in $\widehat\XX(\TT)$, i.e.,
\begin{align}\label{eq:averaging:best approximation}
 \norm{e-E}{\XX} = \min_{V\in\widehat\XX(\TT)}\norm{e-V}{\XX}.
\end{align}
Recall that $E$ is then characterized by the orthogonality
\begin{align*}
 \dual{e-E}{V}_\XX = 0
 \quad\text{for all }V\in\widehat\XX(\TT)
\end{align*}
which implies the Pythagoras theorem
\begin{align*}
 \norm{e-E}\XX^2 + \norm{E}\XX^2 = \norm{e}\XX^2.
\end{align*}
First, note that that the best approximation property~\eqref{eq:averaging:best approximation} and the triangle inequality for $V=Gu+G\widehat U$ prove
\begin{align*}
 \c{cnt}^{-1}\,b(e,e-E) 
 &\le \norm{e}\XX\,\norm{e-E}\XX\\
 &\le \norm{e}\XX\,(\norm{(1-G)u}\XX + \eta_\TT)\\
 &\le \norm{e}\XX\,(q\,\norm{e}\XX + \eta_\TT).
\end{align*}
Second, observe that by definition of $\lambda$ the Galerkin orthogonality
for $e=u-\widehat U$ as well as the estimate $\norm{E}\XX\le\norm{e}\XX$ prove
\begin{align*}
 \c{cnt}^{-1}\,b(e,E)
 = \c{cnt}^{-1}\,\min_{\widehat V\in\XX(\widehat \mesh)}b(e,E-\widehat V)
 &\le \lambda\,\norm{e}\XX\norm{E}\XX\\
 & \le \lambda\,\norm{e}\XX^2.
\end{align*}
Altogether, we see
\begin{align*}
 \c{ell}\norm{e}\XX^2 
 \le b(e,e) &= b(e,e-E) + b(e,E)\\
 &\le \c{cnt}(q+\lambda)\norm{e}\XX^2 + \c{cnt}\,\eta_\TT\,\norm{e}\XX.
\end{align*}
Rearranging this estimate, we conclude the proof.
$\hfill\qed$
\end{proof}

In practice, higher-order polynomials lead to higher-order convergence
rates if the unknown solution $u$ is smooth or if the mesh $\TT$ is appropriately graded. Therefore, one may
expect that the constant $q$ from~\eqref{eq:averaging:q} satisfies
$q\to0$ if the mesh is adaptively refined. The constant
$\lambda$ from~\eqref{eq:averaging:lambda} satisfies 
$0\le\lambda\le1$ by definition. Geometrically, $\lambda<1$ corresponds
to a strengthened Cauchy inequality, cf.\ \cite[Sect 4]{cp07:lecture:notes}.
In practice, $\lambda\ll1$ follows if the mesh $\widehat\TT$ is 
sufficiently fine with respect to $\TT$. We refer to the discussion below.
In conclusion, the assumption $q+\lambda < \c{ell}/\c{cnt}$ required
for the reliability estimate~\eqref{eq:averaging:reliable} can be satisfied
in practice.

As for the $(h-h/2)$-error estimator from Section~\ref{section:est:hh2},
a practical BEM application has, first, to replace the non-local norm
$\norm\cdot\XX$ by some easily computable local norm, e.g., some locally
weighted $L_2$-norm resp.\ $H^1$-seminorm. Moreover, the computationally
expensive Galerkin projection $G$ has to be replaced by some numerically
cheaper operator $\Pi:\XX(\widehat\TT)\to\widehat\XX(\TT)$. Both aspects
are discussed for the weakly singular and hypersingular model problem
in the following subsections.

We finally note that for our applications, i.e., weakly singular and
hypersingular integral equation, averaging on large patches turns out
to be equivalent to $(h-h/2)$-type error estimation.

\paragraph{Weakly singular operator:}
Averaging on large patches for weak\-ly singular integral equations in 2D 
and 3D BEM has first been proposed and analyzed in~\cite{cp06}. We also refer
to~\cite{afp12} for the discussion on anisotropic mesh refinement.
In~\cite{cp06}, it holds $\XX(\widehat\TT) = \Pp^p(\widehat\TT)$ and 
$\widehat\XX(\TT)=\Pp^{p+1}(\TT)$. We suppose that $\widehat\TT$ is obtained 
from $k$ uniform refinements of $\TT$, i.e., the corresponding mesh-sizes 
satisfy
\begin{align}
 \widehat h = 2^{-k}\,h.
\end{align}
Let $\pi_{\widehat\TT}^{p}$ be the $L_2$-projection onto 
$\Pp^{p}(\widehat\TT)$. Fix $V\in\linebreak\Pp^{p+1}(\TT)$. The approximation estimate 
from Lemma~\ref{lem:L2:Pp:apx} yields
\begin{align*}
 \min_{\widehat V\in\Pp^p(\widehat\TT)}&
 \norm{V-\widehat V}{\wilde H^{-1/2}(\Gamma)}
 \le \norm{(1-\pi_{\widehat\TT}^{p})V}{\wilde H^{-1/2}(\Gamma)}
 \\
 &\lesssim \norm{\widehat h^{1/2}V}{L_2(\Gamma)}
 \le \norm{(\widehat h/h)^{1/2}}{L^\infty(\Gamma)}\,
 \norm{h^{1/2}V}{L_2(\Gamma)}.
\end{align*}
By choice of $\widehat\TT$, it holds $\norm{(\widehat h/h)^{1/2}}{L^\infty(\Gamma)} \le 2^{-k/2}$. The inverse estimate of Lemma~\ref{lem:Pp:invest}
proves
\begin{align*}
 \norm{h^{1/2}V}{L_2(\Gamma)}
 \lesssim \norm{V}{\wilde H^{-1/2}(\Gamma)}.
\end{align*}
Combining these observations, we see that the constant $\lambda$
from~\eqref{eq:averaging:lambda} satisfies, for $k$ sufficiently large,
\begin{align*}
\lambda &:= \max_{V\in\Pp^{p+1}(\TT)}\min_{\widehat V\in\Pp^p(\widehat\TT)}
 \frac{\norm{V-\widehat V}{\wilde H^{-1/2}(\Gamma)}}{\norm{V}{\wilde H^{-1/2}(\Gamma)}}
\lesssim 2^{-k/2} \ll1,
\end{align*}
where the hidden constant depends only on 
$\Gamma$, shape regularity of $\TT$, and the polynomial degree $p$.
Moreover, standard approximation results prove (see e.g.~\cite{ss11}) 
that, at least for smooth solutions $u$, the constant $q$ 
from~\eqref{eq:averaging:q} satisfies $q=\OO(h^{p+5/2}/\widehat h^{p+3/2})
=\OO(h)$. 

The following theorem is first found in~\cite[Sect~5]{cp06} and 
formulated in the energy norm $\norm\cdot{\slo}\simeq\norm\cdot{\wilde H^{-1/2}(\Gamma)}$. Note that $\eta_\mesh$ corresponds to the abstract error
estimator $\eta_\TT$ from the abstract Theorem~\ref{theorem:averaging}.
Since the proof is similar to that of Theorem~\ref{thm:def:hh2:weaksing},
we omit the details.

\begin{theorem}\label{theorem:averaging:weaksing}
Let $\pi_\TT^{p+1}$ denote the $L_2$-orthogonal projection onto $\Pp^{p+1}(\TT)$.
Let $G_\TT^{p+1}$ denote the Galerkin projection~\eqref{eq:averaging:galerkin} 
onto $\Pp^{p+1}(\TT)$. Then, the estimators
  \begin{align*}
    \begin{split}
    \begin{array}{rclcrcl}
      \eta_\mesh &:=& \norm{(1-G_\TT^{p+1})\widehat\Phi}{\slo},
      &\,&
      \mu_\mesh &:=& \norm{h_\mesh^{1/2}(1-G_\TT^{p+1})\widehat\Phi}{L_2(\Gamma)},\\
      \widetilde\eta_\mesh &:=& \norm{(1-\pi^{p+1}_\mesh)\widehat\Phi}{\slo},
      &\,&
      \widetilde\mu_\mesh &:=& \norm{h_\mesh^{1/2}(1-\pi^{p+1}_\mesh)\widehat\Phi}{L_2(\Gamma)},
    \end{array}
    \end{split}
  \end{align*}
satisfy the equivalence estimates 
  \begin{align*}
    \eta_\mesh &\leq \wilde\eta_\mesh\leq C_\slo^{1/2}\c{approx}\wilde\mu_\mesh,\\
    \wilde\mu_\mesh &\leq \mu_\mesh \leq 2^{k/2}\c{inv}\c{ell}^{-1/2} \eta_\mesh,
  \end{align*}
  where $C_V = \norm{\slo}{\wilde H^{-1/2}(\Gamma)\rightarrow H^{1/2}(\Gamma)}$ and
  $\c{ell}$ are the stability and ellipticity constants of the single layer operator $\slo$,
  $\c{approx}$ is the constant of Lemma~\ref{lem:L2:Pp:apx}, and $\c{inv}$ is the
  constant of the inverse estimate of Lemma~\ref{lem:Pp:invest}.
\end{theorem}

The numerical experiments in~\cite{cp06,cp07:lecture:notes,effp09,afp12}
give empirical evidence that $k=2$ seems to be sufficient in practice.
As first observed in~\cite[Thm.~5.3]{effp09} for lowest-order 2D BEM $p=0$, one 
can prove that averaging on large patches is equivalent to 
$(h-h/2)$-error estimation. The argument also transfers to 3D and arbitrary
polynomial degree $p\ge0$.

\begin{mycorollary}\label{corollary:averaging:weaksing}
For all $T\in\TT$, it holds 
\begin{align}\label{eq:averaging:equiv}
\begin{split}
 \norm{(1-\pi^{p+1}_\mesh)\widehat\Phi}{L_2(T)}
 &\le 
 \norm{(1-\pi^{p}_\mesh)\widehat\Phi}{L_2(T)}\\
 &\le C_{\rm equiv}\, \norm{(1-\pi^{p+1}_\mesh)\widehat\Phi}{L_2(T)},
\end{split}
\end{align}
where the constant $C_{\rm equiv}$ depends only on the polynomial degree $p$.
Comparing the error estimators $\widetilde\mu_\TT$ of Theorem~\ref{thm:def:hh2:weaksing} and Theorem~\ref{theorem:averaging:weaksing},
this proves that all eight error estimators are equivalent. In particular,
the estimate~\eqref{eq:averaging:equiv} shows that the equivalence of the
respective $\widetilde\mu_\TT$ estimators holds even elementwise.
\end{mycorollary}

\begin{proof}
The lower bound in~\eqref{eq:averaging:equiv} follows from the local
best approximation property
\begin{align*}
 \norm{(1-\pi^{p+1}_\mesh)\widehat\Phi}{L_2(T)}
 = \min_{\Psi\in\Pp^{p+1}(T)}\norm{\widehat\Phi-\Psi}{L_2(T)}
\end{align*}
of the $L_2$-projection $\pi^{p+1}_\TT$ and 
nestedness $\Pp^{p}(T)\subseteq\Pp^{p+1}(T)$. To prove the 
upper bound in~\eqref{eq:averaging:equiv}, observe that 
\begin{align*}
\norm{(1-\pi^{p+1}_\mesh)\widehat\Phi}{L_2(T)}=0
\quad\Longleftrightarrow\quad
\norm{(1-\pi^{p}_\mesh)\widehat\Phi}{L_2(T)}=0.
\end{align*}
Therefore, the equivalence follows from scaling arguments and 
equivalence of seminorms on finite dimensional spaces.
$\hfill\qed$
\end{proof}

\paragraph{Hypersingular operator:}
For hypersingular integral equations, averaging on large patches has been
proposed and analyzed for lowest-order 2D BEM in~\cite{cp07}. The equivalence
of $(h-h/2)$-type error estimators
(cf.\ Theorem~\ref{thm:def:hh2:hypsing})
and averaging on large patches has been
proved in~\cite{efgp12}. These results have been generalized to 3D BEM and 
arbitrary polynomial order $p\ge1$ in~\cite{affkp13}. Altogether, the 
results from Theorem~\ref{theorem:averaging:weaksing} and 
Corollary~\ref{corollary:averaging:weaksing} hold accordingly.
For these reasons, we leave the details to the reader and refer to the
given references.
\subsection{ZZ-type error estimator}\label{section:zzest}
The idea of the ZZ-type error estimator (in the context of FEM also \emph{gradient recovery} estimator) is to \emph{recover} a smoother approximation of the computed solution and to compare it with 
the discrete solution. Since the seminal work~\cite{zz}, the ZZ-type error estimators for FEM became very popular within the engineering community due to their implementational ease. Although ZZ-type error estimators are mathematically well-developed for FEM, see e.g.~\cite{bc02,bc02b,cc04,rod94}, there was no theory for BEM until~\cite{zz2014} which treats the 2D case and lowest-order elements.
In our presentation, we extend the approach to $d=2,3$ but stick with lowest-order elements $p=0$ for weakly singular integral equations resp.\ $p=1$ for hypersingular integral equations.
\paragraph{Weakly singular operator:}\label{section:zzest:weak}
The ZZ-type error estimator from\linebreak
\cite{zz2014} reads
\begin{align*}
\est{\ell}{}^2:=\sum_{\el\in\mesh_\ell}\est{\ell}{\el}^2:=\sum_{\el\in\mesh_\ell}h_\el \norm{(1-A_\ell)\Phi_\ell}{L_2(\el)}^2,
\end{align*}
where the smoothing operator $A_\ell:\,L_2(\Gamma)\to\Pp^1(\mesh_\ell)$ is defined as follows: Let $z$ denote a node of $\mesh_\ell$ and let $\omega_z:= T_1\cup\ldots\cup T_{\#\omega_z}$ be the node patch.
\begin{itemize}
 \item If the normal vector of $\Gamma$ does not jump at $z$, define
 \begin{align}\label{zzest:weak:Adef1}
  (A_\ell \psi)(z):=|\omega_z|^{-1}\int_{\omega_z} \psi\,dz.
 \end{align}
\item If the normal vector of $\Gamma$ jumps at $z$, find sets $C_1,\ldots,C_{m_z}$, with $m_z\leq \#\omega_z$ and $\bigcup_{i=1}^{m_z}C_i = \omega_z$ such that
the normal vector does not jump on the $C_i$, $i=1,\ldots,m_z$. Then, define for all $i=1,\ldots,m_z$
\begin{align}\label{zzest:weak:Adef2}
 (A_\ell \psi)|_{C_i}(z)&:=|C_i|^{-1}\int_{C_i} \psi\,dz.
\end{align}
\end{itemize}
This definition is useful since $\Phi_\ell$ approximates a normal derivative and is supposed to jump at corners and edges of  $\Gamma$.
\begin{remark}
For $d=2$, the definition of $A_\ell$ simplifies as one only has to check if the normal vector jumps at a given node. Then, one integrates separately over the two adjacent elements. The search for continuity components $C_i$ is no longer required. Also for $d=3$, one may save some implementational efforts by just setting $C_i=\overline{T_i}$ for all $T_i\subseteq \omega_z$. This might not be the optimal solution, but still works in practice. 
\end{remark}

\begin{theorem}\label{zzest:thm:rel}
Let $\mesh_\ell$ be the uniform refinement of some mesh $\mesh_\ell^\prime$. Then, there holds
\begin{align*}
\norm{\Phi_\ell- \Phi_\ell^\prime}{{\widetilde H}^{-1/2}(\Gamma)}\leq C_{\rm ZZ} \est{\ell}{}
\end{align*}
for the corresponding Galerkin solutions $\Phi_\ell$ and $\Phi_\ell^\prime$.
Under the saturation assumption (Assumption~\ref{ass:sata}), this implies
\begin{align*}
\norm{\phi-\Phi_\ell}{{\widetilde H}^{-1/2}(\Gamma)}\leq \widetilde C_{\rm ZZ} \est{\ell}{}.
\end{align*}
The constant $C_{\rm ZZ}>0$ depends only on $\Gamma$ and all possible shapes of element patches in $\mesh_\ell$, while $\widetilde C_{\rm ZZ}>0$ depends additionally on $\c{sata}$ from Assumption~\ref{ass:sata}.
\end{theorem}
\begin{proof}
The complete proof for the 2D situation can be found in~\cite[Thm.~5]{zz2014}. Here, we only provide a brief sketch.
First, we use Theorem~\ref{thm:def:hh2:weaksing} to see
\begin{align*}
\norm{\Phi_\ell- \Phi_\ell^\prime}{{\widetilde H}^{-1/2}(\Gamma)}\simeq \norm{h_\ell^{1/2}(1-\pi_\ell^{0\prime})\Phi_\ell}{L_2(\Gamma)},
\end{align*}
where $\pi_\ell^{0\prime}:\,L_2(\Gamma)\to \Pp^0(\mesh_\ell^\prime)$. With the element patch $\omega_T:=\bigcup\set{\el^\prime \in\mesh_\ell}{\overline\el \cap\overline\el^\prime \neq \emptyset}$, the elementwise estimate
\begin{align*}
\norm{h_\ell^{1/2}(1-\pi_\ell^{0\prime})\Phi_\ell}{L_2(\el)}^2\lesssim \sum_{\el^\prime \subseteq \omega_\el}\est{\ell}{\el^\prime}^2
\end{align*}
then follows by scaling arguments, and the hidden constant depends on the number of different patch shapes of $\mesh_\ell$. This proves 
\begin{align*}
 \norm{\Phi_\ell-\Phi_\ell^\prime}{\widetilde H^{-1/2}(\Gamma)}\lesssim \est{\ell}{}.
\end{align*}
Under the saturation assumption, we derive
\begin{align*}
\norm{\phi-\Phi_\ell}{{\widetilde H}^{-1/2}(\Gamma)}&\lesssim \norm{\phi-\Phi_\ell^\prime}{{\widetilde H}^{-1/2}(\Gamma)}
\lesssim \norm{\Phi_\ell-\Phi_\ell^\prime}{{\widetilde H}^{-1/2}(\Gamma)}.
\end{align*}
This concludes the proof.
$\hfill\qed$
\end{proof}

\begin{theorem}\label{zzest:thm:eff}
There holds
\begin{align*}
C_{\rm ZZ}^{-1} \est{\ell}{} \leq \norm{\phi-\Phi_\ell}{{\widetilde H}^{-1/2}(\Gamma)} +
\min_{\Psi\in\Sp^{1}(\mesh_\ell)}\norm{\phi-\Psi}{{\widetilde H}^{-1/2}(\Gamma)}.
\end{align*}
The constant $C_{\rm ZZ}>0$ depends only on $\Gamma$ and all possible patch shapes of $\mesh_\ell$.
\end{theorem}
\begin{proof}
The complete proof can be found in~\cite[Thm.~7]{zz2014}. Here, we only provide a brief sketch. Elementwise arguments show
\begin{align*}
\est{\ell}{}\lesssim \norm{\Phi_\ell-\Psi}{{\widetilde H}^{-1/2}(\Gamma)}
\end{align*}
for all $\Psi\in\Sp^1(\mesh_\ell)$. With this, we obtain
\begin{align*}
\est{\ell}{}&\lesssim \min_{\Psi\in\Sp^{1}(\mesh_\ell)}\norm{\Phi_\ell-\Psi}{{\widetilde H}^{-1/2}(\Gamma)}\\
&\leq \min_{\Psi\in\Sp^{1}(\mesh_\ell)}\norm{\phi-\Psi}{{\widetilde H}^{-1/2}(\Gamma)}+\norm{\phi-\Phi_\ell}{{\widetilde H}^{-1/2}(\Gamma)}.
\end{align*}
This concludes the proof.
$\hfill\qed$
\end{proof}

\paragraph{Hypersingular operator:}\label{section:zzest:hyp}
The ZZ-type error estimator from\linebreak
\cite{zz2014} reads
\begin{align*}
\est{\ell}{}^2:=\sum_{\el\in\mesh_\ell}\est{\ell}{\el}^2:=\sum_{\el\in\mesh_\ell}h_\el \norm{(1-A_\ell)\nabla U_\ell}{L_2(\el)}^2,
\end{align*}
where the smoothing operator $A_\ell:\,\big(L_2(\Gamma)\big)^d\to\big(\Sp^1(\mesh_\ell)\big)^d$ is defined nodewise by
\begin{align*}
  (A_\ell \psi)(z):=|\omega_z|^{-1}\int_{\omega_z} \psi\,dz
 \end{align*}
for all nodes $z$ of $\mesh_\ell$.
The difference to the weakly singular case is the fact that $A_\ell\psi \in \Sp^1(\mesh_\ell)$ is continuous on $\Gamma$, independently of jumps of the normal vector.

\begin{theorem}
Let $\mesh_\ell$ be the uniform refinement of some mesh $\mesh_\ell^\prime$. Then, there holds
\begin{align*}
\norm{U_\ell- U_\ell^\prime}{\widetilde H^{1/2}(\Gamma)}\leq C_{\rm ZZ} \est{\ell}{}.
\end{align*}
Under the saturation assumption (Assumption~\ref{ass:sata}), this implies
\begin{align*}
\norm{u-U_\ell}{{\widetilde H}^{1/2}(\Gamma)}\leq \widetilde C_{\rm ZZ} \est{\ell}{}.
\end{align*}
The constant $C_{\rm ZZ}>0$ depends only on $\Gamma$ and all possible shapes of element patches in $\mesh_\ell$, while $\widetilde C_{\rm ZZ}>0$ depends additionally on $\c{sata}$ from Assumption~\ref{ass:sata}.
\end{theorem}
\begin{proof}
The proof is similar to the weakly singular case in Theorem~\ref{zzest:thm:eff}
and can be found in~\cite[Thm.~1]{zz2014}.
$\hfill\qed$
\end{proof}

\begin{theorem}
There holds
\begin{align*}
C_{\rm ZZ}^{-1} \est{\ell}{} \leq \norm{u-U_\ell}{H^{1/2}(\Gamma)} +
\min_{V\in\widetilde\Sp^{2,1}(\mesh_\ell)}\norm{u-V}{{\widetilde H}^{-1/2}(\Gamma)}.
\end{align*}
The space $\widetilde\Sp^{2,1}_0(\mesh_\ell):=\Sp^2(\mesh_\ell)\cap C^1(\Gamma)\cap \widetilde H^{1/2}(\Gamma)$ denotes the space of all piecewise quadratics which are globally differentiable with zero trace at $\partial\Gamma$.
If $\Gamma$ is closed, i.e.\ $\partial\Gamma=\emptyset$, we have $\widetilde \Sp^{2,1}(\mesh_\ell)=\Sp^{2,1}(\mesh_\ell)$.
The constant $C_{\rm ZZ}>0$ depends only on $\Gamma$ and all possible patch shapes of $\mesh_\ell$.
\end{theorem}
\begin{proof}
The proof is similar to the weakly singular case in Theorem~\ref{zzest:thm:eff} and can be found in~\cite[Thm.~3]{zz2014} for $d=2$.
$\hfill\qed$
\end{proof}

\subsection{Two-equation estimators}\label{section:twoeq}
In this section, we consider only the Dirichlet- or Neumann problem, i.e., $\Gamma$
is always the boundary of a bounded domain. In these cases, there is a method that differs
completely from residual-based methods or approaches based on space enrichment.

In~\cite{ss00:dirichlet,ss00:neumann,s00} it is shown that the error of, e.g., the Dirichlet
problem, i.e., $\phi-\Phi$, fulfills a second-kind integral equation. Indeed,
the representation formula allows to define a potential based on the approximate Neumann data
and the exact Dirichlet data via
$\wilde u = \slp\Phi - \dlp g$. The trace and normal derivative of this potential fulfill
\begin{align*}
  \trace \wilde u &= \slo\Phi + (1/2-K)g\\
  \slo\dn\wilde u &= (1/2+\dlo)\trace \wilde u,
\end{align*}
where $\trace$ denotes the trace operator and $\dn$ denotes the (co-) normal derivative.
The combination of this equations and the identities
\begin{align*}
  (1/2+\dlo)(1/2-\dlo)=\slo\hyp, \quad  \dlo\slo = \slo\adlo
\end{align*}
show
\begin{align*}
  \slo\dn\wilde u = \slo(1/2+\adlo)\Phi + \slo\hyp g,
\end{align*}
and as $\hyp g = (1/2-\adlo)\phi$, this yields the second-kind integral equation for the error
\begin{align*}
  \dn\wilde u - \Phi = (1/2-\adlo)(\phi-\Phi),
\end{align*}
cf.~\cite[Lemma~2.1]{ss00:dirichlet}.
This equation has to be solved approximately in $H^{-1/2}(\Gamma)$ and the corresponding
norm of the solution to be localized. The approximate solution is based on the following observation,
which is proved in~\cite[Thm.~3.1]{ost:wendland:01}.
\begin{theorem}
  There is a constant $\c{cont}<1$, such that for all $\phi \in H^{-1/2}(\Gamma)$ holds
  \begin{align*}
    \norm{(1/2 + \adlo)\phi}{\slo} \leq \c{cont} \norm{\phi}{\slo}.
  \end{align*}
\end{theorem}
\begin{remark}
  The notation used in this section is bounded to Galerkin methods. Error estimators of the
  type presented here do not use orthogonality and can therefore be defined also for
  collocation or qualocation methods,
  where, instead of the factor $1/2$, a function has to be used which represents the
  curvature of the boundary.
\end{remark}
According to the last theorem, the Neumann series
\begin{align*}
  (1/2-\adlo)^{-1} = \sum_{j=0}^\infty (1/2 + \adlo)^j
\end{align*}
converges in the norm $\norm{\cdot}{\slo}$,
so that one may define for $J\in\N_0$ the global error estimator
\begin{align*}
  \eta^{(J)} := \norm{\sum_{j=0}^J (1/2 + \adlo)^j (\dn\wilde u - \Phi)}{\slo}.
\end{align*}
Due to representation via a Neumann series, the estimator is efficient and reliable,
as is shown in~\cite{ss00:dirichlet}.
\begin{theorem}\label{thm:est:ost:eff:rel}
  The estimator $\eta^{(J)}$ is efficient and reliable,
  \begin{align*}
    \frac{1}{1+\c{cont}^{J+1}}\eta^{(J)} \leq \norm{\phi-\Phi}{\slo}
    \leq \frac{1}{1-\c{cont}^{J+1}}\eta^{(J)}
  \end{align*}
\end{theorem}
The same arguments also apply for the Neumann problem, where the error estimator for an approximation
$U$ is defined by
\begin{align*}
  \eta^{(J)} := \norm{\sum_{j=0}^J \left( P(1/2-\dlo)^j P (\trace \wilde u - U) \right)}{\hyp},
\end{align*}
cf.~\cite{ss00:neumann,s00}, where $P$ is an operator that ensures vanishing integral mean.
The analogue to Theorem~\ref{thm:est:ost:eff:rel} is of course valid.\\

On the implementational side, one has to introduce an approximation of the application of
the Neumann series. If, e.g., $\Phi\in\Pp^p(\mesh)$ is an approximation of the solution of a Dirichlet
problem, the $L_2(\Gamma)$-projection $\pi$ on a space finer than $\Pp^p(\mesh)$ may be used to compute
\begin{align*}
  \wilde\eta^{(J)} := \norm{\sum_{j=0}^J \left( \pi(1/2 + \adlo) \right)^j
  \pi(\dn\wilde u - \Phi)}{\slo},
\end{align*}
in which case the following result is valid, cf.~\cite[Thm.~3.3]{ss00:dirichlet}.
\begin{theorem}\label{thm:est:ost:apx:eff:rel}
  Let $(\mesh_\ell)_{\ell\in\N_0}$ be a uniform sequence of meshes with mesh-width $h_\ell$ and
  $(\widehat\mesh_\ell)_{\ell\in\N_0}$ be a uniform sequence of meshes with mesh-width
  $\widehat h_\ell$ such that $\mesh_\ell\subseteq \widehat\mesh_\ell$ for $\ell\in\N_0$.
  If $\Phi_\ell\in\Pp^p(\mesh_\ell)$ is an approximation to the exact solution $\phi$ of
  a Dirichlet problem with data $g$, $\wilde u_\ell := \slp \Phi_\ell - \dlp g$,
  and $\widehat\pi_\ell:L_2\rightarrow\Pp^p(\widehat\mesh_\ell)$ is the $L_2$ orthogonal projection,
  the estimator
  \begin{align*}
    \wilde\eta^{(J)}_\ell := \norm{\sum_{j=0}^J \left( \widehat\pi_\ell(1/2 + \adlo) \right)^j
    \widehat\pi_\ell(\dn\wilde u_\ell - \Phi_\ell)}{\slo},
  \end{align*}
  is efficient and reliable in the sense that there exists a constant $\setc{est:2eq}>0$ such that
  \begin{align*}
    \frac{1}{1+\c{cont}^{J+1}} \left\{ \wilde\eta^{(J)}_\ell 
    - \c{est:2eq} J\widehat h_\ell \eta^{(0)}_\ell \right\}
    \leq
    \norm{\phi-\Phi_\ell}{\slo}\\
    \leq
    \frac{1}{1-\c{cont}^{J+1}} \left\{ \wilde\eta^{(J)}_\ell 
    + \c{est:2eq} J\widehat h_\ell \eta^{(0)}_\ell \right\}
  \end{align*}
\end{theorem}

The localization of $\wilde\eta^{(J)}_\ell$ in this case is a more subtle matter, as no orthogonality or
approximation property can be used. Indeed, the derivation of this type of estimator assumed
no whatsoever special approximation property. In~\cite{ss00:dirichlet} the authors
use the localization
\begin{align*}
  \left( \wilde \eta^{(J)}_\ell \right)^2 = \sum_{\el\in\mesh_\ell} \wilde \eta^{(J)}_\el
\end{align*}
where
\begin{align*}
  \wilde \eta^{(J)}_\el &= \dual{\slo e^{(J)}_\ell}{e^{(J)}_\ell}_\el,\\
  e^{(J)}_\ell &= \sum_{j=0}^J \left( \pi(1/2 + \adlo) \right)^j \pi(\dn\wilde u - \Phi_\ell)
\end{align*}
This is indeed not a fully localized estimator since it involves the single layer operator $\slo$.
In~\cite{s00}, it is suggested to use the multilevel localization from
Section~\ref{section:localization:multilevel}. With the notation from
Theorem~\ref{thm:est:ost:apx:eff:rel}, define the operator
\begin{align*}
  A_\ell^{s} \phi := \sum_{k=0}^\ell h_k^{-2s} (\pi_k - \pi_{k-1})\phi
  + \widehat h_\ell^{-2s} (\widehat\pi_\ell - \pi_\ell)\phi
\end{align*}
and note that for $\widehat\Phi_\ell\in\Pp^0(\widehat\mesh_\ell)$ holds
\begin{align*}
  \sum_{k=0}^\ell &h_k^{-2s} \norm{(\pi_k - \pi_{k-1})\Phi_\ell}{\ell_2(\Gamma)}^2
+ \widehat h_\ell^{-2s} \norm{(\widehat\pi_\ell - \pi_\ell)\Phi_\ell}{\ell_2(\Gamma)}^2\\
&= \dual{A_\ell^{s}\widehat\Phi_\ell}{\widehat\Phi_\ell}_{\ell_2(\Gamma)}
= \dual{A_\ell^{s/2}\widehat\Phi_\ell}{A_\ell^{s/2}\widehat\Phi_\ell}_{\ell_2(\Gamma)},
\end{align*}
where the last identity follows from the properties of the $\pi_k$, cf.~\cite[Prop.~2.1]{s00}.
Finally, Theorem~\ref{thm:multilevel} states that the $H^{-1/2}$-norm of $e_\ell^{(J)}$
can be bounded by
\begin{align*}
  &\dual{A_\ell^{-1/4}e^{(J)}_\ell}{A_\ell^{-1/4}e^{(J)}_\ell}_{\ell_2(\Gamma)}\\
  &\qquad =
  \sum_{\el\in\mesh_\ell} \dual{A_\ell^{-1/4}e^{(J)}_\ell}{A_\ell^{-1/4}e^{(J)}_\ell}_{\ell_2(\el)},
\end{align*}
via
\begin{align*}
  \sum_{\el\in\mesh_\ell} \eta_\el^2
  \lesssim
  \left( \wilde\eta^{(J)}_\ell \right)^2
  \lesssim (\ell+2)^2 \sum_{\el\in\mesh_\ell} \eta_\el^2.
\end{align*}
\subsection{A~posteriori error control of data approximation}\label{section:dataapproximation}
In practice, the right-hand side $F$ of~\eqref{intro:galerkin} cannot be computed analytically. For the weakly singular
integral equation from Proposition~\ref{prop:dirichlet}, it holds for instance $ F = (K+1/2) f$ for some given
$f \in H^{1/2}(\Gamma)$. For the hypersingular integral equation from Proposition~\ref{prop:neumann}, it holds
for instance $ F = (K^\prime-1/2) f$ for some given $f \in H^{-1/2}_0(\Gamma)$. In either case, the
action of the integral operator to the continuous data $f$ is well-defined, but hardly computable.
In practice, the given data $f$ is therefore replaced by some piecewise polynomial data
$f_\ell$. This leads to a computable right-hand side for the Galerkin discretization~\eqref{intro:galerkin}, where
$F$ is replaced by some approximation $F_\ell$. The following short sections give insight
in how to control this additional approximation error.

\subsubsection{Inverse estimates for integral operators}
The following inverse-type estimates have independently first been shown in~\cite{fkmp13,gantumur} for piecewise polynomials. While~\cite{fkmp13} considered lowest-order polynomials on piecewise polygonal geometries,~\cite{gantumur} covers arbitrary-order piecewise polynomials but is restricted to smooth boundaries $\Gamma$. In~\cite{invest}, the results of~\cite{fkmp13,gantumur} are generalized to general densities instead of piecewise polynomials.
\begin{mylemma}\label{estred:lem:invest}
There exists a constant $\setc{opt:weaksing:invest}>0$ such that for all $v\in \widetilde  H^1(\Gamma)$ and all $\psi\in L_2(\Gamma)$
\begin{align}\label{opt:eq:invest}
\begin{split}
 \norm{h_\ell^{1/2}\nabla_\Gamma V\psi&}{L_2(\Gamma)}+ \norm{h_\ell^{1/2}(1/2-K^\prime)\psi}{L_2(\Gamma)}\\
 &\leq  \c{opt:weaksing:invest} \big(\norm{\psi}{\widetilde H^{-1/2}(\Gamma)}+\norm{h_\ell^{1/2}\psi}{L_2(\Gamma)}\big),\\
\norm{h_\ell^{1/2} Wv&}{L_2(\Gamma)}+\norm{h_\ell^{1/2}\nabla_\Gamma (1/2+K)v}{L_2(\Gamma)}\\
&\leq  \c{opt:weaksing:invest}\big(\norm{v}{\widetilde H^{1/2}(\Gamma)}+\norm{h_\ell^{1/2}\nabla_\Gamma v}{L_2(\Gamma)}\big).
\end{split}
\end{align}
The constant $\c{opt:weaksing:invest}$ depends only on the shape regularity of $\mesh_\ell$ and on $\Gamma$. In the special case
$v=V_\ell\in\widetilde\Sp^p(\mesh_\ell)$ and $\psi=\Psi_\ell\in\Pp^p(\mesh_\ell)$, there even holds
\begin{align}\label{opt:eq:invest:discrete}
 \begin{split}
\norm{h_\ell^{1/2}\nabla_\Gamma V\Psi_\ell}{L_2(\Gamma)} &+  \norm{h_\ell^{1/2}(1/2-K^\prime)\Psi_\ell}{L_2(\Gamma)}\\
&\leq \c{opt:weaksing:invest:discrete} \norm{\Psi_\ell}{\widetilde H^{-1/2}(\Gamma)},\\
 \norm{h_\ell^{1/2} WV_\ell}{L_2(\Gamma)}&+ \norm{h_\ell^{1/2}\nabla_\Gamma (1/2+K)V_\ell}{L_2(\Gamma)}\\
 &\leq \c{opt:weaksing:invest:discrete}\norm{V_\ell}{\widetilde H^{1/2}(\Gamma)}.
 \end{split}
\end{align}
The constant $\setc{opt:weaksing:invest:discrete}>$ depends only on $\Gamma$, the shape regularity of $\mesh_\ell$, and on the polynomial degree $p$.
\end{mylemma}

\subsubsection{Weakly singular integral equation}
\label{section:aposteriori:data:weaksing}
We show two eligible ways of data approximation for the weakly singular integral equation from
Proposition~\ref{prop:galerkin:dirichlet} with $\Gamma=\partial\Omega$.
First, the approximation for the right-hand side $F$ can be done via the Scott-Zhang projection $J_\ell:\, L_2(\Gamma)\to \Sp^{p+1}(\mesh_\ell)$ from Section~\ref{section:sz}, i.e.
\begin{align}\label{data:apx:sz}
 F_\ell:=(1/2+K)J_\ell f,
\end{align}
or via the $L_2$-orthogonal projection $\Pi_\ell^{p+1}:\,L_2(\Gamma)\to\linebreak\Sp^{p+1}(\mesh_\ell)$ from
Section~\ref{section:L2}, i.e.,
\begin{align}\label{data:apx:l2}
 F_\ell:=(1/2+K)\Pi_\ell^{p+1} f,
\end{align}
where we additionally assume that $\Pi_\ell^{p+1}$ is $H^1$-stable (cf., Section~\ref{section:L2}).
In the following, we denote with $P_\ell$ either the Scott-Zhang projection $P_\ell=J_\ell$ or the $L_2$-orthogonal projection $P_\ell=\Pi_\ell^{p+1}$.
Let $\widetilde\Phi_\ell\in\Pp^p(\mesh_\ell)$ denote the solution of~\eqref{intro:galerkin} with right-hand side~\eqref{data:apx:sz} or~\eqref{data:apx:l2}. The introduced approximation error
can be controlled with the following result
\begin{mylemma}\label{data:errstab:weaksing}
There exists a constant $\setc{data:errstab:weaksing}>0$ such that
\begin{subequations}
\begin{align}
 \c{data:errstab:weaksing}^{-1}\norm{\Phi_\ell-\widetilde\Phi_\ell}{H^{-1/2}(\Gamma)}&\leq  \norm{h_\ell^{1/2}\nabla_\Gamma (1-P_\ell)f}{L_2(\Gamma)}.\label{data:osc:weak:a}
 \end{align}
 Moreover, there exists a
 constant $\setc{data:errstab:weaksing:nvb}>0$ such that
 \begin{align}
  \c{data:errstab:weaksing:nvb}^{-1}\norm{\Phi_\ell-\widetilde\Phi_\ell}{H^{-1/2}(\Gamma)}&\leq \norm{h_\ell^{1/2}(1-\pi_\ell^p)\nabla_\Gamma f}{L_2(\Gamma)}.\label{data:osc:weak:b}
\end{align}
\end{subequations}
The constant $\c{data:errstab:weaksing}$ depends only on $\Gamma$, the shape regularity of $\mesh_\ell$, and on the polynomial degree $p$. The constant $\c{data:errstab:weaksing:nvb}$ depends additionally on all possible shapes of element patches in $\mesh_\ell$.
\end{mylemma}
\begin{proof}
 The Galerkin formulation~\eqref{intro:galerkin} shows
 \begin{align*}
  \form{\Phi_\ell-\widetilde\Phi_\ell}{\Phi_\ell-\widetilde\Phi_\ell}&=\dual{(1/2+K)(f-P_\ell f)}{\Phi_\ell-\widetilde\Phi_\ell}_\Gamma\\
  &\lesssim \norm{f-P_\ell f}{ H^{1/2}(\Gamma)}\norm{\Phi_\ell-\widetilde\Phi_\ell}{H^{-1/2}(\Gamma)},
 \end{align*}
where we used the stability of $K$ from Theorem~\ref{thm:stability}. With ellipticity from Theorem~\ref{thm:ellipticity}, this yields
\begin{align*}
\norm{\Phi_\ell-\widetilde\Phi_\ell}{H^{-1/2}(\Gamma)}\lesssim \norm{f-P_\ell f}{ H^{1/2}(\Gamma)}.
\end{align*}
We use the assumption on $H^1$-stability on $\Pi_\ell^{p+1}$ or, in case of $P_\ell=J_\ell$, the $H^1$-stability of $J_\ell$ from Lemma~\ref{lem:sz}. With Lemma~\ref{lem:Sp:apx}, it holds
\begin{align*}
 \norm{\Phi_\ell-\widetilde\Phi_\ell}{H^{-1/2}(\Gamma)}\lesssim \min_{V_\ell\in \Sp^{p+1}(\mesh_\ell)}\norm{h_\ell^{1/2}\nablag(f-V_\ell)}{L_2(\Gamma)}.
\end{align*}
This shows~\eqref{data:osc:weak:a}. For~\eqref{data:osc:weak:b}, the result~\cite[Proposition~8]{dirichlet3d} finally shows
\begin{align*}
 &\min_{V_\ell\in \Sp^{p+1}(\mesh_\ell)}\norm{h_\ell^{1/2}\nablag(f-V_\ell)}{L_2(\Gamma)}\\
 &\qquad\qquad\lesssim \norm{h_\ell^{1/2}(1-\pi_\ell^{p})\nablag f}{L_2(\Gamma)},
\end{align*}
where the hidden constant depends on the shapes of the element patches in $\mesh_\ell$.
This concludes the proof.
$\hfill\qed$
\end{proof}
Also the error estimators satisfy certain stability properties.
\begin{mylemma}\label{data:eststab:weaksing}
Let $\est{\ell}{}$ denote one of the $(h-h/2)$-type error estimators defined in Theorem~\ref{thm:hh2:weaksing} or the weighted residual error estimator from Theorem~\ref{thm:est:wres:weaksing:rel}.
Moreover, let $\estd{\ell}{}$ denote the perturbed version of the respective error estimator computed with the perturbed Galerkin approximation $\widetilde \Phi_\ell$ and the perturbed data $F_\ell$.
Then, there exists a constant $\setc{data:eststab:weaksing}>0$ such that
\begin{subequations}
\begin{align}\label{data:osc:weak:est:a}
|\estd{\ell}{}- \est{\ell}{}|\leq\c{data:eststab:weaksing}\norm{h_\ell^{1/2}\nabla_\Gamma (1-P_\ell)f}{L_2(\Gamma)}.
 \end{align}
 Moreover, there exists a constant $\setc{data:eststab:weaksing:nvb}>0$ such that
 \begin{align}
 |\estd{\ell}{}- \est{\ell}{}|\leq\c{data:eststab:weaksing:nvb}\norm{h_\ell^{1/2}(1-\pi_\ell^p)\nabla_\Gamma f}{L_2(\Gamma)}.\label{data:osc:weak:est:b}
\end{align}
\end{subequations}
The constant $\c{data:eststab:weaksing}$ depends only on $\Gamma$, the shape regularity of $\mesh_\ell$, and on $p$. The constant $\c{data:eststab:weaksing:nvb}$ depends additionally on all possible shapes of element patches in $\mesh_\ell$.
\end{mylemma}
\begin{proof}
As example for $(h-h/2)$-type estimators, we choose e.g., $\est{\ell}{}=\norm{h_\ell^{1/2}(1-\pi_\ell^p)\widehat\Phi_\ell}{L_2(\Gamma)}$. Let $\widetilde{\widehat\Phi_\ell}$ denote the solution of~\eqref{intro:galerkin} on the 
uniformly refined space $\widehat\XX_\ell:=\Pp^p(\widehat\mesh_\ell)$ with right-hand side $F_\ell$.
The inverse triangle inequality combined with the inverse estimate from Lemma~\ref{lem:Pp:invest} shows
\begin{align*}
\big| \norm{h_\ell^{1/2}(1-\pi_\ell^{p})\widetilde{\widehat \Phi_\ell}}{L_2(\Gamma)}&-\norm{h_\ell^{1/2}(1-\pi_\ell^{p})\widetilde{\Phi_\ell}}{L_2(\Gamma)}\big|\\
&\leq \norm{h_\ell^{1/2}(\widetilde{\widehat \Phi_\ell}-\widehat{\Phi_\ell})}{L_2(\Gamma)}\\
&\lesssim \norm{\widetilde{\widehat \Phi_\ell}-\widehat{\Phi_\ell}}{\widetilde H^{-1/2}(\Gamma)}.
\end{align*}
The remaining statement follows from~\eqref{data:osc:weak:a}--\eqref{data:osc:weak:b}.

The estimate for the weighted residual error estimator is similar. There holds
\begin{align*}
\big| \norm{h_\ell^{1/2}\nablag &(V\widetilde \Phi_\ell - F_\ell)}{L_2(\Gamma)}- \norm{h_\ell^{1/2}\nablag (V \Phi_\ell - F)}{L_2(\Gamma)}\big|\\
 & \lesssim  \norm{h_\ell^{1/2}\nablag V(\widetilde \Phi_\ell -\Phi_\ell)}{L_2(\Gamma)}\\
 &\qquad+ \norm{h_\ell^{1/2}\nablag (1/2+K)(f-P_\ell f)}{L_2(\Gamma)}.
\end{align*}
We apply the inverse estimate for $V$ from~\eqref{opt:eq:invest:discrete} and for $(1/2+K)$ from~\eqref{opt:eq:invest} to obtain
\begin{align*}
 \big| \norm{h_\ell^{1/2}\nablag &(V\widetilde \Phi_\ell - F_\ell)}{L_2(\Gamma)}- \norm{h_\ell^{1/2}\nablag (V \Phi_\ell - F)}{L_2(\Gamma)}\big|\\
 &\lesssim \norm{\widetilde\Phi_\ell -\Phi_\ell}{\widetilde H^{-1/2}(\Gamma)}\\
&\quad + \norm{f-P_\ell f}{H^{1/2}(\Gamma)} + \norm{h_\ell^{1/2}\nablag (f-P_\ell f)}{L_2(\Gamma)}.
\end{align*}
The remainder follows as in the proof of Lemma~\ref{data:errstab:weaksing}.
$\hfill\qed$
\end{proof}

\subsubsection{Hypersingular integral equation}
\label{section:aposteriori:data:hypsing}
For the hypersingular integral equation from Proposition~\ref{prop:galerkin:neumann} with $\Gamma = \partial\Omega$, the most useful method for data approximation employs the $L_2$-orthogonal projection $\pi_\ell^{p-1}:\,L_2(\Gamma)\to\Pp^{p-1}(\mesh_\ell)$, i.e.,
\begin{align}\label{data:apx:hyp}
 F_\ell:=(1/2-K^\prime)\pi_\ell^{p-1}f.
\end{align}
Note that if $f\in H^{-1/2}_0(\Gamma)$, then also $\pi_\ell^{p-1}f\in H^{-1/2}_0(\Gamma)$.
Let $U_\ell\in \Sp^p(\mesh_\ell)$ denote the solution of~\eqref{intro:galerkin} with right-hand side~\eqref{data:apx:hyp}.
The introduced approximation error
can be controlled with the following result.

\begin{mylemma}\label{data:errstab:hypsing}
There exists a constant $\setc{data:errstab:hypsing}>0$ such that
\begin{align}
 \c{data:errstab:hypsing}^{-1}\norm{U_\ell-\widetilde U_\ell}{H^{1/2}(\Gamma)}&\leq  \norm{h_\ell^{1/2}(1-\pi^{p-1}_\ell)f}{L_2(\Gamma)}.\label{data:osc:hyp:a}
 \end{align}
The constant $\c{data:errstab:hypsing}$ depends only on $\Gamma$, the shape regularity of $\mesh_\ell$, and on the polynomial degree $p$.
\end{mylemma}
\begin{proof}
 The proof follows as for the weakly singular case in Lemma~\ref{data:errstab:weaksing}.
$\hfill\qed$
\end{proof}

Again, also the error estimators satisfy certain stability properties.
\begin{mylemma}\label{data:eststab:hypsing}
Let $\est{\ell}{}$ denote one of the $(h-h/2)$-type error estimators defined in Theorem~\ref{thm:hh2:hypsing} or the weighted residual error estimator from Theorem~\ref{thm:est:wres:hypsing:rel}.
Moreover, let $\estd{\ell}{}$ denote the perturbed version of the respective error estimator computed with the perturbed Galerkin approximation $\widetilde U_\ell$ and the perturbed data $F_\ell$.
Then, there exists a constant $\setc{data:eststab:hypsing}>0$ such that
\begin{align}\label{data:osc:hyp:est:a}
|\estd{\ell}{} - \est{\ell}{}|\leq \c{data:eststab:hypsing}\norm{h_\ell^{1/2}(1-\pi^{p-1}_\ell)f}{L_2(\Gamma)}.
 \end{align}
The constant $\c{data:eststab:hypsing}$ depends only on $\Gamma$, the shape regularity of $\mesh_\ell$, and on $p$.
\end{mylemma}
\begin{proof}
 The proof follows as for  the weakly singular case in Lemma~\ref{data:eststab:weaksing}.
$\hfill\qed$
\end{proof}
\section{A posteriori error estimators for the $p$ and $hp$-versions}\label{section:aposteriori:hp}

The $p$-version of the boundary element method is the extreme case of improving
approximation properties only by increasing polynomial degrees, of piecewise polynomials
on a fixed mesh. A combination of mesh refinement with increasing polynomial degrees
is called $hp$-version. Higher order polynomial degrees are particularly suited for
the approximation of singular functions, the ones that appear due to corners and edges
of domains, recall Section~\ref{sec_reg} for details.

As is well known from finite element error analysis, the $p$-version converges
twice as fast as the $h$-version for problems with singularities when the meshes
match the singularity locations. This is also true of boundary elements. For a
first analysis in two dimensions (on curves) considering hypersingular
and weakly singular operators see~\cite{eps:s:89}. An optimal
analysis of this case has been provided in~\cite{guo:heuer:04}.
In three dimensions (on surfaces) the first $p$-version analysis for the hypersingular
operator appeared in~\cite{schwab:suri:96}. Here, only closed surfaces
are considered, implying $H^1(\Gamma)$ regularity of the solution.
Later, this gap has been closed in~\cite{bespalov:heuer:05} for hypersingular
operators and~\cite{bespalov:heuer:07} presents an analysis for weakly singular operators.
Of course, the pure $p$-version is mainly of theoretical interest since in practice,
mesh refinement is easier to implement than polynomials of high degree in a stable way.
Combining the $h$- and the $p$-version one can choose quasi-uniform or non-uniformly
refined meshes. The $hp$-version with quasi-uniform meshes combines the convergence orders of
both variants (twice the rate with respect to degrees in comparison to mesh refinement).
The corresponding analyses have been given in~\cite{eps:suri:91} and~\cite{guo:heuer:06} (preliminary and optimal estimates, respectively) in two
dimensions for both operators. In three dimensions, and for open surfaces with the
strongest singularities, the publications are~\cite{bespalov:heuer:08}
(hypersingular operator) and~\cite{bespalov:heuer:10}.

The $hp$-version gives full flexibility in choosing any mesh and degree combination,
with analysis for quasi-uniform\linebreak
meshes provided by the publications mentioned before.
In the so-called $hp$-version with geometric meshes one selects a specific combination
of geometrically graded meshes with polynomial degrees that are larger on larger elements.
In three dimensions (on surfaces) this implies the use of an\-isotropic elements and
polynomial degrees that are different in different directions on the same element.
In this way, an exponential rate of convergence (faster than any algebraic order in terms
of numbers of unknowns) can be achieved, cf.~\cite{heuer:maischak:eps:99}.

Finite and boundary element analysis for meshes including anisotropic elements is
challenging. This is due to the fact that no simple scaling arguments related to
affine mappings onto reference elements apply. In many cases, different scaling
properties in different directions get mixed up and make the analysis on distorted
elements cumbersome.

In the case of the $p$-version there is another difficulty. The analysis of low order
methods employs scaling arguments in order to use arguments from the equivalence of
norms in finite-dimensional spaces, defined on reference elements. When considering
the $p$-version, by definition dimensions of approximation spaces on elements are not\linebreak
bounded. This means that simple arguments from the equivalence of different norms
do not apply. Analytical tools for the analysis of $p$- and $hp$-versions are usually
different, and scaling arguments form only a part of the story.

Considering both remarks, on anisotropic elements and on the difficulties with the
$p$-version, is becomes clear that error estimation for the $hp$-version with
geometric meshes is a non-trivial issue. In fact, we are not aware of any publication
analyzing this situation in a satisfying way, neither using residual-based estimators
nor enrichment-based methods. In particular, nothing is known for the a~posteriori
$p$-error estimation of weakly singular operators in three dimensions.
Additive Schwarz theory is the most advanced area dealing with $p$-approximations
of boundary integral operators in three dimensions. This, in particular,
is the case with hypersingular operators. However, let us recall that there is
a satisfying analysis of two-level error estimation on an\-isotropic meshes for
weakly singular integral equations \cite{eh06}, as discussed in Section~\ref{section:est:2level},
cf.~Figure~\ref{fig:weaksing:refine} and Theorem~\ref{thm:weaksing:estim:2level}.

In the following we discuss the existing theory for two-level error estimation
of the $p$- and $hp$-version with quasi-uniform meshes for the solution
of hypersingular integral equations on surfaces \cite{h02,hms02}.
For the $p$- and $hp$-version of the BEM, solving integral equations on curves,
we refer to \cite{hms01}, see also \cite{heuer:thesis}.

Our model problem is the hypersingular integral equation considered in
Proposition~\ref{prop:hypsing}, and for simplicity we consider an open flat surface
$\Gamma$ with polygonal boundary. The meshes $\mesh$ are assumed to be quasi-uniform
with shape-regular elements. Triangles and convex quadrilaterals are allowed.
We will use the notation and framework introduced in Section~\ref{section:cea}.
That means we are considering
\[
   b(u,v) = \dual{\hyp u}{v}_\Gamma,\quad
   \XX=\wilde H^{1/2}(\Gamma),\quad
   \XX_\mesh^p=\wilde\Sp^p(\mesh).
\]
Here, the index $p$ in the discrete space refers to polynomial degrees $p\ge 1$
that can be different on different elements, and even different in different directions.
The exact solution of the problem is $u\in\wilde H^{1/2}(\Gamma)$ and the discrete
solution is denoted by $U\in \wilde\Sp^p(\mesh)$.

Now, in order to define a two-level estimator for the error
$\|u-U\|_{\wilde H^{1/2}(\Gamma)}$, we consider as in Section~\ref{section:est:2level} an
enriched discrete space $\widehat\XX_\mesh^p$ with $\XX_\mesh^p\subset \widehat\XX_\mesh^p\subset\wilde H^{1/2}(\Gamma)$
and decomposition
\begin{align} \label{2level:dec:hypsing}
  \widehat\XX_\mesh^p = \ZZ_{\mesh,0} \oplus \ZZ_{\mesh,1} \oplus \ZZ_{\mesh,2}
  \oplus \cdots \oplus \ZZ_{\mesh,L}.
\end{align}
The resulting error estimator is
\begin{align} \label{hypsing:estim:2level}
   \eta_\mesh := \Bigl(\sum_{j=0}^L \eta_j^2\Bigr)^{1/2},\quad
   \eta_j:=\|P_j(\widehat U-U)\|_b.
\end{align}
Here,
\[
   P_j:\;\widehat\XX_\mesh^p\to \ZZ_{\mesh,j}:\quad
   \dual{\hyp P_j v}{w}_\Gamma=\dual{\hyp v}{w}_\Gamma\quad \forall w\in \ZZ_{\mesh,j}
\]
and
\[
   \|v\|_b^2 = \dual{\hyp v}{v}_\Gamma\quad\forall v\in \ZZ_{\mesh,j}
   \quad (j=0,\ldots,L).
\]
Moreover, $\eta_0=0$ corresponds to $\ZZ_0\subset \XX_\mesh^p$, cf.~Lemma~\ref{lem:ASMprop}.

There are different issues to be considered when selecting the enriched space and
a decomposition.
\begin{itemize}
\item Basis functions for $\widehat\XX_\mesh^p\subset\widetilde H^{1/2}(\Gamma)$ must be continuous.
      This fact restricts the possibility of having stable decompositions of $\widehat\XX_\mesh^p$
      with subspaces that are localized on elements.
\item Decompositions \eqref{2level:dec:hypsing} for the $p$-version based on the separation
      of basis functions are inherently unstable for a standard basis
      (cf.~\cite{babuska:griebel:pitkaranta:89} for the finite element method), or
      require specific basis functions that are partially orthogonal
      (cf.~\cite{heuer:98,heuer:98:erratum}).
\item Partially orthogonal basis functions (as mentioned before) are not hierarchical.
      They can be constructed a priorily for rectangles (through tensor product representations)
      or a~posteriorily through a Schur complement step. This Schur complement is not
      a local construction for boundary integral operators and, thus, expensive.
\item Increasing polynomial degrees by a finite number, e.g., from $p$ to $p+1$, for the
      generation of the enriched space $\widehat\XX_\mesh^p$ does in general not satisfy the saturation
      assumption. On the other hand, increasing polynomial degrees by a fixed factor
      is not practical since polynomials of higher degrees are inherently difficult to
      implement in a stable and efficient way.
\end{itemize}

For the reasons above, we suggest to consider two different enrichments with corresponding
decompositions. One for error estimation with focus on satisfying the saturation assumption
(let's call this estimator $\eta_\mathrm{est}$)
and another one to generate indicators steering the mesh refinement (and/or increase of
polynomial degrees) with focus on providing local information
(let's call this estimator $\eta_\mathrm{ref}$).
The estimator $\eta_\mathrm{est}$ can be relatively expensive since it will be used
only for a stopping criterion, it is not necessary to calculate it after each refinement step.
On the other hand, $\eta_\mathrm{ref}$ is needed for every refinement step and should
be cheap. In the following we discuss both cases.

\paragraph{Error estimator $\eta_\mathrm{est}$.}
We consider the enriched space $\widehat\XX_\mesh^p:=\wilde\Sp^{\widehat p}(\widehat\mesh)$
that one obtains
by refining the mesh $\mesh$ uniformly, i.e., subdividing every triangle and quadrilateral
in an iso\-tropic way. Here, $\widehat\mesh$ denotes the refined mesh. Polynomial degrees
$\widehat p$ can be inherited from father elements or one can select the maximum polynomial
degree from the actual space $\XX_\mesh^p$ for the enriched space, $\widehat p=\max\{p\}$.
In this way, numerical experiments indicate
good compliance with the saturation property, cf.~\cite{h02}. Uniformly stable decompositions
can be obtained through overlapping domain decomposition. To this end, let $\omega_j$
($z_j\in\widehat\NN$) denote the patches of elements that are adjacent to interior nodes of
the refined mesh $\widehat\mesh$. Here, for simplicity, we assume that $\Gamma$ is
convex to avoid the appearance of special situations at incoming corners. This is only
for technical reasons and not essential. Then, we consider the decomposition
$\widehat\XX_\mesh^p=\ZZ_{\mesh,0}\cup \ZZ_{\mesh,1}\cup\ldots\cup \ZZ_{\mesh,L}$ with
\begin{align} \label{2level:dec:hypsing:overl}
  \ZZ_{\mesh,0} = \wilde\Sp^1(\widehat\mesh)\quad\text{and}\quad
   \ZZ_{\mesh,j} = \wilde\Sp^{\widehat p}(\widehat\mesh|_{\omega_j}),\quad z_j\in\widehat\NN.
\end{align}
Two comments are in order. First, this decomposition is not direct. This results in
a slightly more complicated additive Schwarz theory than presented in
Section~\ref{section:est:2level}.
Second, we do not have the inclusion $\ZZ_{\mesh,0}\subset \XX_\mesh^p$ so that the error indicator
\[
   \eta_{\mathrm{est},0} = \|P_0(\widehat U-U)\|_b
\]
corresponding to this subspace does not vanish in general. $\ZZ_{\mesh,0}$ is the so-called
coarse grid space of the decomposition and, since it is defined with lowest order
polynomial degree, its calculation is not too expensive. Additionally, this step
can be accelerated by using efficient low order implementations (though this has not
been studied in this particular situation).

\begin{theorem} \label{thm:hypsing:estim:2level:err}
Let $\widehat\XX_\mesh^p=\wilde\Sp^{\widehat p}(\widehat\mesh)$
be defined with uniform degree $\widehat p=\max\{p\}$.
The error estimator $\eta_\mathrm{est}$ defined through the decomposition of
$\widehat\XX_\mesh^p$ with subspaces \eqref{2level:dec:hypsing:overl} is efficient:
there exists a constant $\c{eff}>0$ such that, for any quasi-uniform mesh $\mesh$ of
shape-regular elements with shape-\linebreak regular refinement $\widehat\mesh$, there holds
for any polynomial degree $\widehat p$
\[
   \eta_\mathrm{est} \le \c{eff} \|u-U\|_b.
\]
Furthermore, let $\widehat\mesh$ be sufficiently refined so that Assumption~\ref{ass:sata}
holds. Then $\eta_\mathrm{est}$ is also reliable:
there exists a constant $c>0$ such that, with $\c{rel}=(1-\c{sata}^2)^{-1/2} c$, there
holds for any mesh $\mesh$ with shape-regular refinement $\widehat\mesh$ and for any
polynomial degree $\widehat p$ the estimate
\[
   \|u-U\|_b \le \c{rel}\; \eta_\mathrm{est}.
\]
\end{theorem}

For a proof of Theorem~\ref{thm:hypsing:estim:2level:err} we refer to \cite{h02}.

\paragraph{Error indicator $\eta_\mathrm{ind}$.}
In order to generate error indicators that are local and useful for adaptive steering we
increase locally polynomial degrees. In particular, we aim at error indicators that
indicate also in which direction to refine (literally an element or in the sense of
increasing polynomial degrees in a certain direction on elements).
Here, we do not focus on satisfying the saturation assumption.
In the following, to keep things simpler, we consider only rectangular elements.
For meshes consisting of triangles, or rectangles and triangles, we refer to \cite{h02}.

Our enriched space $\widehat\XX_\mesh^p$ and decomposition will be like
\begin{align} \label{2level:dec:hypsing:p}
  \widehat\XX_\mesh^p = \ZZ_{\mesh,0} \oplus \bigoplus_{T\in\mesh;\; i=1,2} \ZZ_{Ti}.
\end{align}
Here, $\ZZ_{T1}$ and $\ZZ_{T2}$ consist of functions with support on (the closure of) an element
$T\in\mesh$ and the two spaces will generate direction indicators. The space $\ZZ_{\mesh,0}$ consists
of functions with support on (the closure of) $\Gamma$. In order to have conformity
$\ZZ_{Ti}\subset\widetilde H^{1/2}(\Gamma)$ and locality of functions, the elements of
$\ZZ_{Ti}$ must vanish on the boundary of $T$ (so that they can be continuously extended
by zero onto $\Gamma\setminus T$). Of course we are considering polynomials on $T$,
and in that case these functions (with vanishing trace on the boundary of $T$) are
called bubble functions. The generation of bubble functions requires a minimum
polynomial degree. On a triangle the lowest order bubble function has degree three,
and on rectangles one uses tensor products of polynomials of at least degree two in
both directions. In both cases, the minimum degree allows for only one linearly
independent bubble function. Therefore, in order to have enough unknowns for
indicators in different directions, we need slightly higher polynomial degrees.

For ease of presentation let us now assume that $T$ is a rectangle that is oriented
in the $x_1$-$x_2$ plane. Furthermore, $\Pp^{p_1,p_2}(T)$ indicates the space
of polynomials on $T$ with degrees up to $p_i$ in $x_i$-direction, $i=1,2$.
We then define for any $T\in\mesh$ the spaces $\ZZ_{T1}$, $\ZZ_{T2}$ as follows.
\[
   \ZZ_{T1} :=
   \left\{
   \begin{array}{l}
      {\rm span}
      \left\{\Pp^{p_1+1,p_2}(T)\setminus\Pp^{p_1,p_2}(T)\right\}
      \cap H_0^1(T)
      \\\qquad\qquad
      \mbox{if}\ p_1>1, p_2>1 \quad {\rm (a)}\\[1em]
      {\rm span}
      \left\{\Pp^{p_1+1,2}(T)\setminus\Pp^{p_1,2}(T)\right\}
      \cap H_0^1(T)
      \\\qquad\qquad
      \mbox{if}\ p_1>1, p_2=1 \quad {\rm (b)}\\[1em]
      \Pp^{2,p_2}(T)
      \cap H_0^1(T)
      \\\qquad\qquad
      \mbox{if}\ p_1=1, p_2>1 \quad {\rm (c)}\\[1em]
      \Pp^{3,2}(T)
      \cap H_0^1(T)
      \\\qquad\qquad
      \mbox{if}\ p_1=1, p_2=1 \quad {\rm (d)}
   \end{array}
   \right.
\]
is the space to generate an indicator on $T$ in $x_1$-direction and
\[
   \ZZ_{T2}
   :=
   \left\{
   \begin{array}{l}
      {\rm span}
      \left\{\Pp^{p_1,p_2+1}(T)\setminus\Pp^{p_1,p_2}(T)\right\}
      \cap H_0^1(T)
      \\\qquad\qquad
      \mbox{if}\ p_1>1, p_2>1 \quad {\rm (a)}\\[1em]
      \Pp^{p_1,2}(T)
      \cap H_0^1(T)
      \\\qquad\qquad
      \mbox{if}\ p_1>1, p_2=1 \quad {\rm (b)}\\[1em]
      {\rm span}
      \left\{\Pp^{2,p_2+1}(T)\setminus\Pp^{2,p_2}(T)\right\}
      \cap H_0^1(T)
      \\\qquad\qquad
      \mbox{if}\ p_1=1, p_2>1 \quad {\rm (c)}\\[1em]
      \Pp^{2,3}(T)
      \cap H_0^1(T)
      \\\qquad\qquad
      \mbox{if}\ p_1=1, p_2=1 \quad {\rm (d)}
   \end{array}
   \right.
\]
will generate an indicator in $x_2$-direction.
Here, $(p_1,p_2)$ are the polynomial degrees in $\wilde\Sp^p(\mesh)$ on $T$.
They can be different on every element.
As basis functions for the subspaces $\ZZ_{T1}$, $\ZZ_{T2}$ we take affine
images of the tensor products
\[
   \psi_{p_1}(x_1) \psi_{p_2}(x_2)
   \quad\mbox{with}\quad
   \psi_q(s) := \int_{-1}^s L_{q-1}(t)\,dt
\]
for $p_1, p_2 \ge 2$ defined on $(-1,1)^2$ with $L_{q-1}$ being the
Legendre polynomial of degree $q-1$.

\begin{figure}[hbt]
  \psfrag{Z}{$\ZZ$}
  \psfrag{T1}{$T1$}
  \psfrag{T2}{$T2$}
  \psfrag{p2}{$p_2$}
  \psfrag{p2+1}{$p_2+1$}
  \psfrag{p2+2}{$p_2+2$}
  \psfrag{2}{$2$}
  \psfrag{p1}{$p_1$}
  \psfrag{p1+1}{$p1+1$}
  \psfrag{p1+2}{$p1+2$}
\centerline{\scalebox{0.8}{\includegraphics{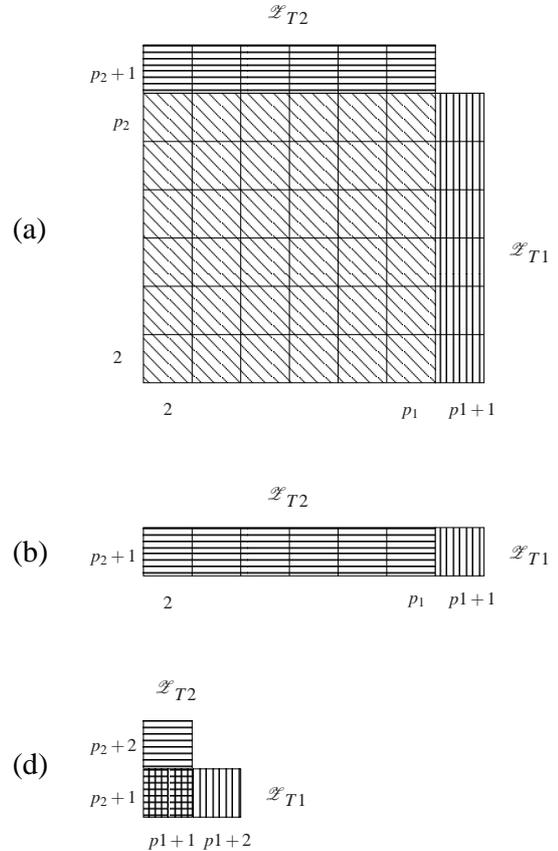}}}
\caption{Illustration of $p$-enrichment and decomposition on an element $T$
         for direction control.}
\label{fig_hyp_pol_dec}
\end{figure}

In Figure~\ref{fig_hyp_pol_dec} we illustrate the increase of polynomial degrees
for the generation of $\ZZ_{T1}$ and $\ZZ_{T2}$ in different situations.
The marked regions represent pairs of polynomial degrees for the two spaces.
We represent only degrees larger than one, that means we illustrate only
bubble functions.
Case (a) is the general case when the polynomial degrees $p_1$, $p_2$ on an
element $T$ are larger than one. We increase the polynomial degrees
by one in both directions and define the local spaces $\ZZ_{T1}$ and
$\ZZ_{T2}$ by the indicated degrees.
The remaining situations, where $p_1$ or $p_2$ is one, are illustrated by
(b), (c) and (d).
These cases are not covered by (a) and we need a special $p$-enrichment
to produce subspaces that indicate different directions for refinement.
Case (c) is analogous to the case (b) when exchanging $p_1$ and $p_2$,
and is omitted.
Only in the case (a) (with $p_1\ge 2$ and $p_2\ge 2$) bubble functions of
the previous space $\XX_\mesh^p$ are present. This is indicated by the diagonal shading.
Case (d) is the only situation where the decomposition of $\ZZ_T=\ZZ_{T1}\cup \ZZ_{T2}$
is not direct.

Now, for elements not being aligned with the $x_1$-$x_2$ directions, the construction
of the two spaces $\ZZ_{T1}$, $\ZZ_{T2}$ works analogously. Concluding, we have defined
the decomposition \eqref{2level:dec:hypsing:p} with the exception of $\ZZ_{\mesh,0}$.
In \cite{h02} and \cite{hms02}, different strategies have been considered to
generate $\ZZ_{\mesh,0}$ so that $\XX_\mesh^p\subset\widehat\XX_\mesh^p$ and \eqref{2level:dec:hypsing:p} is (almost)
stable. These strategies are partial or full orthogonalizations and Schur complement
steps. Here we are not interested in reliable error estimation (which is being
provided by the error estimator based on mesh refinement) and therefore, finish
this section with recalling stability of the decomposition of the enrichment level
\begin{align} \label{2level:dec:hypsing:p:Z}
   \ZZ_\mesh = \bigoplus_{T\in\mesh} \Bigl(\ZZ_{T1}\cup \ZZ_{T2}\Bigr)
\end{align}
and assuming a stable construction of $\ZZ_{\mesh,0}$ without giving more details. This
then implies efficient and reliable estimation of $\|\widehat U-U\|_b$ by $\eta_\mathrm{ind}$.
Note that in some cases, as discussed above, the decomposition $\ZZ_{T1}\cup \ZZ_{T2}$
can be non-direct.

As said before, we consider meshes consisting only of rectangles.
As in \cite{h02,hms02} this can be generalized to meshes including quadrilaterals
and triangles.

\begin{theorem} \label{thm:hypsing:estim:2level:ind}
Let $\ZZ_\mesh$ be defined through decomposition \eqref{2level:dec:hypsing:p:Z} with local
spaces $\ZZ_{Ti}$ ($i=1,2$) as defined previously, and assume that the construction
of $\ZZ_{\mesh,0}$ in \eqref{2level:dec:hypsing:p} is stable.
Then the corresponding error indicator
\[
   \eta_\mathrm{ind} :=
   \Bigl(\|P_0(\widehat U-U)\|_b^2
         + \sum_{T\in\mesh,\; i=1,2} \|P_{Ti}(\widehat U-U)\|_b^2\Bigr)^{1/2}
\]
is reliable and efficient for the estimation of $\|\widehat U-U\|_b$ in the
following sense. 
Assume that the mesh $\mesh$ is locally quasi-uniform, i.e.
the ratio of largest side length and smallest side length
on each element is bounded from above by a global positive constant.
Then there exist constants $c_1$, $c_2>0$ which are independent of
the mesh and polynomial degrees $p$ such that
\[
   c_1 \eta_\mathrm{ind}
   \le
   \|\widehat U - U\|_b
   \le
   c_2 (1+\log p_{\max}) \eta_\mathrm{ind}.
\]
Here, $p_{\max}$ is the maximum of all polynomial degrees in $\XX_\mesh^p$.
\end{theorem}

For a proof we refer to \cite{h02}.
We finish this section with presenting a three-level refinement algorithm
that decides where to add unknowns and whether to refine the mesh or increase
polynomial degrees at those places. This algorithm has proved to work appropriately
in standard situations. The definition and analysis of optimal algorithms
for direction control and decision for $h$ or $p$ refinement in boundary element
methods is an open problem.

{\bf Three-fold algorithm} \cite{h02}:
Define an initial ansatz space $\widetilde\Sp^p(\mesh)$ with
initial mesh $\mesh$ and low polynomial degrees. Choose
an error tolerance $\epsilon>0$ and steering parameters
$\delta_1$, $\delta_2$, $\delta_3$ with $0<\delta_2 <\delta_1<1$
and $0<\delta_3<1$.
Then perform the following steps.

\begin{itemize}
\item[] {\bf 1. Galerkin solution.}
Compute the Galerkin solution $U\in \widetilde\Sp^p(\mesh)$.

\item[] {\bf 2. Error estimator.}
Calculate the terms $\eta_j=\|P_j(\widehat U - U)\|_b$
and the estimator $\eta_\mathrm{est}$, based on the decomposition
(\ref{2level:dec:hypsing:overl}).

{\bf Stop} if $\eta_\mathrm{est}\le\epsilon$.

\item[] {\bf 3. Adaption steps.} 

\begin{itemize}
\item[] {\bf (i) Indicators.}
Compute the error indicators
$\eta_{Ti}=\|P_{Ti}(\widehat U - U)\|_b$,
$\eta_T:=(\eta_{T1}^2+\eta_{T2}^2)^{1/2}$
($T\in\mesh$, $i=1,2$) based on the decomposition (\ref{2level:dec:hypsing:p:Z}),
and set $\eta_\mathrm{max}:=\max_{T\in\mesh}\eta_T$.

\item[] {\bf (ii) Classification of elements.}
Classify quadrilaterals $T$ as follows
(in pseudo Fortran90 language, directions are understood with respect to local
 coordinates):

\hspace*{1em}{\tt if ($\eta_T>\delta_1\eta_\mathrm{max}$) then}
\quad ! $h$-refinement\\
\hspace*{2em}{\tt if ($\eta_{T1}<\delta_3\eta_{T2}$) then}\\
\hspace*{3em} classify $T$ for horizontal intersection\\
\hspace*{2em}{\tt elseif ($\eta_{T2}<\delta_3\eta_{T1}$) then}\\
\hspace*{3em} classify $T$ for vertical intersection\\
\hspace*{2em}{\tt else}\\
\hspace*{3em} classify $T$ for intersections in both directions\\
\hspace*{2em}{\tt endif}\\
\hspace*{1em}{\tt elseif ($\eta_T>\delta_2\eta_\mathrm{max}$) then}
\quad ! $p$-increase\\
\hspace*{2em}{\tt if ($\eta_{T1}<\delta_3\eta_{T2}$) then}\\
\hspace*{3em} classify $T$ for increase of polynomial degree
              in vertical direction\\
\hspace*{2em}{\tt elseif ($\eta_{T2}<\delta_3\eta_{T1}$) then}\\
\hspace*{3em} classify $T$ for increase of polynomial degree
              in horizontal direction\\
\hspace*{2em}{\tt else}\\
\hspace*{3em} classify $T$ for increase of polynomial degrees
              in both directions\\
\hspace*{2em}{\tt endif}\\
\hspace*{1em}{\tt endif}

Triangles are classified without direction control, i.e., they are
refined by halving all their edges if $\eta_T>\delta_1\eta_\mathrm{max}$ and
the polynomial degree is increased if
$\delta_2\eta_\mathrm{max}<\eta_T\le\delta_1\eta_\mathrm{max}$.

\item[] {\bf (iii) Adaption.}
\begin{itemize}
\item[] (a) Go through all the elements and refine as classified.
\item[] (b) Go through all the elements and refine as necessary to remove
            hanging nodes.
\item[] (c) Go through all the elements and increase polynomial degrees
            as classified if the corresponding element has not been
            refined in (b).
\item[]{\bf goto 1}.
\end{itemize}
\end{itemize}
\end{itemize}

\begin{remark}
When an element is refined then polynomial degrees for the new
elements need to be given. To avoid high polynomial degrees on small
elements one can inherit the degrees reduced by one for the refinement direction
whenever possible (i.e., when the degree is larger than one).
A more sophisticated algorithm may include the refinement of quadrilaterals into
triangles and vice versa. This has been studied on \cite{h02,hms02}.
\end{remark}

\begin{remark}
An adaptive $h$-version can be realized by choosing $\delta_2\ge\delta_1$.
Pure $p$-adaptivity occurs when choosing $\delta_1>1$.
Isotropic adaption (no direction control) can be performed by
taking $\delta_3=0$.
\end{remark}
\section{Estimator reduction}\label{section:estred}
This section explains the concept of estimator reduction and its use to
prove plain convergence of ABEM, i.e., the validity of~\eqref{eq:plain:conv}.
The general idea that will be presented here applies to error estimators
whose local contributions are weighted by the local mesh-size $h_\ell$. The 
approach is illustrated for some $(h-h/2)$-type error estimators from
Section~\ref{section:est:hh2}, the ZZ-type error estimators from 
Section~\ref{section:zzest}, and the weighted residual error estimators 
from Section~\ref{section:est:wres}.
To that end, we consider a sequence of meshes $\mesh_\ell$ which, e.g., is generated by
the adaptive Algorithm~\ref{opt:algorithm}.
We only need some minor assumptions on the mesh refinement operation $\refine(\cdot)$.

\subsection{Assumptions on mesh refinement}\label{section:estred:meshrefinementprelim}
We say that a mesh $\mesh_\star$ is a refinement of
another mesh $\mesh_\ell$, written $\mesh_\star\in\refine(\mesh_\ell)$, if the following applies:
\begin{itemize}
\item For all $\el\in\mesh_\ell$, there holds
\begin{align}\label{dp1:estred}
\overline\el = \bigcup \set{\overline\el^\prime}{\el^\prime \in \mesh_\star\text{ with }\el^\prime \subseteq \el},
\end{align}
i.e., each coarse-mesh element $T\in\TT_\ell$ is basically the union of fine-mesh elements
$T'\in\TT_\star$.
\item For all $\el\in\mesh_\ell$ and $\el^\prime\in\mesh_\star$, there holds
\begin{align}\label{dp2:estred}
\el^\prime \subsetneqq \el\quad\implies\quad |\el^\prime|\leq |\el|/2,
\end{align}
i.e., the area of refined elements is at least halved.
\end{itemize}

A sequence of meshes $(\mesh_\ell)_{\ell\in\N_0}$ is called nested, if 
for all $\ell\in\N_0$ it holds
$\mesh_{\ell+1}\in\refine(\mesh_\ell)$ and if
the shape regularity constant $\sigma_\ell$ from Section~\ref{section:bem:discrete}
is uniformly bounded\linebreak $\sup_{\ell\in\N_0}\sigma_\ell<\infty$.

Recall that with each mesh $\TT_\star$, we associate the local mesh-size function
$h_\star\in\Pp^0(\TT_\star)$ defined by $h_\star|_\el:= h_\el=|T|^{1/(d-1)}$.

\subsection{Abstract error estimator}
Given the mesh $\mesh_\ell$, suppose that there exists a computable error estimator 
\begin{align*}
 \est{\ell}{}:=\Big(\sum_{\el\in\mesh_\ell}\est{\ell}{\el}^2\Big)^{1/2}\text{ with } \est{\ell}{\el}\in [0,\infty)\text{ for all }\el\in\mesh_\ell.
\end{align*}
The estimator usually depends on the computed solution $U_\ell$ of~\eqref{intro:galerkin} as well as on the right-hand side $F$. 

\subsection{Abstract adaptive algorithm}
Convergence of the adaptive algorithm~\ref{opt:algorithm} has first been addressed in
the frame of AFEM in the pioneering work~\cite{doerfler} which also introduced the bulk 
chasing~\eqref{opt:bulkchasing}. While \cite{doerfler} only proved convergence
up to the resolution of the given data on the initial mesh $\TT_0$, the
work~\cite{mns00} included the adaptive resolution of the data and contained the first plain convergence result. 
For ABEM, convergence of this algorithm has mathematically been 
addressed first in~\cite{fop} and~\cite{afp12} for $(h-h/2)$-type and
averaging error estimators, while the analysis of~\cite{cp12} relied on
an additional (and practically artificial and unnecessary) feedback control.

\begin{remark}
  In practice, step~(iv) of the Algorithm~\ref{opt:algorithm}
  provides the \emph{coarsest} refinement 
  $\TT_{\ell+1}$ of $\TT_\ell$ such that all marked elements have been refined
  by the mesh refinement strategy used, written
  $\TT_{\ell+1}:=\refine(\TT_\ell,\MM_\ell)$. We refer to 
  Section~\ref{section:meshrefinement} for possible concrete strategies for 
  local mesh refinement of 2D and 3D BEM meshes.
\end{remark}

\begin{remark}\label{rem:estred:algorithm}
 To find a set $\MM_\ell\subseteq \mesh_\ell$ which satisfies the bulk chasing~\eqref{opt:bulkchasing}, one arbitrarily adds elements $\el\in\mesh_\ell$ to
 $\MM_\ell$ until~\eqref{opt:bulkchasing} is satisfied (at least $\MM_\ell=\mesh_\ell$ will do the job). If one seeks a set of minimal cardinality $\MM_\ell$, 
 it is necessary to sort the  elementwise error indicators, i.e., $\est{\ell}{\el_1}\geq \est{\ell}{\el_2}\geq \ldots \geq \est{\ell}{\el_{\#\mesh_\ell}}$.
 Then, determine the minimal $1\leq n\leq \#\mesh_\ell$ such that $\theta\est{\ell}{}^2\leq \sum_{j=1}^n\est{\ell}{\el_j}^2$. By construction,\linebreak
 $\MM_\ell:=\{\el_1,\ldots,\el_n\}$ satisfies~\eqref{opt:bulkchasing} with minimal cardinality.
 Obviously, the set $\MM_\ell$ is not unique in general. This may lead to non-symmetric meshes
 even for completely symmetric problems.
\end{remark}

\subsection{Estimator reduction principle}\label{section:estred:principle}
The estimator reduction principle~\cite{afp12} is an elementary yet very useful starting point for the a~posteriori analysis of  any error estimator.
The following lemma states the main idea of the principle.

\begin{mylemma}\label{estred:lem:estreddef}
Given a sequence of error estimators $(\eta_\ell)_{\ell\in\N_0}$, suppose a contraction constant $0<q_{\rm est}<1$ as well as a perturbation sequence $(\alpha_\ell)_{\ell\in\N_0} \subset \R$ such that the error estimator
satisfies the perturbed contraction
\begin{align}\label{estred:eq:estred}
\est{\ell+1}{}^2\leq q_{\rm est} \est{\ell}{}^2 + \alpha_\ell^2\quad\text{for all }\ell\in\N_0.
\end{align}
Then, $\lim_{\ell\to\infty}\alpha^2_\ell =0$ implies estimator convergence
\begin{align}\label{estred:eq:estconv}
\lim_{\ell\to\infty}\est{\ell}{} = 0.
\end{align}
\end{mylemma}
\begin{proof}
Apply the limes superior $\overline\lim$ on both sides of the estimator reduction~\eqref{estred:eq:estred} to obtain
\begin{align*}
\overline{\lim_{\ell\to\infty}}\est{\ell+1}{}^2\leq q_{\rm est} \overline{\lim_{\ell\to\infty}}\est{\ell}{}^2 +\overline{\lim_{\ell\to\infty}}\alpha_\ell^2.
\end{align*}
Since $\alpha^2_\ell$ converges towards zero, there holds $\overline \lim_{\ell\to\infty}\alpha_\ell^2=0$. This implies
\begin{align*}
\overline{\lim_{\ell\to\infty}}\est{\ell+1}{}^2\leq q_{\rm est} \overline{\lim_{\ell\to\infty}}\est{\ell}{}^2 =q_{\rm est} \overline{\lim_{\ell\to\infty}}\est{\ell+1}{}^2.
\end{align*}
Since $0<q_{\rm est}<1$, this leaves the possibilities $\overline \lim_{\ell\to\infty}\est{\ell+1}{}^2=0$ or $\overline \lim_{\ell\to\infty}\est{\ell+1}{}^2=\infty$.
To rule out the second option, apply the estimator reduction~\eqref{estred:eq:estred} inductively to see
\begin{align*}
\est{\ell}{}^2&\leq q_{\rm est}\est{\ell-1}{}^2+\alpha_{\ell-1}^2\\
&\leq q_{\rm est}^2\est{\ell-2}{}^2+q_{\rm est}\alpha_{\ell-2}^2+\alpha_{\ell-1}^2\\
&\leq q_{\rm est}^\ell\est{0}{}^2 + \sum_{k=0}^{\ell-1} q_{\rm est}^k \alpha_{\ell-k-1}^2.
\end{align*}
Convergence of $\alpha_\ell^2$ implies the boundedness $\sup_{\ell\in\N_0}\alpha_\ell^2<\infty$ and the convergence of the geometric series concludes
\begin{align*}
\est{\ell}{}^2
&\leq q_{\rm est}^\ell\est{0}{}^2 + \frac{1}{1-q_{\rm est}}\sup_{\ell\in\N_0}\alpha_\ell^2<\infty.
\end{align*}
This proves $\overline \lim_{\ell\to\infty}\est{\ell+1}{}^2=0$ and elementary calculus yields
\begin{align*}
  0\leq \lim_{\ell\to\infty}\est{\ell}{}^2\leq \overline{\lim_{\ell\to\infty}}\est{\ell+1}{}^2=0.
\end{align*}
This concludes the proof.
$\hfill\qed$
\end{proof}

Before we apply Lemma~\ref{estred:lem:estreddef} to concrete error estimators $\est{\ell}{}$, we collect a number of auxiliary results on the convergence of projections and quasi-interpolants. Later, these will prove that the perturbation terms $\alpha_\ell$ vanish as $\ell\to\infty$.
The first lemma has already been proved in the pioneering work~\cite{bv84}
for symmetric problems and reinvented in~\cite{msv,cp12}.

\begin{mylemma}\label{estred:lem:convortho}Suppose a sequence of nested spaces $(\XX_\ell)_{\ell\in\N_0}\subset \XX$, i.e., $\XX_\ell\subseteq \XX_{\ell+1}$ for all $\ell\in\N_0$.
Then, the Galerkin approximations $U_\ell$ of~\eqref{intro:galerkin} satisfy
\begin{align}\label{estred:eq:apriorigal}
\lim_{\ell\to\infty}\norm{U_\infty-U_\ell}{\XX}=0
\end{align}
for an a~priori limit $U_\infty\in\XX$ which is unknown in general and
depends on the concrete sequence of spaces.
\end{mylemma}
\begin{proof}
Define the closed space $\XX_\infty:=\overline{\bigcup_{\ell\in\N_0} \XX_\ell}\subseteq \XX$, where the closure is understood in the space $\XX$. The Lax-Milgram lemma
guarantees a unique solution $U_\infty\in\XX_\infty$ of~\eqref{intro:galerkin}, where $\XX_\ell$ is replaced with $\XX_\infty$. By use of the Galerkin orthogonality, we prove the C\'ea lemma~\eqref{intro:cea} also for $U_\infty$, i.e., any $V_\ell\in\XX_\ell$ satisfies
\begin{align*}
 C_{\rm ell}\norm{U_\infty-U_\ell}{\XX}^2&\leq\form{U_\infty-U_\ell}{U_\infty-U_\ell}\\
 &=\form{U_\infty-U_\ell}{U_\infty-V_\ell}\\
 &\leq C_{\rm cont} \norm{U_\infty-U_\ell}{\XX}\norm{U_\infty-V_\ell}{\XX}.
\end{align*}
Hence, we are led to
\begin{align*}
\norm{U_\infty-U_\ell}{\XX}\lesssim \min_{V_\ell\in\XX_\ell}\norm{U_\infty-V_\ell}{\XX}.
\end{align*}
Let $\eps>0$. 
The definition of $\XX_\infty$ implies the existence of $k\in\N$
and $W_k\in\XX_k$ such that $\norm{U_\infty-W_k}\XX\le\eps$.
Combining this with the nestedness $\XX_k\subseteq\XX_\ell$ for $\ell\ge k$
and the C\'ea lemma, we obtain $\norm{U_\infty-U_\ell}\XX\lesssim\eps$.
This concludes the proof.
$\hfill\qed$
\end{proof}

The following lemma provides a similar result for quasi-interpolation operators and is proved
in~\cite[Proposition~11]{zz2014}.

\begin{mylemma}\label{estred:lem:convclement}
Given a sequence of nested meshes $(\mesh_\ell)_{\ell\in\N_0}$ as well as corresponding linear
operators $(J_\ell:\, L_2(\Gamma)\to L_2(\Gamma))_{\ell\in\N}$ which satisfy
for all $\el\in\mesh_\ell$ the following properties (i)--(iii):
\begin{itemize}
\item[(i)] local $L_2$-stability
\begin{align*}
\norm{J_\ell v}{L_2(\el)}\leq C_{\rm J} \norm{ v}{L_2(\omega_\el)}\quad\text{for all }v\in L_2(\Gamma);
\end{align*}
\item[(ii)] local first-order approximation property
\begin{align*}
\norm{(1-J_\ell)v}{L_2(\el)}\leq C_{\rm J} \norm{ h_\ell\nabla_\Gamma v}{L_2(\omega_\el)}\quad\text{for all }v\in H^1(\Gamma);
\end{align*}
\item[(iii)] local definition, i.e., $(J_\ell v)|_\el $ depends only on the values of $v|_{\omega_\el}$.
\end{itemize}
Then, there exists a linear and continuous limit operator $J_\infty:\,L_2(\Gamma)\to L_2(\Gamma)$ with
\begin{align}\label{estred:eq:convclement:L2}
\lim_{\ell\to\infty}\norm{J_\infty v - J_\ell v}{L_2(\Gamma)} 
= 0 \quad\text{for all } v\in L_2(\Gamma).
\end{align}
Suppose in addition that $J_\ell(L_2(\Gamma)\big)\subseteq H^1(\Gamma)$ and that
the following property holds:
\begin{itemize}
\item[(iv)] local $H^1$-stability
\begin{align*}
\norm{\nabla_\Gamma (J_\ell v)}{L_2(\el)}\leq C_{\rm J} \norm{ v}{H^1(\omega_\el)}\quad\text{for all }v\in H^1(\Gamma).
\end{align*}
\end{itemize}
Then, the limit operator $J_\infty$ has the following additional properties:
\begin{itemize}
\item $J_\infty:\,H^s(\Gamma)\to H^s(\Gamma)$ is well-defined and continuous for all $0\leq s\leq 1$;
\item for all $0\leq s < 1$, $J_\infty$ is the pointwise limit of $J_\ell$, i.e.,
\begin{align}\label{estred:eq:convclement}
\lim_{\ell\to\infty}\norm{J_\infty v - J_\ell v}{H^s(\Gamma)} = 0 \quad\text{for all } v\in H^s(\Gamma);
\end{align}
\item for $s=1$ and all $v\in H^1(\Gamma)$, $J_\ell v$ converges weakly to $J_\infty v$ as $\ell\to\infty$.
\end{itemize}
\end{mylemma}

The Scott-Zhang projection $J_\ell$ from Lemma~\ref{lem:sz} satisfies even stronger convergence results. We stress that $J_\ell$ satisfies the assumptions (i)--(iv) from Lemma~\ref{estred:lem:convclement}. The following lemma is proved in~\cite[Lemma~18]{dirichlet2d}.

\begin{mylemma}\label{estred:lem:szconv}
  Suppose a sequence of nested meshes $(\mesh_\ell)_{\ell\in\N_0}$ as well as the corresponding
  Scott-Zhang operators $(J_\ell:\, L_2(\Gamma)\to L_2(\Gamma))_{\ell\in\N}$. Then, the limit
  operator $J_\infty:\,L_2(\Gamma)\to L_2(\Gamma)$, which exists due to
  Lemma~\ref{estred:lem:convclement}, is a projection in the sense of $J_\infty  v =  v$
  for all
  $v\in \Sp^p_\infty:=\overline{\bigcup_{\ell\in\N_0} \Sp^p(\mesh_\ell)}\subseteq L_2(\Gamma)$
  where the closure is understood with respect to $L_2(\Gamma)$,
  and satisfies pointwise convergence
  \begin{align}\label{estred:eq:szconv}
    \lim_{\ell\to\infty}\norm{J_\infty v - J_\ell v}{H^s(\Gamma)} = 0
    \quad\text{for all } v\in H^s(\Gamma)
  \end{align}
  and all $0\leq s \leq 1$.
\end{mylemma}

Finally, also the non-local $L_2$-orthogonal projection
$\Pi^p_\ell:\,L_2(\Gamma)\to\Sp^p(\mesh_\ell)$ from Definition~\ref{def:L2projection}
satisfies the convergence properties 
of Lemma~\ref{estred:lem:convclement}--\ref{estred:lem:szconv}. The following lemma improves an
observation of~\cite{kop} to general $0\leq s\leq 1$. We note that Lemma~\ref{estred:lem:l2conv}
does not follow from Lemma~\ref{estred:lem:convclement}, since $\Pi^p_\ell$ is a \emph{non-local}
operator and fails to satisfy the \emph{local} properties (i)--(iv) from
Lemma~\ref{estred:lem:convclement}.

\begin{mylemma}\label{estred:lem:l2conv}
  Given a sequence of nested meshes $(\mesh_\ell)_{\ell\in\N_0}$, suppose uniform $H^1$-stability
  of the $L_2$-orthogonal projection $\Pi^p_\ell:\, L_2(\Gamma)\to \Sp^p(\mesh_\ell)$ for
  all $\ell\in\N_0$, i.e.,
  \begin{align}\label{estred:eq:l2stab}
    \norm{\nabla_\Gamma \Pi^p_\ell v}{L_2(\Gamma)}\leq \c{stab} \norm{v}{H^1(\Gamma)}
    \quad\text{for all }v\in H^1(\Gamma).
  \end{align}
  Then, there exists a linear and continuous limit operator
  $\Pi^p_\infty:\,L_2(\Gamma)\to L_2(\Gamma)$
  which is a projection in the sense of $\Pi^p_\infty  v = v$ for all
  $v\in \Sp^p_\infty:=\overline{\bigcup_{\ell\in\N_0} \Sp^p(\mesh_\ell)}\subseteq L_2(\Gamma)$
  where the closure is understood with respect to $L_2(\Gamma)$, and satisfies
  \begin{itemize}
    \item $\Pi^p_\infty:\,H^s(\Gamma)\to H^s(\Gamma)$ is well-defined and continuous for all
      $0\leq s\leq 1$
    \item For all $0\leq s \leq 1$, $\Pi^p_\infty$ is the pointwise limit of $\Pi^p_\ell$, i.e.
      \begin{align}\label{estred:eq:l2conv}
        \lim_{\ell\to\infty}\norm{\Pi^p_\infty v - \Pi^p_\ell v}{H^s(\Gamma)} = 0
        \quad\text{for all } v\in H^s(\Gamma)
      \end{align}
  \end{itemize}
\end{mylemma}

\begin{proof}
Since $\Pi^p_\ell$ is an orthogonal projection for the $L_2$-scalar product,
one proves analogously to Lemma~\ref{estred:lem:convortho} that there exists an operator
$\Pi^p_\infty:\,L_2(\Gamma)\to L_2(\Gamma)$ such that
\begin{align}\label{estred:eq:l2l2conv}
  \lim_{\ell\to\infty}\norm{\Pi^p_\ell v-\Pi^p_\infty v}{L_2(\Gamma)}=0
  \quad\text{for all }v\in L_2(\Gamma).
\end{align}
Clearly, this and the projection property of $\Pi^p_\ell$ imply in particular 
that $v=\Pi^p_\infty v$ for all $v\in\Sp^p_\infty$.
As $\Pi^p_\ell$ is stable in $L_2(\Gamma)$ by definition and stable in $H^1(\Gamma)$ by
assumption~\eqref{estred:eq:l2stab}, it follows from deeper mathematical techniques
(see, e.g.,~\cite{t07}) that it is also stable in $H^s(\Gamma)$ for $s\in(0,1)$.
For general $v\in H^s(\Gamma)$, the boundedness
$\sup_{\ell\in\N_0}\norm{\Pi^p_\ell v}{H^s(\Gamma)}<\infty$ implies for a subsequence 
$\Pi^p_{\ell_k} v\rightharpoonup w$ weakly in $H^s(\Gamma)$ and hence
$\Pi^p_\infty v=w\in H^s(\Gamma)$. To see $H^s$-convergence for all $0\leq s\leq 1$, we
apply the projection property of $\Pi^p_\ell$ to see
\begin{align*}
  \norm{\Pi^p_\infty v - \Pi^p_\ell v}{H^s(\Gamma)} &=
  \norm{(1-\Pi^p_\ell) \Pi^p_\infty v}{H^s(\Gamma)}\\
  &= \norm{(1-\Pi^p_\ell) (1-J_\ell)\Pi^p_\infty v}{H^s(\Gamma)},
\end{align*}
where $J_\ell:\, L_2(\Gamma)\to \Sp^p(\mesh_\ell)$ denotes the Scott-Zhang projection from
Lemma~\ref{lem:sz}. The $H^s$-stability of $\Pi^p_\ell$ then shows
\begin{align*}
  \norm{\Pi^p_\infty &v - \Pi^p_\ell v}{H^s(\Gamma)}\lesssim
  \norm{(1-J_\ell)\Pi^p_\infty v}{H^s(\Gamma)}\\
  &= \norm{(J_\infty - J_\ell)\Pi^p_\infty v}{H^s(\Gamma)} \to 0,
\end{align*}
as $\Pi^p_\infty v\in\Sp^p_\infty$ and hence $\Pi^p_\infty v = J_\infty \Pi^p_\infty v$
by Lemma~\ref{estred:lem:szconv}.
$\hfill\qed$
\end{proof}

\begin{remark}
For 2D BEM, the $H^1$-stability~\eqref{estred:eq:l2stab} is well-ana\-lyzed
and found in~\cite{ct87}. For 3D BEM, available results 
include~\cite{by13,bps02,bx91,c02,c04,kpp13}, and we refer 
to Section~\ref{section:L2projection:H1stability} below.
\end{remark}

\begin{remark}
Suppose that $J_\ell:L_2(\Gamma)\to L_2(\Gamma)$ satisfies\linebreak
$J_\ell(H^1_0(\Gamma))\subseteq H^1_0(\Gamma)$, i.e., $J_\ell$ incorporates 
homogeneous Diri\-chlet conditions. Suppose that $J_\ell$ satisfies the 
properties (i)--(iv) of Lemma~\ref{estred:lem:convclement} with $H^1(\Gamma)$
replaced by $H^1_0(\Gamma)=\wilde H^1(\Gamma)$. Then, the according a~priori convergence holds
in $\wilde H^s(\Gamma)$ for $0\le s\le1$ instead of $H^s(\Gamma)$. This
observation applies, in particular, for the Scott-Zhang projection from
Lemma~\ref{estred:lem:szconv}, and we refer to~\cite{affkp13} for stable 
Scott-Zhang projectors in $\wilde H^s(\Gamma)$. Finally, also 
Lemma~\ref{estred:lem:l2conv} transfers to this case, if $\Pi^p_\ell$ denotes
the $L_2$-orthogonal projection onto $\widetilde\Sp^p(\TT_\ell)$. We refer
to~\cite{kpp13A} for $H^1$-stability of this $L_2$-projection for the lowest-order
case $p=1$, see also Section~\ref{section:L2projection:H1stability} below.
\end{remark}

\begin{remark}
  For 2D BEM and lowest-order elements, nodal interpolation $J_\ell:C(\overline\Gamma)\to\Sp^1(\TT_\ell)$ from Section~\ref{section:nodal}
  satisfies the identity $(J_\ell v)'=\pi_\ell^0(v')$
 for all $v\in H^1(\Gamma)$, where $\pi_\ell^0:L_2(\Gamma)\to\Pp^0(\TT_\ell)$ 
 denotes the $L_2$-orthogonal projection onto $\Pp^0(\TT_\ell)$.
 Given a sequence of nested meshes $(\mesh_\ell)_{\ell\in\N_0}$, 
 it is part of \cite[Proof of Prop.~5.2]{afgkmp12}
 that therefore the limit of $J_\ell v$ exists in $H^1(\Gamma)$, i.e.,
 $\norm{v_\infty-J_\ell v}{H^1(\Gamma)}\to0$ as $\ell\to\infty$ for some 
 appropriate $v_\infty\in H^1(\Gamma)$.
\end{remark}

\subsection{$(h-h/2)$-type error estimators}\label{section:estred:hh2}

This section follows~\cite{afgkmp12,afp12,kop} and discusses the estimator reduction~\eqref{estred:eq:estred} for the easy-to-implement $(h-h/2)$ error estimator
from~\cite{affkp13,efgp12,fp08}. Given any $\TT_\ell$, we assume that 
$\widehat\TT_\ell\in\refine(\TT_\ell)$ is the uniform refinement of $\TT_\ell$,
i.e., it holds nestedness 
\begin{align}\label{ass1:hh2}
 \Pp^0(\TT_\ell)\subseteq\Pp^0(\TT_{\ell+1})\subseteq\Pp^0(\widehat\TT_\ell)
 \subseteq\Pp^0(\widehat\TT_{\ell+1}),
\end{align} 
and for all $T\in\TT_\ell$ and $T'\in\widehat\TT_\ell$ 
holds
\begin{align}\label{ass2:hh2}
 T'\subseteq T
 \quad\Longrightarrow\quad
 q\,|T| \le |T'| \le |T|/2,
\end{align}
for some fixed and $\ell$-independent $0<q\le1/2$,
i.e., the local mesh-sizes of $\TT_\ell$ and $\widehat\TT_\ell$ are comparable.
\subsubsection{Weakly singular integral equation}\label{section:estred:hh2weaksing}
As model problem serves the weakly singular integral equation from Proposition~\ref{prop:galerkin:weaksing}.
The corresponding $(h-h/2)$-type error estimator from Theorem~\ref{thm:def:hh2:weaksing} employs the $L_2(\Gamma)$-ortho\-gonal projection $\pi_\ell^p:=\pi_{\mesh_\ell}^p:\,L_2(\Gamma)\to\Pp^p(\mesh_\ell)$
from Lemma \ref{def:L2projection} as well as the solution $\widehat \Phi_\ell$ of~\eqref{intro:galerkin}, where $\XX_\ell=\Pp^p(\TT_\ell)$ is replaced with the uniform refinement $\widehat\XX_\ell=\Pp^p(\widehat\TT_\ell)$ corresponding to $\widehat\mesh_\ell$, i.e.,
\begin{align}\label{estred:def:hh2weaksing}
\est{\ell}{}^2:=\sum_{\el\in\mesh_\ell}\est{\ell}{\el}^2:=\sum_{\el\in\mesh_\ell}h_\el \norm{(1-\pi_\ell^p)\widehat\Phi_\ell}{L_2(\el)}^2,
\end{align}
where $h_\el:=|\el|^{1/(d-1)}\simeq {\rm diam}(\el)$. The following lemma
proves the estimator reduction estimate~\eqref{estred:eq:estred}. The proof
reveals that the contraction constant $0<q_{\rm set}<1$ essentially follows from
the contraction~\eqref{dp2:estred} of the local mesh-size on refined elements.

\begin{mylemma}\label{estred:lem:weaksing}
Given a sequence of nested meshes $(\mesh_\ell)_{\ell\in\N}$, which additionally satisfy the bulk chasing~\eqref{opt:bulkchasing}
for all $\ell\in\N$ and some $0<\theta\le1$, the $(h-h/2)$ error estimator 
$\est{\ell}{}$ from~\eqref{estred:def:hh2weaksing} satisfies the estimator reduction~\eqref{estred:eq:estred} with $\alpha_\ell:=\setc{opt:estred:aux}\norm{\widehat\Phi_{\ell+1}-\widehat\Phi_\ell}{\widetilde H^{-1/2}(\Gamma)}$. 
While $q_{\rm est}$ depends only on  the\linebreak marking parameter $\theta$,
the constant $\c{opt:estred:aux}$ depends additionally on $\Gamma$, the polynomial degree $p$, 
and the uniform shape regularity of $\TT_\ell$.
\end{mylemma}

\begin{proof}
The triangle inequality yields
\begin{align}
\begin{split}\label{dp:hh2:eq1}
 \eta_{\ell+1}
 \le &\norm{h_{\ell+1}^{1/2}(1-\pi_{\ell+1}^p)\widehat\Phi_\ell}{L_2(\Gamma)}
 \\&+ \norm{h_{\ell+1}^{1/2}(1-\pi_{\ell+1}^p)(\widehat\Phi_{\ell+1}-\widehat\Phi_\ell)}{L_2(\Gamma)}.
 \end{split}
\end{align}
Note that the projection $\pi_{\ell+1}^p$ is even the $\TT_{\ell+1}$-piecewise best approximation, i.e.,
\begin{align*}
 \norm{(1-\pi_{\ell+1}^p)\psi}{L_2(T')}
 = \min_{\Psi_{\ell+1}\in\Pp^p(T')}\norm{\psi-\Psi_{\ell+1}}{L_2(T')}.
\end{align*}
This and the inverse estimate from Lemma~\ref{lem:Pp:invest} prove
\begin{align*}
 \norm{h_{\ell+1}^{1/2}&(1-\pi_{\ell+1}^p)(\widehat\Phi_{\ell+1}-\widehat\Phi_\ell)}{L_2(\Gamma)}
 \\&\le \norm{h_{\ell+1}^{1/2}(\widehat\Phi_{\ell+1}-\widehat\Phi_\ell)}{L_2(\Gamma)}
 \lesssim \norm{\widehat\Phi_{\ell+1}-\widehat\Phi_\ell}{\widetilde H^{-1/2}(\Gamma)}.
\end{align*}
The first summand in~\eqref{dp:hh2:eq1} is split into the contributions on refined and non-refined elements
\begin{align*}
 &\norm{h_{\ell+1}^{1/2}(1-\pi_{\ell+1}^p)\widehat\Phi_\ell}{L_2(\Gamma)}^2
 =\\
 &\qquad\sum_{T\in\TT_\ell\backslash\TT_{\ell+1}}\!\!\norm{h_{\ell+1}^{1/2}(1-\pi_{\ell+1}^p)\widehat\Phi_\ell}{L_2(T)}^2
 \\&\qquad\quad+ \sum_{T\in\TT_\ell\cap\TT_{\ell+1}}\!\!\norm{h_{\ell+1}^{1/2}(1-\pi_{\ell+1}^p)\widehat\Phi_\ell}{L_2(T)}^2.
\end{align*}
For non-refined elements $T\in\TT_\ell\cap\TT_{\ell+1}$ holds
\begin{align*}
 \norm{h_{\ell+1}^{1/2}(1-\pi_{\ell+1}^p)\widehat\Phi_\ell}{L_2(T)}^2
 &= \norm{h_{\ell}^{1/2}(1-\pi_{\ell}^p)\widehat\Phi_\ell}{L_2(T)}^2\\
 &= \eta_\ell(T)^2.
\end{align*}
For refined elements $T\in\TT_\ell\backslash\TT_{\ell+1}$ holds
\begin{align}\nonumber
 \norm{h_{\ell+1}^{1/2}(1-\pi_{\ell+1}^p)\widehat\Phi_\ell}{L_2(T)}^2
 &\le 2^{-1/(d-1)}\,\norm{h_{\ell}^{1/2}(1-\pi_{\ell}^p)\widehat\Phi_\ell}{L_2(T)}^2
 \\&= 2^{-1/(d-1)}\,\eta_\ell(T)^2.\label{dp:estred:refined}
\end{align}
Combining this with the bulk chasing~\eqref{opt:bulkchasing} and 
$\MM_\ell\subseteq\TT_\ell\backslash\TT_{\ell+1}$, we obtain
\begin{align*}
 \norm{h_{\ell+1}^{1/2}&(1-\pi_{\ell+1}^p)\widehat\Phi_\ell}{L_2(\Gamma)}^2
 \\&\le \eta_\ell^2
 -(1-2^{-1/(d-1)})\sum_{T\in\TT_\ell\backslash\TT_{\ell+1}}\eta_\ell(T)^2
 \\&\le \big(1-\theta(1-2^{-1/(d-1)})\big)\,\eta_\ell^2.
\end{align*}
This concludes the proof with $q_{\rm est} = \sqrt{1-\theta(1-2^{-1/(d-1)})}$.
$\hfill\qed$
\end{proof}

\begin{myproposition}\label{estred:prop:aprioriweaksing}
Algorithm~\ref{opt:algorithm} guarantees convergence\linebreak
$\lim_{\ell\to\infty}\est{\ell}{}=0$
of the $(h-h/2)$-type estimator $\eta_\ell$ from~\eqref{estred:def:hh2weaksing}.
\end{myproposition}

\begin{proof}
According to Lemma~\ref{estred:lem:estreddef} and Lemma~\ref{estred:lem:weaksing}, it remains to prove $\alpha_\ell\to 0$ as $\ell\to\infty$. By 
nestedness~\eqref{ass1:hh2}, Lemma~\ref{estred:lem:convortho} provides some
limit $\widehat\Phi_\infty\in\widetilde H^{-1/2}(\Gamma)$ such that
$\lim_{\ell\to\infty}\norm{\widehat\Phi_\infty-\widehat\Phi_\ell}{\widetilde H^{-1/2}(\Gamma)}=0$. Hence,
 \begin{align*}
  \norm{\widehat\Phi_{\ell+1}&-\widehat\Phi_\ell}{\widetilde H^{-1/2}(\Gamma)}
  \to 0
 \end{align*}
as $\ell\to\infty$.  This concludes the proof.
$\hfill\qed$
\end{proof}

\begin{remark}
Usual implementations of uniform mesh-refine\-ment ensure
$\TT_{\ell+1}\backslash\TT_{\ell}\subseteq\widehat\TT_\ell$. 
This implies $$\norm{h_{\ell+1}^{1/2}(1-\pi_{\ell+1}^p)\widehat\Phi_\ell}{L_2(T)}^2=0$$
in~\eqref{dp:estred:refined}
for all refined elements $T\in\TT_\ell\backslash\TT_{\ell+1}$ and thus leads to
$q_{\rm est}=\sqrt{1-\theta}$ in Lemma~\ref{estred:lem:weaksing}.
\end{remark}

\begin{remark}
Analogous results hold for other $(h-h/2)$-type estimators like 
$\eta_\ell = \norm{h_\ell^{1/2}(\widehat\Phi_\ell-\Phi_\ell)}{L_2(\Gamma)}$,
where\linebreak $\alpha_\ell \simeq\norm{\widehat\Phi_{\ell+1}-\widehat\Phi_\ell}{\wilde H^{-1/2}(\Gamma)}+\norm{\Phi_{\ell+1}-\Phi_\ell}{\wilde H^{-1/2}(\Gamma)}\to0$.
We note that, in practice, the variant from~\eqref{estred:def:hh2weaksing} is
preferred as it avoids the computation of the coarse-mesh Galerkin solution
$\Phi_\ell$.
\end{remark}
\subsubsection{Hypersingular integral equation}\label{section:estred:hh2hypsing}
As model problem serves the hypersingular integral equation from Proposition~\ref{prop:galerkin:hypsing}.
One possible $(h-h/2)$-type error estimator from Theorem~\ref{thm:def:hh2:hypsing:grad} employs the $L_2(\Gamma)$-ortho\-go\-nal
projection $\pi_\ell^{p-1}:=\pi_{\mesh_\ell}^{p-1}:\,L_2(\Gamma)\to\Pp^{p-1}(\mesh_\ell)$  as well as the solution $\widehat U_\ell$
of~\eqref{intro:galerkin}, where $\XX_\ell=\wilde\Sp^p(\TT_\ell)$ is replaced with the uniform refinement $\widehat\XX_\ell=\wilde\Sp^p(\widehat\TT_\ell)$ and reads
\begin{align}\label{estred:def:hh2:hypsing}
\est{\ell}{}^2:=\!\sum_{\el\in\mesh_\ell}\est{\ell}{\el}^2:=\!\sum_{\el\in\mesh_\ell}h_\el
\norm{(1-\pi_\ell^{p-1})\nabla_\Gamma\widehat U_\ell}{L_2(\el)}^2,
\end{align}
where $h_\el:=|\el|^{1/(d-1)}\simeq {\rm diam}(\el)$. 

\begin{mylemma}\label{estred:lem:hypsing}
Given a sequence of nested meshes $(\mesh_\ell)_{\ell\in\N}$, which additionally satisfy the bulk chasing~\eqref{opt:bulkchasing}
for all $\ell\in\N$ and some $0<\theta\le1$, the $(h-h/2)$ error estimator $\est{\ell}{}$ from~\eqref{estred:def:hh2:hypsing} satisfies the estimator reduction~\eqref{estred:eq:estred} with $\alpha_\ell:= \c{opt:estred:aux} \norm{\widehat U_{\ell+1}-\widehat U_\ell}{\widetilde H^{1/2}(\Gamma)}$. The constant $q_{\rm est}$ depends only on the marking parameter $\theta$, while $\c{opt:estred:aux}$ additionally depends on $\Gamma$, the polynomial degree $p$, and uniform shape regularity
of $\TT_\ell$.
\end{mylemma}

\begin{proof}
 The proof is very similar to that for the weakly singular integral equation from Lemma~\ref{estred:lem:weaksing} and therefore omitted.  The only difference is that at the present case, we need the inverse estimate from Lemma~\ref{lem:Sp:invest}.
$\hfill\qed$
\end{proof}

\begin{myproposition}\label{estred:prop:apriorihypsing}
Algorithm~\ref{opt:algorithm} guarantees convergence\linebreak $\lim_{\ell\to\infty}\est{\ell}{}=0$
of the $(h-h/2)$-type estimator $\eta_\ell$ from~\eqref{estred:def:hh2:hypsing}.
\end{myproposition}

\begin{proof}
As the proof of Proposition~\ref{estred:prop:aprioriweaksing}, the statement follows with Lemma~\ref{estred:lem:convortho} and Lemma~\ref{estred:lem:hypsing}.
$\hfill\qed$
\end{proof}

\begin{remark}
The proofs and assertions of Lemma~\ref{estred:lem:hypsing} and Pro\-position~\ref{estred:prop:apriorihypsing} also transfer to other $(h-h/2)$-type error estimators from Theorem~\ref{thm:def:hh2:hypsing}, e.g.,
\begin{align*}
 \eta_\ell = \norm{h_\ell^{1/2}\nabla_\Gamma(1-J_\ell)\widehat U_\ell}{L_2(\Gamma)},
\end{align*}
where 
\begin{align*}
 \eta_{\ell+1} 
 \le &\norm{h_{\ell+1}^{1/2}\nabla_\Gamma(1-J_\ell)\widehat U_\ell}{L_2(\Gamma)}
 \\&+ \norm{h_{\ell+1}^{1/2}\nabla_\Gamma\big((1-J_{\ell+1})\widehat U_{\ell+1}-(1-J_\ell)\widehat U_\ell\big)}{L_2(\Gamma)}.
\end{align*}
Arguing with the mesh-size reduction as in the proof of\linebreak Lemma~\ref{estred:lem:weaksing}, one sees
$\norm{h_{\ell+1}^{1/2}\nabla_\Gamma(1-J_\ell)\widehat U_\ell}{L_2(\Gamma)}
\le q_{\rm est}\,\eta_\ell$.
Suppose that $J_\ell$ satisfies the properties (i)--(iv) of 
Lemma~\ref{estred:lem:convclement}, e.g., $J_\ell$ is the Scott-Zhang projection
from Section~\ref{section:sz}. Then, the second term in the above estimate
is bounded by
\begin{align*}
 &\norm{h_{\ell+1}^{1/2}\nabla_\Gamma\big((1-J_{\ell+1})\widehat U_{\ell+1}-(1-J_\ell)\widehat U_\ell\big)}{L_2(\Gamma)}
 \\&\lesssim \norm{h_{\ell+1}^{1/2}\nabla_\Gamma(J_{\ell+1}-J_\ell)\widehat U_{\ell+1}}{L_2(\Gamma)} + \norm{h_{\ell+1}^{1/2}(\widehat U_{\ell+1}-\widehat U_\ell)}{L_2(\Gamma)} 
 \\&\lesssim \norm{(J_{\ell+1}-J_\ell)\widehat U_{\ell+1}}{\widetilde H^{1/2}(\Gamma)} + \norm{\widehat U_{\ell+1}-\widehat U_\ell}{\widetilde H^{1/2}(\Gamma)}
 =:\alpha_\ell
\end{align*}
where we have used the inverse estimate of Lemma~\ref{lem:Sp:invest}. With 
the a~priori convergence results of Lemma~\ref{estred:lem:convortho}
and Lemma~\ref{estred:lem:convclement}, one sees that $\alpha_\ell\to0$
as $\ell\to\infty$. This concludes the proof of the estimator reduction
also for other variants of the $(h-h/2)$ error estimator.
\end{remark}

\subsection{ZZ-type error estimators}\label{section:estred:zz}
For this section, we consider only the lowest-order case, which is $p=0$ for the weakly singular
integral equation and $p=1$ for the hypersingular integral equation.
\subsubsection{Weakly singular integral equation}
We consider the problem from Proposition~\ref{prop:galerkin:weaksing}.
The ZZ-type error estimator from Section~\ref{section:zzest:weak} reads
\begin{align}\label{estred:weaksingzz:def}
\est{\ell}{}^2:=\sum_{\el\in\mesh_\ell}\est{\ell}{\el}^2:=\sum_{\el\in\mesh_\ell}h_\el \norm{(1-A_\ell)\Phi_\ell}{L_2(\el)}^2,
\end{align}
where the smoothing operator $A_\ell:\,L_2(\Gamma)\to\Pp^1(\mesh_\ell)$ is defined in~\eqref{zzest:weak:Adef1}--\eqref{zzest:weak:Adef2}.

\begin{mylemma}\label{estred:lem:weaksingzz}
Given a sequence of nested meshes $(\mesh_\ell)_{\ell\in\N}$, which additionally satisfy the bulk chasing~\eqref{opt:bulkchasing}
for all $\ell\in\N$ and some $0<\theta\le1$, the ZZ-type error estimator $\est{\ell}{}$ from~\eqref{estred:weaksingzz:def} satisfies the estimator
reduction~\eqref{estred:eq:estred} with
$\alpha_\ell:=(\norm{h_{\ell+1}^{1/2}(1-A_\ell)(\Phi_{\ell+1}-\Phi_\ell)}{L_2(\Gamma)}+
\norm{h_{\ell+1}^{1/2}(A_{\ell+1}-A_\ell)\Phi_{\ell+1}}{L_2(\Gamma)})$. The
constant $q_{\rm est}$ depends only on $\theta$.
\end{mylemma}

\begin{proof}
The same arguments as used in the proof of Lemma~\ref{estred:lem:weaksing} apply.
The triangle inequality and reduction of the mesh-size~\eqref{dp2:estred} on
marked elements result in
\begin{align*}
\eta_{\ell+1}&\le q_{\rm est}\,\eta_\ell 
+ \norm{h_{\ell+1}^{1/2}\big((1-A_{\ell+1})\Phi_{\ell+1}-(1-A_\ell)\Phi_\ell}{L_2(\Gamma)}\\
&\le q_{\rm est}\,\eta_\ell + \alpha_\ell.
\end{align*}
This concludes the proof.
$\hfill\qed$
\end{proof}

\begin{myproposition}\label{estred:prop:aprioriweaksingzz}
Algorithm~\ref{opt:algorithm} guarantees convergence\linebreak $\lim_{\ell\to\infty} \est{\ell}{} = 0$ of the ZZ-type estimator $\eta_\ell$
from~\eqref{estred:weaksingzz:def}.
\end{myproposition}

\begin{proof}
Lemma~\ref{estred:lem:estreddef} and Lemma~\ref{estred:lem:weaksingzz} prove that it suffices to show $\alpha_\ell\to 0$ as $\ell\to\infty$.
 The operator $A_\ell$ satisfies the assumptions (i)--(iv) of Lemma~\ref{estred:lem:convclement} (see~\cite{zz2014}).
For the first contribution of $\alpha_\ell$, we use the $L_2$-stability (i) from Lemma~\ref{estred:lem:convclement}
and the inverse estimate from Lemma~\ref{lem:Pp:invest} to see
\begin{align*}
 \norm{h_\ell^{1/2}(1-A_\ell)(\Phi_{\ell+1}-\Phi_\ell)}{L_2(\Gamma)}&\lesssim \norm{h_\ell^{1/2}(\Phi_{\ell+1}-\Phi_\ell)}{L_2(\Gamma)}\\
 &\lesssim \norm{\Phi_{\ell+1}-\Phi_\ell}{\widetilde H^{-1/2}(\Gamma)}.
\end{align*} 
  Moreover, Lemma~\ref{estred:lem:convclement} provides a limit operator $A_\infty:\,L_2(\Gamma)\to L_2(\Gamma)$. For any $k\leq \ell$, there holds
 \begin{align*}
   \norm{&h_\ell^{1/2}(A_{\ell+1}-A_\ell)\Phi_\ell}{L_2(\Gamma)}\\
   &\leq  \norm{h_\ell^{1/2}(A_{\ell+1}-A_\ell)\Phi_k}{L_2(\Gamma)}\\
   &\qquad
  + \norm{h_\ell^{1/2}(A_{\ell+1}-A_\ell)(\Phi_k-\Phi_\ell)}{L_2(\Gamma)}\\
  &\lesssim  \norm{h_\ell^{1/2}(A_{\ell+1}-A_\ell)\Phi_k}{L_2(\Gamma)}
  + \norm{h_\ell^{1/2}(\Phi_k-\Phi_\ell)}{L_2(\Gamma)}.
 \end{align*}
where we used the local stability (i) from Lemma~\ref{estred:lem:convclement}. The inverse estimate from Lemma~\ref{lem:Pp:invest} shows
\begin{align*}
  \norm{h_\ell^{1/2}&(A_{\ell+1}-A_\ell)\Phi_\ell}{L_2(\Gamma)}\\
  &\lesssim  \norm{h_\ell^{1/2}(A_{\ell+1}-A_\ell)\Phi_k}{L_2(\Gamma)}
  + \norm{\Phi_k-\Phi_\ell}{\widetilde H^{-1/2}(\Gamma)}.
\end{align*}
Given any $\eps>0$, Lemma~\ref{estred:lem:convortho} allows to choose $k\in\N$ sufficiently large such that
$\norm{\Phi_\ell-\Phi_k}{\widetilde H^{-1/2}(\Gamma)}<\eps$ for all $\ell\geq k$.
For sufficiently large $\ell$, there also holds due to Lemma~\ref{estred:lem:convclement}
\begin{align*}
\norm{h_\ell^{1/2}(A_{\ell+1}-A_\ell)\Phi_k}{L_2(\Gamma)}
\lesssim \norm{(A_{\ell+1}-A_\ell)\Phi_k}{L_2(\Gamma)}\leq \eps.
\end{align*}
Altogether, we get $\alpha_\ell\lesssim \eps$ for all $\ell\ge k$ and 
therefore  conclude $\alpha\to 0 $ as $\ell\to\infty$. 
$\hfill\qed$
\end{proof}
\subsubsection{Hyper singular integral equation}
We consider the problem from Proposition~\ref{prop:galerkin:hypsing}.
The ZZ-type error estimator from Section~\ref{section:zzest:hyp} reads
\begin{align}\label{estred:hypsingzz:def}
\est{\ell}{}^2:=\sum_{\el\in\mesh_\ell}\est{\ell}{\el}^2:=\sum_{\el\in\mesh_\ell}h_\el \norm{(1-A_\ell)\nabla_\Gamma U_\ell}{L_2(\el)}^2,
\end{align}
where the smoothing operator $A_\ell:\,L_2(\Gamma)\to\Sp^1(\mesh_\ell)$ is defined by
 \begin{align*}
  (A_\ell \psi)(z):=|\omega_z|^{-1}\int_{\omega_z} \psi\,dz\quad\text{for all nodes }z \text{ of }\mesh_\ell
 \end{align*}
with the node patch $\omega_z := \bigcup \{\el\in\mesh_\ell \,:\, z\in\overline\el\}$.
The difference to the weakly singular case is, that $A_\ell\psi$ is continuous everywhere on $\Gamma$.

\begin{mylemma}\label{estred:lem:hypsingzz}
Given a sequence of nested meshes $(\mesh_\ell)_{\ell\in\N}$, which additionally satisfy the bulk chasing~\eqref{opt:bulkchasing}
for all $\ell\in\N$ and some $0<\theta\le1$, the ZZ-type error estimator $\est{\ell}{}$ from~\eqref{estred:hypsingzz:def} satisfies the estimator
reduction~\eqref{estred:eq:estred} with
$\alpha_\ell:=(\norm{h_{\ell+1}^{1/2}(1-A_\ell)\nabla_\Gamma(U_{\ell+1}-U_\ell)}{L_2(\Gamma)}+
\norm{h_{\ell+1}^{1/2}(A_{\ell+1}-A_\ell)\nabla_\Gamma U_{\ell+1}}{L_2(\Gamma)})$. The constant $q_{\rm est}$ depends only on $\theta$.
\end{mylemma}

\begin{proof}
The same arguments as in the proof of Lemma~\ref{estred:lem:weaksing}, prove the estimator reduction~\eqref{estred:eq:estred} for the ZZ-type error estimator. 
$\hfill\qed$
\end{proof}

\begin{myproposition}\label{estred:prop:apriorihypsingzz}
Algorithm~\ref{opt:algorithm} guarantees convergence\linebreak$\lim_{\ell\to\infty}\est{\ell}{}=0$ of the ZZ-type estimator $\eta_\ell$
from~\eqref{estred:hypsingzz:def}.
\end{myproposition}
\begin{proof}
The proof follows analogously to the proof of Proposition~\ref{estred:prop:aprioriweaksingzz}.
$\hfill\qed$
\end{proof}
\subsection{Weighted residual error estimators}\label{section:estred:res}
The weighted residual error estimator for BEM is more complex  than $(h-h/2)$-type- or ZZ-type error
estimators in the sense that it requires the evaluation of a
non-local integral operator. Therefore, the techniques are more involved in this section.
\subsubsection{Weakly singular integral equation}\label{section:estred:resweaksing}
We consider the problem from Proposition~\ref{prop:galerkin:weaksing}.
The standard weighted residual error estimator from Section~\ref{section:est:wres} for this problem reads
\begin{align}\label{estred:weaksingres:def}
\est{\ell}{}^2:=\sum_{\el\in\mesh_\ell}\est{\ell}{\el}^2:=\sum_{\el\in\mesh_\ell}h_\el \norm{\nabla_\Gamma(V\Phi_\ell-f)}{L_2(\el)}^2,
\end{align}
where $\nabla_\Gamma(\cdot)$ denotes the surface gradient on $\Gamma$.
Note that, while~\eqref{prop:galerkin:weaksing} is well-stated for $f\in H^{1/2}(\Gamma)$, the definition of $\est{\ell}{}$
requires additional regularity $f\in H^1(\Gamma)$ of the data.

\begin{mylemma}\label{estred:lem:weaksingres}
Given a sequence of nested meshes $(\mesh_\ell)_{\ell\in\N}$, which additionally satisfy the bulk chasing~\eqref{opt:bulkchasing}
for all $\ell\in\N$ and some $0<\theta\le1$, the weighted residual error estimator $\est{\ell}{}$ from~\eqref{estred:weaksingres:def} satisfies the estimator
reduction~\eqref{estred:eq:estred} with $\alpha_\ell:=\c{opt:estred:aux}\norm{\Phi_{\ell+1}-\Phi_\ell}{\widetilde
H^{-1/2}(\Gamma)}$. The constant $q_{\rm est}$ depends only on $\theta$, while $\c{opt:estred:aux}$ depends additionally on $\Gamma$, the polynomial degree $p$, and uniform shape regularity of $\mesh_\ell$.
\end{mylemma}

\begin{proof}
The proof follows the lines of the proof of Lemma~\ref{estred:lem:weaksing}.
The triangle inequality and contraction~\eqref{dp2:estred} of the mesh-size
on marked elements result in
\begin{align*}
 \eta_{\ell+1} \le q_{\rm est}\,\eta_\ell 
 + \norm{h_{\ell+1}^{1/2}\nabla V(\Phi_{\ell+1}-\Phi_\ell)}{L_2(\Gamma)}.
\end{align*}
Instead of the standard inverse estimates, one needs to 
employ~\eqref{opt:eq:invest:discrete} to obtain
\begin{align*}
\norm{h_{\ell+1}^{1/2}\nabla V(\Phi_{\ell+1}-\Phi_\ell)}{L_2(\Gamma)}
\lesssim \norm{\Phi_{\ell+1}-\Phi_\ell}{\wilde H^{-1/2}(\Gamma)}.
\end{align*}
This concludes the proof.
$\hfill\qed$
\end{proof}

\begin{myproposition}\label{estred:prop:aprioriweaksingres}
Algorithm~\ref{opt:algorithm} guarantees convergence\linebreak $\lim_{\ell\to\infty}\est{\ell}{}=0$ of the weighted residual error estimator $\eta_\ell$ from~\eqref{estred:weaksingres:def}.
\end{myproposition}
\subsubsection{Hypersingular integral equation}\label{section:estred:reshypsing}
We consider the problem from Proposition~\ref{prop:galerkin:hypsing}.
The standard weighted residual error estimator from Section~\ref{section:est:wres} for this problem reads
\begin{align}\label{estred:hypsingres:def}
\est{\ell}{}^2:=\sum_{\el\in\mesh_\ell}\est{\ell}{\el}^2:=\sum_{\el\in\mesh_\ell}h_\el \norm{WU_\ell-f}{L_2(\el)}^2.
\end{align}
Note that, while~\eqref{prop:galerkin:hypsing} is well-stated for $f\in \widetilde H^{-1/2}(\Gamma)$, the definition of $\est{\ell}{}$
requires additional regularity $f\in L_2(\Gamma)$ of the data.

\begin{mylemma}\label{estred:lem:hypsingres}
Given a  sequence of nested meshes $(\mesh_\ell)_{\ell\in\N}$, which additionally satisfy the bulk chasing~\eqref{opt:bulkchasing}
for all $\ell\in\N$ and some $0<\theta\le1$, the weighted residual error estimator $\est{\ell}{}$ from~\eqref{estred:hypsingres:def} satisfies the estimator reduction~\eqref{estred:eq:estred} with 
$\alpha_\ell:=\c{opt:estred:aux}\norm{U_{\ell+1}-U_\ell}{\wilde H^{1/2}(\Gamma)}$. The constant $q_{\rm est}$ depends only on $\theta$, while $\c{opt:estred:aux}$ depends
additionally on $\Gamma$, the polynomial degree $p$, and uniform shape regularity
of $\TT_\ell$.
\end{mylemma}

\begin{proof}
  The proof works analogously to the proof of Lemma \ref{estred:lem:weaksingres}, only this time employ
  the inverse-type estimate
  \begin{align*}
    \norm{h_{\ell+1}^{1/2}\hyp(U_{\ell+1}-U_\ell)}{L_2(\Gamma)}
    \lesssim \norm{U_{\ell+1}-U_\ell}{\wilde H^{1/2}(\Gamma)}.
  \end{align*}
  from~\eqref{opt:eq:invest:discrete}.
  $\hfill\qed$
\end{proof}

Arguing as before, Lemma~\ref{estred:lem:hypsingres} allows to derive convergence
of the related ABEM with Lemma~\ref{estred:lem:estreddef}.

\begin{myproposition}\label{estred:prop:apriorihypsingres}
Algorithm~\ref{opt:algorithm} guarantees convergence\linebreak $\lim_{\ell\to\infty}\est{\ell}{}=0$ of the weighted residual error estimator $\eta_\ell$ from~\eqref{estred:hypsingres:def}.
\end{myproposition}
\subsection{Approximation of right-hand side data with $(h-h/2)$-type estimators}
\label{section:estred:data}

In many cases, the right-hand side $F$ in~\eqref{intro:weakform} involves the application of
integral operators to the given data which can hardly be computed analytically in practice.
To circumvent this bottleneck, the aim  of data approximation is to replace the right-hand
side $F$ in~\eqref{intro:weakform} with some computable approximation $F_\ell$ on any mesh
$\mesh_\ell$ and to solve
\begin{align*}
 \form{U_\ell}{V}=F_\ell(V)\quad\text{for all }V\in\XX_\ell
\end{align*}
instead of~\eqref{intro:galerkin}.
\subsubsection{Weakly singular integral equation}\label{section:estred:hh2weaksingdata}
We consider the problem from Proposition~\ref{prop:galerkin:dirichlet}, where
$\Gamma=\partial\Omega$ and the right-hand side
in~\eqref{intro:weakform} reads $F(\psi):=\dual{(1/2+K)f}{\psi}_\Gamma$ for all $\psi\in \widetilde H^{-1/2}(\Gamma)$.
We approximate the right-hand side by approximating $f\in H^{1/2}(\Gamma)$ via the Scott-Zhang projection $J_\ell^{p+1}$,
i.e., $f_\ell:=J_\ell^{p+1} f\in \Sp^{p+1}(\mesh_\ell)$, and hence $F_\ell(\psi):=\dual{(1/2+K)f_\ell}{\psi}_\Gamma$.
We thus end up with the formulation given in Proposition~\ref{prop:galerkin:dirichlet:data}.
The additional error is controlled via extending the $(h-h/2)$-type error estimator
from~\eqref{estred:def:hh2weaksing} by a data oscillation term
\begin{align}\label{estred:weaksinghh2data:def}
\est{\ell}{}^2&:=\sum_{\el\in\mesh_\ell}\est{\ell}{\el}^2\\
&:=\sum_{\el\in\mesh_\ell}h_\el
\big(\norm{(1-\pi_\ell^{p})\widehat\Phi_{\ell}}{L_2(\el)}^2 
+\norm{\nabla_\Gamma(1-J_\ell^{p+1}) f}{L_2(\el)}^2\big), \nonumber
\end{align}
cf.\ Section~\ref{section:aposteriori:data:weaksing}.

\begin{mylemma}\label{estred:lem:weaksingdata}
Given a  sequence of nested meshes $(\mesh_\ell)_{\ell\in\N}$, which additionally satisfy the bulk chasing~\eqref{opt:bulkchasing}
for all $\ell\in\N$ and some $0<\theta\le1$, the $(h-h/2)$-type error estimator $\eta_\ell$ with data oscillation
term from~\eqref{estred:weaksinghh2data:def} satisfies the estimator reduction~\eqref{estred:eq:estred} with 
$\alpha_\ell:=\c{opt:estred:aux}\big(\norm{\widehat\Phi_{\ell+1}-\widehat\Phi_\ell}{H^{-1/2}(\Gamma)}+\norm{(J_{\ell+1}^{p+1}-J_\ell^{p+1})f}{H^{1/2}(\Gamma)}\big)$. The
constant $q_{\rm est}$ depends only on $\theta$, while $\c{opt:estred:aux}$ depends 
additionally on $\Gamma$, the polynomial degree
$p$, and uniform shape regularity of $\mesh_\ell$.
\end{mylemma}

\begin{proof}
The same arguments as used in the proof of Lemma \ref{estred:lem:weaksing} apply.
The triangle inequality and reduction of the mesh-size~\eqref{dp2:estred} on
marked elements result in
\begin{align*}
\eta_{\ell+1}&\le q_{\rm est}\,\eta_\ell 
+ \norm{h_{\ell+1}^{1/2}(1-\pi_{\ell+1}^p)(\widehat\Phi_{\ell+1}-\widehat\Phi_\ell)}{L_2(\Gamma)}\\
&\qquad+\norm{h_{\ell+1}^{1/2}\nablag(J_{\ell+1}^{p+1}-J_\ell^{p+1})f}{L_2(\Gamma)}
\end{align*}
As before,  the inverse estimate from Lemma~\ref{lem:Pp:invest} proves
\begin{align*}
 \norm{h_{\ell+1}^{1/2}(1-\pi_{\ell+1}^p)(\widehat\Phi_{\ell+1}-\widehat\Phi_\ell)}{L_2(\Gamma)}
 \lesssim \norm{\widehat\Phi_{\ell+1}-\widehat\Phi_\ell}{H^{-1/2}(\Gamma)}.
\end{align*}
Moreover, the inverse estimate from Lemma~\ref{lem:Sp:invest} gives
\begin{align*}
  \norm{h_{\ell+1}^{1/2}\nablag(J_{\ell+1}^{p+1}-J_\ell^{p+1})f}{L_2(\Gamma)}
 \lesssim \norm{(J_{\ell+1}^{p+1}-J_\ell^{p+1})f}{H^{1/2}(\Gamma)}
\end{align*}
and concludes the proof.
$\hfill\qed$
\end{proof}

\begin{myproposition}\label{estred:prop:aprioriweaksingdata}
Algorithm~\ref{opt:algorithm} guarantees convergence\linebreak $\lim_{\ell\to\infty}\est{\ell}{}=0$ of the $(h-h/2)$-type
estimator $\eta_\ell$ with data approximation from~\eqref{estred:weaksinghh2data:def}.
\end{myproposition}

\begin{proof}
Lemma~\ref{estred:lem:szconv} proves a~priori convergence
\begin{align*}
\norm{(J_\infty-J_\ell)f}{H^{1/2}(\Gamma)}\xrightarrow{\ell\to\infty}0.
\end{align*}
Consequently, it holds
\begin{align*}
 \norm{(J_{\ell+1}-J_\ell)f}{H^{1/2}(\Gamma)}\xrightarrow{\ell\to\infty}0.
\end{align*}
It remains to prove that
\begin{align*}
 \norm{\widehat\Phi_{\ell+1}-\widehat\Phi_\ell}{H^{-1/2}(\Gamma)}.
\end{align*}
Note that one cannot directly employ Lemma~\ref{estred:lem:convortho}, since
$\widehat\Phi_{\ell+1}$ and $\widehat\Phi_\ell$ are computed with respect to
\emph{different} right-hand sides. To tackle this issue, let 
$\widehat\Phi_{\ell,\infty}\in \Pp^p(\widehat\mesh_\ell)$ be the unique
solution of 
\begin{align*}
\form{\widehat\Phi_{\ell,\infty}}{V}=\dual{(1/2+K)J_\infty f}{V}_{L_2(\Gamma)}\quad\text{for all }V\in\widehat\XX_\ell.
\end{align*}
For this, Lemma~\ref{estred:lem:convortho} applies and proves convergence
$\norm{\widehat\Phi_\infty-\widehat\Phi_{\ell,\infty}}{H^{-1/2}(\Gamma)}\to0$
as $\ell\to\infty$. The triangle inequality proves
\begin{align*}
\norm{&\widehat\Phi_{\ell+1}-\widehat\Phi_\ell}{H^{-1/2}(\Gamma)}\\
&\leq\norm{\widehat\Phi_{\ell,\infty}-\widehat\Phi_\ell}{H^{-1/2}(\Gamma)}+
\norm{\widehat\Phi_{\ell+1}-\widehat\Phi_{\ell+1,\infty}}{H^{-1/2}(\Gamma)}\\
&\quad+ \norm{\widehat\Phi_{\ell,\infty}-\widehat\Phi_{\ell+1,\infty}}{H^{-1/2}(\Gamma)}.
\end{align*}
The third term on the right-hand side already vanishes as $\ell\to\infty$.
For the remaining to terms, the stability of the problem and Lemma~\ref{estred:lem:l2conv} show
\begin{align*}
\norm{&\widehat\Phi_{\ell,\infty}-\widehat\Phi_\ell}{H^{-1/2}(\Gamma)}+
\norm{\widehat\Phi_{\ell+1}-\widehat\Phi_{\ell+1,\infty}}{H^{-1/2}(\Gamma)}\\
&\lesssim \norm{f_\ell-J_\infty f}{H^{1/2}(\Gamma)} +
\norm{f_{\ell+1}-J_\infty f}{H^{1/2}(\Gamma)}\to 0
\end{align*}
as $\ell\to\infty$. Altogether, we obtain $\lim_{\ell\to\infty}\alpha_\ell=0$ and
conclude the proof.
$\hfill\qed$
\end{proof}

\begin{remark}
As a consequence of Proposition~\ref{estred:prop:aprioriweaksingdata}, one
obtains that $\norm{f-J_\ell f}{H^{1/2}(\Gamma)} \lesssim \eta_\ell \to0$ as 
$\ell\to\infty$, i.e., $J_\infty f = f$.
\end{remark}

\begin{remark}
If the $L_2$-orthogonal projection $\Pi_\ell^{p+1}:L_2(\Gamma)\to\Sp^{p+1}(\TT_\ell)$
is $H^1$-stable~\eqref{estred:eq:l2stab}, Lemma~\ref{estred:lem:weaksingdata}
and Proposition~\ref{estred:prop:aprioriweaksingdata} transfer to data approximation
with $f_\ell = \Pi_\ell^{p+1}f$. In practice, this approach is preferred, since
it might lead to superconvergence for pointwise errors inside of $\Omega$.
\end{remark}

\begin{remark}
The proofs of Lemma~\ref{estred:lem:weaksingdata}
and Proposition~\ref{estred:prop:aprioriweaksingdata} transfer to situations, 
where the approximation error of $f_\ell = J_\ell f \approx f$ is 
controlled by $\norm{h_\ell^{1/2}(1-\pi_\ell^p)\nabla_\Gamma f}{L_2(\Gamma)}$ 
in~\eqref{estred:weaksinghh2data:def} instead
of $\norm{h_\ell^{1/2}\nabla_\Gamma(1-J_\ell)f}{L_2(\Gamma)}$. This requires additional
care with the mesh-refine\-ment to ensure equivalence of these two norms,
cf.\ Section~\ref{section:aposteriori:data:weaksing}. Possible mesh refinement strategies
are discussed in Section~\ref{section:meshrefinement} below. In this case, one may
even use more general $H^{1/2}$-stable projections
$J_\ell:H^{1/2}(\Gamma)\to\Sp^{p+1}(\TT_\ell)$ instead of the Scott-Zhang projection
to discretize the data, see~\cite{dirichlet3d,ffkmp13}.
This approach will be presented in Section~\ref{section:estred:resweaksingdata}.
\end{remark}
\subsubsection{Hyper singular integral equation}\label{section:estred:hh2hypsingdata}
We consider the problem from Proposition~\ref{prop:galerkin:neumann}, where 
$\Gamma=\partial\Omega$ and the right-hand side
in~\eqref{intro:weakform} reads $F(\psi):=\dual{(1/2-K^\prime)f}{v}_\Gamma$ for all
$v\in H^{1/2}(\Gamma)$. We approximate the right-hand side by approximating
$f\in H^{-1/2}(\Gamma)$ by $f_\ell:=\pi^{p-1}_\ell f\in \Pp^{p-1}(\mesh_\ell)$
and let $F_\ell(v):=\dual{(1/2-K^\prime)f_\ell}{v}_\Gamma$. Recall that
$\pi_\ell^{p-1}$ is the $L_2(\Gamma)$-orthogonal projection onto $\Pp^{p-1}(\mesh_\ell)$.
We thus end up with the formulation given in Proposition~\ref{prop:galerkin:neumann:data}.
The additional error is controlled via extending the error estimator by a data oscillation term
\begin{align}\label{estred:hypsinghh2data:def}
  \est{\ell}{}^2&:=\sum_{\el\in\mesh_\ell}\est{\ell}{\el}^2 \\
  \begin{split}
    &:=\sum_{\el\in\mesh_\ell}h_\el \big(\norm{(1-\pi_\ell^{p-1})\nablag\widehat
    U_{\ell}}{L_2(\el)}^2 \\&\qquad\qquad+\norm{(1-\pi_\ell^{p-1}) f}{L_2(\el)}^2\big),
  \end{split}\nonumber
\end{align}
cf.\ Section~\ref{section:aposteriori:data:hypsing}

\begin{mylemma}\label{estred:lem:hypsingdata}
  Given a  sequence of nested meshes $(\mesh_\ell)_{\ell\in\N}$, which additionally satisfy
  the bulk chasing~\eqref{opt:bulkchasing} for all $\ell\in\N$ and some $0<\theta\le1$,
  the $(h-h/2)$-type error estimator  with data oscillation term
  from~\eqref{estred:hypsinghh2data:def} satisfies the estimator reduction~\eqref{estred:eq:estred}
  with $\alpha_\ell:=\c{opt:estred:aux}\norm{\widehat U_{\ell+1}-\widehat U_\ell}{H^{1/2}(\Gamma)}$.
  While $q_{\rm est}$ depends only on $\theta$, the constant $\c{opt:estred:aux}$ depends
  additionally on $\Gamma$, the polynomial degree $p$, the marking parameter $\theta$,
  and uniform shape regularity of $\mesh_\ell$.
\end{mylemma}

\begin{proof}
  The proof works analogously to that of the weakly singular case from
  Lemma~\ref{estred:lem:weaksingdata}, but may additionally use that
  $\norm{(1-\pi_{\ell+1}^{p-1})f}{L_2(T)}\le \norm{(1-\pi_{\ell}^{p-1})f}{L_2(T)}$ for all
  $T\in\TT_\ell$. This leads to an improved perturbation term $\alpha_\ell$.
  $\hfill\qed$
\end{proof}

As for the weakly singular case in Proposition~\ref{estred:prop:aprioriweaksingdata},
we obtain convergence of data perturbed ABEM for the hypersingular integral equation.

\begin{myproposition}\label{estred:prop:apriorihypsingdata}
  Algorithm~\ref{opt:algorithm} guarantees convergence\linebreak
  $\lim_{\ell\to\infty}\est{\ell}{}=0$ of the $(h-h/2)$-type estimator $\est{\ell}{}$
  with data approximation from~\eqref{estred:hypsinghh2data:def}.
\end{myproposition}

\subsection{Approximation of right-hand side data with weighted residual estimators}
\label{section:estred:data:res}
\subsubsection{Weakly singular integral equation}\label{section:estred:resweaksingdata}
We consider the problem from Proposition~\ref{prop:galerkin:dirichlet}, where
$\Gamma = \partial\Omega$ and the right-hand side is given by
$F(\psi):=\dual{(1/2+K)f}{\psi}_\Gamma$ for all $\psi\in \widetilde H^{-1/2}(\Gamma)$.
Let $J_\ell^{p+1}:H^{1/2}(\Gamma)\rightarrow\linebreak\Sp^{p+1}(\mesh_\ell)$ be an arbitrary $H^{1/2}(\Gamma)$
stable projection such that the pointwise limit
\begin{align*}
  J^{p+1}_\infty v = \lim_{\ell\rightarrow\infty}J_\ell^{p+1} v
\end{align*}
exists for any $v\in H^{1/2}(\Gamma)$.
We approximate the right-hand side by approximating $f\in H^{1/2}(\Gamma)$ by
$f_\ell:= J^{p+1}_\ell f\in \Sp^{p+1}(\mesh_\ell)$ and let
$F_\ell(\psi):=\dual{(1/2+K)f_\ell}{\psi}_\Gamma$, i.e., we arrive at the
discrete formulation of Proposition~\ref{prop:galerkin:dirichlet:data}.
This additional error is controlled via extending the error estimator by a data oscillation term
\begin{align}\label{estred:weaksingresdata:def}
\est{\ell}{}^2&:=\sum_{\el\in\mesh_\ell}\est{\ell}{\el}^2\\ \nonumber
&:=\sum_{\el\in\mesh_\ell}h_\el \big(\norm{\nabla_\Gamma(V\Phi_\ell-(1/2+K)f_\ell)}{L_2(\el)}^2\\
&\qquad+\norm{(1-\pi_\ell^{p})\nabla_\Gamma f}{L_2(\el)}^2\big). \nonumber
\end{align}
Possible example for $J^{p+1}_\ell$ include the Scott-Zhang projection onto $\Sp^{p+1}(\mesh_\ell)$,
cf. Lemma~\ref{estred:lem:szconv}, as well as the $L_2$ -orthogonal projection
onto $\Sp^{p+1}(\mesh_\ell)$, provided that the latter is $H^1(\Gamma)$ stable, cf
Lemma~\ref{estred:lem:l2conv}.

\begin{mylemma}\label{estred:lem:weaksingdatares}
  Given a  sequence of nested meshes $(\mesh_\ell)_{\ell\in\N}$, which additionally satisfy
  the bulk chasing~\eqref{opt:bulkchasing} for all $\ell\in\N$ and some $0<\theta\le1$,
  the weighted residual error estimator $\est{\ell}{}$ with data oscillation
  term from~\eqref{estred:weaksingresdata:def} satisfies the estimator
  reduction~\eqref{estred:eq:estred} with 
  $\alpha_\ell:=\c{opt:estred:aux}(\norm{\Phi_{\ell+1}-\Phi_\ell}{\widetilde H^{-1/2}(\Gamma)}+
  \norm{f_{\ell+1}-f_\ell}{H^{1/2}(\Gamma)})$. While $q_{\rm est}$ depends only on $\theta$,
  the constant $\c{opt:estred:aux}$ depends additionally on $\Gamma$, the polynomial degree $p$,
  the marking parameter $\theta$, and uniform shape regularity of $\mesh_\ell$.
\end{mylemma}
\begin{proof}
  The data oscillation term is treated as in the proof of Lemma~\ref{estred:lem:weaksingdata}.
  For the estimator, one needs additionally the inverse estimate~\eqref{opt:eq:invest:discrete}
  to estimate
  \begin{align*}
    \norm{h_{\ell+1}^{1/2}\nabla_\Gamma (1/2+K)(f_{\ell+1}-f_\ell)}{L_2(\Gamma)}\lesssim
    \norm{f_{\ell+1}-f_\ell}{H^{1/2}(\Gamma)}.
  \end{align*}
  The remainder however, follows exactly the lines of the proof of Lemma~\ref{estred:lem:weaksingres}.
  $\hfill\qed$
\end{proof}

\begin{myproposition}\label{estred:prop:aprioriweaksingdatares}
  Algorithm~\ref{opt:algorithm} guarantees convergence\linebreak $\lim_{\ell\to\infty}\est{\ell}{}=0$
  of the weighted residual estimator with data approximation from~\eqref{estred:weaksingresdata:def}.
\end{myproposition}
\begin{proof}
  The proof follows the lines of Proposition~\ref{estred:prop:aprioriweaksingdata}.
  Additionally, we need to employ the convergence
  $\lim_{\ell\to\infty}\norm{f_{\ell+1}-f_\ell}{H^{1/2}(\Gamma)}^2=0$
  from Lemma~\ref{estred:lem:l2conv}.
  $\hfill\qed$
\end{proof}
\subsubsection{Hypersingular integral equation}\label{section:estred:reshypsingdata}
We consider the problem from Proposition~\ref{prop:galerkin:neumann}, where the right-hand side
in~\eqref{intro:weakform} reads $F(v):=\dual{(1/2-K^\prime)f}{v}_\Gamma$ for all
$v\in H^{1/2}(\Gamma)$.
We approximate the right-hand side by approximating $f\in \widetilde H^{-1/2}(\Gamma)$
by $f_\ell:=\pi^{p-1}_\ell f\in \Pp^{p-1}(\mesh_\ell)$ and let
$F_\ell(v):=\dual{(1/2-K^\prime)f_\ell}{v}_\Gamma$, i.e., we arrive at the discrete formulation
of Proposition~\ref{prop:galerkin:neumann:data}.
The additional error is controlled via extending the error estimator by a data oscillation term
\begin{align}\label{estred:hypsingresdata:def}
  \est{\ell}{}^2&:=\sum_{\el\in\mesh_\ell}\est{\ell}{\el}^2\\ \nonumber
  &:=\sum_{\el\in\mesh_\ell}h_\el
  \big(\norm{\nabla_\Gamma(W U_\ell-(1/2-K^\prime)f_\ell)}{L_2(\el)}^2\\
  &\qquad+\norm{(1-\pi_\ell^{p-1})f}{L_2(\el)}^2\big). \nonumber
\end{align}

\begin{mylemma}\label{estred:lem:hypsingdatares}
  Given a  sequence of nested meshes $(\mesh_\ell)_{\ell\in\N}$, which additionally satisfy
  the bulk chasing~\eqref{opt:bulkchasing} for all $\ell\in\N$ and some $0<\theta\le1$,
  the weighted residual error estimator with data oscillation term $\est{\ell}{}$
  from~\eqref{estred:hypsingresdata:def} satisfies the estimator reduction~\eqref{estred:eq:estred}
  with $\alpha_\ell:=\c{opt:estred:aux}(\norm{U_{\ell+1}-U_\ell}{H^{1/2}(\Gamma)}+
  \norm{f_{\ell+1}-f_\ell}{H^{-1/2}(\Gamma)})$. The constants $q_{\rm est}, \c{opt:estred:aux}$
  depend only on $\Gamma$, the marking parameter $\theta$, the polynomial degree $p$, and the
  uniform $\sigma_\ell$-shape regularity.
\end{mylemma}
\begin{proof}
  The data oscillation term is treated as in the proof of Lemma~\ref{estred:lem:weaksingdata}.
  For the estimator, one needs additionally the inverse estimate~\eqref{opt:eq:invest:discrete}
  to estimate
  \begin{align*}
    \norm{h_{\ell+1}^{1/2} (1/2-K^\prime)(f_{\ell+1}-f_\ell)}{L_2(\Gamma)}\lesssim
    \norm{f_{\ell+1}-f_\ell}{H^{-1/2}(\Gamma)}.
  \end{align*}
  The remainder however, follows exactly the lines of the proof of Lemma~\ref{estred:lem:hypsingres}.
  $\hfill\qed$
\end{proof}

\begin{myproposition}\label{estred:prop:apriorihypsingdatares}
Algorithm~\ref{opt:algorithm} guarantees convergence\linebreak $\lim_{\ell\to\infty}\est{\ell}{}=0$
of the weighted residual estimator with data approximation from~\eqref{estred:hypsingresdata:def}.
\end{myproposition}
\begin{proof}
  The proof follows along the lines of Proposition~\ref{estred:prop:aprioriweaksingdata}.
  Additionally, we need to employ the convergence\linebreak
  $\lim_{\ell\to\infty}\norm{f_{\ell+1}-f_\ell}{H^{-1/2}(\Gamma)}^2\leq
  \lim_{\ell\to\infty}\norm{f_{\ell+1}-f_\ell}{L_2(\Gamma)}^2 = 0$ from Lemma~\ref{estred:lem:l2conv}.
  $\hfill\qed$
\end{proof}

\begin{figure*}[t]
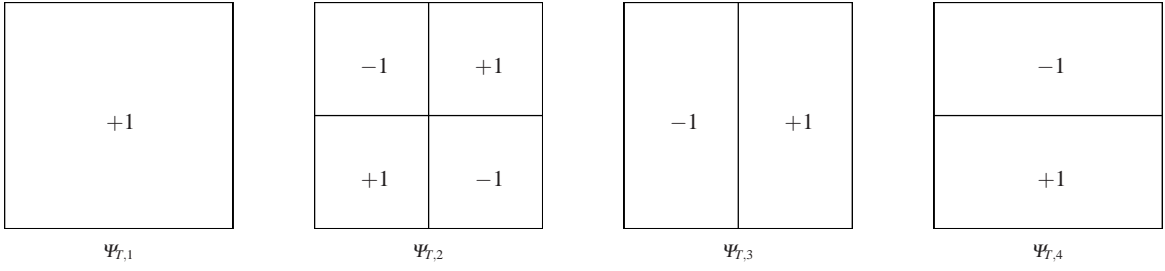

\psfrag{T}[cc][cc][1][0]{$\displaystyle +1$}%
\psfrag{T1}[cc][cc][1][0]{$\displaystyle -1$}%
\psfrag{T2}[cc][cc][1][0]{$\displaystyle +1$}%
\psfrag{T3}[cc][cc][1][0]{$\displaystyle -1$}%
\psfrag{T4}[cc][cc][1][0]{$\displaystyle +1$}%
 \def\fig#1#2{\begin{minipage}[t]{40mm}%
   \centering\hspace*{-5.5mm}\includegraphics[height=37mm]{intro_#1.eps}\\[-4mm]{\scriptsize#2}%
 \end{minipage}}
 \fig{T0}{$\Psi_{\el,1}$}
 \fig{T0unif}{$\Psi_{\el,2}$}
 \fig{T0hori}{$\Psi_{\el,3}$}
 \fig{T0vert}{$\Psi_{\el,4}$}
\caption{Functions $\Psi_{\el,i}$ with their values on the element $\el\in\mesh_\ell$ used for the computation of the
error estimator from Section~\ref{section:estred:anisotropic}.}
\label{estred:fig:ani}
\end{figure*}
\subsection{Anisotropic mesh refinement}\label{section:estred:anisotropic}
The presence of edge singularities in solutions of simple problems like $V\phi=1$ on some boundary $\Gamma:=\partial\Omega$ with $\Omega\subseteq \R^3$ makes it necessary to
allow for anisotropic mesh refinement if one aims to achieve optimal convergence rates. This implies that the
shape-regularity constant $\sigma_\ell$ from Section~\ref{section:bem:discrete} cannot remain bounded for
a given sequence of meshes $(\mesh_\ell)_{\ell\in\N_0}$, but satisfies $\sup_{\ell\in\N_0}\sigma_\ell = \infty$.
The following variant of the $(h-h/2)$-type error estimator accounts for this:
\begin{align}\label{estred:anisotropic:def}
\est{\ell}{}^2:=\sum_{\el\in\mesh_\ell}\est{\ell}{\el}^2:=\sum_{\el\in\mesh_\ell}\rho_\el \norm{(1-\pi_\ell^0)\widehat\Phi_\ell}{L_2(\el)}^2,
\end{align}
where $\rho_\el>0$ denotes the radius of the largest inscribed circle of the element $\el\in\mesh_\ell$. Obviously, there holds $\rho_\el\leq h_\el$, and $\sup_{\el\in\mesh_\ell}h_\el/\rho_\el$ depends only on $\sigma_\ell$.
We briefly discuss the lowest-order case $p=0$ for rectangular elements $\el\in\mesh_\ell$, which provides an
easy-to-implement criterion to decide how to refine the elements, while the general case $p\geq 0$ is discussed
in~\cite{afp12}.
As depicted in Figure~\ref{estred:fig:ani}, we define four element functions $\Psi_{\el,i}\in \Pp^0(\widehat \mesh_\ell)$ with $\supp(\Psi_{\el,i})\subseteq \el$.
Note that $\set{\Psi_{\el,i}}{\el\in\mesh_\ell,\,i=1,\ldots,4}$ defines a basis of $\Pp^0(\widehat\mesh_\ell)$. Hence, for each $\el\in\mesh_\ell$, there exist (computable) coefficients 
\begin{align*}
 c_{\el,i}:=\frac{\int_\el \Psi_{\el,i}\, \widehat\Phi_\ell\,dx}{\norm{\Psi_{\el,i}}{L_2(\el)}}\quad\text{for }i=1,2,3,4,
\end{align*}
such that
\begin{align*}
\widehat\Phi_\ell = \sum_{i=1}^4 c_{\el,i}\Psi_{\el,i}\quad\text{on }\el. 
\end{align*}
If one intends to refine $\el$ (i.e., $\el$ is marked for refinement by the bulk chasing~\eqref{opt:bulkchasing}), the following set of rules decides the direction of refinement: Choose an additional parameter $0<\tau<1$ which steers the sensitivity to directional refinement (cf.~Figure~\ref{intro:3d:refinement} from the introduction).
\begin{itemize}
 \item[(R1)] If $c_{\el,2}^2+c_{\el,3}^2\leq \tau/(1-\tau)c_{\el,4}^2$, split $\el$ along the vertical direction to generate two sons $\el_1,\el_2\in\mesh_{\ell+1}$.
 \item[(R2)] If $c_{\el,2}^2+c_{\el,4}^2\leq \tau/(1-\tau)c_{\el,3}^2$, split $\el$ along the horizontal direction to generate two sons $\el_1,\el_2\in\mesh_{\ell+1}$.
 \item[(R3)] If none of the above applies split $\el$ along both directions to generate four sons $\el_1,\el_2,\el_3,\el_4\in\mesh_{\ell+1}$.
\end{itemize}
We note that (R1) and (R2) are exclusive, i.e., if the criterion from (R1) is satisfied, the  criterion from (R2) fails to hold (cf.~\cite{afp12}).
Moreover, we need to ensure the following two refinement rules to guarantee the validity of the inverse
estimate~\cite[Thm.~3.6]{ghs05} of Lemma~\ref{lem:Pp:invest} on anisotropic meshes.
\begin{itemize}
 \item[(R4)] Hanging nodes are at most of order one, i.e., each side $e$ of an elements $\el\in\mesh_\ell$ contains at most one node $z$ which is not an endpoint of $e$.
 \item[(R5)] $K$-mesh property: there holds for some $\kappa_\ell>0$
 \begin{align*}
  \frac{\rho_\el}{\rho_{\el^\prime}}+ \frac{h_\el}{h_{\el^\prime}} \leq \kappa_\ell<\infty,
 \end{align*}
 for all $\el,\el^\prime\in\mesh_\ell$ with $\el\cap\el^\prime\neq \emptyset$.

\end{itemize}

The following lemma is proved in~\cite{fp08,afp12}.
\begin{mylemma}\label{estred:lem:weaksingani}
Given a sequence of meshes $(\mesh_\ell)_{\ell\in\N}$ with\linebreak $\mesh_{\ell+1}\in\refine(\mesh_\ell)$ for all $\ell\in\N_0$
and $K:=\sup_{\ell\in\N_0}\kappa_\ell<\infty$, which additionally satisfy the bulk chasing~\eqref{opt:bulkchasing}
for all $\ell\in\N$ and some $0<\theta\le1$. Suppose that all marked elements $\MM_\ell\subseteq\mesh_\ell\setminus\mesh_{\ell+1}$ are refined according to the rules (R1)--(R5).
Then, the modified $(h-h/2)$ error estimator $\est{\ell}{}$ from~\eqref{estred:anisotropic:def} satisfies the
estimator reduction~\eqref{estred:eq:estred} with $\alpha_\ell:=\c{opt:estred:aux}\norm{\widehat
\Phi_{\ell+1}-\widehat\Phi_\ell}{\widetilde H^{-1/2}(\Gamma)}^2$. The constants 
$q_{\rm est}$ depends only on $\theta$ and $\tau$, while $\c{opt:estred:aux}$ depends
additionally on $\Gamma$, the polynomial degree $p$, and the K-mesh constant $K$.
\end{mylemma}

As before, Lemma~\ref{estred:lem:estreddef} implies convergence of ABEM.

\begin{myproposition}\label{estred:prop:anisotropic}
Algorithm~\ref{opt:algorithm} guarantees convergence\linebreak $\lim_{\ell\to\infty}\est{\ell}{}=0$ of the modified $(h-h/2)$-type
estimator $\est{\ell}{}$ from~\eqref{estred:anisotropic:def} for
anisotropic mesh refinement.
\end{myproposition}
\section{Mesh refinement}\label{section:meshrefinement}
When it comes to the mathematical proof of optimal convergence rates
of ABEM (Section~\ref{section:convergence}), it is
clear that this requires certain properties of the mesh refinement which
go beyond the elementary properties from Section~\ref{section:estred:meshrefinementprelim}. While
those are sufficient to prove plain convergence of ABEM by means of the 
estimator reduction principle from Section~\ref{section:estred}, they formally do not prevent that marking of
\emph{one single} element $\MM_\ell=\{T\}\subset\TT_\ell$ results in a 
refinement $\TT_{\ell+1}$, where \emph{all} elements have been refined, i.e., 
$\TT_\ell\backslash\TT_{\ell+1}=\TT_\ell$. Moreover, the contemporary
mathematical proofs of optimal convergence rates require certain additional
properties. 

\subsection{General notation}
Suppose a fixed mesh refinement strategy $\refine(\cdot)$ and an admissible
mesh $\TT$, i.e., the mesh refinement $\refine(\cdot)$ can be used to refine $\TT$
and provides a refined admissible mesh.
For $\MM\subseteq\TT$ being a set of marked elements, we write
$\TT' = \refine(\TT,\MM)$ if $\TT'$ is the coarsest admissible mesh which
is obtained from $\TT$ by refinement of at least the marked elements $\MM$,
i.e., $\TT\backslash\TT'\supseteq\MM$. For some admissible mesh $\TT$,
we write $\TT'\in\refine(\TT)$, if there exists some $k\in\N_0$ and
sets of marked elements $\widetilde\MM_0,\ldots,\widetilde\MM_{k-1}$ as 
well as meshes $\widetilde\mesh_0,\widetilde\mesh_1,\ldots,\widetilde\mesh_k$ such that $\widetilde \MM_j\subseteq \widetilde\mesh_j$ and 
$\widetilde\mesh_{j+1}=\refine(\widetilde\mesh_j,\widetilde\MM_j)$ 
for all $j=0,\ldots,k-1$, with $\mesh=\widetilde\mesh_0$ and $\mesh'=\widetilde\mesh_k$.

\subsection{Optimality conditions on mesh refinement}\label{section:meshopt}
Besides the naive properties from Section~\ref{section:estred:meshrefinementprelim}, the contemporary 
mathematical proofs of optimal convergence rates require certain additional
properties of the mesh refinement. Suppose that $\TT_0$ is a given admissible
initial mesh for the adaptive algorithm and that $\TT_\ell$ for $\ell\ge1$ is 
obtained inductively by $\TT_{\ell} = \refine(\TT_{\ell-1},\MM_{\ell-1})$. 
Let $\eta_\ell$ be the a~posteriori error estimator used to mark elements 
$\MM_\ell\subseteq\TT_\ell$ for refinement. Then, the analysis of 
Section~\ref{section:convergence} relies on the following three properties
of $\refine(\cdot)$, where $\#(\cdot)$ denotes the number of elements of a finite
set:
\begin{itemize}
\item[\Large$\bullet$]
{\bf Bounded shape-regularity:}
The mesh refinement strategy has to ensure that all estimator related constants
in, e.g., reliability or efficiency estimates~\eqref{intro:reliable}--\eqref{intro:efficient}, remain uniformly bounded as $\ell\to\infty$.
\item[\Large$\bullet$]
{\bf Mesh-closure estimate:}
The number of refined elements can (at least in average and up to some multiplicative constant) 
be controlled by the number of marked elements in the sense that
\begin{align}\label{dp:meshclosure}
 \#\TT_{\ell+1} - \#\mesh_0 \leq \c{nvb}\sum_{k=0}^{\ell}\#\MM_k
\end{align}
for some constant $\c{nvb}>0$.
\item[\Large$\bullet$]
{\bf Overlay estimate:}
To compare the adaptively generated meshes $\TT_\ell$ with some (purely theoretical) optimal
mesh, one requires that for all $\TT_\star\in\refine(\TT_0)$ exists a coarsest common
refinement $\TT_\star\oplus\TT_\ell$ of both $\TT_\star$ and $\TT_\ell$ such that
\begin{align}\label{dp:overlay}
\#(\TT_\star\oplus\TT_\ell) \le \#\TT_\star + \#\TT_\ell - \#\TT_0.
\end{align}
\end{itemize}
Different methods for local mesh refinement are available in the literature. To
the best of our knowledge, only three strategies are available which ensure these
properties: for 2D BEM, the extended 1D bisection algorithm from~\cite{affkp13-A};
for 3D BEM, the 2D NVB algorithm\footnote{newest vertex bisection (NVB).}, see e.g.~\cite{steve08,kpp13},
as well as red-refinement with hanging nodes of maximum order $1$, 
see~\cite{bn2010}. In the following, we shall discuss the extended 1D bisection
algorithm from~\cite{affkp13-A} as well as the results of~\cite{kpp13} on 2D
NVB.

We close this section with some historical remarks. The mesh-closure estimate~\eqref{dp:meshclosure}
has first been proved in~\cite{bdd04} for 2D NVB and later
for NVB and general dimension $d\ge2$ in~\cite{steve08}. Either work requires an additional
assumption on the initial mesh $\TT_0$. For 2D, this assumption has recently
been removed in~\cite{kpp13}. The overlay estimate~\eqref{dp:overlay} first appeared in~\cite{steve07}
for 2D NVB. In~\cite{ckns}, the proof is generalized to NVB in arbitrary dimension $d\ge2$.
Bisection in 1D has only been considered and analyzed in~\cite{affkp13-A}. 
Even though the above mesh refinement strategies seem fairly arbitrary, to the
best of the authors' knowledge, NVB is the only refinement strategy for $d\ge2$
known to satisfy~\eqref{dp:meshclosure}--\eqref{dp:overlay}. Even the simple
red-green-blue refinement, see e.g.~\cite{c04}, fails to satisfy~\eqref{dp:overlay},
while the mesh-closure estimate~\eqref{dp:meshclosure} can still be proved,
see~\cite{kpp13} and the references therein.

\subsection{Extended 1D bisection for 2D BEM}

In 2D BEM, the constants in the a~priori or a~posteriori error analysis usually
depend on a uniform upper bound $\sigma>0$ of the shape-regularity constant
(or: bounded local mesh-ratio)
\begin{align}\label{dp:meshration}
 \frac{\diam(T)}{\diam(T')}
 \le \sigma
 \text{ for all neighbors }T,T'\in\TT_\ell\text{ and }\ell\ge0.
\end{align}
Since this property is not guaranteed by simple 1D bisection algorithms, it has
to be ensured explicitly. With the shape-regularity constant of the initial mesh
\begin{align*}
 \sigma_{\TT_0} := \max\set{\frac{\diam(T)}{\diam(T')}}{T,T'\in\TT_0\text{ with }\overline T\cap \overline T'\neq\emptyset},
\end{align*}
we use the following algorithm from~\cite{affkp13-A}.

\begin{algorithm}[Extended 1D bisection]\label{alg:bisect1d}\ \linebreak
\def\UU{\mathcal U}
\textsc{Input}: local mesh-ratio $\sigma_{\TT_0}$, current mesh $\mesh_\ell$, and
  set of marked elements $\MM_\ell\subseteq\mesh_\ell$.\\
\textsc{Output}: refined mesh $\mesh_{\ell+1}$.
\begin{itemize}
 \item[\rm(o)] Set counter $k:=0$ and define $\MM_\ell^{(0)}:=\MM_\ell$ .
 \item[\rm(i)] Define $\UU^{(k)}:=\bigcup_{T\in\MM_\ell^{(k)}}\set{T'\in\TT_\ell\backslash\MM_\ell^{(k)}
\text{ neighbor of }T}{\diam(T')>\sigma_{\TT_0}\,\diam(T)}$ and $\MM_\ell^{(k+1)}:=\MM_\ell^{(k)}\cup\UU^{(k)}$.

    \item[\rm(ii)] If $\MM_\ell^{(k)}\subsetneqq\MM_\ell^{(k+1)}$, increase counter $k\mapsto k+1$ and goto (i).
    \item[\rm(iii)] Otherwise bisect all elements $T\in\MM^{(k)}$ to obtain the new mesh $\mesh_{\ell+1}$.
    \end{itemize}
\end{algorithm}

The following result is proved in~\cite[Thm.~2.3]{affkp13-A}.

\begin{theorem}
  Suppose that $\mesh_0$ is a partition of $\Gamma$, and\linebreak $\left( \mesh_\ell \right)_{\ell\in\N_0}$ is generated by
  Algorithm~\ref{alg:bisect1d}, i.e., for all $\ell\in\N_0$ holds
  \begin{align*}
    \mesh_{\ell+1} = \refine\left( \mesh_\ell, \MM_\ell \right).
  \end{align*}
  Then, bounded shape regularity~\eqref{dp:meshration} with
$\sigma = 2\,\sigma_{\TT_0}$ is guaranteed.
As a consequence of bisection and~\eqref{dp:meshration}, only
finitely many shapes of node and element patches can occur.
 Moreover, the mesh closure estimate~\eqref{dp:meshclosure}
as well as the overlay estimate~\eqref{dp:overlay} are valid, where the constant 
$\c{nvb}>0$ depends only on $\TT_0$. Finally, the coarsest common refinement
$\TT_\star\oplus\TT_\ell$ of $\TT_\ell,\TT_\star\in\refine(\TT_0)$ is the
overlay
\begin{align}\label{dp:def:overlay}
\begin{split}
 \TT_\star\oplus\TT_\ell
 = \big\{T\in\TT_\star\cup\TT_\ell\,:\,&\forall T'\in\TT_\star\cup\TT_\ell\\
 &\quad(T'\subseteq T\,\Rightarrow T'=T)\big\},
\end{split}
\end{align}
i.e., the union of the locally finest elements.\qed
\end{theorem}

\begin{figure}
\centering%
\psfrag{T1}[cc][cc][1][1]{\tiny{\color{red}$T_1$}}%
\psfrag{TT0}[cc][cc][1][1]{\tiny{$\TT_0$}}%
\psfrag{TT1}[cc][cc][1][1]{\tiny{$\TT_\ell$}}%
\psfrag{TT2}[cc][cc][1][1]{\tiny{$\TT_{\ell+1}$}}%
\psfrag{1}[cc][cc][1][1]{\tiny{$1$}}%
\psfrag{1/2}[cc][cc][1][1]{\tiny{$1/2$}}%
\psfrag{1/4}[cc][cc][1][1]{\tiny{$1/4$}}%
\psfrag{1/8}[cc][cc][1][1]{\tiny{$1/8$}}%
\psfrag{0}[cc][cc][1][1]{\tiny{$0$}}%
\includegraphics[width=.46\textwidth]{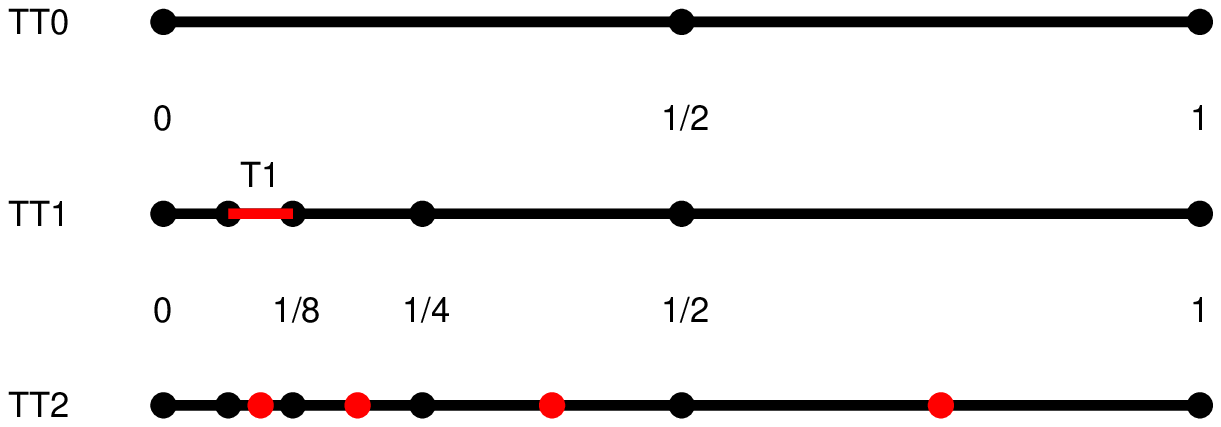}
\caption{For 1D bisection, refined elements $\TT_\ell\setminus \TT_{\ell+1}$ are bisected into two sons, whence $\#\MM_\ell\leq \#(\TT_\ell\setminus\TT_{\ell+1}) = \#\TT_{\ell+1}-\#\TT_\ell$. With $\sigma_{\TT_\ell}\le2\,\sigma_{\TT_0}$, the converse inequality $\#\TT_{\ell+1}-\#\TT_\ell\lesssim\#\MM_\ell$ cannot hold in general as the following elementary example proves: Let $\TT_0$ denote the partition of $[0,1]$ in two elements of length $1/2$, i.e., $\sigma_{\TT_0}=1$. Repeated marking of the leftmost elements of $\TT_0,\TT_1,\ldots,\TT_{\ell-1}$ generates the mesh $\TT_\ell$ with $\sigma_{\TT_\ell}=2$ and $\#\TT_\ell=\ell$. Marking the highlighted element $T_1 \in\TT_\ell$ results in the mesh $\TT_{\ell+1}:=\refine(\TT_\ell,\{T_1\})$,
where $\ell-1$ elements are refined to ensure $\sigma_{\TT_{\ell+1}}=2$.
Consequently, the number of additional refinements can be arbitrarily large,
and~\eqref{dp:counterexample1d} cannot hold in general.}
\label{fig:bisect1d}
\end{figure}

We note that, while the mesh-closure estimate~\eqref{dp:meshclosure} is true,
a stepwise variant
\begin{align}\label{dp:counterexample1d}
 \#\TT_{\ell+1}-\#\TT_\ell \le \c{nvb}\,\#\MM_\ell
 \quad\text{for all }\ell\in\N_0
\end{align}
cannot hold with an $\ell$-independent constant $\c{nvb}>0$. We refer to a 
simple counter example from~\cite{affkp13-A} which is also illustrated in
Fig.~\ref{fig:bisect1d}.
\subsection{2D newest vertex bisection for 3D BEM}

\begin{figure*}[t]
  \centering
  \psfrag{T0}{}
  \psfrag{T1}{}
  \psfrag{T2}{}
  \psfrag{T3}{}
  \psfrag{T4}{}
  \psfrag{T12}{}
  \psfrag{T34}{}
  \includegraphics[width=35mm]{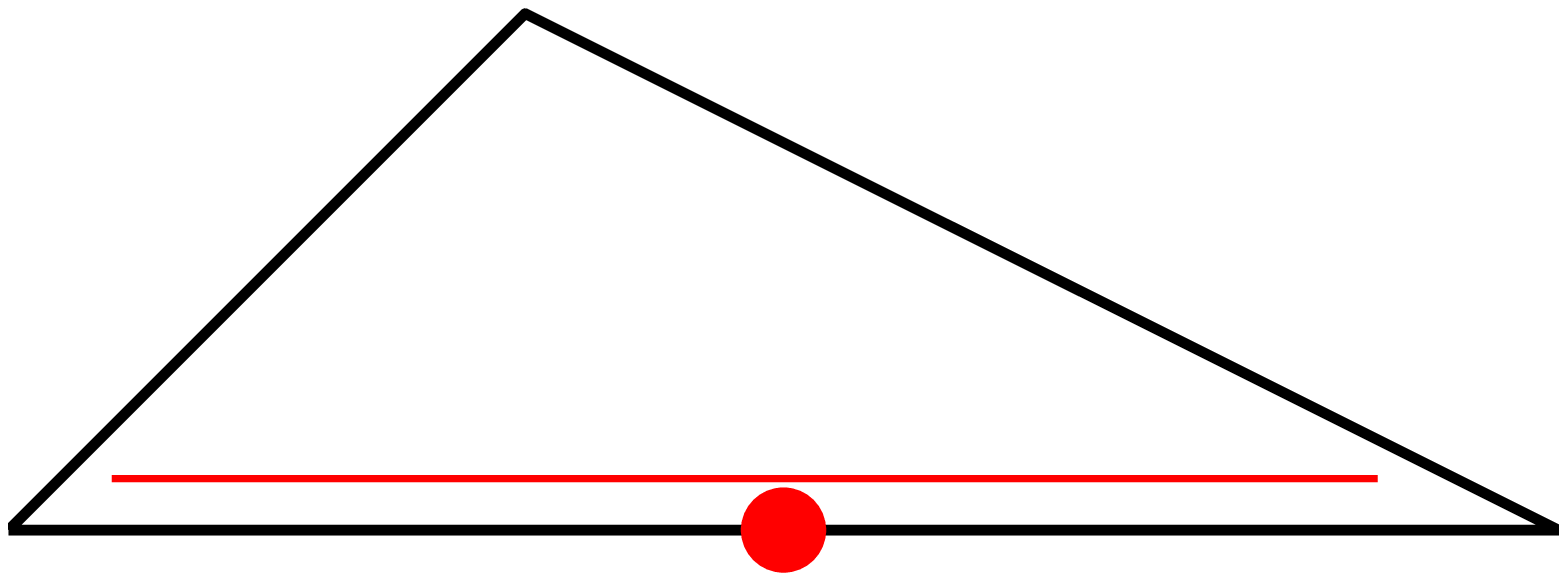} \quad
  \includegraphics[width=35mm]{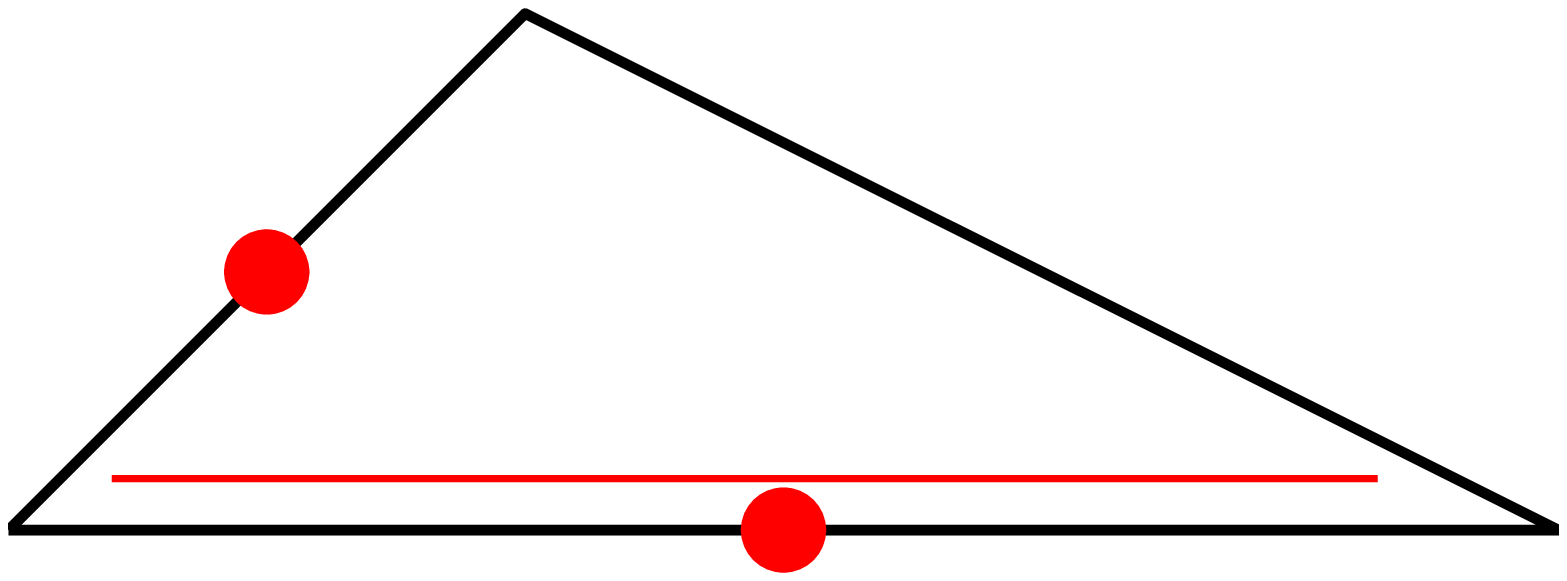} \quad
  \includegraphics[width=35mm]{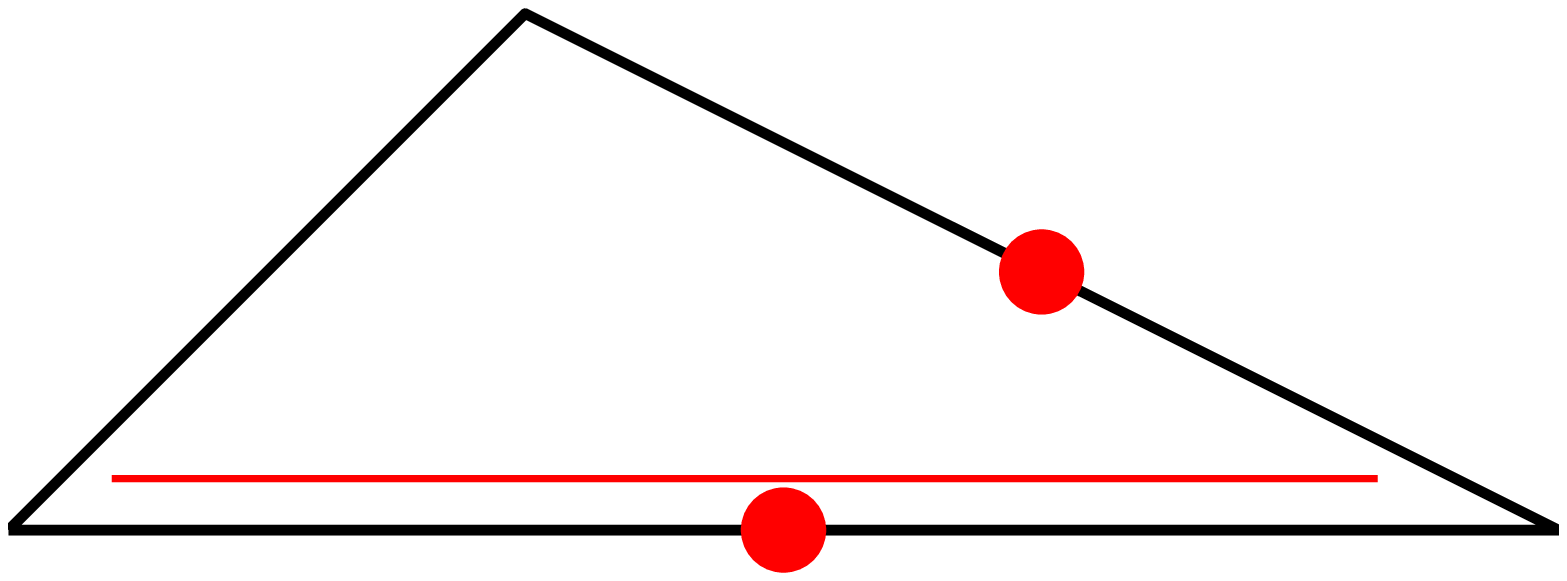} \quad
  \includegraphics[width=35mm]{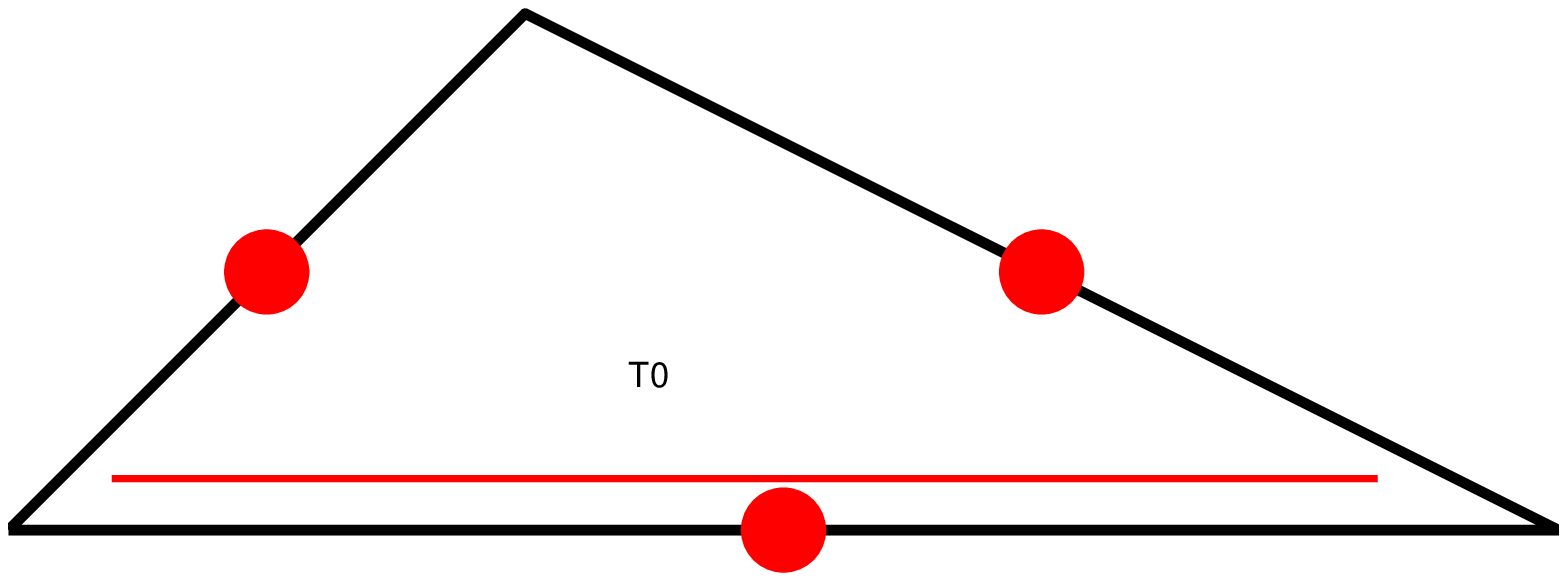} \\
  \includegraphics[width=35mm]{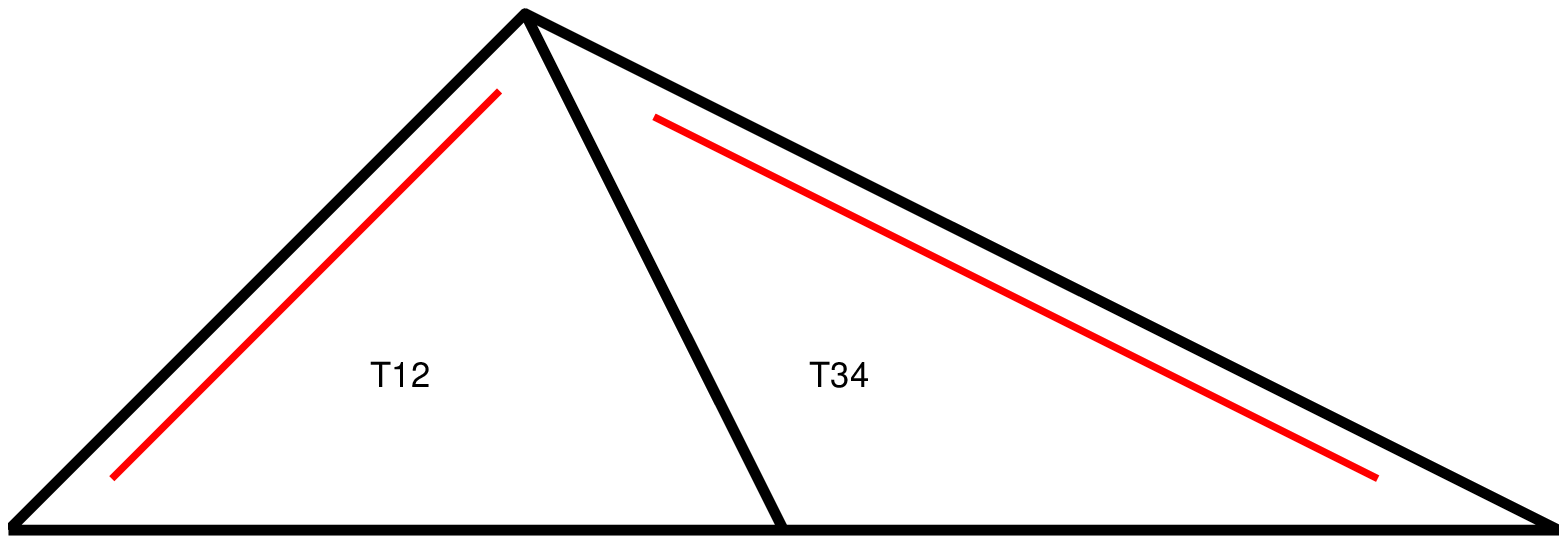} \quad
  \includegraphics[width=35mm]{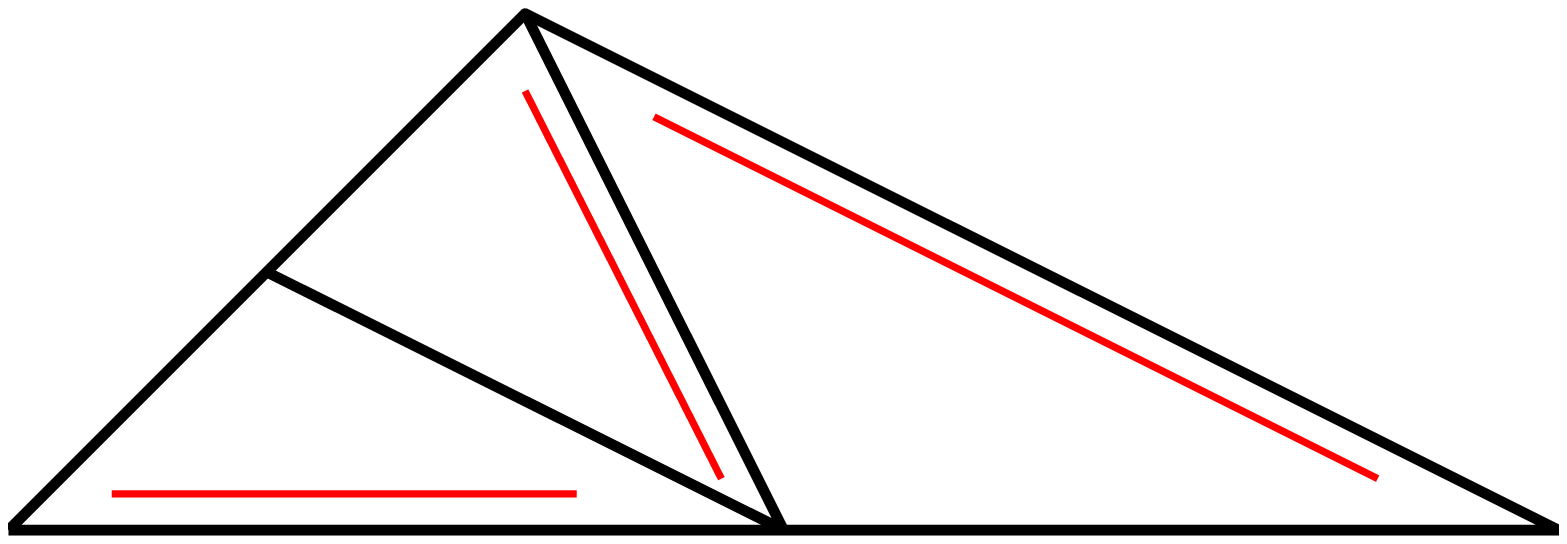}\quad
  \includegraphics[width=35mm]{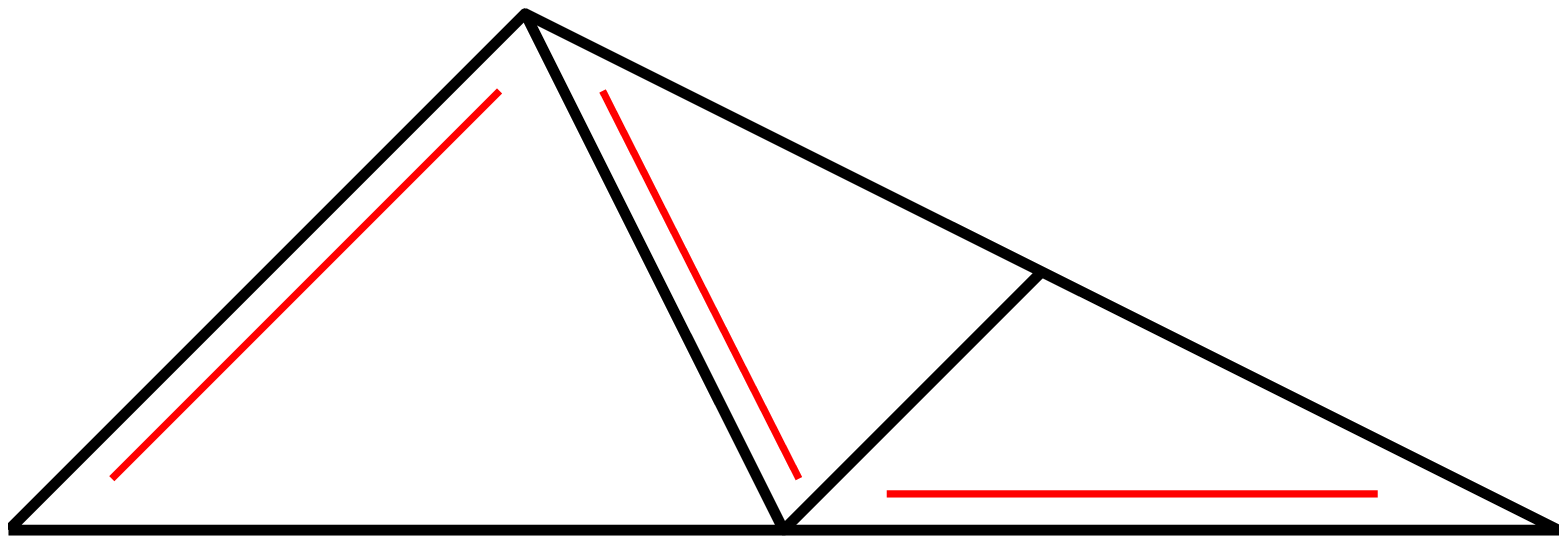}\quad
  \includegraphics[width=35mm]{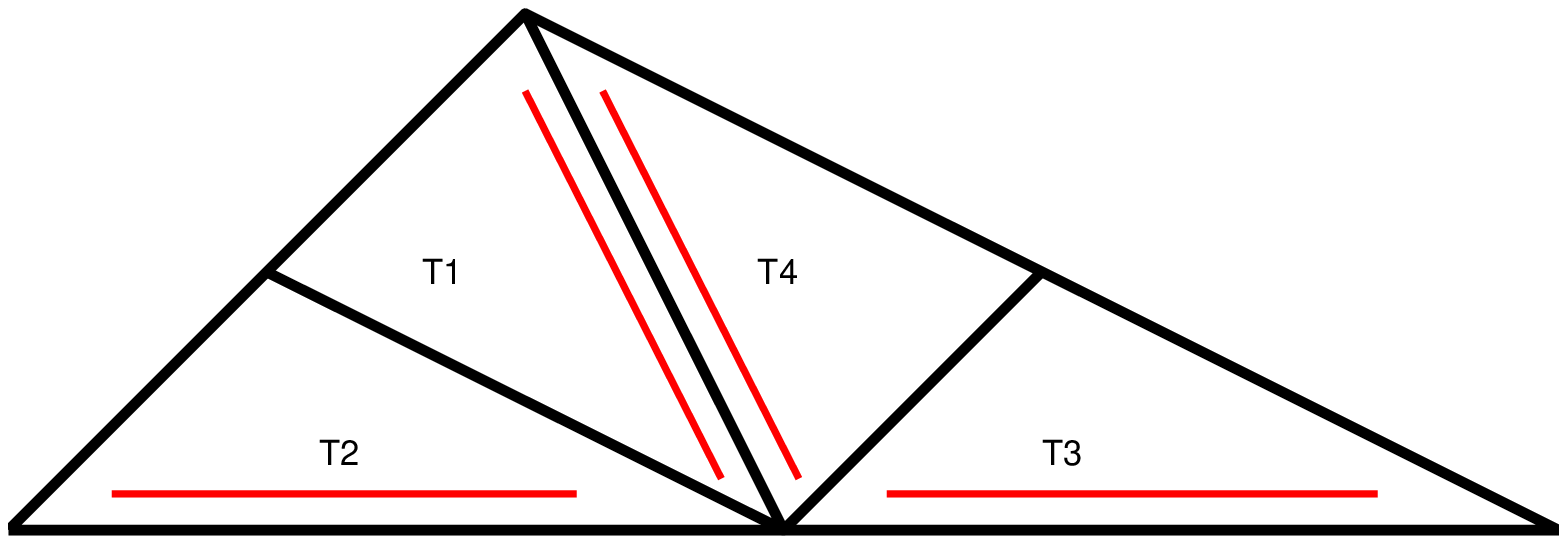}
  \caption{
  For each triangle $\el\in\mesh_\ell$, there is one fixed \emph{reference edge},
  indicated by the double line (left, top). Refinement of $\el$ is done by bisecting
  the reference edge, where its midpoint becomes a new node. The reference
  edges of the son triangles $\el'\in\mesh_{\ell+1}$ are opposite to this newest
  vertex (left, bottom).
  To avoid hanging nodes, one proceeds as follows:
  We assume that certain edges of $\el$, but at least the reference edge,
  are marked for refinement (top).
  Using iterated newest vertex bisection, the element is then split into
  $2$, $3$, or $4$ son triangles (bottom).
  If all elements are refined by three bisections (right, bottom), we obtain the
  so-called uniform bisec(3)-refinement which is denoted by $\widehat\mesh_\ell$.
  }
  \label{fig:nvb}
\end{figure*}

\begin{figure*}
 \centering
 \includegraphics[width=35mm]{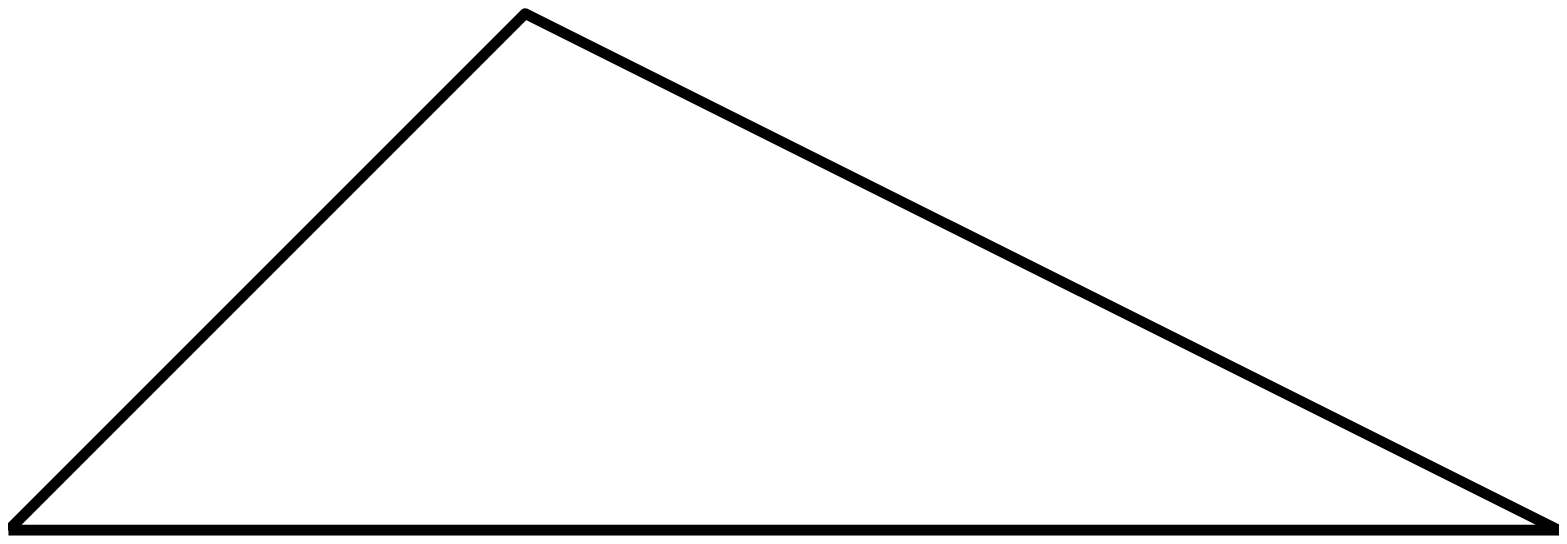}\quad
 \includegraphics[width=35mm]{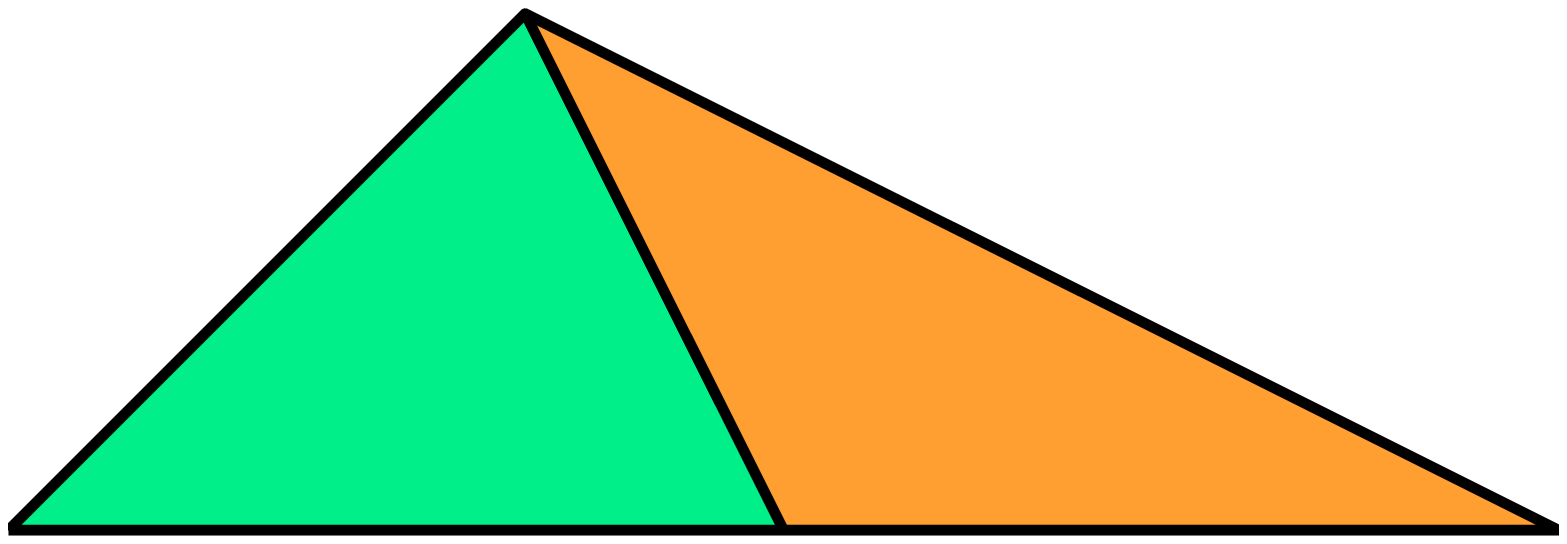}\quad
 \includegraphics[width=35mm]{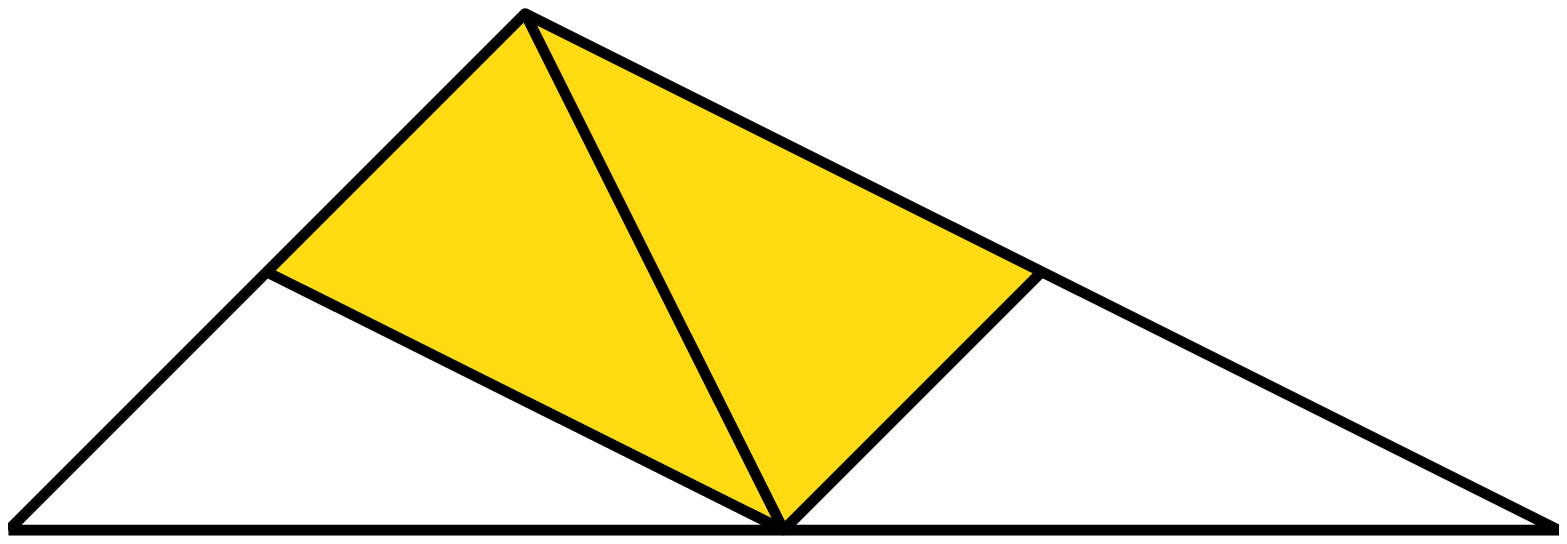}\quad
 \includegraphics[width=35mm]{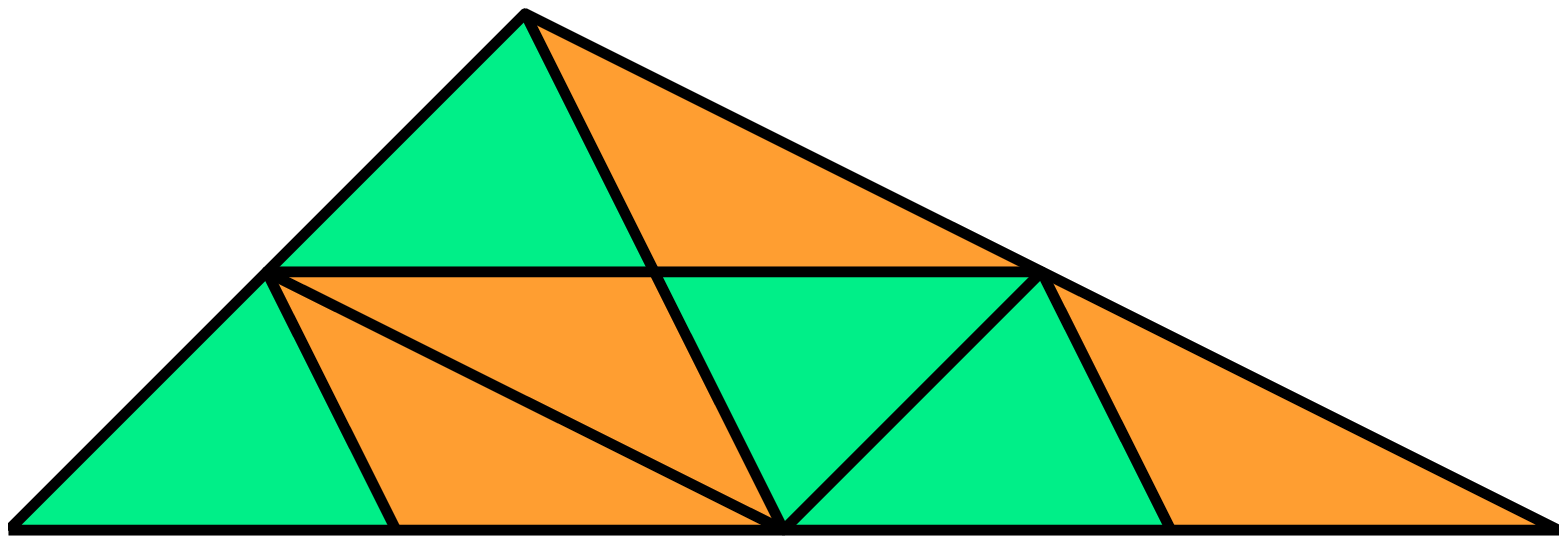}
 \caption{NVB refinement only leads to
 finitely many shapes of triangles for the family of all possible triangulations
 obtained by arbitrary newest vertex bisections. To see this,
 we start from a macro element (left), where the bottom edge is the
 reference edge. Using iterated newest vertex bisection, one observes
 that only four similarity classes of triangles occur, which
 are indicated by the coloring. After three steps of bisections
 (right), no additional similarity class appears.}
 \label{fig:nvb:equi}
\end{figure*}

Denote by $\mesh_\ell$ a given mesh and by $\edges_\ell$ its edges. Suppose that
for each triangle $\el\in\mesh_\ell$, there is a so-called \textit{reference edge}
$\edge_\el\in\edges_\ell$ with $\edge_\el \subset \partial\overline\el$.
To refine a specific element $\el\in\mesh_\ell$, the midpoint $m_\edge$ of its reference edge
$\edge_\el$ becomes a new node, and $\el$ is bisected along $m_\edge$ and the node opposite to
$\edge_\el$ into its two sons, see Fig.~\ref{fig:nvb} (left). The edges opposite to $m_\edge$
become the reference edges of the two sons of $\el$.

If the marked elements $\MM_\ell$ of a mesh are refined according to this rule, the new mesh automatically
inherits a distribution of reference edges. Hence, only the initial mesh $\mesh_0$, with which
an adaptive algorithm would be initialized, needs to be equipped with a distribution of reference
edges.

In order to keep the mesh conforming, i.e., to avoid hanging nodes, different approaches
are available: Sewell~\cite{sewell} proposes to bisect $\el$ either if its reference edge
is on the boundary, or if it is \textit{compatibly divisible}, i.e., the neighbor $\el'$ on the other
side of the reference edge $\edge_\el$ of $\el$ also uses the common edge as reference edge.
This approach was refined by Mitchell in~\cite{mitchell91}, who proposes to recursively call
the bisection algorithm on the neighbor $\el'$ of $\el$ until a compatibly divisible element is found.
This approach is reasonable under certain conditions, however, situations exists where the recursion
per se cannot terminate, cf.~\cite{kossaczky}. However, if all elements in $\mesh_0$
are compatibly divisible, the recursion terminates on every following mesh $\mesh_\ell$.
Distributing the reference edges on $\mesh_0$ this way, i.e., all elements end up being compatibly
divisible, is always possible as proven in~\cite{bdd04}. However, no scalable algorithm is known
which performs this task. To circumvent this problem, as shown in~\cite{kpp13},
it is possible to cast the 2D NVB into an iterative algorithm,
which does not need a special distribution of the reference edges to terminate:

\begin{algorithm}[Iterative formulation of 2D NVB]\label{alg:nvb}\ \linebreak
    \textsc{Input}: mesh $\mesh_\ell$ and
  set of marked elements $\MM_\ell\subseteq\mesh_\ell$.\\
  \textsc{Output}: refined mesh $\mesh_{\ell+1}$.
  \begin{itemize}
    \item[\rm(o)] Set counter $k:=0$ and define set of marked reference edges
      $\MM_\ell^{(0)}:= \left\{ \edge_\el \mid \el\in\MM_\ell \right\}$.
    \item[\rm(i)] Define
      $\MM_\ell^{(k+1)} := \left\{ e_\el \mid
      \exists\edge\in\MM_{\ell}^{(k)} \text{ with } \edge\subset\partial\el
      \right\}$.
    \item[\rm(ii)] If $\MM_\ell^{(k)}\subsetneqq\MM_\ell^{(k+1)}$, increase counter $k\mapsto k+1$ and goto (i).
    \item[\rm(iii)] Otherwise and with $\MM^{(k)}_\ell$ being the set of marked edges,
	use newest vertex bisection to refine all elements $\el\in\mesh_\ell$ with $\edge_\el\in\MM_\ell^{(k)}$ according to Fig.~\ref{fig:nvb}
	to obtain the new mesh $\mesh_{\ell+1}$.
    \end{itemize}
\end{algorithm}

The following proposition collects the elementary properties of NVB, and we
also refer to Fig.~\ref{fig:nvb:equi}.

\begin{myproposition}\label{mesh:prop:nvb}
Algorithm~\ref{alg:nvb} terminates regardless of the distribution of reference edges. The output
$\mesh_{\ell+1} = \refine(\linebreak \mesh_\ell,\MM_\ell)$ is the coarsest conforming mesh such that all elements in $\MM_\ell$ are refined.
Moreover, NVB ensures that only finitely many shapes of elements (and hence also patches) may occur. In particular, NVB generated meshes are uniformly 
$\sigma$-shape regular, cf.~Section~\ref{section:bem:discrete}, where $\sigma>0$ depends only on
the initial mesh $\TT_0$.
\end{myproposition}
\subsubsection{Mesh closure and overlay estimate}
Another look back to Algorithm~\ref{alg:nvb} reveals that, while only advised to refine elements of
$\MM_\ell$, it refines all the elements $T\in\TT_\ell$ with $e_T\in\MM_\ell^{(k)}$. It does this to circumvent the generation
of hanging nodes. We refer to a counter example in~\cite[p. 462]{nsv} that,
as in 1D, an elementary estimate of the type
\begin{align*}
  \#\mesh_{\ell+1} - \#\mesh_\ell \le \c{nvb}\, \#\MM_\ell
  \quad\text{for all }\ell\in\N_0
\end{align*}
cannot hold with an $\ell$-independent constant $\c{nvb}>0$. The following 
theorem is proved in~\cite{kpp13} for 2D NVB.

\begin{theorem}\label{thm:nvb:closure}
  Suppose that $\mesh_0$ is a mesh on $\Gamma$ with an arbitrary distribution of reference edges,
  and $\left( \mesh_\ell \right)_{\ell\in\N_0}$ is generated by
  Algorithm~\ref{alg:nvb}, i.e., for all $\ell\in\N_0$ holds
  \begin{align*}
    \mesh_{\ell+1} = \refine\left( \mesh_\ell, \MM_\ell \right).
  \end{align*}
  Then, the mesh-closure estimate~\eqref{dp:meshclosure} is valid, and the 
  constant $\c{nvb}>0$ depends only on $\TT_0$.
\end{theorem}

Theorem~\ref{thm:nvb:closure} was first proved in~\cite{bdd04} for $d=2$, under the additional assumption that the
distribution of reference edges in $\mesh_0$ is such that all elements are compatibly divisible.
In~\cite{steve08}, the theorem was extended to $d\geq 2$, and in~\cite{demlow} it was shown to hold also if
additional refinements are made to keep the mesh mildly graded. 
The work~\cite{kpp13} finally removed the assumption on the special distribution 
of the reference edges in $d=2$.

The following theorem is proved in~\cite{bdd04} for $d=2$ and~\cite{ckns} for 
$d\ge3$. To guarantee termination of their recursive formulations of the NVB algorithm, these works require that
the distribution of reference edges in $\mesh_0$ is such that all elements are compatibly divisible. However, their proofs of the overlay estimate~\eqref{dp:overlay}
do not use this assumption and also apply to the inductive formulation
of 2D NVB from~\cite{kpp13}.

\begin{theorem}\label{thm:nvb:overlay}
  Suppose that $\mesh_0$ is a mesh on $\Gamma$ with an arbitrary distribution of reference edges
  and that $\left( \mesh_\ell \right)_{\ell\in\N_0}$ is generated by
  Algorithm~\ref{alg:nvb}, i.e., for all $\ell\in\N_0$ holds
  \begin{align*}
    \mesh_{\ell+1} = \refine\left( \mesh_\ell, \MM_\ell \right).
  \end{align*}
  Then, the overlay estimate~\eqref{dp:overlay} is valid.
  Moreover, the coarsest common refinement $\TT_\star\oplus\TT_\ell$ of $\TT_\ell,\TT_\star\in\refine(\TT_0)$ is the overlay~\eqref{dp:def:overlay}.
\end{theorem}

\subsubsection{$H^s$ stability of the $L_2$ projection}\label{section:L2projection:H1stability}
If the $L_2(\Gamma)$ projection onto $\Sp^p(\mesh)$ (or $\wilde\Sp^p(\mesh)$)
is used for localization of a fractional
order Sobolev norm, it needs to fulfill the approximation estimate from Lemma~\ref{lem:Sp:apx}.
According to this Lemma, stability in $H^s(\Gamma)$ (or $\wilde H^s(\Gamma)$) is
a sufficient condition, and it is seen easily that it is also necessary.
Indeed, choosing $s=1$ in Lemma~\ref{lem:Sp:apx}, it holds
\begin{align*}
  \norm{\Pi_\mesh^p v}{H^1(\Gamma)} &\leq \norm{v}{H^1(\Gamma)} + \norm{v - \Pi_\mesh^p v}{H^1(\Gamma)}\\
  &\lesssim (1+\diam(\Gamma)) \norm{v}{H^1(\Gamma)}.
\end{align*}
By deeper mathematical results, it follows from this estimate that $\Pi_\mesh^p$ is
$H^s(\Gamma)$ stable. Hence there is no other way than analyzing $\Pi_\mesh^p$'s
stability in $H^1$.
For quasi-uniform meshes, it follows with arguments from~\cite{bx91} that
$\Pi_\mesh^p$ is $H^1(\Gamma)$-stable.
In fact, with an arbitrary $H^1(\Gamma)$-stable Cl\'ement-type operator $J_\mesh$
(e.g., the Scott-Zhang projection from~\ref{section:sz}), it follows 
with the inverse estimate from Lemma~\ref{lem:Sp:invest} that
\begin{align*}
  &\norm{\nabla\Pi_\mesh^p v}{L_2(\Gamma)} \leq \norm{\nabla(\Pi_\mesh^p - J_\mesh) v}{L_2(\Gamma)} +
  \norm{\nabla J_\mesh v}{L_2(\Gamma)}\\
  &\quad\lesssim \norm{h_\mesh^{-1}}{L_\infty(\Gamma)} \norm{(\Pi_\mesh^p - J_\mesh) v}{L_2(\Gamma)} + 
  \norm{\nabla J_\mesh v}{L_2(\Gamma)},
\end{align*}
and due to the projection property of $\Pi_\mesh^p$ and its $L_2(\Gamma)$-stability,
\begin{align*}
  \norm{(\Pi_\mesh^p - J_\mesh) v}{L_2(\Gamma)}
  &\leq \norm{(1 - J_\mesh) v}{L_2(\Gamma)}\\
  &\lesssim \norm{h_\mesh}{L_\infty(\Gamma)}\norm{v}{H^1(\Gamma)},
\end{align*}
where we have finally used the first-order approximation property of $J_\mesh$.
Combining these two estimates and regarding the fact that for quasi-uniform meshes
\begin{align*}
  \norm{h_\mesh^{-1}}{L_\infty(\Gamma)}\norm{h_\mesh}{L_\infty(\Gamma)} \lesssim 1,
\end{align*}
the $H^1(\Gamma)$-stability of $\Pi_\mesh^p$ follows.

Unfortunately, this argument cannot be used in this straight forward manner
on adaptively refined meshes. The $H^s(\Gamma)$-stability
can be shown, though, under certain conditions on the mesh. There are basically two approaches:
\begin{itemize}
  \item Imposing global or local growth-conditions on the mesh. This approach is used in the works
    \cite{bps02,c02,ct87,ej95,s01,s02}.
  \item Using only a sequence of adaptively generated meshes such that an equivalent mesh-size function
    can be used, which takes care of the fact that the mesh is not quasi-uniform.
    This approach is used in the works~\cite{by13,c04,kpp13,kpp13A}.
\end{itemize}
The strength of the first approach is that it can be used for an \textit{arbitrary} sequence of
meshes which does not have to be the output of an adaptive mesh refinement strategy. However,
certain growth-conditions may be too restrictive if already the coarsest mesh violates them.
Therefore, the second approach will yield more general results
when it comes to adaptive mesh refinement.
\begin{theorem}\label{thm:L2projection:cc}
  Suppose that $\mesh_0$ is a mesh on $\Gamma$ with an arbitrary distribution of reference edges.
  Then, if $\left( \mesh_\ell \right)_{\ell\in\N_0}$ is generated by
  Algorithm~\ref{alg:nvb}, the sequence of $L_2(\Gamma)$-ortho\-gonal projections $\Pi_\ell$ onto
  $\Sp^1(\mesh_\ell)$ is uniformly $H^1(\Gamma)$-stable, i.e., for all $\ell\in\N_0$ holds
  \begin{align*}
    \norm{\Pi_\ell v}{H^1(\Gamma)} \leq \c{stab}\norm{v}{H^1(\Gamma)},\quad
    \text{ for all } v\in H^1(\Gamma),
  \end{align*}
  and the constant $\c{stab}$ depends only on $\mesh_0$.
  The same result holds for the $L_2(\Gamma)$-projection onto $\wilde\Sp^p(\mesh)$.
\end{theorem}

  The first proof of this type of result is due to Carsten\-sen~\cite{c04}. In the latter work, the distribution
  of reference edges is supposed to fulfill an additional assumption, and instead of 2D NVB, a modified 2D red-green-blue
  mesh refinement strategy is considered. In the work~\cite{kpp13}, the assumptions on the initial distribution
  of reference edges has been removed, and 2D NVB was considered as underlying refinement strategy. In~\cite{kpp13A} the analysis of~\cite{kpp13} has been
  generalized to NVB in arbitrary dimension $d\ge3$.

The last theorem can be employed in lowest-order\linebreak Galerkin boundary element methods,
but higher-order methods require the $H^1(\Gamma)$-stability of the $L_2(\Gamma)$-projection
onto $\Sp^p(\mesh)$. Results of this kind have been shown by Bank and Yserentant in~\cite{by13}.
To state their result, the concept of the so-called
\textit{level-function} $\gen_\ell:\mesh_\ell\rightarrow \N_0$ has to be introduced,
which measures the number of bisections needed to create a specific element.
For all $\el\in\mesh_0$, define $\gen_0(\el):=0$. Then, the two sons $\el_1$ and $\el_2$ of an element
$\el$ that arise due to a bisection, see Fig.~\ref{fig:nvb} (left),
have level $\gen_{\ell+1}(\el_1)= \gen_{\ell+1}(\el_2) := \gen_\ell(\el)+1$.
The following theorem is the main result of~\cite{by13}.

\begin{theorem}\label{thm:L2projection:by}
  Assume that for the sequence of meshes\linebreak
  $\left( \mesh_\ell \right)_{\ell\in\N_0}$ holds
  \begin{align*}
    \abs{\gen_\ell(\el) - \gen_\ell(\el')} \leq 1 \quad\text{ if }\overline\el\cap\overline{\el'}\neq 0.
  \end{align*}
  Then, the sequence $\left( \Pi_\ell^p \right)_{\ell\in\N_0}$ of $L_2(\Gamma)$-orthogonal projections onto
  $\Sp^p(\mesh_\ell)$ is uniformly $H^1(\Gamma)$ stable for $p\leq 12$ in $d=2$ and for $p\leq 7$ in $d=3$.
\end{theorem}

The assumption on the level function in Theorem~\ref{thm:L2projection:by} suggests that also elements
sharing a vertex need to have a difference in their level-functions of at most $1$.
This assumption needs to be enforced via additional refinements, as suggested in~\cite{by13}.
\section{Optimal convergence of adaptive BEM}\label{section:convergence}
Whereas plain convergence of error estimators was the concern of Section~\ref{section:estred},
this section deals with convergence of Algorithm~\ref{opt:algorithm} even with optimal rates.
The first result on convergence rates~\cite{bdd04} considered AFEM for the 2D Poisson model problem
and required an additional coarsening step which has later been
proved to be unnecessary~\cite{steve07}. The latter work introduced the 
assumption that the set of marked elements in Step~(iii) of Algorithm~\ref{opt:algorithm} has minimal
cardinality and proved that the marking criterion~\eqref{opt:bulkchasing} is (in some sense) even
necessary (see Lemma~\ref{opt:lem:bulkchasing} below).

For ABEM, optimal convergence rates for weighted-resi\-dual
error estimators have independently first been proved in~\cite{fkmp13,gantumur}.
While~\cite{fkmp13} considers ABEM for the 3D Laplacian on polyhedral domains,
\cite{gantumur} considers ABEM for general operators, but the analysis requires
smooth boundaries.

The goal of this section is to explain the concept of convergence with optimal rates and
to provide abstract results which cover the BEM model problems from
Section~\ref{opt:sec:example1}--\ref{opt:sec:example4}. Since most of the analysis can be done in
an abstract mathematical setting, we stick with the frame and the notation of the Lax-Milgram
lemma from Section~\ref{section:introduction}.

\subsection{Necessary approximation property}\label{section:opt:approx}
We suppose that the discrete spaces $\XX_\ell$ are nested in the sense that
$\XX_\ell\subseteq \XX_{\star}\subseteq \XX$ if $\mesh_{\star}\in \refine(\mesh_\ell)$ and that
they satisfy the following approximation property: For all $\mesh_\ell\in\refine(\mesh_0)$ and all
$\eps>0$, there exists $\mesh_\star\in\refine(\mesh_\ell)$ such that
\begin{align}\label{opt:eq:aproxprop}
  \norm{u-U_\star}{\XX}\leq \eps,
\end{align}
where $U_\star \in\XX_\star$ is the solution of~\eqref{intro:galerkin}  on the mesh $\mesh_\star$.
\begin{remark}
  Although one could theoretically construct spac\-es $\XX$, where assumption~\eqref{opt:eq:aproxprop}
  is violated, the authors are unaware of any practical example, where this is the case.
  In practice,~\eqref{opt:eq:aproxprop} is satisfied for $\mesh_\star$ being a sufficiently fine
  uniform refinement of $\mesh_\ell$.
\end{remark}

\subsubsection{Assumptions on the adaptive algorithm}\label{section:opt:adaptive}
Algorithm~\ref{opt:algorithm} needs to be modified to allow for convergence with optimal rates.
In contrast to Section~\ref{section:estred}, we now have to ensure that
we choose a set of minimal cardinality $\MM_\ell$ in Step~(iii) of Algorithm~\ref{opt:algorithm}
(see also Remark~\ref{rem:estred:algorithm} for details on the realization).
In step~(iv) of Algorithm~\ref{opt:algorithm}, we suppose that $\mesh_{\ell+1}$ is the coarsest
refinement of $\mesh_\ell$ such that all marked elements $\el\in\MM_\ell$ have been refined,
written $\mesh_{\ell+1}=\refine(\mesh_\ell,\MM_\ell)$. Finally, we suppose that the mesh refinement
strategy used satisfies the properties of Section~\ref{section:meshopt} plus the fact that it
produces only finitely many shapes of element patches. These assumptions are, for example, satisfied
for the bisection strategies discussed in Section~\ref{section:meshrefinement}.

\subsubsection{Assumptions on the error estimator}\label{opt:section:assumptions}
The following four assumptions on the error estimator $\est{\ell}{}$ are first found in~\cite{axioms}
and distilled from the literature on AFEM~\cite{dirichlet3d,ckns,nonsymm,dirichlet2d,steve07}
and ABEM~\cite{affkp13-A,ffkmp13,ffkmp13-A,fkmp13,gantumur}.
We will use these assumptions to prove the main results on convergence
(Theorem~\ref{opt:thm:convergence}) and optimal rates (Theorem~\ref{opt:thm:optimality}).
In Section~\ref{opt:sec:example1}--\ref{opt:sec:example4} below, these assumptions will
be verified for concrete model problems.

\begin{enumerate}
  \renewcommand{\theenumi}{{\rm A\arabic{enumi}}}%
  \renewcommand{\labelenumi}{({\rm A\arabic{enumi}})}%
  \item\label{opt:ass:stable} Stability on non-refined elements: There exists a constant
    $\setc{opt:stable}>0$ such that any refinement $\mesh_\star\in \refine(\mesh_\ell)$
    of $\mesh_\ell\in\refine(\mesh_0)$ satisfies
    \begin{align*}
      \Big| \Big(\sum_{\el\in\mesh_\ell\cap \mesh_\star} \est{\ell}{\el}^2\Big)^{1/2} &-
      \Big(\sum_{\el\in\mesh_\ell\cap \mesh_\star} \est{\star}{\el}^2\Big)^{1/2}\Big| \\
      &\leq \c{opt:stable}\norm{U_\star-U_\ell}{\XX}.
    \end{align*}

  \item\label{opt:ass:reduction} Reduction on refined elements: There exist constants
    $\setc{opt:reduction}>0$ and $0<q_{\rm red}< 1$ such that any refinement
    $\mesh_\star\in\linebreak \refine(\mesh_\ell)$  of $\mesh_\ell\in\refine(\mesh_0)$ satisfies
    \begin{align*}
      \sum_{\el\in\mesh_\star\setminus \mesh_\ell} \est{\star}{\el}^2 \leq q_{\rm red}
      \sum_{\el\in\mesh_\ell\setminus \mesh_\star} \est{\ell}{\el}^2
      + \c{opt:reduction}\norm{U_\star-U_\ell}{\XX}^2.
    \end{align*}

  \item\label{opt:ass:reliable} Reliability: There exists a constant $\setc{opt:reliable}>0$
    such that any mesh $\mesh_\ell\in\refine(\mesh_0)$ satisfies
    \begin{align*}
      \norm{u-U_\ell}{\XX}\leq \c{opt:reliable} \est{\ell}{}.
    \end{align*}

  \item\label{opt:ass:drel} Discrete reliability: There exist constants $\setc{opt:drel}>0$
    and $\setc{opt:refined}>0$ such that any refinement $\mesh_\star\in \refine(\mesh_\ell)$
    of $\mesh_\ell\in\refine(\mesh_0)$ satisfies
    \begin{align*}
      \norm{U_\star-U_\ell}{\XX}\leq \c{opt:drel}\Big(\sum_{\el\in\RR{\ell}{\star}}
      \est{\ell}{\el}^2\Big)^{1/2},
    \end{align*}
    where the set $\RR{\ell}{\star}\supseteq \mesh_\ell\setminus\mesh_\star$ satisfies\linebreak
    $\#\RR{\ell}{\star}\leq \c{opt:refined} \#(\mesh_\ell\setminus\mesh_\star)$.
\end{enumerate}

\begin{remark}
  The assumptions~\eqref{opt:ass:stable}--\eqref{opt:ass:reduction} are fulfilled by many error
  estimators with $h$-weighting factor, e.g., weighted residual error-estimators, $(h-h/2)$-type
  error estimator, and ZZ-type error estimators as shown implicitly in Section~\ref{section:estred}.
  They form the main ingredients of the abstract proof of estimator reduction~\eqref{estred:eq:estred}.
  Discrete reliability~\eqref{opt:ass:drel} is stronger than reliability~\eqref{opt:ass:reliable}
  as is shown in the following. So far, it is only proved for the weighted residual error estimator
  with $\RR{\ell}{\star}$ being the refined elements plus one additional element layer around them.
\end{remark}

\begin{mylemma}
  The discrete reliability~\eqref{opt:ass:drel} implies the reliability~\eqref{opt:ass:reliable}
  with $\c{opt:reliable}=\c{opt:drel}$.
\end{mylemma}

\begin{proof}
  Suppose~\eqref{opt:ass:drel} and use the approximation property~\eqref{opt:eq:aproxprop}
  for $\eps>0$ to see
  \begin{align*}
    \norm{u-U_\ell}{\XX}&\leq \norm{u-U_\star}{\XX}+ \norm{U_\star-U_\ell}{\XX}\\
    &\leq \eps + \c{opt:drel}\Big(\sum_{\el\in\RR{\ell}{\star}} \est{\ell}{\el}^2\Big)^{1/2}\leq
    \eps + \c{opt:drel}\est{\ell}{}.
  \end{align*}
  Since $\eps>0$ is arbitrary, this implies~\eqref{opt:ass:reliable} with
  $\c{opt:reliable}=\c{opt:drel}$.
  $\hfill\qed$
\end{proof}

\subsection{Convergence of error estimator and error}\label{opt:section:convergence}

\begin{theorem}\label{opt:thm:convergence}
  Suppose that the error estimator satisfies stability~\eqref{opt:ass:stable} and
  reduction~\eqref{opt:ass:reduction}. Then, Algorithm~\ref{opt:algorithm} drives the
  estimator to zero, i.e.
  \begin{align}\label{opt:eq:convergence}
    \lim_{\ell\to\infty} \est{\ell}{} = 0.
  \end{align}
  Suppose that the error estimator additionally satisfies reliability~\eqref{opt:ass:reliable}.
  Then, Algorithm~\ref{opt:algorithm} converges even $R$-linearly in the sense that there exist
  constants $\setc{opt:Rlin}>0$ and $0<q_{\rm R}<1$ such that
  \begin{align}\label{opt:eq:Rlin}
    \c{opt:reliable}^{-2}\norm{u\!-\!U_{\ell+n}}{\XX}^2\leq  \est{\ell+n}{}^2\leq
    \c{opt:Rlin}q_{\rm R}^n \est{\ell}{}^2\;\text{for all }\ell,n\in\N_0,
  \end{align}
  which particularly implies
  \begin{align}
    \norm{u-U_{\ell}}{\XX}^2\leq \c{opt:reliable}^2\c{opt:Rlin}\est{0}{}^2\,q_{\rm R}^\ell
    \quad\text{for all }\ell\in\N_0.
  \end{align}
  The constants $\c{opt:Rlin}$ and $q_{\rm R}$ depend only on $\theta$ as well as on the constants
  in~\eqref{opt:ass:stable}--\eqref{opt:ass:reliable}.
\end{theorem}

The proof of the above theorem is split into several steps. The first lemma
states that the assumptions~\eqref{opt:ass:stable} and~\eqref{opt:ass:reduction} imply the estimator
reduction~\eqref{estred:eq:estred} and thus render an abstract version of the results in
Lemmas~\ref{estred:lem:weaksing}--\ref{estred:lem:hypsingres}.
Such an estimate is first but implicitly found in~\cite{ckns}.
Together with the a~priori convergence of Lemma~\ref{estred:lem:convortho}, it proves that the
adaptive algorithm drives the estimator to zero~\eqref{opt:eq:convergence}. This so-called
\emph{estimator reduction principle} is discussed in Section~\ref{section:estred} and has been
proposed in~\cite{afp12} for $(h-h/2)$-type estimators and was generalized afterwards to data-perturbed 
BEM~\cite{afgkmp12,kop}, residual-based error estimators~\cite{fkmp13}, as well as the FEM-BEM
coupling~\cite{afpfembem,fembem}.
Instead of the respective concrete settings, the following proof relies only
on the abstract assumptions~\eqref{opt:ass:stable}--\eqref{opt:ass:reduction}
as well as the marking criterion~\eqref{opt:bulkchasing}.

\begin{mylemma}\label{opt:lem:estred}
  Suppose~\eqref{opt:ass:stable}--\eqref{opt:ass:reduction}. Then, there exist constants
  $\setc{opt:estred}>0$ and $0<q_{\rm est}<1$ such that Algorithm~\ref{opt:algorithm} satisfies
  \begin{align}\label{opt:eq:estred}
    \est{\ell+1}{}^2\leq q_{\rm est} \est{\ell}{}^2 + \c{opt:estred}\norm{U_{\ell+1}-U_\ell}{\XX}^2
  \end{align}
  for all $\ell\in\N$. The constant $q_{\rm est}$ depends only on $\theta$ and $q_{\rm red}$
  from~\eqref{opt:ass:reduction}. The constant $\c{opt:estred}$ depends additionally on
  $\c{opt:reduction}$ as well as $\c{opt:stable}$.
\end{mylemma}
\begin{proof}
  We split the error estimator into refined and non-refined elements and
  apply~\eqref{opt:ass:stable}--\eqref{opt:ass:reduction}. With Young's inequality
  $(a+b)^2\leq (1+\delta)a^2+ (1+\delta^{-1})b^2$ for all $a,b\in\R$ and all $\delta>0$, this shows
  \begin{align*}
    \est{\ell+1}{}^2&=\sum_{\el\in\mesh_{\ell+1}\setminus\mesh_\ell}\est{\ell+1}{\el}^2 +
    \sum_{\el\in\mesh_{\ell+1}\cap\mesh_\ell}\est{\ell+1}{\el}^2\\
    &\leq q_{\rm red}\sum_{\el\in\mesh_{\ell}\setminus\mesh_{\ell+1}}\est{\ell}{\el}^2 +
    (1+\delta)\sum_{\el\in\mesh_{\ell}\cap\mesh_{\ell+1}}\est{\ell}{\el}^2\\
    &\qquad +(\c{opt:reduction}+(1+\delta^{-1})\c{opt:stable}^2)\norm{U_{\ell+1}-U_\ell}{\XX}^2.
  \end{align*}
  The bulk chasing~\eqref{opt:bulkchasing} together with
  $\MM_\ell\subseteq \mesh_{\ell+1}\setminus\mesh_\ell$ implies
  \begin{align*}
    \est{\ell+1}{}^2&\leq (q_{\rm red}-(1+\delta))
    \sum_{\el\in\mesh_{\ell}\setminus\mesh_{\ell+1}}\est{\ell}{\el}^2 + (1+\delta)\est{\ell}{}^2\\
    &\qquad +(\c{opt:reduction}+(1+\delta^{-1})\c{opt:stable}^2)\norm{U_{\ell+1}-U_\ell}{\XX}^2\\
    &\leq ((1+\delta)-\theta((1+\delta)- q_{\rm red}))\est{\ell}{}^2 +
    \c{opt:estred}\norm{U_{\ell+1}-U_\ell}{\XX}^2,
  \end{align*}
  where $\c{opt:estred}:=(\c{opt:reduction}+(1+\delta^{-1})\c{opt:stable}^2)$ and
  $q_{\rm est}:=(1+\delta)-\theta((1+\delta)- q_{\rm red})\in (0,1)$ for $\delta>0$ sufficiently small.
  $\hfill\qed$
\end{proof}

With the previous results, we are able to prove the first part of Theorem~\ref{opt:thm:convergence}. 
\begin{proof}[of estimator convergence~\eqref{opt:eq:convergence}]
By assumption (Section~\ref{section:opt:approx}), the discrete spaces are nested.
Therefore, Lemma~\ref{estred:lem:convortho} proves a~priori convergence of $U_\ell$ towards
some limit $U_\infty$, and hence the perturbation term in~\eqref{opt:eq:estred} vanishes.
Consequently, Lemma~\ref{estred:lem:estreddef} applies and concludes the proof.
  $\hfill\qed$
\end{proof}

The next goal is the $R$-linear convergence~\eqref{opt:eq:Rlin}. To this end, the literature
usually employs the Pythagoras identity
\begin{align}\label{opt:eq:pyth}
  \norm{u-U_{\ell+1}}{\XX}^2+\norm{U_{\ell+1}-U_\ell}{\XX}^2=\norm{u-U_\ell}{\XX}^2,
\end{align}
see, e.g.,~\cite{ckns,fkmp13,gantumur,steve07}.
Under reliability~\eqref{opt:ass:reliable}, estimator convergence~\eqref{opt:eq:convergence} implies 
\begin{align}\label{opt:eq:convhelp}
  \lim_{\ell\to\infty}\norm{u-U_\ell}{\XX}=0.
\end{align}
Consequently,~\eqref{opt:eq:pyth} results in
\begin{align}\label{opt:eq:pyth2}
  \sum_{k=\ell}^\infty \norm{U_{k+1}-U_k}{\XX}^2 = \norm{u-U_\ell}{\XX}^2.
\end{align}
However,~\eqref{opt:eq:pyth} and thus~\eqref{opt:eq:pyth2} rely heavily on the symmetry of the
bilinear form $\form{\cdot}{\cdot}$. This may suffice for many problems, however,
when it comes to mixed boundary value problems or the FEM-BEM coupling,~\eqref{opt:eq:pyth} is
wrong in general.
The recent work~\cite{quasipyth} proves a weaker version of~\eqref{opt:eq:pyth2} which holds for
general continuous and elliptic bilinear forms $\form{\cdot}{\cdot}$ and is sufficient for the
upcoming analysis. With~\eqref{intro:continuous}--\eqref{intro:elliptic} and as the sequence
$(U_\ell)_{\ell\in\N_0}$ is convergent~\eqref{opt:eq:convhelp}, there exists a constant
$\setc{opt:qosum}>0$ such that all $\ell\in\N_0$ satisfy
\begin{align}\label{opt:eq:qosum}
  \sum_{k=\ell}^\infty \norm{U_{k+1}-U_k}{\XX}^2\leq \c{opt:qosum}\norm{u-U_\ell}{\XX}^2.
\end{align}
We note that~\eqref{opt:eq:qosum} is weaker than~\eqref{opt:eq:pyth2}, but avoids the use of the
symmetry of $\form{\cdot}{\cdot}$. This will be employed in the following proof.
\begin{proof}[of $R$-linear estimator convergence~\eqref{opt:eq:Rlin}]
  Let $N,\ell\in\N$. Use the estimator reduction~\eqref{opt:eq:estred} to see
  \begin{align*}
    \sum_{k=\ell+1}^{\ell+N} \est{k}{}^2 \leq \sum_{k=\ell+1}^{\ell+N} \Big(q_{\rm est}
    \est{k-1}{}^2 +\c{opt:estred}\norm{U_k-U_{k-1}}{\XX}^2\Big).
  \end{align*}
  This implies
  \begin{align*}
    (1-q_{\rm est})\sum_{k=\ell+1}^{\ell+N} \est{k}{}^2 \leq \est{\ell}{}^2 +\c{opt:estred}
    \sum_{k=\ell+1}^{\ell+N}\norm{U_k-U_{k-1}}{\XX}^2.
  \end{align*}
  The estimate~\eqref{opt:eq:qosum} together with reliability~\eqref{opt:ass:reliable}  then shows
  \begin{align*}
    \sum_{k=\ell}^{\ell+N} \est{k}{}^2\leq \frac{2+\c{opt:estred}\c{opt:qosum}
    \c{opt:reliable}^2}{1-q_{\rm est}}\est{\ell}{}^2.
  \end{align*}
  Since the right-hand side does not depend on $N$, there also holds with
  $\c{opt:Rlin}:=(2+\c{opt:estred}\c{opt:qosum}\c{opt:reliable}^2)(1-q_{\rm est})^{-1}\geq 1$
  \begin{align*}
    \sum_{k=\ell}^{\infty} \est{k}{}^2\leq\c{opt:Rlin}\est{\ell}{}^2.
  \end{align*}
  This result is employed several times to conclude the proof.
  Mathematical induction on $n\in\N$ shows
  \begin{align*}
    \est{\ell+n}{}^2&\leq\sum_{k=\ell+n-1}^{\infty} \est{k}{}^2-\est{\ell+n-1}{}^2\\
    &\leq (1-\c{opt:Rlin}^{-1})\sum_{k=\ell+n-1}^{\infty} \est{k}{}^2\\
    &\vdots \\
    &\leq (1-\c{opt:Rlin}^{-1})^n\sum_{k=\ell}^{\infty} \est{k}{}^2
    \leq \c{opt:Rlin}(1-\c{opt:Rlin}^{-1})^n\est{\ell}{}^2.
  \end{align*}
  This proves~\eqref{opt:eq:Rlin} with $q_{\rm R}=(1-\c{opt:Rlin}^{-1})$.
  $\hfill\qed$
\end{proof}

\subsection{Convergence with optimal rates}\label{opt:section:optimal}

Having fixed the error estimator $\est{}{}$ for Step~(ii) and the 
mesh refinement strategy for Step~(iv) of Algorithm~\ref{opt:algorithm},
the overall goal in this section is to prove algebraic convergence
rates. For $s>0$, we define the approximability norm
\begin{align}\label{opt:approxnorm}
  \norm{\est{}{}}{\A_s}:=\sup_{N\in\N_0}(N+1)^s\big(\inf_{\mesh_\star\in \refine(\mesh_0)\atop
  \#\mesh_\star -\#\mesh_0\leq N} \est{\star}{} \big)\in [0,\infty].
\end{align}
By definition, $\norm{\est{}{}}{\A_s}<\infty$ means that one can find 
a sequence $(\overline\TT_\ell)_{\ell\in\N_0}$ of meshes such that the corresponding
sequence of estimators $(\overline{\est{\ell}{}})_{\ell\in\N_0}$ satisfies
\begin{align}\label{eq:dp:convrate}
  \overline{\est{\ell}{}} 
  \lesssim (\#\overline\TT_\ell-\#\TT_0)^{-s}
  \quad\text{for all }\ell\in\N_0,
\end{align}
i.e., the estimators decay with algebraic rate $-s$. We note that the
meshes $\overline\TT_\ell$ are not necessarily nested.

The following theorem proves that for each $s>0$, 
$\norm{\est{}{}}{\A_s}<\infty$ is equivalent to~\eqref{eq:dp:convrate} 
for $\overline\TT_\ell:=\TT_\ell$ being the adaptively generated mesh,
i.e., each possible algebraic convergence rate (constrained by estimator
and mesh refinement) will in fact by achieved by Algorithm~\ref{opt:algorithm}
(which produces only nested meshes).
In particular, adaptive mesh refinement is superior to uniform mesh refinement.

\begin{theorem}\label{opt:thm:optimality}
  Suppose~\eqref{opt:ass:stable}--\eqref{opt:ass:reduction} as well as discrete
  reliability~\eqref{opt:ass:drel}. Let the adaptivity parameter satisfy
  $\theta<\theta_0:=(1+\c{opt:stable}^2\c{opt:drel}^2)^{-1}$. Then, Algorithm~\ref{opt:algorithm}
  converges with the best possible rate in the sense that for all $s>0$, it holds
  $\norm{\est{}{}}{\A_s}<\infty$ if and only if
  \begin{align}
    \est{\ell}{}\leq \c{opt:optimality} (\#\mesh_\ell-\#\mesh_0)^{-s}\quad\text{for all }\ell\in\N,
  \end{align}
  where $\setc{opt:optimality}>0$ depends only on the constants
  in~\eqref{opt:ass:stable}--\eqref{opt:ass:drel} as well as on $\norm{\est{}{}}{\A_s}$ and $\theta$.
\end{theorem}

An essential part of the proof of the above theorem is that the bulk chasing
criterion~\eqref{opt:bulkchasing} does not mark too many elements. This is the concern of the first
lemma, which states that if one observes linear convergence~\eqref{opt:eq:Rlin}, the refined elements
satisfy the bulk chasing~\eqref{opt:bulkchasing}. In this respect the marking
strategy~\eqref{opt:bulkchasing} appears to be sufficient as well as necessary for linear
convergence~\eqref{opt:eq:Rlin}. We note that at this stage the discrete
reliability~\eqref{opt:ass:drel} enters. This observation has first been proved for AFEM
in~\cite{steve07}. Unlike the AFEM literature~\cite{steve07,ckns}, our statement and proof relies
only on the error estimator and avoids the use of any efficiency estimate (or lower error bound)
for the estimator.

\begin{mylemma}\label{opt:lem:bulkchasing}
  Let the error estimator satisfy stability~\eqref{opt:ass:stable} and discrete
  reliability~\eqref{opt:ass:drel}. Then, there exists $0<\kappa_0<1$ such that any refinement
  $\mesh_\star\in\refine(\mesh_\ell)$ which satisfies
  \begin{align}\label{opt:eq:kappa}
    \est{\star}{}^2\leq \kappa_0 \est{\ell}{}^2
  \end{align}
  fulfils the bulk chasing~\eqref{opt:bulkchasing} in the sense
  \begin{align}\label{opt:eq:bulkopt}
    \theta\est{\ell}{}^2 \leq \sum_{\el\in\RR{\ell}{\star}}\est{\ell}{\el}^2
  \end{align}
  for all $0\leq \theta<\theta_0$. The constant $\theta_0$ is defined in
  Theorem~\ref{opt:thm:optimality} whereas $\RR{\ell}{\star}\subseteq \mesh_\ell$ is guaranteed
  by~\eqref{opt:ass:drel}.
\end{mylemma}
\begin{proof}
  Similar to the proof of the estimator reduction in Lem\-ma~\ref{opt:lem:estred}, we split the
  error estimator and apply~\eqref{opt:ass:stable} together with Young's inequality
  $(a+b)^2\leq (1+\delta)a^2+(1+\delta^{-1})b^2$ for all $a,b\in\R$ and $\delta>0$. This yields
  \begin{align*}
    \est{\ell}{}^2&=\sum_{\el\in\mesh_\ell\setminus\mesh_\star}\est{\ell}{\el}^2 +
    \sum_{\el\in\mesh_\ell\cap\mesh_\star}\est{\ell}{\el}^2\\
    &\leq \sum_{\el\in\mesh_\ell\setminus\mesh_\star}\est{\ell}{\el}^2 +(1+\delta) \est{\star}{}^2
    \\
    &\qquad+ (1+\delta^{-1}) \c{opt:stable}^2 \norm{U_\star-U_\ell}{\XX}^2.
  \end{align*}
  The assumption~\eqref{opt:eq:kappa} as well as~\eqref{opt:ass:drel} apply and show
  \begin{align*}
    \est{\ell}{}^2&\leq \sum_{\el\in\mesh_\ell\setminus\mesh_\star}\est{\ell}{\el}^2 +
    (1+\delta) \kappa_0\est{\ell}{}^2\\
    &\qquad+ (1+\delta^{-1}) \c{opt:stable}^2 \norm{U_\star-U_\ell}{\XX}^2\\
    &\leq (1+(1+\delta^{-1}) \c{opt:stable}^2\c{opt:drel}^2)\sum_{\el\in\RR{\ell}{\star}}
    \est{\ell}{\el}^2+(1+\delta) \kappa_0\est{\ell}{}^2.
  \end{align*}
  This implies
  \begin{align*}
    \frac{1-(1+\delta) \kappa_0}{1+(1+\delta^{-1}) \c{opt:stable}^2\c{opt:drel}^2}
    \est{\ell}{}^2\leq \sum_{\el\in\RR{\ell}{\star}}\est{\ell}{\el}^2.
  \end{align*}
  If $\theta<\theta_0$, there exist $\delta>0$ and $\kappa_0>0$ such that
  \begin{align*}
    \theta \leq \frac{1-(1+\delta) \kappa_0}{1+(1+\delta^{-1}) \c{opt:stable}^2\c{opt:drel}^2}\leq
    \frac{1}{1+\c{opt:stable}^2\c{opt:drel}^2}=: \theta_0.
  \end{align*}
  The combination of the last two estimates concludes the proof.
  $\hfill\qed$
\end{proof}

\begin{proof}[of Theorem~\ref{opt:thm:optimality}]
  The very technical proof of Theorem~\ref{opt:thm:optimality} is found in great detail
  in~\cite[Section~4]{axioms}. Therefore, we only provide a brief sketch here. Given $\lambda>0$
  and by use of $\norm{\est{}{}}{\A_s}<\infty$, 
  one can find a mesh  $\mesh_\eps\in\refine(\mesh_0)$ with
  $\est{\eps}{}^2\leq \eps^2 :=\lambda\est{\ell}{}^2$ and $\#\mesh_\eps-\#\mesh_0\lesssim \eps^{-1/s}$.
  The overlay property~\eqref{dp:overlay} proves for
  $\mesh_\star:=\mesh_\eps\oplus\mesh_\ell\in \refine(\mesh_\ell)$ that
  \begin{align}\label{opt:eq:optmesh1}
    \#(\mesh_\ell\setminus\mesh_\star)\leq \#\mesh_\eps-\#\mesh_0\lesssim \eps^{-1/s}.
  \end{align}
  The arguments of the proof of Lemma~\ref{opt:lem:estred} apply also for
  $\mesh_\star\in\refine(\mesh_\eps)$ and show
  \begin{align*}
    \est{\star}{}^2\lesssim \est{\eps}{}^2 + \norm{U_\star-U_\eps}{\XX}^2.
  \end{align*}
  Then, the discrete reliability~\eqref{opt:ass:drel} implies
  \begin{align*}
    \est{\star}{}^2\lesssim \est{\eps}{}^2\leq \eps^2= \lambda\est{\ell}{}^2.
  \end{align*}
  Finally, we choose $\lambda>0$ sufficiently small such that there holds
  \begin{align}\label{opt:eq:optmesh2}
    \est{\star}{}^2\leq \kappa_0 \est{\ell}{}^2.
  \end{align}
  Note that the choice of $\lambda$ depends only on the constants
  in~\eqref{opt:ass:stable}--\eqref{opt:ass:reduction} and~\eqref{opt:ass:drel}, as well as on
  $\norm{\est{}{}}{\A_s}$. With~\eqref{opt:eq:optmesh2}, Lemma~\ref{opt:lem:bulkchasing} applies
  and proves that $\RR{\ell}{\star}$ satisfies the bulk chasing~\eqref{opt:eq:bulkopt}. Since
  $\MM_\ell$ is chosen in Step~(iii) of Algorithm~\ref{opt:algorithm} as a set with minimal
  cardinality which satisfies the bulk chasing criterion, there holds with~\eqref{opt:eq:optmesh1}
  \begin{align*}
    \#\MM_\ell\leq \#\RR{\ell}{\star}\leq \c{opt:refined}\#(\mesh_\ell\setminus\mesh_\star)\lesssim
    \eps^{-1/s}= \lambda^{-1/{2s}}\est{\ell}{}^{-1/s}.
  \end{align*}
  Next, the mesh closure estimate~\eqref{dp:meshclosure} provides
  \begin{align*}
    \#\mesh_\ell-\#\mesh_0 \lesssim \sum_{k=0}^{\ell-1}\#\MM_k \lesssim
    \sum_{k=0}^{\ell-1}\#\est{k}{}^{-1/s}
    \quad\text{for all }\ell\in\N_0
  \end{align*}
  The $R$-linear convergence~\eqref{opt:eq:Rlin} shows
  $\est{\ell}{}^2\lesssim q_{\rm R}^{\ell-k}\est{k}{}^2$, which implies
  $\est{k}{}^{-1/s}\lesssim q_{\rm R}^{(\ell-k)/(2s)} \est{\ell}{}^{-1/s}$. By convergence of the
  geometric series, this concludes
  \begin{align*}
    \#\mesh_\ell-\#\mesh_0 \lesssim \est{\ell}{}^{-1/s} \sum_{k=0}^{\ell-1} q_{\rm R}^{(\ell-k)/(2s)}
    \leq \est{\ell}{}^{-1/s} \frac{1}{1-q_{\rm R}^{1/(2s)}}.
  \end{align*}
  Taking the estimate to the power of $-s$ shows
  \begin{align*}
    \est{\ell}{} \lesssim ( \#\mesh_\ell-\#\mesh_0 )^{-s}\quad\text{for all }\ell\in\N_0.
  \end{align*}
  This concludes the proof.
  $\hfill\qed$
\end{proof}

\subsection{Linear convergence of a $(h-h/2)$-type error estimator for the weakly
singular integral equation}\label{opt:sec:example1}
This section discusses the $(h-h/2)$ error estimator from~\cite{fp08},
cf.\ Section~\ref{section:est:hh2} and extends the estimator reduction result from
Section~\ref{section:estred:hh2weaksing} in the context of the abstract framework.
In the terms of the previous section, it holds according to Proposition~\ref{prop:galerkin:weaksing}
\begin{align}\label{opt:eq:weaksing}
  \form{\phi}{\psi}&:=\dual{V\phi}{\psi}_\Gamma\quad\text{for all }\phi,\psi\in\XX:=
  \widetilde H^{-1/2}(\Gamma),\nonumber\\
  F(\psi)&:= \dual{f}{\psi}_{\Gamma}\quad\text{for all }\psi\in\XX,
\end{align}
where $f\in H^{1/2}(\Gamma)$. The exact solution $\phi\in\XX$ satisfies
\begin{align}\label{opt:weaksing:continuous}
  \form{\phi}{\psi}=F(\psi)\quad\text{for all }\psi\in\XX.
\end{align}
Since the $(h-h/2)$ error-estimator defined in~\eqref{estred:def:hh2weaksing} uses the uniformly refined
mesh $\widehat\mesh_\ell:=\refine(\mesh_\ell,\mesh_\ell)$ instead of the original mesh $\mesh_\ell$,
it is natural to consider the discrete spaces $\XX_\ell:=\widehat\XX_\ell:=\Pp^p(\widehat \mesh_\ell)$
for the discrete Galerkin formulation~\eqref{intro:galerkin} which reads in this setting: find
$\widehat\Phi_\ell\in\widehat\XX_\ell$ such that
\begin{align}\label{opt:weaksing:discrete}
  \form{\widehat \Phi_\ell}{\Psi}=F(\Psi)\quad\text{for all }\Psi\in\widehat\XX_\ell.
\end{align}
Recall the $(h-h/2)$-type error estimator from Theorem~\ref{thm:def:hh2:weaksing}, i.e.,
\begin{align}\label{opt:est:weaksing:hh2}
  \est{\ell}{}^2:=\sum_{\el\in\mesh_\ell}\est{\ell}{\el}^2:=\sum_{\el\in\mesh_\ell}h_\el
  \norm{(1-\pi_\ell^p)\widehat\Phi_\ell}{L_2(\el)}^2,
\end{align}
where $h_\el:=|\el|^{1/(d-1)}\simeq {\rm diam}(\el)$. 
The next step is to prove the assumptions~\eqref{opt:ass:stable}--\eqref{opt:ass:reduction}
for $\est{\ell}{}$.

\begin{mylemma}\label{opt:weaksing:ass}
  The $(h-h/2)$ error estimator $\est{\ell}{}$ from~\eqref{opt:est:weaksing:hh2} satisfies
  stability~\eqref{opt:ass:stable} and reduction~\eqref{opt:ass:reduction}. The constants
  $\c{opt:stable},q_{\rm red},\c{opt:reduction}$ depend only on $\Gamma$, the polynomial
  degree $p$, and on the shape regularity of $\mesh_\ell$.
\end{mylemma}
\begin{proof}
  Stability~\eqref{opt:ass:stable} and reduction~\eqref{opt:ass:reduction} are implicitly shown
  in the proof of Lemma~\ref{estred:lem:weaksing}.
  $\hfill\qed$
\end{proof}

\begin{theorem}\label{opt:thm:weaksingconv}
  For all $0<\theta\leq 1$, Algorithm~\ref{opt:algorithm} with the $(h-h/2)$ error estimator
  $\est{\ell}{}$ converges in the sense
  \begin{align}\label{opt:weaksing:estconv}
    \lim_{\ell\to\infty}\est{\ell}{}=0.
  \end{align}
  If the saturation assumption (Assumption~\ref{ass:sata}) is satisfied, then $\est{\ell}{}$
  from~\eqref{opt:est:weaksing:hh2} satisfies reliability~\eqref{opt:ass:reliable}, and there
  holds $R$-linear convergence
  \begin{align}\label{opt:weaksing:rlin}
    \c{opt:reliable}^{-2}\norm{\phi-\Phi_{\ell+n}}{\widetilde H^{-1/2}(\Gamma)}^2\leq
    \est{\ell+n}{}^2\leq \c{opt:Rlin}q_{\rm R}^n \est{\ell}{}^2
  \end{align}
  for all $\ell,n\in\N_0$. The constant $\c{opt:reliable}$ depends only on the saturation constant
  $\c{sata}$, $\Gamma$, the shape regularity of $\mesh_\ell$, and the polynomial degree $p$.
  The constants $\c{opt:Rlin}$ and $q_{\rm R}$ depend additionally on $\theta$. 
\end{theorem}

\begin{proof}
  The convergence~\eqref{opt:weaksing:estconv} follows from Theorem~\ref{opt:thm:convergence},
  since $\est{\ell}{}$ satisfies the assumptions~\eqref{opt:ass:stable}--\eqref{opt:ass:reduction}.
  The reliability of $\est{\ell}{}$ under the saturation assumption is proved in
  Theorem~\ref{thm:hh2:weaksing} together with the equivalence of
  $\norm{\cdot}{\slo}\simeq \norm{\cdot}{\widetilde H^{-1/2}(\Gamma)}$. The remaining statement
  follows from Theorem~\ref{opt:thm:convergence}.
  $\hfill\qed$
\end{proof}

\subsection{Optimal convergence of weighted residual error estimator for the
weakly singular integral equation}\label{opt:sec:example2}
We consider the model problem~\eqref{opt:weaksing:continuous}, but in contrast to the previous
section, the discrete problem employs the original mesh $\mesh_\ell$ instead of its uniform
refinement $\widehat\mesh_\ell$, i.e., find $\Phi_\ell\in\XX_\ell:=\Pp^p(\mesh_\ell)$ such that
\begin{align}\label{opt:weaksing:discrete2}
  \form{\Phi_\ell}{\Psi}=F(\Psi)\quad\text{for all }\Psi\in\XX_\ell.
\end{align}
As in Section~\ref{section:estred:resweaksing}, the standard weighted residual error estimator
from Section~\ref{section:est:wres} reads
\begin{align}\label{opt:est:weaksing:res}
  \est{\ell}{}^2:=\sum_{\el\in\mesh_\ell}\est{\ell}{\el}^2:=\sum_{\el\in\mesh_\ell}h_\el
  \norm{\nabla_\Gamma(V\Phi_\ell-f)}{L_2(\el)}^2,
\end{align}
where $\nabla_\Gamma(\cdot)$ denotes the surface gradient on $\Gamma$.
Note that, while~\eqref{intro:weakform} and~\eqref{opt:weaksing:discrete2}
are well-stated for $f\in H^{1/2}(\Gamma)$, the definition of $\est{\ell}{}$
requires additional regularity $f\in H^1(\Gamma)$ of the data.

\begin{mylemma}\label{opt:lem:weaksing:res}
  The weighted residual error estimator $\est{\ell}{}$ from
  \eqref{opt:est:weaksing:res} satisfies
  the assumptions~\eqref{opt:ass:stable}--\eqref{opt:ass:drel}.
  The set $\RR{\ell}{\star}$ from~\eqref{opt:ass:drel} satisfies
  $\RR{\ell}{\star}:=\set{\el\in\mesh_\ell}{\exists \el^\prime\in\mesh_\ell\setminus\mesh_\star,
  \overline{\el}^\prime\cap\overline\el\neq \emptyset}$.
  The constant $\c{opt:refined}$ depends only on the shape regularity of the mesh $\mesh_\ell$,
  whereas the constants $\c{opt:stable},\c{opt:reduction},q_{\rm red},\c{opt:reliable},\linebreak \c{opt:drel}$
  depend additionally on $\Gamma$ and the polynomial degree $p$.
\end{mylemma}

\begin{proof}
  Reliability~\eqref{opt:ass:reliable} is well-known for $\est{\ell}{}$ since its invention
  in~\cite{cs96} for $d=2$ and~\cite{cms} for $d=3$.
  The assumptions~\eqref{opt:ass:stable}--\eqref{opt:ass:reduction} are shown in the proof of
  Lemma~\ref{estred:lem:weaksingres}.
  The proof of discrete reliability~\eqref{opt:ass:drel} analyzes the original reliability proof
  from~\cite{cms}. The technical proof is found 
  in~\cite{fkmp13,gantumur} for $p=0$ and in~\cite{ffkmp13} for general $p\ge0$.
  $\hfill\qed$
\end{proof}

\begin{theorem}\label{opt:thm:weaksing:res}
  For all $0<\theta\leq 1$, Algorithm~\ref{opt:algorithm} with the residual error estimator
  $\est{\ell}{}$ from~\eqref{opt:est:weaksing:res} converges in the sense
  \begin{align}\label{opt:weaksing:res:rlin}
    \c{opt:reliable}^{-2}\norm{\phi-\Phi_{\ell+n}}{\widetilde H^{-1/2}(\Gamma)}^2\leq
    \est{\ell+n}{}^2\leq \c{opt:Rlin}q_{\rm R}^n \est{\ell}{}^2
  \end{align}
  for all $\ell,n\in\N_0$. For $0<\theta<\theta_0$, Algorithm~\ref{opt:algorithm} converges
  with the best possible rate $s>0$ in the sense that $\norm{\est{}{}}{\A_s}<\infty$ if and only if
  \begin{align}
    \est{\ell}{}\leq \c{opt:optimality} (\#\mesh_\ell-\#\mesh_0)^{-s}\quad\text{for all }\ell\in\N,
  \end{align}
  where the constants $\c{opt:Rlin},q_{\rm R}$ depend only on $\Gamma$, the shape regularity of
  the meshes $\mesh_\ell$, the polynomial degree $p$, and $\theta$. The constant
  $\c{opt:optimality}>0$ depends additionally on $\norm{\est{}{}}{\A_s}$.
\end{theorem}
\begin{proof}
  Lemma~\ref{opt:lem:weaksing:res} shows that the
  assumption~\eqref{opt:ass:stable}--\eqref{opt:ass:drel} are satisfied.
  Theorem~\ref{opt:thm:convergence} and Theorem~\ref{opt:thm:optimality} prove the statements.
  $\hfill\qed$
\end{proof}

\subsection{Linear convergence of a $(h-h/2)$-type error estimator for the hypersingular
integral equation}\label{opt:sec:example3}
This section extends the estimator reduction result from Section~\ref{section:estred:hh2hypsing}
in the context of the abstract framework. In the terms of previous section, the variational
formulation reads according to Proposition~\ref{prop:galerkin:hypsing}
\begin{align}\label{opt:eq:hypsing}
  \begin{split}
    \form{u}{v}&:=
    \begin{cases}
      \dual{\hyp u}{v}_\Gamma &\text{ for } \Gamma\subsetneq\partial\Omega,\\
      \dual{Wu}{v}_\Gamma+\dual{u}{1}_\Gamma\dual{v}{1}_\Gamma
      &\text{ for } \Gamma = \partial\Omega,
    \end{cases}
    \\
    F(v)&:= \dual{\phi}{v}_{\Gamma}\quad\text{for all }u,v\in\XX:= \widetilde H^{1/2}(\Gamma),
  \end{split}
\end{align}
where $\phi\in H^{-1/2}(\Gamma)$, respectively
$\phi\in H^{-1/2}_0(\Gamma):=\set{\psi\in H^{-1/2}(\Gamma)}{\dual{\psi}{1}_\Gamma=0}$.
The exact solution $u\in\XX$ satisfies
\begin{align}\label{opt:hypsing:continuous}
  \form{u}{v}=F(v)\quad\text{for all }v\in\XX.
\end{align}
Since the $(h-h/2)$ error-estimator from Theorem~\ref{thm:def:hh2:hypsing} uses the uniformly
refined mesh $\widehat\mesh_\ell:=\refine(\mesh_\ell,\mesh_\ell)$ instead of the original mesh
$\mesh_\ell$, it is natural to consider the discrete spaces
$\XX_\ell:=\widehat\XX_\ell:=\wilde\Sp^p(\widehat \mesh_\ell)$ for the discrete Galerkin
formulation~\eqref{intro:galerkin} which reads in this setting: find
$\widehat U_\ell\in\widehat\XX_\ell$ such that
\begin{align}\label{opt:hypsing:discrete}
  \form{\widehat U_\ell}{V}=F(V)\quad\text{for all }V\in\widehat\XX_\ell.
\end{align}
Recall the $(h-h/2)$-type error estimator from Theorem~\ref{thm:def:hh2:hypsing}
\begin{align}\label{opt:est:hypsing:hh2}
  \est{\ell}{}^2:=\sum_{\el\in\mesh_\ell}\est{\ell}{\el}^2:=\sum_{\el\in\mesh_\ell}h_\el
  \norm{(1-\pi_\ell^p)\nabla_\Gamma\widehat U_\ell}{L_2(\el)}^2,
\end{align}
where $\nabla_\Gamma$ denotes the surface gradient on $\Gamma$.
The next step is to prove the assumptions~\eqref{opt:ass:stable}--\eqref{opt:ass:reduction}
for $\est{\ell}{}$.

\begin{mylemma}\label{opt:hypsing:ass}
  The $(h-h/2)$ error estimator $\est{\ell}{}$ from~\ref{opt:est:hypsing:hh2} satisfies
  stability~\eqref{opt:ass:stable} and reduction~\eqref{opt:ass:reduction}. The constants
  $\c{opt:stable},\linebreak q_{\rm red},\c{opt:reduction}$ depend only on $\Gamma$, the polynomial degree
  $p$, and on the shape regularity of $\mesh_\ell$.
\end{mylemma}
\begin{proof}
  The statement is essentially proved in Lemma~\ref{estred:lem:hypsing}.
  $\hfill\qed$
\end{proof}

\begin{theorem}\label{opt:thm:hypsingconv}
  For all $0<\theta\leq 1$, Algorithm~\ref{opt:algorithm} with the\linebreak $(h-h/2)$ error estimator
  $\est{\ell}{}$ from~\eqref{opt:est:hypsing:hh2} converges in the sense
  \begin{align}\label{opt:hypsing:estconv}
    \lim_{\ell\to\infty}\est{\ell}{}=0.
  \end{align}
  If the saturation assumption (Assumption~\ref{ass:sata}) is satisfied, then $\est{\ell}{}$
  satisfies reliability~\eqref{opt:ass:reliable}, and there holds $R$-linear convergence
  \begin{align}\label{opt:hypsing:rlin}
    \c{opt:reliable}^{-2}\norm{u-U_{\ell+n}}{\widetilde H^{1/2}(\Gamma)}^2\leq
    \est{\ell+n}{}^2\leq \c{opt:Rlin}\est{\ell}{}^2\,q_{\rm R}^n 
  \end{align}
  for all $\ell,n\in\N_0$. The constant $\c{opt:reliable}$ depends only on the saturation constant
  $\c{sata}$, $\Gamma$, the shape regularity of $\mesh_\ell$, and the polynomial degree $p$.
  The constants  $\c{opt:Rlin}$ and $q_{\rm R}$ depend additionally on $\theta$.
\end{theorem}

\begin{proof}
  The convergence~\eqref{opt:hypsing:estconv} follows from Theorem~\ref{opt:thm:convergence},
  since $\est{\ell}{}$ satisfies the assumptions~\eqref{opt:ass:stable}--\eqref{opt:ass:reduction}.
  The reliability of $\est{\ell}{}$ under the saturation assumption is proved in
  Theorem~\ref{thm:hh2:hypsing} together with the equivalence of
  $\norm{\cdot}{\hyp}\simeq \norm{\cdot}{\widetilde H^{1/2}(\Gamma)}$, and
  the remaining statement follows from Theorem~\ref{opt:thm:convergence}.
  $\hfill\qed$
\end{proof}

\subsection{Optimal convergence of weighted residual error estimator for the hyper
singular integral equation}\label{opt:sec:example4}
We consider the model problem~\eqref{opt:hypsing:continuous}, but in contrast to the previous
section, the discrete problem employs the original mesh $\mesh_\ell$ instead of its uniform
refinement $\widehat\mesh_\ell$, i.e., find $U_\ell\in\XX_\ell:=\wilde\Sp^p(\mesh_\ell)$ such that
\begin{align}\label{opt:hypsing:discrete2}
  \form{U_\ell}{V}=F(V)\quad\text{for all }V\in\XX_\ell.
\end{align}
As in Section~\ref{section:estred:reshypsing}, the weighted residual error estimator from
Section~\ref{section:est:wres} requires more regularity,
i.e., $\phi\in L_2(\Gamma)$ needs to be assumed. The error estimator then reads
\begin{align}\label{opt:est:hypsing:res}
  \est{\ell}{}^2:=\sum_{\el\in\mesh_\ell}\est{\ell}{\el}^2:=\sum_{\el\in\mesh_\ell}h_\el
  \norm{W U_\ell-\phi}{L_2(\el)}^2.
\end{align}

\begin{mylemma}\label{opt:lem:hypsing:res}
  The weighted residual error estimator $\est{\ell}{}$ from \eqref{opt:est:hypsing:res}
  satisfies the assumptions~\eqref{opt:ass:stable}--\eqref{opt:ass:drel}.
  Moreover,~\eqref{opt:ass:drel} holds with $\RR{\ell}{\star}:=\mesh_\ell\setminus\mesh_\star$
  and $\c{opt:refined}=1$. The constants $\c{opt:stable}$, $\c{opt:reduction}$, $q_{\rm red}$,
  $\c{opt:reliable}$, $\c{opt:drel}$ depend only on the shape regularity of the mesh $\mesh_\ell$,
  $\Gamma$ and the polynomial degree $p$.
\end{mylemma}
\begin{proof}
  Reliability~\eqref{opt:ass:reliable} is well-known for $\est{\ell}{}$ since its invention
  in~\cite{cs96} for $d=2$ and~\cite{cmps} for $d=3$.
  The assumptions~\eqref{opt:ass:stable}--\eqref{opt:ass:reduction} are shown in the proof of
  Lemma~\ref{estred:lem:hypsingres}.
  The proof of discrete reliability~\eqref{opt:ass:drel} employs the Scott-Zhang projection
  from Lemma~\ref{lem:sz} to obtain the local statement. The technical proof refines the
  arguments from~\cite{cmps} and is found in~\cite[Proposition~4]{ffkmp13-A}.
  Alternatively, the proof of~\cite{gantumur} built on the localization
  techniques from~\cite{f00,f02}. This, however, restricts the analysis to
  lowest-order elements $p=1$, where  $\RR{\ell}{\star}$ consists of
  $\TT_\ell\backslash\TT_\star$ plus one layer of non-refined elements.
  $\hfill\qed$
\end{proof}

\begin{theorem}\label{opt:thm:hypsing:res}
  For all $0<\theta\leq 1$, Algorithm~\ref{opt:algorithm} with the residual error estimator
  $\est{\ell}{}$ from~\eqref{opt:est:hypsing:res}  converges in the sense
  \begin{align}\label{opt:hypsing:res:rlin}
    \c{opt:reliable}^{-2}\norm{u-U_{\ell+n}}{\widetilde H^{1/2}(\Gamma)}^2\leq
    \est{\ell+n}{}^2\leq \c{opt:Rlin}q_{\rm R}^n \est{\ell}{}^2
  \end{align}
  for all $\ell,n\in\N_0$. For $0<\theta<\theta_0$, Algorithm~\ref{opt:algorithm} converges
  with the best possible rate $s>0$ in the sense that $\norm{\est{}{}}{\A_s}<\infty$ if and only if
  \begin{align}
    \est{\ell}{}\leq \c{opt:optimality} (\#\mesh_\ell-\#\mesh_0)^{-s}\quad\text{for all }\ell\in\N,
  \end{align}
  where the constants $\c{opt:Rlin},q_{\rm R}$ depend only on $\Gamma$, the shape regularity of
  the meshes $\mesh_\ell$, the polynomial degree $p$, and $\theta$. The constant
  $\c{opt:optimality}>0$ depends additionally on $\norm{\est{}{}}{\A_s}$.
\end{theorem}
\begin{proof}
  Lemma~\ref{opt:lem:hypsing:res} shows that the
  assumption~\eqref{opt:ass:stable}--\eqref{opt:ass:drel} are satisfied.
  Theorem~\ref{opt:thm:convergence} and Theorem~\ref{opt:thm:optimality} prove the statements.
  $\hfill\qed$
\end{proof}

\subsection{Inclusion of data approximation}\label{opt:section:data}
Also the data approximation, which is already discussed in Section~\ref{section:estred:data},
can be analyzed towards optimal convergence rates.
As in Section~\ref{section:estred:data}, we replace the right-hand side $F$ in~\eqref{intro:weakform}
with some computable approximation $F_\ell$ on any mesh $\mesh_\ell$ and solve
\begin{align}\label{opt:galerkind}
  \form{\Ud_\ell}{V}=F_\ell(V)\quad\text{for all }V\in\XX_\ell
\end{align}
instead of~\eqref{intro:galerkin}. To control the additional error $\norm{\Ud_\ell-U_\ell}{\XX}$
introduced by this approximation, the error estimator $\est{\ell}{}$ is extended by some
data approximation term
\begin{align}\label{opt:def:dataapprox}
  \data{\ell}{}^2:=\sum_{\el\in\mesh_\ell}\data{\ell}{\el}^2\geq
  C_{\rm data}^{-1}\norm{\Ud_\ell-U_\ell}{\XX}^2.
\end{align}
This is an abstract approach to the concrete results of Section~\ref{section:dataapproximation},
where several examples for $\data{\ell}{}$ are given.
The extended error estimator reads elementwise for all $\el\in\mesh_\ell$
\begin{align*}
  \estd{\ell}{\el}^2:=\est{\ell}{\el}^2+\data{\ell}{\el}^2,
\end{align*}
where $\est{\ell}{\el}$ uses the computable approximate solution $\Ud_\ell$ and the
approximate data $F_\ell$ instead of the non-computable solution $U_\ell$.
Note the difference with the notation of Section~\ref{section:dataapproximation}.
With this, the extension of Algorithm~\ref{opt:algorithm} reads:
\begin{algorithm}[adaptive mesh refinement]\label{opt:algorithmd}\ \linebreak
    \textsc{Input}: initial mesh $\mesh_0$ and adaptivity parameter $0<\theta\leq 1$.\\
  \textsc{Output}: sequence of solutions $(\Ud_\ell)_{\ell\in\N_0}$,
  sequence of estimators $(\estd{\ell}{})_{\ell\in\N_0}$, and sequence of meshes
  $(\mesh_\ell)_{\ell\in\N_0}$.\\
  \textsc{Iteration}: For all $\ell=0,1,2,3,\ldots$ do {\rm (i)}--{\rm (iv)}
  \begin{itemize}
    \item[\rm(i)] Compute solution $\Ud_\ell$ of~\eqref{opt:galerkind}.
    \item[\rm(ii)] Compute error indicators $\estd{\ell}{\el}$ for all elements $\el\in\mesh_\ell$.
    \item[\rm(iii)] Find a set of minimal cardinality $\MM_\ell\subseteq\mesh_\ell$ such that
    \begin{align}\label{opt:bulkchasingd}
    \theta\estd{\ell}{}^2\leq \sum_{\el\in\MM_\ell}\estd{\ell}{\el}^2.
    \end{align}
    \item[\rm(iv)] Refine (cf.~Section~\ref{section:opt:adaptive}) at least the marked
      elements to obtain the new mesh $\mesh_{\ell+1}:=\refine(\mesh_\ell,\MM_\ell)$.
    \end{itemize}
\end{algorithm}

To account for the new estimator term, we have to adopt the assumptions from
Section~\ref{opt:section:assumptions} slightly. To that end, we introduce a theoretical data
approximation term $\datad{\ell}{}^2$ which satisfies
$C_{\rm data}^{-1}\data{\ell}{}^2\leq \datad{\ell}{}^2\leq C_{\rm data}\data{\ell}{}^2$ for
some constant $C_{\rm data}>0$. The only reason for this is that we want to allow ourselves to
use a slightly different oscillation term for implementation than we use for the analysis.
This simplifies the realization of Algorithm~\ref{opt:algorithmd}.
\begin{enumerate}
    \renewcommand{\theenumi}{\mbox{$\widetilde{\rm A\arabic{enumi}}$}}%
    \renewcommand{\labelenumi}{(\mbox{$\widetilde{\rm A\arabic{enumi}}$})}%
  \item\label{opt:ass:stabled} Stability on non-refined elements: There exists a constant
    $\c{opt:stable}>0$ such that any refinement $\mesh_\star\in \refine(\mesh_\ell)$ of
    $\mesh_\ell\in\refine(\mesh_0)$ satisfies
    \begin{align*}
      \Big| \Big(\sum_{\el\in\mesh_\ell\cap \mesh_\star} &\estd{\ell}{\el}^2\Big)^{1/2} -
      \Big(\sum_{\el\in\mesh_\ell\cap \mesh_\star} \estd{\star}{\el}^2\Big)^{1/2}\Big|^2 \\
      &\leq \c{opt:stable}^{2}\Big( \norm{\Ud_\star-\Ud_\ell}{\XX}^2+\datad{\ell}{}^2-
      \datad{\star}{}^2\Big).
    \end{align*}

  \item\label{opt:ass:reductiond} Reduction on refined elements: There exist constants\linebreak
    $\c{opt:reduction}>0$ and $0<q_{\rm red}< 1$ such that any refinement
    $\mesh_\star\in \refine(\mesh_\ell)$  of $\mesh_\ell\in\refine(\mesh_0)$ satisfies
    \begin{align*}
      \sum_{\el\in\mesh_\star\setminus \mesh_\ell} \estd{\star}{\el}^2 &\leq q_{\rm red}
      \sum_{\el\in\mesh_\ell\setminus \mesh_\star} \estd{\ell}{\el}^2\\
      &\quad+ \c{opt:reduction}\Big(\norm{\Ud_\star-\Ud_\ell}{\XX}^2
      + \datad{\ell}{}^2-\datad{\star}{}^2\Big).
    \end{align*}

  \item\label{opt:ass:reliabled} Reliability: There exists a constant $\c{opt:reliable}>0$ such
    that any mesh $\mesh_\ell\in\refine(\mesh_0)$ satisfies
    \begin{align*}
      \norm{u-\Ud_\ell}{\XX}\leq \c{opt:reliable} \estd{\ell}{}.
    \end{align*}

  \item\label{opt:ass:dreld} Discrete reliability: There exist constants $\c{opt:drel}>0$
    and $\c{opt:refined}>0$ such that any refinement $\mesh_\star\in \refine(\mesh_\ell)$
    of $\mesh_\ell\in\refine(\mesh_0)$ satisfies
    \begin{align*}
      \norm{\Ud_\star-\Ud_\ell}{\XX}^2+\datad{\ell}{}^2-\datad{\star}{}^2\leq
      \c{opt:drel}^2\sum_{\el\in\RR{\ell}{\star}} \estd{\ell}{\el}^2,
    \end{align*}
    where the set $\RR{\ell}{\star}\supseteq \mesh_\ell\setminus\mesh_\star$ satisfies\linebreak
    $\#\RR{\ell}{\star}\leq \c{opt:refined} \#(\mesh_\ell\setminus\mesh_\star)$.
\end{enumerate}

Moreover, and in contrast to the unperturbed case in Section~\ref{opt:section:assumptions},
the a~priori convergence~\eqref{estred:eq:apriorigal} as well as the 
generalized Pythagoras estimate~\eqref{opt:eq:qosum}
are not available in this general setting. Hence, we also have to verify
\begin{enumerate}
    \renewcommand{\theenumi}{\mbox{$\widetilde{\rm A\arabic{enumi}}$}}%
    \renewcommand{\labelenumi}{(\mbox{$\widetilde{\rm A\arabic{enumi}}$})}%
    \setcounter{enumi}{4}
  \item \label{opt:ass:apriorid}A~priori convergence of data: There exists a continuous linear
    functional $F_\infty:\,\XX_\infty\to \R$ such that
    \begin{align*}
      \lim_{\ell\to\infty}\norm{F_\infty- F_\ell}{\XX_\ell^\prime}:= \lim_{\ell\to\infty}
      \sup_{V\in\XX_\ell\atop \norm{V}{\XX}=1} |F_\infty(V)-F_\ell(V)|=0,
    \end{align*}
    where $\XX_\infty:=\overline{\bigcup_{\ell\in\N_0}\XX_\ell}\subseteq \XX$
    (the closure is understood with respect to $\XX$). Moreover, there exists $\datad{\infty}{}\geq 0$ such that
    \begin{align*}
     \lim_{\ell\to\infty}\datad{\ell}{}=\datad{\infty}{}.
    \end{align*}

  \item \label{opt:ass:qosumd} Pythagoras estimate: For all $\eps>0$, there exists a constant
    $\c{opt:qosum}(\eps)>0$ such that for all $k\in\N$ holds
    \begin{align*}
      \sum_{k=\ell}^\infty \norm{\Ud_{k+1}-\Ud_k}{\XX}^2-\eps \estd{k}{}^2\leq
      \c{opt:qosum}(\eps)\estd{\ell}{}^2.
    \end{align*}
\end{enumerate}

\subsection{Data approximation and convergence of ABEM}\label{opt:section:data:conv}
The proofs in this section differ only slightly from the unperturbed case in
Section~\ref{opt:section:convergence}. Therefore, we only highlight the differences.

\begin{theorem}\label{opt:thm:convergenced}
  Suppose that there hold~\eqref{opt:ass:stabled}--\eqref{opt:ass:reductiond} as
  well\\as~\eqref{opt:ass:apriorid}. Then, Algorithm~\ref{opt:algorithmd} drives
  the estimator to zero, i.e.,
  \begin{align}\label{opt:eq:convergenced}
    \lim_{\ell\to\infty} \estd{\ell}{} = 0.
  \end{align}
  Suppose that the error estimator additionally satisfies~\eqref{opt:ass:reliabled}
  and~\eqref{opt:ass:qosumd}. Then, Algorithm~\ref{opt:algorithmd} converges even linearly in
  the sense that there exist constants $\c{opt:Rlin}>0$ and $0<q_{\rm R}<1$ such that
  \begin{align}\label{opt:eq:Rlind}
    \c{opt:reliable}^{-2}\norm{u-\Ud_{\ell+n}}{\XX}^2\leq  \estd{\ell+n}{}^2\leq
    \c{opt:Rlin}q_{\rm R}^n \estd{\ell}{}^2\;\text{for all }\ell,n\in\N_0,
  \end{align}
  which particularly implies
  \begin{align}
    \norm{u-\Ud_{\ell}}{\XX}^2\leq \c{opt:reliable}^2\c{opt:Rlin}\estd{0}{}^2\,q_{\rm R}^\ell
    \quad\text{for all }\ell\in\N_0.
  \end{align}
  The constants $\c{opt:Rlin}$ and $q_{\rm R}$ depend only on $\theta$ as well as on the constants
  in~\eqref{opt:ass:stabled}--\eqref{opt:ass:reliabled} and~\eqref{opt:ass:qosumd}.
\end{theorem}

Again, an important ingredient is the a~priori convergence of $\Ud_\ell$.

\begin{mylemma}[a~priori convergence]\label{estred:lem:convorthod}
  Suppose~\eqref{opt:ass:apriorid}. Then,\linebreak there exists $\Ud_\infty \in \XX$ such that
  Algorithm~\ref{opt:algorithmd} satisfies
  \begin{align}\label{opt:eq:apriorid}
    \lim_{\ell\to\infty}\norm{\Ud_\infty-\Ud_\ell}{\XX}=0.
  \end{align}
\end{mylemma}
\begin{proof}
  Replace $F$ in~\eqref{intro:galerkin} by $F_\infty$ from~\eqref{opt:ass:apriorid} and consider the corresponding solution $(U_{\infty,\ell})_{\ell\in\N}$. Lemma~\ref{estred:lem:convortho} shows the existence of $\Ud_\infty\in\XX$ such that
\begin{align}\label{opt:eq:aprioriinfty}
 \lim_{\ell\to\infty}\norm{\Ud_\infty-U_{\ell,\infty}}{\XX} = 0.
\end{align}
From stability
\begin{align*}
 \norm{U_{\ell,\infty}-\Ud_\ell}{\XX} \leq \norm{F_\infty-F_\ell}{\XX_\ell^\prime},
\end{align*}
it follows
\begin{align*}
\norm{\Ud_\infty-\Ud_\ell}{\XX}&\leq \norm{\Ud_\infty-U_{\ell,\infty}}{\XX}+
 \norm{U_{\ell,\infty}-\Ud_\ell}{\XX}\\
 &\lesssim \norm{\Ud_\infty-U_{\ell,\infty}}{\XX}+\norm{F_\infty-F_\ell}{\XX_\ell^\prime}\to 0
\end{align*}
as $\ell\to\infty$ by assumption~\eqref{opt:ass:apriorid} and~\eqref{opt:eq:aprioriinfty}.
  $\hfill\qed$
\end{proof}

Also the estimator reduction follows accordingly.

\begin{mylemma}\label{opt:lem:estredd}
  Suppose~\eqref{opt:ass:stabled}--\eqref{opt:ass:reductiond}. Then, there
  exist constants $\c{opt:estred}>0$ and $0<q_{\rm est}<1$ such that Algorithm~\ref{opt:algorithmd}
  satisfies
  \begin{align}\label{opt:eq:estredd}
    \estd{\ell+1}{}^2\leq q_{\rm est} \estd{\ell}{}^2 +
    \c{opt:estred}\Big(\norm{\Ud_{\ell+1}-\Ud_\ell}{\XX}^2+\datad{\ell}{}^2-\datad{\ell+1}{}^2\Big)
  \end{align}
  for all $\ell\in\N$. The constant $q_{\rm est}$ depends only on $\theta$ and $q_{\rm red}$
  from~\eqref{opt:ass:reductiond}. The constant $\c{opt:estred}$ depends additionally on
  $\c{opt:reduction}$ as well as $\c{opt:stable}$.
\end{mylemma}
\begin{proof}
  The proof is identical to that of Lemma~\ref{opt:lem:estred}.
  $\hfill\qed$
\end{proof}

This implies the first part of Theorem~\ref{opt:thm:convergenced}. 
\begin{proof}[of estimator convergence~\eqref{opt:eq:convergenced}]
 Due to~\eqref{opt:ass:apriorid}, there holds $\datad{\ell}{}^2-\datad{\ell+1}{}^2\to 0$ as $\ell\to\infty$. With this and the arguments of the proof of~\eqref{opt:eq:convergence}, we see with a~priori
  convergence~\eqref{opt:eq:apriorid} for the limes superior
  $\overline \lim_{\ell\to\infty}\estd{\ell+1}{}^2=0$ or
  $\overline \lim_{\ell\to\infty}\estd{\ell+1}{}^2=\infty$.
  To rule out the second option, apply the estimator reduction iteratively to see
  \begin{align*}
    \estd{\ell}{}^2
    &\leq q_{\rm est}^\ell\estd{0}{}^2 + \c{opt:estred}\sum_{k=0}^{\ell-1} q_{\rm est}^k
    \Big(\norm{\Ud_{\ell-k}-\Ud_{\ell-k-1}}{\XX}^2\\
    &\qquad\qquad\qquad\qquad\qquad+\datad{\ell-k-1}{}^2-\datad{\ell-k}{}^2\Big)\\
    &\leq q_{\rm est}^\ell\estd{0}{}^2 + \c{opt:estred}\Big(\datad{0}{}^2+\sum_{k=0}^{\ell-1}
    q_{\rm est}^k \norm{\Ud_{\ell-k}-\Ud_{\ell-k-1}}{\XX}^2\Big),
  \end{align*}
  by exploiting the telescoping series.
  The a~priori convergence of Lemma~\ref{estred:lem:convorthod} implies
  $\sup_{\ell\in\N} \norm{\Ud_{\ell}-\Ud_{\ell-1}}{\XX}^2\leq C_{\rm max}<\infty$. This and the
  convergence of the geometric series show
  \begin{align*}
    \estd{\ell}{}^2\leq  \estd{0}{}^2 + 2\c{opt:estred}(\data{0}{}^2+C_{\rm max})<\infty\quad
    \text{for all }\ell\in\N
  \end{align*}
  and consequently $\overline \lim_{\ell\to\infty}\eta_\ell^2=0$.
  This concludes the proof of~\eqref{opt:eq:convergenced}.
  $\hfill\qed$
\end{proof}

The convergence~\eqref{opt:eq:convergenced} leads us to the $R$-linear convergence.
\begin{proof}[of $R$-linear estimator convergence~\eqref{opt:eq:Rlind}]
  Let $N,\ell\in\N$. Use the estimator reduction~\eqref{opt:eq:estredd} to see
\begin{align*}
\sum_{k=\ell+1}^{\ell+N} \estd{k}{}^2 \leq \sum_{k=\ell+1}^{\ell+N} \Big(q_{\rm est}\estd{k-1}{}^2 +\c{opt:estred}&\big(\norm{\Ud_k-\Ud_{k-1}}{\XX}^2\\
&+\datad{k-1}{}^2-\datad{k}{}^2\big)\Big).
\end{align*}
By use of the telescoping series, this implies
\begin{align*}
(1-q_{\rm est}&-\c{opt:estred}\eps)\sum_{k=\ell+1}^{\ell+N} \estd{k}{}^2\\
&\leq \estd{\ell}{}^2 
+\c{opt:estred}\Big(\datad{\ell}{}^2+\sum_{k=\ell+1}^{\ell+N}(\norm{\Ud_k-\Ud_{k-1}}{\XX}^2-\eps\estd{k}{}^2)\Big).
\end{align*}
The assumption~\eqref{opt:ass:qosumd} together with $C_{\rm data}^{-1}\datad{\ell}{}^2\leq \data{\ell}{}^2 \leq \estd{\ell}{}^2$  then shows
\begin{align*}
\sum_{k=\ell}^{\ell+N} \estd{k}{}^2\leq \frac{2+\c{opt:estred}(\c{opt:qosum}(\eps)+C_{\rm data})}{1-q_{\rm est}-\c{opt:estred}\eps}\estd{\ell}{}^2:=\c{opt:Rlin}\estd{\ell}{}^2.
\end{align*}
Clearly, $\c{opt:Rlin}\geq 1$ for sufficiently small $\eps>0$. Moreover, the right-hand side is independent of $N$ and hence
\begin{align*}
\sum_{k=\ell}^{\infty} \estd{k}{}^2\leq\c{opt:Rlin}\estd{\ell}{}^2.
\end{align*}
The remainder of the proof follows as in the proof of~\eqref{opt:eq:Rlin}.
  $\hfill\qed$
\end{proof}

\subsection{Data approximation and optimal rates}\label{opt:section:data:optimal}

The approximability norm now also contains the data approximation term $\data{\ell}{}$, i.e.,
\begin{align}\label{opt:approxnormd}
\norm{\estd{}{}}{\A_s}:=\sup_{N\in\N_0}(N+1)^s\big(\inf_{\mesh_\star\in \refine(\mesh_0)\atop \#\mesh_\star -\#\mesh_0\leq N} \estd{\star}{} \big)\in [0,\infty].
\end{align}

This allows us to formulate the following theorem.
\begin{theorem}\label{opt:thm:optimalityd}
Suppose~\eqref{opt:ass:stabled}--\eqref{opt:ass:reductiond} as well as discrete reliability~\eqref{opt:ass:dreld} and~\eqref{opt:ass:apriorid}--\eqref{opt:ass:qosumd}. Let the adaptivity parameter satisfy $\theta<\theta_0:=(1+\c{opt:stable}^2\c{opt:drel}^2)^{-1}$. Then, Algorithm~\ref{opt:algorithm} converges with the best possible rate in the sense
that for all $s>0$, it holds $\norm{\estd{}{}}{\A_s}<\infty$ if and only if
\begin{align}
\est{\ell}{}\leq \c{opt:optimality} (\#\mesh_\ell-\#\mesh_0)^{-s}\quad\text{for all }\ell\in\N,
\end{align}
where $\c{opt:optimality}>0$ depends only on the constants in~\eqref{opt:ass:stabled}--\eqref{opt:ass:qosumd} as well as on $\norm{\estd{}{}}{\A_s}$ and $\theta$.
\end{theorem}

The optimality of the marking criterion still holds with data approximation.
\begin{mylemma}\label{opt:lem:bulkchasingd}
Let the error estimator satisfy stability~\eqref{opt:ass:stabled} and discrete reliability~\eqref{opt:ass:dreld}. Then, there exists $0<\kappa_0<1$ such that any refinement $\mesh_\star\in\refine(\mesh_\ell)$ which satisfies
\begin{align}\label{opt:eq:kappad}
\estd{\star}{}^2\leq \kappa_0 \estd{\ell}{}^2,
\end{align}
fulfils the bulk chasing~\eqref{opt:bulkchasingd} in the sense
\begin{align}\label{opt:eq:bulkoptd}
\theta\estd{\ell}{}^2 \leq \sum_{\el\in\RR{\ell}{\star}}\estd{\ell}{\el}^2
\end{align}
for all $0\leq \theta<\theta_0$. The constant $\theta_0$ is defined in Theorem~\ref{opt:thm:optimalityd} whereas $\RR{\ell}{\star}\subseteq \mesh_\ell$ is defined in~\eqref{opt:ass:dreld}.
\end{mylemma}
\begin{proof}
The proof is identical to that of Lemma~\ref{opt:lem:bulkchasing}.
  $\hfill\qed$
\end{proof}

\begin{proof}[of Theorem~\ref{opt:thm:optimalityd}]
The proof combines only the previous results and is therefore identical to the proof of Theorem~\ref{opt:thm:optimalityd}.
  $\hfill\qed$
\end{proof}

The final lemma proves that the overall best rate is now determined  by the respective best rates for data approximation terms and for the non-perturbed problem.

\begin{mylemma}\label{opt:lem:approxchar}
 Suppose that for $s_1,s_2>0$, there holds
 \begin{align}\label{opt:eq:unperturbedbest}
  \sup_{N\in\N_0}(N+1)^{s_1}\big(\inf_{\mesh_\star\in \refine(\mesh_0)\atop \#\mesh_\star -\#\mesh_0\leq N} \est{\star}{} \big)<\infty
 \end{align}
as well as
\begin{align}\label{opt:eq:databest}
  \sup_{N\in\N_0}(N+1)^{s_2}\big(\inf_{\mesh_\star\in \refine(\mesh_0)\atop \#\mesh_\star -\#\mesh_0\leq N} \data{\star}{} \big)<\infty.
\end{align}
Then, this implies $ \norm{\estd{}{}}{\A_s}<\infty$ for $s:=\min\{s_1,s_2\}$. Conversely,  $ \norm{\estd{}{}}{\A_s}<\infty$ implies~\eqref{opt:eq:unperturbedbest}--\eqref{opt:eq:databest} with $s_1=s=s_2$.
\end{mylemma}
\begin{proof}
   The proof is technical and can be found in~\cite{ffkmp13},
   but essentially only relies on the overlay estimate~\eqref{dp:overlay}.
  $\hfill\qed$
\end{proof}

\subsection{Optimal convergence of weighted residual error estimator for the weakly singular integral equation with data approximation}\label{opt:sec:example2d}
As in Section~\ref{section:estred:resweaksingdata}, we consider the model problem from Proposition~\ref{prop:galerkin:dirichlet}, i.e.
\begin{align*}
  \form{\phi}{\psi}&:=\dual{V\phi}{\psi}\quad\text{for all }\phi,\psi\in\XX:=\widetilde
  H^{-1/2}(\Gamma),\\
  F(\psi)&:=\dual{(1/2+K)f}{\psi}\quad\text{for all }\phi\in\XX,
\end{align*}
where $f\in H^1(\Gamma)$. In contrast to Section~\ref{section:estred:resweaksingdata},
the data approximation is done via the Scott-Zhang operator
$J_\ell^{p+1}:=J_{\mesh_\ell}^{p+1}:\, L_2(\Gamma)\to \Sp^{p+1}(\mesh_\ell)$
from Lemma~\ref{lem:sz}. We define
\begin{align*}
  F_\ell(\psi):=\dual{(1/2+K)J_\ell^{p+1} f}{\psi}\quad\text{for all }\ell\in\N_0.
\end{align*}
With $\XX_\ell:=\Pp^p(\mesh_\ell)$, the discrete version~\eqref{opt:galerkind} reads: Find $\Phid_\ell\in\XX_\ell$ such that
\begin{align*}
 \form{\Phid_\ell}{\Psi}=F_\ell(\Psi)\quad\text{for all }\Psi\in\XX_\ell.
\end{align*}

There holds $J_\ell f\in H^1(\Gamma)$ and hence the standard\linebreak
weighted residual error estimator reads
\begin{align*}
\begin{split}
\est{\ell}{}^2&:=\sum_{\el\in\mesh_\ell}\est{\ell}{\el}^2\\
&:=\sum_{\el\in\mesh_\ell}h_\el \norm{\nabla_\Gamma(V\Phid_\ell-(1/2+K)J_\ell^{p+1} f)}{L_2(\el)}^2,
\end{split}
\end{align*}
where $\nabla_\Gamma(\cdot)$ denotes the surface gradient on $\Gamma$. 
The data approximation term is defined as
\begin{align*}
 \data{\ell}{}^2:=\sum_{\el\in\mesh_\ell} \data{\ell}{\el}^2:=\sum_{\el\in\mesh_\ell} h_\el\norm{(1-\pi_\ell^p)\nabla_\Gamma f}{L_2(\el)}^2.
\end{align*}
It is proved in Lemma~\ref{data:errstab:weaksing}, that $C_{\rm data}^{-1}\norm{\Phi_\ell-\Phid_\ell}{\XX}^2\leq \data{\ell}{}^2$, where the constant $C_{\rm data}=\c{data:errstab:weaksing:nvb}>0$ depends only on the polynomial degree $p$, on $\mesh_0$ (since Proposition~\ref{mesh:prop:nvb} states that Algorithm~\ref{alg:nvb} produces only finitely many different shapes of element patches), and on the shape regularity of $\mesh_\ell$.
Altogether, the extended error estimator reads
\begin{align}\label{opt:est:weaksing:res:data}
\begin{split}
  \estd{\ell}{}^2&=\norm{h_\ell^{1/2}\nabla_\Gamma(V\Phid_\ell-(1/2+K)J_\ell^{p+1} f)}{L_2(\Gamma)}^2\\
 &\qquad+\norm{h_\ell^{1/2}(1-\pi_\ell^p)\nabla_\Gamma f}{L_2(\el)}^2.
 \end{split}
\end{align}
For the abstract analysis of Section~\ref{opt:section:data}, we define an elementwise equivalent data approximation term, which is only of theoretical purpose and does not have to be computed at all.
This term is defined as
\begin{align*}
 \datad{\ell}{}^2=\norm{\widetilde{h}_\ell^{1/2}(1-\pi_\ell^p)\nabla_\Gamma f}{L_2(\el)}^2,
\end{align*}
where we exchanged the mesh-size function $h_\ell$ with the modified mesh-size function $\widetilde h_\ell$ from~\cite[Section~8]{axioms} in the data approximation term. This modified mesh-width function satisfies the following.
\begin{mylemma}\label{opt:lem:modmesh}
   The modified mesh-size function $\widetilde h_\ell \in \Pp^0(\mesh_\ell)$ from~\cite[Section~8]{axioms} satisfies for $\mesh_\star\in\refine(\mesh_\ell)$
 the following properties {\rm (i)--(iii)}:
 \begin{enumerate}
  \item[(i)] Elementwise equivalence: $C_{\rm h}^{-1}h_\ell|_T\leq \widetilde h_\ell|_T \leq h_\ell|_T$ for all $T\in\TT_\ell$;
  \item[(ii)] Monotonicity: $\widetilde h_{\star}|_T\leq \widetilde h_\ell|_T$ for all $T\in\TT_\ell$;
  \item[(iii)] Reduction: $\widetilde h_{\star}\leq q_{\rm h} \widetilde h_\ell$ on $\RR{\ell}{\star}$,
 \end{enumerate}
where $\RR\ell\star \supseteq \TT_\ell\setminus\TT_\star$ is defined as 
\begin{align}\label{opt:eq:rrdata}
 \begin{split}
  \RR{\ell}{\star}:=\{&\el\in\mesh_\ell\,:\,\text{exists } \el_1=\el,\el_2,\ldots,\el_6\\
  &\text{with } \el_6\in\mesh_\ell\setminus\mesh_\star\text{ and } \overline{\el}_j\cap\overline{\el}_{j+1}\neq \emptyset\\
  &\text{for all }j=0,\ldots,4\}
  \end{split}
\end{align}
which roughly means $\TT_\ell\setminus\TT_\star$ plus five additional layers of elements around $\TT_\ell\setminus\TT_\star$. The constants $C_{\rm h}>0$ and $0<q_{\rm h}<1$ depend only on the $\sigma_0$-shape regularity of $\TT_0$ and the space dimension $d$.
\end{mylemma}
Due to~(i) in the above Lemma~\ref{opt:lem:modmesh}, there holds obviously $\datad{\ell}{}\simeq \data{\ell}{}$.

\begin{mylemma}\label{opt:lem:weaksing:resd}
 The weighted residual error estimator $\estd{\ell}{}$ from \eqref{opt:est:weaksing:res:data} satisfies the assumptions~\eqref{opt:ass:stabled}--\eqref{opt:ass:qosumd}.
 The set $\RR{\ell}{\star}$ from~\eqref{opt:ass:drel} is defined in~\eqref{opt:eq:rrdata}.
 The constant $\c{opt:refined}$ depends only on the shape regularity of the mesh $\mesh_\ell$, whereas the constants $\c{opt:stable}$, $\c{opt:reduction}$, $q_{\rm red}$, $\c{opt:reliable}$, $\c{opt:drel}$, $\c{opt:qosum}$ depend additionally on $\Gamma$ and the polynomial degree $p$.
\end{mylemma}

\begin{proof}[of~\eqref{opt:ass:stabled}--\eqref{opt:ass:reductiond}]
 The assumptions~\eqref{opt:ass:stabled}--\eqref{opt:ass:reductiond} are proved very similar to the assumptions~\eqref{opt:ass:stable}--\eqref{opt:ass:reduction} for the unperturbed case in Lemma~\ref{opt:lem:weaksing:resd}. The proof can be found in great detail in~\cite{ffkmp13} and is only sketched
 in the following.
 
 The main difference to the proofs of~\eqref{opt:ass:stable}--\eqref{opt:ass:stable} in the unperturbed case of Lemma~\ref{opt:lem:weaksing:res} is that one obtains analogously to the proof of Lemma~\ref{data:eststab:weaksing} the additional term
 \begin{align*}
  \norm{J_\star f -J_\ell f}{H^{1/2}(\Gamma)}^2
 \end{align*}
on the right-hand side. 
The elementwise equivalence of $h_\ell$ and $\widetilde h_\ell$ from Lemma~\ref{opt:lem:modmesh} together with an approximation property of the Scott-Zhang projection proved in\linebreak
\cite[Lemma~2]{ffkmp13} imply
\begin{align*}
  \norm{J_\star f -J_\ell f}{H^{1/2}(\Gamma)}^2&\lesssim\norm{h_\ell^{1/2}(1-\pi_\ell)\nabla_\Gamma f}{L_2(\bigcup \RR{\ell}{\star})}^2\\
  &\simeq \norm{\widetilde h_\ell^{1/2}(1-\pi_\ell)\nabla_\Gamma f}{L_2(\bigcup \RR{\ell}{\star})}^2
\end{align*}
The fact that $\widetilde h_\star\leq q_{\rm h}\widetilde h_\ell$ on $\RR{\ell}{\star}$ together with the monotonicity of $\widetilde h_\ell$ show
\begin{align*}
 (1-q_{\rm h})\widetilde h_\ell|_{\RR{\ell}{\star}}\leq \widetilde h_\ell -\widetilde h_\star 
\end{align*}
pointwise almost everywhere on ${\RR{\ell}{\star}}$. Hence, we get
\begin{align}\label{opt:eq:ingredient}
\begin{split}
 \norm{\widetilde h_\ell^{1/2}&(1-\pi_\ell)\nabla_\Gamma f}{L_2(\bigcup \RR{\ell}{\star})}^2\\
 &=\int_{\bigcup \RR{\ell}{\star}} \widetilde h_\ell\big((1-\pi_\ell)\nabla_\Gamma f\big)^2\,dx\\
 &\leq (1-q_{\rm h})^{-1} \int_{\Gamma} (\widetilde h_\ell-\widetilde h_\star)\big((1-\pi_\ell)\nabla_\Gamma f\big)^2\,dx\\
 &\lesssim \norm{\widetilde h_\ell^{1/2}(1-\pi_\ell)\nabla_\Gamma f}{L_2(\Gamma)}^2 -\norm{\widetilde h_\star^{1/2}(1-\pi_\ell)\nabla_\Gamma f}{L_2(\Gamma)}^2\\
 &\leq \datad{\ell}{}^2-\datad{\star}{}^2.
 \end{split}
\end{align}
This result is the main ingredient of the proof.
  $\hfill\qed$
\end{proof}

\begin{proof}[of~\eqref{opt:ass:reliabled}]
 Define $\phi_\ell\in H^{-1/2}(\Gamma)$ as solution of~\eqref{opt:weaksing:continuous} when replacing the right-hand side with $F_\ell:=(1/2+K)J_\ell f$.
 By definition of $\data{\ell}{}$ and the reliability of the unperturbed problem~\eqref{opt:ass:reliable} from Lemma~\ref{opt:lem:weaksing:res}, there holds
  \begin{align*}
 \norm{\phi-\Phid_\ell}{\XX}&\leq \norm{\phi-\phid_\ell}{\XX}+\norm{\phid_\ell-\Phid_\ell}{\XX}\\
 &\lesssim \estd{\ell}{}+\data{\ell}{},
 \end{align*}
 where we used Lemma~\ref{data:errstab:weaksing} for the last estimate. This proves~\eqref{opt:ass:reliabled}.
  $\hfill\qed$
 \end{proof}

\begin{proof}[of~\eqref{opt:ass:dreld}]
The proof is very similar to the unperturbed case but additionally utilizes Lemma~\ref{data:errstab:weaksing}. It can be found in~\cite{ffkmp13}. 
  $\hfill\qed$
\end{proof}

\begin{proof}[of~\eqref{opt:ass:apriorid}]
The proof of a~priori convergence~\eqref{opt:ass:apriorid} for $F_\ell:=(1/2+K)J_\ell f$ uses the a~priori convergence of the Scott-Zhang projection proved in Lemma~\ref{estred:lem:szconv} in the sense
\begin{align*}
 \lim_{\ell\to\infty} J_\ell f = J_\infty f\in H^{1/2}(\Gamma).
\end{align*}
This and the stability of $K:\,H^{1/2}(\Gamma)\to H^{1/2}(\Gamma)$ imply for $F_\infty:=(1/2+K)J_\infty f$
\begin{align*}
 \norm{F_\infty-F_\ell}{H^{1/2}(\Gamma)}\lesssim \norm{J_\infty f -J_\ell f}{H^{1/2}(\Gamma)}\to 0
\end{align*}
as $\ell\to\infty$.

We define $\datad{\infty}{}:=\lim_{\ell\to\infty}\norm{\widetilde{h}_\ell^{1/2}(1-\pi_\infty^p)\nabla_\Gamma f}{L_2(\Gamma)}^2$,\linebreak
where $\pi_\infty^p\,:L_2(\Gamma)\to\XX_\infty$ is the $L_2$-orthogonal projection.
By definition of $\XX_\infty$, there holds $\lim_{\ell\to\infty}\datad{\ell}{}=\datad{\infty}{}$. This concludes the proof.
  $\hfill\qed$
\end{proof}

\begin{proof}[of~\eqref{opt:ass:qosumd}]
 Define $\phid_k\in H^{-1/2}(\Gamma)$ as solution of~\eqref{opt:weaksing:continuous} when replacing the right-hand side with $F_k:=(1/2+K)J_k f$.
 Then, there holds for all $k\in\N_0$ that $\form{\phi_{k+1}-\Phid_{k+1}}{\Phid_{k+1}-\Phid_k}=0$ and hence
 \begin{align*}
  \norm{\Phid_{k+1}-\Phid_k}{\slp}^2 &= \norm{\phid_{k+1}-\Phid_k}{\slp}^2 \\
  &\qquad- \norm{\phid_{k+1}-\Phid_{k+1}}{\slp}^2.
 \end{align*}
Young's inequality $(a+b)^2\leq (1+\delta)a^2+(1-\delta^{-1})b^2$ for all $a,b\in\R$ and $\delta>0$ shows
\begin{align*}
  \norm{\Phid_{k+1}&-\Phid_k}{\slp}^2 \\
  &\leq (1+\delta)\norm{\phid_k-\Phid_k}{\slp}^2 - \norm{\phid_{k+1}-\Phid_{k+1}}{\slp}^2\\
  &\quad+ (1+\delta^{-1})\norm{\phid_{k+1}-\phid_k}{\slp}^2.
\end{align*}
The stability of the problem~\eqref{intro:weakform} implies $\norm{\phid_{k+1}-\phid_k}{\slp}^2\simeq \norm{\phid_{k+1}-\phid_k}{\slp}^2\lesssim \norm{J_{k+1}f - J_k f}{H^{1/2}(\Gamma)}$ and as in the proof of~\eqref{opt:ass:stabled}--\eqref{opt:ass:reductiond} above, we see
\begin{align*}
 \norm{\phid_{k+1}-\phid_k}{\slp}^2\lesssim \datad{k}{}^2-\datad{k+1}{}^2.
\end{align*}
As in the proof of~\eqref{opt:ass:reliabled}, one shows that $\norm{\phid_k-\Phid_k}{\slp}\leq \c{opt:reliable}\estd{k}{}$.
Altogether, this shows for $\eps\c{opt:reliable}^{-1}=\delta$
\begin{align*}
 \sum_{k=\ell}^\infty &(\norm{\Phid_{k+1}-\Phid_k}{\slp}^2-\eps\estd{k}{}^2)\\
 &\leq
  \sum_{k=\ell}^\infty \big(\norm{\Phid_{k+1}-\Phid_{k}}{\slp}^2-\delta\norm{\phid_k-\Phid_k}{\slp}^2\big)\\
  &\lesssim   
   \sum_{k=\ell}^\infty \big(\norm{\phid_{k}-\Phid_k}{\slp}^2 - \norm{\phid_{k+1}-\Phid_{k+1}}{\slp}^2\\
  &\quad+ \datad{k}{}^2-\datad{k+1}{}^2\big).
\end{align*}
The telescoping series reveals
\begin{align*}
 \sum_{k=\ell}^\infty (\norm{\Phid_{k+1}&-\Phid_k}{\slp}^2-\eps\estd{k}{}^2)\\
 &\lesssim
 \norm{\phid_{\ell}-\Phid_\ell}{\slp}^2 + \datad{\ell}{}^2
 \lesssim \estd{\ell}{}^2.
\end{align*}
This concludes the proof.
  $\hfill\qed$
\end{proof}

\begin{theorem}\label{opt:thm:weaksing:resd}
For $0<\theta\leq 1$, Algorithm~\ref{opt:algorithm} with the residual error estimator $\estd{\ell}{}$ from~\eqref{opt:est:weaksing:res:data} converges in the sense of
\begin{align}\label{opt:weaksing:res:rlind}
\c{opt:reliable}^{-2}\norm{\phi-\Phid_{\ell+n}}{H^{-1/2}(\Gamma)}^2\leq \estd{\ell+n}{}^2\leq \c{opt:Rlin}q_{\rm R}^n \estd{\ell}{}^2
\end{align}
for all $\ell,n\in\N_0$. For $0<\theta<\theta_0$, Algorithm~\ref{opt:algorithm} converges with the best possible rate $s>0$ in the sense that
$\norm{\estd{}{}}{\A_s}<\infty$ if and only if
\begin{align}
\estd{\ell}{}\leq \c{opt:optimality} (\#\mesh_\ell-\#\mesh_0)^{-s}\quad\text{for all }\ell\in\N,
\end{align}
where the constants $\c{opt:Rlin},q_{\rm R}$ depend only on $\Gamma$, the shape regularity of the meshes $\mesh_\ell$, the polynomial degree $p$, and $\theta$. The constant $\c{opt:optimality}>0$ depends additionally on $\norm{\estd{}{}}{\A_s}$.
\end{theorem}
\begin{proof}
 Lemma~\ref{opt:lem:weaksing:resd} shows that the assumption~\eqref{opt:ass:stabled}--\eqref{opt:ass:qosumd} are satisfied. Theorem~\ref{opt:thm:convergenced} and Theorem~\ref{opt:thm:optimalityd} prove the statements.
  $\hfill\qed$
\end{proof}

\subsection{Optimal convergence of weighted residual error estimator for the hypersingular integral equation with data approximation}\label{opt:sec:example4d}
As in Section~\ref{section:estred:reshypsingdata}, we consider the model problem from Proposition~\ref{prop:galerkin:neumann}, i.e.
\begin{align*}
 \form{u}{v}&:=\dual{Wu}{v}_\Gamma + \dual{u}{1}_\Gamma\dual{v}{1}_\Gamma\quad\text{for all }u,v\in\XX,\\
  F(v)&:=\dual{(1/2-K^\prime)\phi}{v}\quad\text{for all }v\in\XX:= H^{1/2}(\Gamma),
\end{align*}
where $\phi\in L_2(\Gamma)$. The data approximation is done via the $L_2$-orthogonal projection $\pi_\ell:=\pi^{p-1}_{\mesh_\ell}:\, L_2(\Gamma)\to \Pp^{p-1}(\mesh_\ell)$. We define
\begin{align*}
 F_\ell(\psi):=\dual{(1/2-K^\prime)\pi_\ell \phi}{\psi}\quad\text{for all }\ell\in\N_0.
\end{align*}
With $\XX_\ell:=\Sp^p(\mesh_\ell)$ for $p\geq 1$, the discrete version~\eqref{opt:galerkind} of~\eqref{opt:hypsing:continuous} reads: Find $\Ud_\ell\in\XX_\ell$ such that
\begin{align*}
 \form{\Ud_\ell}{V}=F_\ell(V)\quad\text{for all }V\in\XX_\ell.
\end{align*}

The standard weighted residual error estimator reads
\begin{align*}
\est{\ell}{}^2&:=\sum_{\el\in\mesh_\ell}\est{\ell}{\el}^2\\
&:=\sum_{\el\in\mesh_\ell}h_\el \norm{W\Ud_\ell-(1/2-K^\prime)\pi_\ell \phi)}{L_2(\el)}^2.
\end{align*}
The data approximation term is defined as
\begin{align*}
 \data{\ell}{}^2:=\sum_{\el\in\mesh_\ell} \data{\ell}{\el}^2:=\sum_{\el\in\mesh_\ell} h_\el\norm{(1-\pi_\ell)\phi}{L_2(\el)}^2,
\end{align*}
with $\datad{\ell}{}=\data{\ell}{}$.
Lemma~\ref{data:errstab:hypsing} shows that $C_{\rm data}^{-1}\norm{U_\ell-\Ud_\ell}{\XX}^2\leq \data{\ell}{}^2$, where the constant $C_{\rm data}=\c{data:errstab:hypsing}>0$ depends only on the polynomial degree $p$ and on the shape regularity of $\mesh_\ell$.
Altogether, the extended error estimator reads
\begin{align}\label{opt:est:hypsing:res:data}
\begin{split}
 \estd{\ell}{}^2&=\norm{h_\ell^{1/2}(W\Ud_\ell-(1/2-K^\prime)\pi_\ell \phi)}{L_2(\Gamma)}^2\\
 &\qquad+\norm{h_\ell^{1/2}(1-\pi_\ell)\phi}{L_2(\el)}^2.
 \end{split}
\end{align}

\begin{mylemma}\label{opt:lem:hypsing:resd}
  The weighted residual error estimator $\estd{\ell}{}$ from \eqref{opt:est:hypsing:res:data}
  satisfies the assumptions~\eqref{opt:ass:stabled}--\eqref{opt:ass:qosumd}.
  The set $\RR{\ell}{\star}$ from~\eqref{opt:ass:dreld} satisfies $\RR{\ell}{\star}:=\mesh_\ell\setminus\mesh_\star$ and $\c{opt:refined}=1$. The constants $\c{opt:stable}$, $\c{opt:reduction}$, $q_{\rm red}$, $\c{opt:reliable}$, $\c{opt:drel}$, $\c{opt:qosum}$ depend on $\Gamma$, the shape regularity of $\mesh_\ell$, and the polynomial degree $p$.
\end{mylemma}
\begin{proof}
  The assumptions~\eqref{opt:ass:stabled}--\eqref{opt:ass:reliabled} are straightforward to\linebreak
  prove.
  A detailed proof is found in~\cite{ffkmp13-A}.
  The proof of the discrete reliability is very similar to the unperturbed case and is also found in~\cite{ffkmp13-A}. To see~\eqref{opt:ass:apriorid}, define
 $\Pp^p(\mesh_\infty):=\overline{\bigcup_{\ell\in\N_0}\Pp^p(\mesh_\ell)}\subseteq L_2(\Gamma)$ and $F_\infty(v):=\dual{(1/2-K^\prime)\pi_\infty \phi}{v}_\Gamma$, where $\pi_\infty:\,L_2(\Gamma)\to\Pp^p(\mesh_\infty)$ denotes the $L_2$-orthogonal projection. By definition and with the stability of
 $K^\prime:\linebreak H^{-1/2}(\Gamma)\to H^{-1/2}(\Gamma)$, there holds
 \begin{align*}
 \norm{F_\infty-F_\ell}{\XX_\ell^\prime}&\leq\norm{(1/2-K^\prime)(\pi_\infty-\pi_\ell)\phi}{H^{-1/2}(\Gamma)}\\
 &\lesssim 
 \norm{(\pi_\infty-\pi_\ell)\phi}{L_2(\Gamma)}^2.
 \end{align*}
The term on the right-hand side tends to zero as $\ell\to\infty$. The convergence of $\datad{\ell}{}$ follows as in the weakly singular case in the proof of Lemma~\ref{opt:lem:weaksing:resd}. This shows~\eqref{opt:ass:apriorid}.
Finally,~\eqref{opt:ass:qosumd} follows analogously to the proof for the weakly singular case in Lemma~\ref{opt:lem:weaksing:resd}. 
With $u_k\in H^{1/2}(\Gamma)$ the solution of~\eqref{opt:hypsing:continuous} with right-hand side $F_k:=(1/2-K^\prime)\pi_k \phi$, the only difference is the proof of $\norm{u_{k+1}-u_k}{H^{1/2}(\Gamma)}\lesssim\data{k}{}^2-\data{k+1}{}^2$, which is much easier now. By ellipticity of $\form{\cdot}{\cdot}$, there holds
\begin{align*}
 \norm{u_{k+1}-u_k}{H^{1/2}(\Gamma)}&\lesssim\norm{(\pi_{k+1}-\pi_k)\phi}{H^{-1/2}(\Gamma)} \\
 &= 
 \norm{(1-\pi_k)\pi_{k+1}\phi}{H^{-1/2}(\Gamma)}.
\end{align*}
The approximation property of the $L_2$-orthogonal projection $\pi_k$ (see Lemma~\ref{lem:L2:Pp:apx}) implies
\begin{align*}
 \norm{(1-\pi_k)\pi_{k+1}\phi}{H^{-1/2}(\Gamma)}&\lesssim\norm{h_\ell^{1/2}(\pi_{k+1}-\pi_k)\phi}{L_2(\Gamma)}\\
 &=\norm{h_\ell^{1/2}(\pi_{k+1}-\pi_k)\phi}{L_2(\bigcup(\TT_k\setminus\TT_{k+1})},
\end{align*}
where we used the elementwise definition of $\pi_k$ and $\pi_{k+1}=\pi_k$ on $\TT_k\cap\TT_{k+1}$. There exists a constant $0<q<1$ which depends on the space dimension $d$, such that $h_{k+1}|_T\leq q h_k|_T$ for all $T\in\RR{k}{k+1}=\TT_k\setminus\TT_{k+1}$. Hence there holds
\begin{align*}
 (1-q)h_k|_{\RR{k}{k+1}}\leq h_k-h_{k+1}
\end{align*}
and therefore
\begin{align*}
 \norm{h_\ell^{1/2}(\pi_{k+1}-\pi_k)\phi}{L_2(\bigcup(\RR{k}{k+1})}^2\lesssim \data{k}{}^2-\data{k+1}{}^2.
\end{align*}
The remainder follows analogously to the proof for the weak\-ly singular case in Lemma~\ref{opt:lem:weaksing:resd}.
  $\hfill\qed$
\end{proof}

\begin{theorem}\label{opt:thm:hypsing:resd}
For all $0<\theta\leq 1$, Algorithm~\ref{opt:algorithmd} with the residual error estimator $\estd{\ell}{}$ from~\eqref{opt:est:hypsing:res:data} converges in the sense
\begin{align}\label{opt:hypsing:res:rlind}
\c{opt:reliable}^{-2}\norm{u-\Ud_{\ell+n}}{H^{1/2}(\Gamma)}^2\leq \estd{\ell+n}{}^2\leq \c{opt:Rlin}q_{\rm R}^n \estd{\ell}{}^2
\end{align}
for all $\ell,n\in\N_0$. For $0<\theta<\theta_0$, Algorithm~\ref{opt:algorithmd} converges with the best possible rate $s>0$ in the sense that
$\norm{\estd{}{}}{\A_s}<\infty$ if and only if
\begin{align}
\estd{\ell}{}\leq \c{opt:optimality} (\#\mesh_\ell-\#\mesh_0)^{-s}\quad\text{for all }\ell\in\N,
\end{align}
where the constants $\c{opt:Rlin},q_{\rm R}$ depend only on $\Gamma$, the shape regularity of the meshes $\mesh_\ell$, the polynomial degree $p$, and $\theta$. The constant $\c{opt:optimality}>0$ depends additionally on $\norm{\estd{}{}}{\A_s}$.
\end{theorem}
\begin{proof}
 Lemma~\ref{opt:lem:hypsing:resd} shows that the assumption~\eqref{opt:ass:stabled}--\eqref{opt:ass:qosumd} are satisfied. Theorem~\ref{opt:thm:convergenced} and Theorem~\ref{opt:thm:optimalityd} prove the statements.
  $\hfill\qed$
\end{proof}
\section{Implementational details}\label{section:implementation}
This section deals with implementational issues for the $L_2$-orthogonal
projection $\pi_\ell^p$ and the
Scott-Zhang operator $J_\ell$ as well as for some of the error
estimators discussed above. For the ease of presentation, we consider only $d=2$ and
give precise examples for the lowest-order cases $p\in\{0,1\}$.
However, the following considerations are elementary and most of the ideas directly transfer
to higher-order discretizations as well as $d\geq 3$.
\subsection{Implementation of the $L_2$-orthogonal projection $\pi_\ell^p$}\label{section:implementation:l2proj}
Let $\{\Psi_j^{\rm ref}\}_{j=1}^{m} \in\Pp^p(\elref)$ denote a basis of the
space of piecewise polynomials on the reference element $\elref$ with
\begin{align}
  m := \dim(\Pp^p(\elref)) 
  = \begin{cases}
    p+1 & d=2, \\
    \frac12(p+1)(p+2) & d=3.
  \end{cases}
\end{align}
Define the mass matrix $\massmat_{\elref}\in\R^{m\times m}$ associated to
the reference element $\elref$ by 
\begin{align}
  (\massmat_{\elref})_{jk} = \dual{\Psi_k^{\rm ref}}{\Psi_j^{\rm ref}}_{\elref}.
\end{align}
Recall the affine mapping $F_\el : \elref \to \el$ from
Section~\ref{section:bem:discrete}, which maps the reference
element $\elref$ to $\el\in\mesh_\ell$.
Define the basis $\{\Psi_j^\el\}_{j=1}^m$ of $\Pp^p(\el)$ by $\Psi_j^\el \circ
F_\el = \Psi^{\rm ref}_j$. Let $\massmat_\el \in\R^{m\times m}$ denote the local
mass matrix with entries
\begin{align}
  (\massmat_\el)_{jk} := \dual{\Psi_k^\el}{\Psi_j^\el}_\el.
\end{align}
By using the transformation $F_\el$, the computation of the entries
$(\massmat_\el)_{jk}$ can be reduced to the computation of
$(\massmat_{\elref})_{jk}$, i.e.,
\begin{align}
\begin{split}
  \dual{\Psi_k^\el}{\Psi_j^\el}_\el &= \int_\el \Psi_k^\el \Psi_j^\el \,dx 
  = \sqrt{\det(B_\el^T B_\el)}\int_{\elref} \Psi^{\rm ref}_k\Psi^{\rm ref}_j
  \,dy  \\ &= \sqrt{\det(B_\el^T B_\el)}
  \dual{\Psi^{\rm ref}_k}{\Psi^{\rm ref}_j}_{\elref}.
\end{split}
\end{align}
Hence,
\begin{align}
  \massmat_\el = \sqrt{\det(B_\el^T B_\el)} \massmat_{\elref} =
  \frac{|\el|}{|\elref|} \massmat_{\elref}.
\end{align}
Note that $\massmat_{\elref}$ can be computed analytically.
This is useful in practice, as we have to compute $\massmat_{\elref}$ only once
and then multiply it by $|\el|/|\elref|$ to get $\massmat_\el$ for
each element $\el\in\mesh_\ell$.

Let $g\in L_2(\Gamma)$ and let $\pi_\ell^p : L_2(\Gamma) \to \Pp^p(\mesh_\ell)$
denote the $L_2$-orthogonal projection onto $\Pp^p(\mesh_\ell)$. Then, $\pi_\ell^p
g$ satisfies
\begin{align}\label{eq:implementation:l2proj}
  \dual{\pi_\ell^p g}{\Psi}_\Gamma = \dual{g}{\Psi}_\Gamma \quad\text{for all }
  \Psi\in\Pp^p(\mesh_\ell).
\end{align}
A basis $\{\Psi_j\}_{j=1}^\dimPp$ with $\dimPp:=m \, \#\mesh_\ell$ is given by
the combination of all basis functions $\Psi_k^\el$ for each element
$\el\in\mesh_\ell$.
Define the mass matrix $\massmat \in \R^{\dimPp\times \dimPp}$ by 
\begin{align}
  \massmat_{jk} = \dual{\Psi_k}{\Psi_j}_\Gamma.
\end{align}
Then,~\eqref{eq:implementation:l2proj} is equivalent to 
\begin{align}
  \massmat \mathbf{x} = \mathbf{g} \quad\text{with } \mathbf g_j :=
  \dual{g}{\Psi_j}_\Gamma,
\end{align}
where $\pi_\ell^p g = \sum_{j=1}^\dimPp \mathbf x_j \Psi_j$.
However, a simple calculation shows that $\pi_\ell^p$ is local in the sense that
\begin{align}
  (\pi_\ell^p) g|_\el = \pi_{\ell,\el}^p g =  \pi_{\ell,\el}^p (g|_\el),
\end{align}
where $\pi_{\ell,\el}^p$ denotes the $L_2$-orthogonal projection onto $\Pp^p(\el)$.
This allows us to reduce the computation of the orthogonal projection $\pi_\ell^p
g$ to the solution of the local problems
\begin{align}
  \massmat_\el \mathbf x_\el = \mathbf g_\el \quad\text{for all }\el\in\mesh_\ell,
\end{align}
where $(\mathbf g_\el)_j = \dual{g}{\Psi_j^\el}_\el$ and $\pi_{\ell,\el}^p g =
\sum_{j=1}^m (\mathbf x_\el)_j \Psi_j^\el$.
The coefficients $\dual{g}{\Psi_j}_\el$ can be computed by use of, e.g., Gaussian
quadrature rules, cf.~\cite{stroud}, as follows: First, we transform the
integral over $\el$ to the reference element $\elref$. Then, we apply an appropriate
quadrature rule of order $q\in\N$ with weights $\{w_i\}_{i=1}^q$ and evaluation points
$\{y_i\}_{i=1}^q \subseteq \elref$, i.e.
\begin{align}
  \dual{g}{\Psi_j^\el}_\el = \int_{\elref} g\circ F_\el \Psi^{\rm ref}_j \,dy
  \approx \sum_{j=1}^q w_i (g\circ F_\el)(y_i) \Psi^{\rm ref}_j(y_i).
\end{align}

We give some examples for the matrix $\massmat_\el$ for $d=2,3$ and $p=0,1$. 
Let $d=2$ with the reference element $\elref = (-1,1)$. For $p=0$ and the basis
function $\Psi^{\rm ref}_1(x) = 1$ for $x\in\elref$, we get
\begin{align}
  \massmat_\el = |T|.
\end{align}
For $d=2$ and $p=1$ with the basis
\begin{align*}
  \Psi^{\rm ref}_1(x) = 1 \quad\text{and}\quad \Psi^{\rm ref}_2(x) = x
  \quad\text{for } x\in\elref,
\end{align*}
the local mass matrix reads
\begin{align}
  \massmat_\el = |T| \begin{pmatrix}
    1 & 0 \\
    0 & \tfrac13
  \end{pmatrix}.
\end{align}

Let $d=3$ with reference element given by
\begin{align*}
  \elref = {\rm conv}\{ (0,0),(1,0),(0,1)\}.
\end{align*}
For $p=0$ with basis $\Psi^{\rm ref}_1(x,y) = 1$ for $(x,y)\in\elref$, we have
\begin{align}
  \massmat_\el = |T|.
\end{align}
For $p=1$ with basis
\begin{align*}
  \Psi^{\rm ref}_1(x,y) = 1, \quad \Psi^{\rm ref}_2(x,y) = x, \quad
  \Psi^{\rm ref}_3(x,y) = y
\end{align*}
for $(x,y)\in\elref$, we get
\begin{align}
  \dual{\Psi_k^\el}{\Psi_j^\el}_\el = 2|T| \int_0^1 \int_0^{1-x}
  \Psi^{\rm ref}_k\Psi^{\rm ref}_j \,dy\,dx.
\end{align}
Thus, the local mass matrix reads
\begin{align}
  \massmat_\el = \frac{|T|}{12} \begin{pmatrix}
    12 & 4 & 4 \\
    4 & 2 & 1 \\
    4 & 1 & 2
  \end{pmatrix}.
\end{align}

\subsection{Implementation of the Scott-Zhang projection $J_\mesh$}\label{section:implementation:sz}
The implementation of the Scott-Zhang projection defined in Section~\ref{section:sz} requires the computation of an $L_2$-dual basis
and a numerical integration.
We stick with the setting and notations of Section~\ref{section:sz}.
Suppose that $\el_i\in\mesh$ is the element chosen for the computation
of $J_\mesh v$ at the node $z_i$ and $\left\{ \phi_{i,j} \right\}_{j=1}^d$ is 
the nodal basis of $\Pp^1(\el_i)$. 
Recall the definition
\begin{align*}
  \int_{\el_i} \psi_{i,k}\phi_{i,j}\,dx = \delta_{k,j}
\end{align*}
of the dual basis $\{\psi_{i,k}\}_{k=1}^d$ from~\eqref{eq:dualbasis}.
Hilbert space theory predicts the representation
\begin{align*}
  \psi_{i,k} = \sum_{m=1}^d a_{k,m}^{(i)} \phi_{i,m} \quad\text{with}\quad a_{k,m}^{(i)}\in\R.
\end{align*}
Using the duality~\eqref{eq:dualbasis}, the
coefficients $a_{k,m}^{(i)}$ can be computed by solving the $d$ systems of $d\times d$ linear equations
\begin{align*}
  \sum_{m=1}^d a_{k,m}^{(i)} \int_{\el_i}\phi_{i,m}(x)\phi_{i,j} = \delta_{k,j}.
\end{align*}
Denoting by $\mathbf{A}^{(i)}\in\R^{d\times d}$ the matrix with
$\mathbf{A}^{(i)}_{k,m} = a_{k,m}^{(i)}$, this yields
\begin{align*}
  \mathbf{A}^{(i)} &= \abs{\el_i}^{-1}
  \begin{pmatrix}
    4 & -2\\
    -2 & 4
  \end{pmatrix}
  \text{ for } d=2,\\
  \mathbf{A}^{(i)} &= \abs{\el_i}^{-1}
  \begin{pmatrix}
    18 & -6 & -6 \\
    -6 & 18 & -6 \\
    -6 & -6 & 18
  \end{pmatrix}
  \text{ for } d=3.
\end{align*}
Let $\psi_i$ denote the dual basis function $\psi_{i,k}$ with $k\in\left\{ 1, \dots, d \right\}$ such that
$\phi_{i,k}(z_i) = 1$. Hence,
\begin{align*}
  \psi_i = \sum_{m=1}^d \mathbf{A}_{k,m}^{(i)} \phi_{i,m}.
\end{align*}
To compute $\int_{\el_i}\psi_i v\,dx$, standard quadrature rules on intervals respectively 
triangles can be used,
cf.~\cite{stroud}.

\subsection{Assumptions on uniform refinement for $d=2$}
We suppose that the uniform refinement $\widehat \mesh_\ell$ of $\mesh_\ell$ is
obtained by splitting each element $T\in\mesh_\ell$ into $k$ sons
$T'\in\widehat\mesh_\ell$ for some fixed $k\geq 2$. A natural approach for $d=2$ employs $k=2$
and ensures $h_{\widehat\mesh_\ell} = h_{\mesh_\ell}/2$ for the respective mesh-size
functions $h_\ell,\widehat h_\ell : \Gamma\to\R$ with, e.g.,
$h_\ell|_T = h_T = \diam(T)$ for all $T\in\mesh_\ell$.
For the remainder of this section, let $\el^+,\el^- \in\widehat\mesh_\ell$ denote
the unique elements with $\overline\el^+\cup\overline\el^- = \overline\el$ for
all $T\in\mesh_\ell$.

\subsection{Two-level estimator}
Recall the hierarchical two-level decomposition
\begin{align}\label{eq:implementation:decomposition}
  \widehat\XX_\mesh = \XX_\mesh \oplus \ZZ_\mesh
\end{align}
from~\eqref{2level}, where $Z$ is further decomposed into
\begin{align}
  \ZZ_\mesh = \ZZ_{\mesh,1} \oplus \cdots \oplus \ZZ_{\mesh,L} \quad\text{with } \dim(\ZZ_{\mesh,i}) = 1.
\end{align}
Let $\Psi_j\neq 0$ denote an appropriate element of $\ZZ_{\mesh,j}$. The basis functions $\Psi_j$, hence $\ZZ_{\mesh,j}$, will be specified accordingly for the
weakly singular integral equation as well as the hypersingular integral equation
later on.
Note that $\widehat\XX_\mesh = \XX_{\widehat\mesh}$ denotes the space
associated to the uniformly refined mesh $\widehat\mesh$.

The computation of the local indicators
\begin{align}
  \est{j}{} = \norm{P_j(\widehat U - U)}{b},
\end{align}
where $b(\cdot,\cdot)$ denotes the bilinear form corresponding to the
weakly singular or hypersingular integral equation, requires the representation
of the projection operators $P_j$. 
Since $\ZZ_{\mesh,j}$ is one-dimensional, the subsequent identity
follows immediately from the definition of $P_j$,
\begin{align}
  P_j \widehat V = \frac{b(\widehat V,\Psi_j)}{\norm{\Psi_j}{b}^2} \Psi_j
  \quad\text{for all } \widehat V \in \widehat \XX_\mesh.
\end{align}
With $\widehat V = \widehat U-U$, Lemma~\ref{lem:ASMprop} shows
\begin{align}
  \est{j}{} = \frac{|f(\Psi_j) - b(U,\Psi_j)|}{\norm{\Psi_j}{b}^2}.
\end{align}
The computation of $\est{j}{}$ involves the assembling of the Galerkin matrix
corresponding to $b(\cdot,\cdot)$ with ansatz space $\XX_\mesh$ and test function
$\Psi_j$ as well as the computation of $b(\Psi_j,\Psi_j)$.
Therefore, we have to compute the entries
\begin{align*}
  b(V_k,\Psi_j) \quad\text{as well as}\quad b(\Psi_j,\Psi_j),
\end{align*}
where $\{V_k\}_{k=1}^{\dim(\XX_\mesh)}$ denotes a basis of $\XX_\mesh$.
We note that this
quantities can be obtained from the Galerkin matrix with respect to the fine
space $\widehat\XX_\mesh$.
\subsubsection{2D weakly singular integral equation}
We consider $\XX_\mesh = \Pp^0(\mesh)$ and $\widehat\XX_\mesh = \XX_{\widehat\mesh} = \Pp^0(\widehat\mesh)$.
For each $T_j \in\mesh$, let $T_{j}^\pm \in \widehat\mesh$ denote the two elements
with $\overline T_{j}^+ \cup \overline T_{j}^- = \overline T_j$ and $|T_{j}^\pm| =
|T_j|/2$.
Define the basis function $\Psi_j\in\widehat \XX_\mesh$ by
\begin{align}
  \Psi_j|_{T_{j}^+} = +1 \quad \Psi_j|_{T_{j}^-} = -1 \quad \Psi_j|_{\Gamma
  \backslash\overline T_j} = 0.
\end{align}
It holds
\begin{align}
  \dual{\Psi_j}1_T = 0.
\end{align}
The stability of the corresponding
decomposition~\eqref{eq:implementation:decomposition} is proved
in~\cite{hms01} resp.~\cite{effp09} for $d=2$ and~\cite{eh06} for $d=3$.

\subsubsection{2D hypersingular integral equation}
We consider $\XX_\mesh = \Sp^1(\mesh)$ and $\widehat \XX_\mesh = \Sp^1(\widehat\mesh)$.
For each $T_j \in\mesh$, let $T_{j}^\pm \in \widehat\mesh$ denote the two elements
with $\overline T_{j}^+ \cup \overline T_{j}^- = \overline T$ and $|T_{j}^\pm| =
|T_j|/2$. 
Define the midpoint $m_j = \overline T_{j}^+\cap \overline T_{j}^-$ and the
basis function $\Psi_j\in\widehat \XX_\mesh$ by
\begin{align}
  \Psi_j(m_j) = 1 \quad\text{and}\quad \Psi_j|_{\Gamma\backslash \overline T_j}
  = 0.
\end{align}
The stability of the corresponding
decomposition~\eqref{eq:implementation:decomposition} is proved
in~\cite{hms01} resp.~\cite{effp09} for $d=2$ and~\cite{h02}
resp.~\cite{affkp13} for $d=3$.
\subsection{$(h-h/2)$ error estimators in 2D}\label{section:implementation:hh2}
For the $(h-h/2)$-based error estimators $\hhhalfweak_\ell$,
$\hhhalfweaktilde_\ell$ from Section~\ref{section:est:hh2}, we have to compute local
refinement indicators of the form
\begin{align*}
  h_T \norm{\widehat \Psi}{L_2(T)}^2 \quad\text{for all } T\in\mesh_\ell 
\end{align*}
with $\widehat \Psi \in \Pp^p(\widehat\mesh_\ell)$ and $h_T = \diam(T)$.

Let $T^\pm \in \widehat \mesh_\ell$ denote the son elements of the
father $T\in\mesh_\ell$, i.e., $\overline T = \overline T^+ \cup \overline T^-$.
We define the local mesh $\localmesh$ as the restriction of $\widehat\mesh_\ell$
to $T$, i.e., $\localmesh := \{T^+,T^-\}$.
Let $\{\widehat\Psi_j\}_{j=1}^{2(p+1)}$ denote a basis of the local subspace
$\Pp^p(\localmesh)$. With the representation of $\widehat \Psi
\in\Pp^p(\localmesh)$ on the father element $T\in\mesh_\ell$
\begin{align}
  \widehat \Psi|_T = \sum_{j=1}^{2(p+1)} \alpha_j \widehat\Psi_j,
\end{align}
and the local mass matrix $\massmathat_T \in \R_{\rm sym}^{2(p+1)\times 2(p+1)}$ with
entries
\begin{align}
  (\massmathat_T)_{jk} = \dual{\widehat\Psi_j}{\widehat\Psi_k}_{T},
\end{align}
the computation of the local indicators read
\begin{align}\label{eq:implementation:localind}
  h_T \norm{\widehat \Psi}{L_2(T)}^2 = h_T \alpha^T \massmathat_T \alpha.
\end{align}
In the following, this observation is employed for the weakly singular and hypersingular
integral equation. Note that the estimators for the hypersingular integral
equation require the computation of the arclength derivative $\nabla_\Gamma(\cdot)$ of a discrete
function.

\subsubsection{2D weakly singular integral equation}\label{section:implementation:hh2:weak}
We consider the local refinement indicators of the estimator
\begin{align}
  \hhhalfweak_\ell^2 = \sum_{T\in\mesh_\ell} h_T \norm{\widehat
  \Phi_\ell-\Phi_\ell}{L_2(T)}^2 =:
  \sum_{T\in\mesh_\ell} h_T \norm{\widehat \Psi}{L_2(T)}^2,
\end{align}
where $\Phi_\ell \in \Pp^p(\mesh_\ell)$ and $\widehat\Phi_\ell
\in\Pp^p(\widehat\mesh_\ell)$ are the respective Galerkin solutions. 

For the lowest-order case $p=0$, we choose the characteristic functions as
basis, i.e., $\widehat\Psi_i$ is the characteristic function on the element $T_i \in
\localmesh$.
The corresponding local mass matrix reads $\massmathat_T = h_T/2 \, \idmat$, where $\idmat$ denotes
the $2\times 2$ identity matrix.
Hence, the computation of $\hhhalfweak_\ell^2(T) = h_T \norm{\widehat
\Phi_\ell-\Phi_\ell}{L_2(T)}^2$ reads
\begin{align}
\begin{split}
  \hhhalfweak_\ell^2(T) &= h_T \alpha^T\cdot \massmathat_T \alpha = h_T^2/2 \,
  \alpha^T\cdot \idmat \alpha \\
  &= \frac{h_T^2}2 \Big( (\widehat
  \Phi_\ell|_{T^+} - \Phi_\ell|_T)^2 + (\widehat
  \Phi_\ell|_{T^-} - \Phi_\ell|_T)^2 \Big).
\end{split}
\end{align}

The computation of the local refinement indicators
\begin{align}\label{eq:implementation:mutildeslp}
  \hhhalfweaktilde_\ell^2(T) = h_T \norm{\widehat\Phi_\ell-
  \pi_\ell\widehat\Phi_\ell}{L_2(T)}^2
\end{align}
of the estimator $\hhhalfweaktilde_\ell$ is done in the same manner.
For the computation of $\pi_\ell\widehat\Phi_\ell$ we proceed as in
Section~\ref{section:implementation:l2proj}.
Let $\Psi_T$ denote the characteristic function on $T\in\mesh_\ell$. It holds
$\dual{\widehat\Phi_\ell}{\Psi_T}_T = h_T/2(\widehat \Phi_\ell|_{T^+} +
\widehat\Phi_\ell|_{T^-})$. Hence, 
\begin{align}
  \pi_\ell\widehat\Phi_\ell|_T = \frac12 (\widehat\Phi_\ell|_{T^+} +
  \widehat\Phi_\ell|_{T^-}).
\end{align}
For the local refinement indicators in~\eqref{eq:implementation:mutildeslp}, we
get
\begin{align}
\begin{split}
  \hhhalfweaktilde_\ell^2(T) &= \frac{h_T^2}2 ( \widehat\Phi_\ell|_{T^+} -
  \frac12 (\widehat\Phi_\ell|_{T^+} + \widehat\Phi_\ell|_{T^-}))^2  \\
  &\qquad  +\frac{h_T^2}2 ( \widehat\Phi_\ell|_{T^-} -
  \frac12 (\widehat\Phi_\ell|_{T^+} + \widehat\Phi_\ell|_{T^-}))^2 \\
  &= \frac{h_T^2}4 ( \widehat\Phi_\ell|_{T^+}-\widehat\Phi_\ell|_{T^-})^2.
\end{split}
\end{align}
From a practical point of view, the estimator $\hhhalfweaktilde_\ell$ is more
attractive than $\hhhalfweak_\ell$, since the computation involves only the Galerkin
solution on the fine mesh $\widehat\mesh_\ell$.

For $p=1$, we use the basis $\{\widehat\Psi_j\}_{j=1}^{4} \subseteq
\Pp^1(\localmesh)$ with
\begin{align}\label{eq:implementation:basisP1hat}
  \widehat\Psi_1|_{\el^+} &= 1, & \widehat\Psi_1|_{\el^-}&=0,\nonumber \\
  \widehat\Psi_2|_{\el^+} &= 0, & \widehat\Psi_1|_{\el^-}&=1, \\
  \widehat\Psi_3|_{\el^+} \circ F_{\el^+}(x) &= x, & \widehat\Psi_3|_{\el^-}&=0,\nonumber \\
  \widehat\Psi_4|_{\el^+} &=0, & \widehat\Psi_4|_{\el^-} \circ F_{\el^-}(x) &= x,\nonumber
\end{align}
for all $x\in\elref = (-1,1)$.
The local mass matrix $\massmathat_\el$ then reads
\begin{align}
  \massmathat_\el = \frac{h_\el}2 \begin{pmatrix}
    1 & 0 & 0 & 0 \\
    0 & 1 & 0 & 0 \\
    0 & 0 & \tfrac13 & 0 \\
    0 & 0 & 0 & \tfrac13
  \end{pmatrix},
\end{align}
and with $\widehat\Psi =\widehat\Phi_\ell|_\el - \Phi_\ell|_\el = \sum_{j=1}^4 \alpha_j \widehat\Psi_j$, we get
\begin{align}
  \hhhalfweak_\ell^2(\el) = \frac{h_\el}2 \norm{\widehat\Psi}{L_2(\el)}^2 = 
  \frac{h_\el^2}2 ( \alpha_1^2 + \alpha_2^2 + \frac13(\alpha_3^2+\alpha_4^2)).
\end{align}
For the computation of the local error indicator $\hhhalfweaktilde_\ell^2(\el)$,
we proceed as before and compute $\pi_\ell^1$. Let $\beta\in\R^4$ satisfy
$\widehat\Phi_\ell|_\el = \sum_{j=1}^4 \beta_j \widehat\Psi_j$. We end up with
\begin{align}
  \hhhalfweaktilde_\ell^2(\el) = \frac{h_\el^2}4 ( (\beta_1-\beta_2)^2 +
  \frac13(\beta_3-\beta_4)^2).
\end{align}

\subsubsection{2D hypersingular integral equation}
We consider the local refinement indicators of the estimator
\begin{align}\label{eq:implementation:hh2hyp}
  \hhhalfhyp_\ell^2 = \sum_{T\in\mesh_\ell} h_T \norm{\nabla_\Gamma (\widehat
  U_\ell -  U_\ell)}{L_2(T)}^2 =:
  \sum_{T\in\mesh_\ell} h_T \norm{\widehat \Psi}{L_2(T)}^2,
\end{align}
where $U_\ell\in\Sp^{p+1}(\mesh_\ell)$ resp. $\widehat
U_\ell\in\Sp^{p+1}(\widehat\mesh_\ell)$ are the respective Galerkin solutions. Hence, 
$\widehat \Psi := \nabla_\Gamma (\widehat U_\ell-U_\ell) \in \Pp^p(\widehat
\mesh_\ell)$.
To apply~\eqref{eq:implementation:localind}, it remains to provide a formula to
compute the arclength derivative $\nabla_\Gamma
(\widehat U_\ell-U_\ell)|_T \in \Pp^p(\localmesh)$.

Let $p=0$ and let $\{\eta_1,\eta_2,\eta_3\}$ denote a basis of the local subspace
$\Sp^{1}(\localmesh)$. We define the matrix $\localgradinPP_T \in\R^{2\times 3}$,
which represents the gradients of the basis $\{\eta_j\}$ with respect to the
basis $\{\widehat\Psi_1,\widehat\Psi_2\}$ of $\Pp^0(\localmesh)$, i.e.
\begin{align}\label{eq:implementation:defgradinPP}
  \nabla_\Gamma \eta_j = \sum_{k=1}^{2} (\localgradinPP_T)_{kj} \widehat\Psi_k.
\end{align}
Then, $\nabla_\Gamma (\widehat U_\ell-U_\ell) = \sum_j \beta_j \nabla_\Gamma \eta_j$, where $\alpha =
\localgradinPP_T \beta$.
Testing equation~\eqref{eq:implementation:defgradinPP} with $\widehat\Psi_j$ in
$L_2(T)$, we obtain the equivalent matrix
equation
\begin{align}
  \localgradproj_T = \massmathat_T \localgradinPP_T
\end{align}
with $(\localgradproj_T)_{jk} = \dual{\nabla_\Gamma \eta_k}{\widehat\Psi_j}_T$. Hence,
$\localgradinPP_T = (\massmathat_T)^{-1} \localgradproj_T$.
Together with~\eqref{eq:implementation:localind}, the computation of the local
indicators in~\eqref{eq:implementation:hh2hyp} with $\widehat\Psi = \nabla_\Gamma (\widehat U-U)$ reads
\begin{align}\label{eq:implementation:hh2hyp2}
\begin{split}
  h_T \norm{\widehat \Psi}{L_2(T)}^2 &= h_T \alpha^T \massmathat_T\alpha = h_T
  \beta^T \localgradinPP_T^T \massmathat_T \localgradinPP_T \beta \\
  &= h_T \beta^T (\massmathat_T^{-1} \localgradproj_T)^T \massmathat_T
  \massmathat_T^{-1}
  \localgradproj_T \beta \\
  &= h_T \beta^T \localgradproj_T^T \massmathat_T^{-1} \localgradproj_T \beta.
\end{split}
\end{align}
Let $x_1,x_2$ denote the endpoints
of the element $T\in\mesh_\ell$ and let $x_3 = (x_1+x_2)/2$ denote the midpoint of
$T$. We choose the nodal basis $\{\eta_j\}$ with respect to the nodes $\{x_j\}$.
Let $\widehat\Psi_1,\widehat\Psi_2$ denote the characteristic functions on $T^+,T^- \in \widehat
\mesh_\ell$ with $\overline T^+ \cup \overline T^- = \overline T$ and
$x_1\in\overline T^+$, $x_2\in\overline T^-$.
Then,
\begin{align*}
  \localgradproj_T = \begin{pmatrix}
    -1 & 0 & +1 \\
    0 & +1 & -1
  \end{pmatrix},
  \qquad 
  \massmathat_T = \frac{h_T}2 \begin{pmatrix}
    1 & 0 \\
    0 & 1
  \end{pmatrix},
\end{align*}
and~\eqref{eq:implementation:hh2hyp2} becomes with $\beta_j = (\widehat
U_\ell-U_\ell)(x_j)$
\begin{align}\begin{split}
  h_T \beta^T \localgradproj_T^T \massmathat_T^{-1} \localgradproj_T \beta
  &= 2 \beta^T \begin{pmatrix}
    1 & 0 & -1 \\
    0 & 1 & -1 \\
    -1 & -1 & 2
  \end{pmatrix}\beta\\
  &= 2(\beta_1-\beta_3)^2 + 2(\beta_2-\beta_3)^2.
\end{split}\end{align}
Using $U_\ell(x_3) = (U_\ell(x_1)+U(x_2))/2$, the indicator
$\hhhalfhyp_\ell(T)^2$ is computed by
\begin{align}
\begin{split}
  \hhhalfhyp_\ell(T)^2 &= 2(\widehat U_\ell(x_1)-\widehat U_\ell(x_3) +
  (U_\ell(x_2)-U_\ell(x_1))/2)^2 \\
  &\qquad + 2(\widehat U_\ell(x_2)-\widehat U_\ell(x_3) +
  (U_\ell(x_1)-U_\ell(x_2))/2)^2.
\end{split}
\end{align}
For the computation of the local error indicators
\begin{align}
  \hhhalfhyptilde_\ell^2(T) = h_T \norm{(1-\pi_\ell)\nabla_\Gamma \widehat
  U_\ell}{L_2(T)}^2  =: h_T \norm{\widehat\Psi}{L_2(T)}^2,
\end{align}
we proceed as in Section~\ref{section:implementation:hh2:weak} to compute the
$L_2$-projection
\begin{align}
  \pi_\ell \nabla_\Gamma \widehat U_\ell = \frac12 \Big( (\nabla_\Gamma
  \widehat U_\ell)|_{\el^+} + (\nabla_\Gamma
  \widehat U_\ell)|_{\el^-} \Big).
\end{align}
With $\alpha$ resp. $\beta$ given by $\widehat\Psi = \sum_{j=1}^2 \alpha_j
\widehat\Psi_j$ resp. $\widehat U_\ell = \sum_{j=1}^3 \beta_j \eta_j$, a
straightforward computation shows that
\begin{align}
  \alpha = \frac12 \begin{pmatrix}
    1 & -1 \\
    -1 & 1
  \end{pmatrix}
  \localgradinPP_\el \beta,
\end{align}
Putting this into~\eqref{eq:implementation:localind}, we get
\begin{align}
  h_T \norm{\widehat\Psi}{L_2(T)}^2 = (2\beta_3-\beta_1-\beta_2)^2.
\end{align}
Using $\beta_j = \widehat U_\ell(x_j)$, the local refinement
indicators become
\begin{align}
  \hhhalfhyptilde_\ell^2(T) = ( 2 \widehat U_\ell(x_3) -\widehat U_\ell(x_1) -\widehat U_\ell(x_2))^2.
\end{align}
Again, the computation of $\hhhalfhyptilde_\ell$ is more attractive compared to
the computation $\hhhalfhyp_\ell$, since only the solution $\widehat U_\ell$ on
the fine mesh is needed.

Let $p=1$ and let $\{\widehat\Psi_j\}_{j=1}^4$ be given as
in~\eqref{eq:implementation:basisP1hat}.
We choose the basis $\{\eta_j\}_{j=1}^5$, where $\eta_1,\eta_2,\eta_3$ are the
(linear) nodal basis functions with respect to $x_1,x_2,x_3$, and $\eta_4 :=
\eta_1\eta_3$ as well as $\eta_5:=\eta_2\eta_3$.
Arguing as for $p=0$, we obtain
\begin{align}
  \localgradproj_\el = \begin{pmatrix}
    -1 & 0 & 1 & 0 & 0 \\
    0 & 1 & -1 & 0 & 0 \\
    0 & 0 & 0 & -\tfrac13 & 0 \\
    0 & 0 & 0 & 0 & -\tfrac13 
  \end{pmatrix}.
\end{align}
Using~\eqref{eq:implementation:hh2hyp2} with $\widehat\Psi := \nabla_\Gamma(\widehat
U_\ell-U_\ell)|_\el =\sum_{j=1}^5 \beta_j \nabla_\Gamma \eta_j$, we infer 
\begin{align}
  \hhhalfhyp_\ell^2(\el) = 2( (\beta_1-\beta_3)^2 + (\beta_2-\beta_3)^2 +
  \frac13(\beta_4^2+\beta_5^2)).
\end{align}
For the computation of the local error indicators of $\hhhalfhyptilde_\ell$, let
$\beta\in\R^5$ be given by $\widehat U_\ell = \sum_{j=1}^5 \beta_j \eta_j$. We
proceed as in the case $p=0$ and get
\begin{align}
  \hhhalfhyptilde_\ell^2(\el) = (2\beta_3-\beta_1-\beta_2)^2 +
  \frac13(\beta_4-\beta_5)^2.
\end{align}

Similar results hold for $d=3$ resp. $p>1$.

\subsection{Weighted residual error estimator}

\subsubsection{2D weakly singular integral equation}
We consider the computation of the local error indicators 
\begin{align}\label{eq:implementation:defestres}
  \estres_\ell^2(\el) = h_\el \norm{\nabla_\Gamma(\slo\Phi_\ell-f)}{L_2(\el)}^2
  =: h_\el
  \norm{\nabla_\Gamma R_\ell}{L_2(\el)}^2
\end{align}
of the weighted-residual error estimator for $d=2$.
By the definition of $\nabla_\Gamma(\cdot)$, we have
\begin{align}
  (\nabla_\Gamma R_\ell) \circ F_\el = 2h_\el^{-1} (\nabla_\Gamma(R_\ell\circ F_\el)),
\end{align}
where $F_\el$ denotes the affine mapping from $\elref = (-1,1)$ to $\el = {\rm
conv}\{x_1,x_2\}$.
For~\eqref{eq:implementation:defestres}, this yields
\begin{align}
  \begin{split}
  h_\el \norm{\nabla_\Gamma R_\ell}{L_2(\el)}^2 &= h_\el \int_\el (\nabla_\Gamma
  R_\ell)^2 \,dx\\
  &=2\int_{\elref} ( \nabla_\Gamma(R_\ell\circ F_\el))^2 \,dy.
  \end{split}
\end{align}
Recall that $\nabla_\Gamma (\cdot) = (\cdot)'$, where $(\cdot)'$ denotes the
arclength derivative.
We approximate $R_\ell\circ F_\el$ by some polynomial $\Psi\in\Pp^{2q}(\elref)$ with
$q\geq 1$, i.e.
\begin{align}\label{eq:implementation:resapprox1}
  \int_{\elref} (\nabla_\Gamma(R_\ell\circ F_\el))^2 \,dy \approx \int_{\elref}
  (\Psi')^2 \,dy.
\end{align}
Note hat $\Psi' \in\Pp^{2q-1}(\elref)$, whence $(\Psi')^2 \in \Pp^{4q-2}$.
To compute the integral on the right-hand side of~\eqref{eq:implementation:resapprox1},
we use a $2q$-point Gaussian quadrature rule, which is exact for polynomials of
order $2(2q)-1$, and thus for $(\Psi')^2$.
Let $\{y_i\}_{i=1}^{2q}$ denote the quadrature nodes on $\elref$ with
corresponding weights\linebreak
$\{w_i\}_{i=1}^{2q}$. This leads for the local error indicators
\begin{align}\label{eq:implementation:resapprox2}
  h_T \norm{\nabla_\Gamma R_\ell}{L_2(T)}^2 \approx 2 \int_{\elref} (\Psi')^2 \,dy =
  2\sum_{i=1}^{2q} w_i( \Psi'(y_i))^2.
\end{align}
For the construction of $\Psi\in\Pp^{2q}(\elref)$ by interpolation, we use the
$2q$ points $\{y_i\}_{i=1}^{2q}$ from the Gaussian quadrature rule plus the
midpoint of the reference element $y_{2q+1} = 0$.
According to~\eqref{eq:implementation:resapprox2}, we need to evaluate $\Psi'$
at the points $y_i$ for $i=1,\dots,2q$. 
To that end, let $\{L_i\}_{i=1}^{2q+1}$ denote the Lagrange basis with respect
to the points $\{y_i\}_{i=1}^{2q+1}$.
It holds 
\begin{align*}
  \Psi = \sum_{k=1}^{2q+1} \beta_k L_k
\end{align*}
with $\beta_k := \Psi(y_k)$.
Define the matrix $\lagrangemat' \in \R^{2q\times (2q+1)}$ by
\begin{align}
  (\lagrangemat')_{jk} := L'_k(y_j).
\end{align}
Then, $\alpha_j := \Psi'(y_j)$ can be obtained by the matrix-vector
multiplication
\begin{align}
  \alpha = \lagrangemat' \beta.
\end{align}
We consider the lowest-order case $p=0$, where we use a 2-point Gaussian
quadrature rule ($q=1$) on $\elref$ with $y_1 = -1/\sqrt{3}$, $y_2 = 1/\sqrt{3}$
and weights $w_1=1=w_2$. The derivatives of the Lagrange basis $L_1,L_2,L_3$ with respect to the
points $y_1,y_2,y_3:=0$ are given by
\begin{align}
  L_1'(y) &= 3y-\frac{\sqrt{3}}2,  & L_2'(y) &= 3y+\frac{\sqrt{3}}2, &
  L_3'&=1-3y^2.
\end{align}
For $\alpha_1 = \Psi'(y_1), \alpha_2 = \Psi'(y_2)$, we get
\begin{align}
\begin{split}
  \alpha &= \begin{pmatrix}
    L_1'(y_1) & L_2'(y_1) & L_3'(y_1) \\
    L_1'(y_2) & L_2'(y_2) & L_3'(y_2)
  \end{pmatrix} \beta \\
  &= \sqrt{\frac34}\begin{pmatrix}
    -3 & -1 & 4 \\
    1 & 3 & -4
  \end{pmatrix}\beta.
\end{split}
\end{align}
Altogether, we approximate the local error indicators $\estres_\ell^2(T)$ by
\begin{align}
  \rho_\ell^2(T) \approx \sum_{i=1}^2 w_i \alpha_i^2 = \frac34 \beta^T
  \begin{pmatrix}
    10 & 6 & -16 \\
    6 & 10 & -16 \\
    -16 & -16 & 32
  \end{pmatrix}
  \beta,
\end{align}
where $\beta_j := (R_\ell\circ F_T)(y_j)$ for $j=1,2,3$.
The advantage of this approach is that we do not need to evaluate the
arclength derivative $R_\ell'$ of the function $R_\ell$ numerically, but instead evaluate the function
at specific points directly.

In practice, it suffices to use a small number of quadrature points $q$. For
$p=0$, we stress that the choice $q=1$ is sufficient. 
Define $\estres_{\ell}^{(q)}$ as the residual error estimator
$\estres_\ell$ with the difference that the local indicators $\rho_\ell(T)^2$ are approximated as
given above, and $2q$ denotes the number of points in the chosen Gaussian
quadrature rule.
On uniformly refined meshes, we stress that for a sufficiently smooth
  function $R_\ell$ it holds
\begin{align}
  |\rho_\ell-\rho_\ell^{(q)}| \lesssim N^{-(q+1/2)} \quad\text{with }
  N=\#\mesh_\ell.
\end{align}
Thus for $p=1$, we use $q=2$, i.e., four quadrature points $y_1,y_2,y_3,y_4$
and the additional evaluation point $y_5:=0$.

\subsubsection{2D hypersingular integral equation}
We consider the computation of the local refinement indicators 
\begin{align}
  \estres_\ell^2(\el) = h_\el \norm{\hyp U_\ell - (1/2-\dlo')\phi}{L_2(\el)}^2 =:
  h_\el \norm{R_\ell}{L_2(\el)}^2
\end{align}
of the weighted residual error estimator $\estres_\ell$ for $d=2$.
Let $F_\el$ denote the affine mapping from $\elref = (-1,1)$ to $\el$.
We use a $q$-point Gaussian quadrature rule with nodes $\{y_i\}_{i=1}^q$ and
weights $\{w_i\}_{i=1}^q$ on the reference element $\elref$,
which is exact for polynomials in $\Pp^{2q-1}(\elref)$.
The local refinement indicator $\estres_\ell^2(\el)$ is approximated by
\begin{align}
  \begin{split}
  h_\el \norm{R_\ell}{L_2(\el)}^2 &= \frac{h_\el}2 \int_{\elref} ( R_\ell \circ
  F_\el)^2 \, dx\\
  &\approx \frac{h_\el}2  \sum_{i=1}^q w_i (R_\ell\circ
  F_\el)^2(y_i).
  \end{split}
\end{align}
Define $\estres_\ell^{(q)}$ as $\estres_\ell$ with the difference that the local
refinement indicators $\estres_\ell^2(\el)$ are approximated as given above. On
uniformly refined meshes and for sufficiently smooth residual $R_\ell$, we
stress that
\begin{align}
  |\estres_\ell-\estres_\ell^{(q)}| \lesssim N^{-(2q+1/2)} \quad\text{with }
  N=\#\,\mesh_\ell.
\end{align}
For the lowest order case $p=0$, it is sufficient to use $q=2$ quadrature
points, whereas for $p=1$ it is sufficient to use $q=3$ quadrature points.
\section{Conclusion}\label{section:conclusion}
In this work, we presented all a~posteriori error estimators for Galerkin BEM that are available in the
mathematical literature. Up to now, it is known how to
apply contemporary convergence and optimality analysis
only to several of them, and an overview of this analysis and its application to
these estimators was given.

Although all estimators behave well in numerical experiments
(cf. the references given above, where the respective estimators haven been introduced),
they differ with respect to
overhead in implementation,
computational expense, and mathematically guaranteed convergence as well as optimality of the
related ABEM algorithms.
The choice is also affected by the regularity of the right-hand side data.

With respect to these requirements, a short summary of the advantages and disadvantages of
the estimators follows. We consider only estimators that have been mathematically analyzed on
locally refined meshes.

The $\bm{(h-h/2)}$\textbf{-type estimators} (Section~\ref{section:est:hh2})
are, in general, structurally easy. They require
nearly no overhead for their implementation and can be used
to steer an\-isotropic mesh refinement in a straightforward manner
(Section~\ref{section:estred:anisotropic}), which are their biggest
advantages. Also, there is no additional requirement on the data.
They are always efficient, but reliability is equivalent to the saturation assumption,
which is still an open problem. They are proven to converge, but the convergence is not
proven to be optimal.

\textbf{Averaging estimators} (Section~\ref{section:averaging})
are globally equivalent to $(h-h/2)$-type estimators, and the implementationally and
computationally interesting variants are even locally equivalent to their $(h-h/2)$ counterparts.

Advantages of the $\bm{ZZ}$\textbf{-type estimators} (Section~\ref{section:zzest})
are that the implementation is also very easy,
and as they avoid artificial mesh refinement, they are much cheaper than $(h-h/2)$-type estimators.
While reliability relies on an appropriate saturation assumption, efficiency involves
higher-order terms and is therefore weaker than for $(h-h/2)$-type estimators.
In addition, it is not clear how to steer anisotropic refinement.
Convergence is proven, but optimality remains open.

The \textbf{two-level estimators} (Section~\ref{section:est:2level})
also need the assembly of the Galerkin data on a finer mesh.
Moreover, their reliability depends also on the saturation assumption. 
Up to now, nothing is known about convergence or optimality.
However, their great strength is that they are analyzed
with respect to $hp$-methods and that they are reliable on anisotropic refinement in the case
of weakly singular integral operators (assuming saturation).

The \textbf{weighted residual estimators} (Section~\ref{section:est:wres})
are the only ones which are proven to
converge with optimal rates. As they are also reliable, they are preferred in theory.
However, their disadvantage is that their requirements on the data
is quite strong and that the implementation requires a certain amount of overhead.
Anisotropic meshes can be used, but there is no analysis in this respect.

The \textbf{local double norm estimators} (Section~\ref{section:est:faermann})
are the only estimators that are
known to be efficient \textit{and} reliable without any further assumption on data
or saturation. However, their stable implementation is non-trivial,
and nothing is known with respect to convergence or optimality.

\textbf{Approximation of the given right-hand side} data renders an important aspect for BEM
implementations and can be used and mathematically controlled in adaptive boundary element
methods on isotropic meshes.
Convergence and optimality is proven as long as the data satisfy additional regularity and
appropriate approximation operators are used. There is no analysis on anisotropic meshes.

Several important \textbf{open problems} in this field center around
anisotropic mesh refinement.
There is a need for estimators which are reliable on anisotropic
meshes (two-level estimators for weakly singular operators are
reliable on an\-isotropic meshes under the saturation assumption),
and optimality theory needs to be extended in this respect.
To that end, suitable mesh refinement strategies need to be developed and analyzed.
Furthermore, the approximation of
smooth geometries in isoparametric BEM algorithms,
its incorporation into ABEM, as well as a mathematical analysis with respect to
convergence and optimality remains an interesting open problem.
Also, as stated in the introduction, competitive BEM algorithms need to employ fast methods
for matrix compression. Hence, it is an inevitable task to control the introduced error
in adaptive BEM algorithms.
\bibliographystyle{plain}
\bibliography{literature}

\end{document}